\documentclass[a4paper,12pt]{article}

\usepackage[T1]{fontenc}
\usepackage[utf8]{inputenc}
\usepackage{graphicx}

\usepackage{amsmath}
\usepackage{amsfonts}
\usepackage{amssymb}
\usepackage{amsthm}
\usepackage{array}
\usepackage{color}
\usepackage{enumerate}
\usepackage{tikz}
\usetikzlibrary{arrows.meta} 
\usepackage{appendix} 
\usepackage[nottoc,notlot,notlof]{tocbibind} 

\newtheorem{theorem}{Theorem}[section]
\newtheorem{conjecture}[theorem]{Conjecture}
\newtheorem{definition}[theorem]{Definition}
\newtheorem{proposition}[theorem]{Proposition}
\newtheorem{lemma}[theorem]{Lemma}
\newtheorem{corollary}[theorem]{Corollary}

\newtheorem{remark}[theorem]{Remark}
\newtheorem{example}[theorem]{Example}
\newtheorem{hypothesis}[theorem]{Hypothesis}

\numberwithin{equation}{section}

\newcommand{\N}{\mathbb{N}}
\newcommand{\Z}{\mathbb{Z}}

\newcommand{\R}{\mathbb{R}}
\newcommand{\C}{\mathbb{C}}

\newcommand{\Ebb}{\mathbb{E}}

\newcommand{\Lbb}{\mathbb{L}}

\newcommand{\Pbb}{\mathbb{P}}

\newcommand{\Sbb}{\mathbb{S}}
\newcommand{\Tbb}{\mathbb{T}}
\newcommand{\Ubb}{\mathbb{U}}

\newcommand{\Acal}{\mathcal{A}}
\newcommand{\Bcal}{\mathcal{B}}

\newcommand{\Fcal}{\mathcal{F}}

\newcommand{\Kcal}{\mathcal{K}}
\newcommand{\Lcal}{\mathcal{L}}
\newcommand{\Mcal}{\mathcal{M}}
\newcommand{\Ncal}{\mathcal{N}}

\newcommand{\Pcal}{\mathcal{P}}

\newcommand{\norm}[2]{\left\| #1 \right\|_{#2}}
\newcommand{\dd}{\;{\rm d}}

\newcommand{\transposee}[1]{{\vphantom{#1}}^{\mathit t}{#1}}
\newcommand{\acts}{\curvearrowright}
\DeclareMathOperator{\ad}{ad}
\DeclareMathOperator{\Cov}{Cov}
\DeclareMathOperator{\ess}{ess}

\DeclareMathOperator{\Id}{Id}
\DeclareMathOperator{\Jac}{Jac}
\DeclareMathOperator{\Ker}{Ker}

\DeclareMathOperator{\Leb}{Leb}

\DeclareMathOperator{\Lip}{Lip}

\DeclareMathOperator{\Supp}{Supp}

\DeclareMathOperator{\sgn}{sgn}

\DeclareMathOperator{\Card}{Card}
\DeclareMathOperator{\Sp}{Sp}
\DeclareMathOperator{\Var}{Var}
\DeclareMathOperator{\Vect}{Vect}

\DeclareFontFamily{U}{mathx}{}
\DeclareFontShape{U}{mathx}{m}{n}{ <-> mathx10 }{}
\DeclareSymbolFont{mathx}{U}{mathx}{m}{n}
\DeclareFontSubstitution{U}{mathx}{m}{n}
\DeclareMathAccent{\widecheck}{0}{mathx}{"71}

\usepackage[top=3cm, bottom=3cm, left=1.5cm, right=1.5cm]{geometry}

\title{Potential theory and $\Z^d$-extensions}
\author{Damien THOMINE}
\date{}

\begin{document}

\maketitle

\begin{abstract}

We study hitting probabilities for $\Z^d$-extensions of Gibbs-Markov maps. The goal is to estimate, given a finite $\Sigma \subset \Z^d$ and $p$, $q \in \Sigma$, the probability $P_{pq}$ that the process starting from $p$ returns to $\Sigma$ at site $q$.

Our study generalizes the methods available for random walks. We are able to give in many settings (square integrable jumps, jumps in the basin of a L\'evy or Cauchy random variable) asymptotics for the transition matrix $(P_{pq})_{p, q \in \Sigma}$ when the elements of $\Sigma$ are far apart.

We use three main tools: a variant of the balayage identity using a transfer operator as a Markov transition kernel, a study inspired from fast-slow systems and the hitting time of small sets in hyperbolic systems to relate transfer operators and the transition matrices we seek to compute, and finally Fourier transform and perturbations of transfer operators \textit{\`a la} Nagaev-Guivarc'h to effectively compute these transition matrices in an asymptotic regime.
\end{abstract}

\setcounter{tocdepth}{2}
\tableofcontents

\section*{Introduction}
\addcontentsline{toc}{section}{\protect\numberline{}Introduction}

Large classes of uniformly hyperbolic dynamical systems, when endowed with suitable probability measures, behave similarly to Markov chains. Many theorems pertaining to sequences of i.i.d.\ random variables have found generalizations for sequences of observables of such systems: the law of large numbers (Birkhoff's ergodic theorem), the central limit theorem~\cite{Nagaev:1957, Sinai:1960, Nagaev:1961, GuivarchHardy:1988}, the local central limit theorem and asymptotics of occupation times~\cite{DarlingKac:1957, PollicottSharp:1994a, Aaronson:1997, SzaszVarju:2004, DolgopyatNandori:2020} and large deviation principles~\cite{OreyPelikan:1988, OreyPelikan:1989, Young:1990}, to cite only the most widely known.

\smallskip

This article is part of a project to understand how an important part of the general theory of Markov processes, namely the probabilistic theory of potentials, can be adapted to dynamical systems. The specific problem we are interested in is inspired by the computation of \emph{hitting probabilities for random walks}. Given a recurrent random walk $(S_n)_{n \geq 0}$ on $\Z^d$ and a finite set $\Sigma \subset \Z^d$, one can compute the transition probabilities between points of $\Sigma$ -- in other words, the probability, given any $p$, $q \in \Sigma$, that a random walk starting from point $p$ returns to $\Sigma$ at point $q$.

\begin{figure}[!h]
\centering
\scalebox{0.8}{
\begin{tikzpicture}
\foreach \x in {-2,...,2}{
    \draw [dotted] (\x,-2.5) -- (\x,2.5) ;
    }
\foreach \y in {-2,...,2}{
    \draw [dotted] (-2.5,\y) -- (2.5,\y) ;
    }
\foreach \x in {-2,...,2}{
    \foreach \y in {-2,...,2}{
        \node at (\x,\y) [circle,fill,inner sep=1.5pt]{} ;
        }
    }
\draw node[below left] (A) at (0,0) {$A$};
\draw node[below left] (B) at (-1,0) {$B$};
\draw node[above right] (C) at (1,1) {$C$};
\draw [-{Latex[length=2mm,width=2mm]}, blue] (0,0) -- (0,-1) -- (1,-1) -- (1,0) -- (1,1) ;
\end{tikzpicture}
}
\caption{Three sites $A$, $B$, $C$ in $\Z^2$, and a transition from $A$ to $C$.}
\label{fig:MarcheAleatoireZ2}
\end{figure}
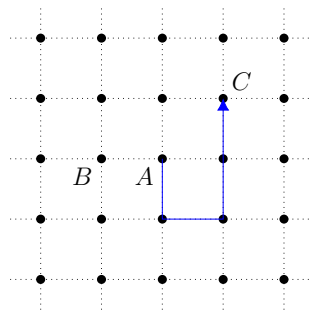

This problem has an interpretation in terms of resistances in electrical networks, and may be solved using potential theory, more precisely the so-called \emph{balayage identity} (see e.g.~\cite[Corollary~1.11]{Revuz:1975}, or 
Theorem~\ref{thm:IdentiteBalayage} below). With this method, for the simple random walk on $\mathbb{Z}^2$ and the subset $\Sigma = \{A, B, C\}$ of Figure~\ref{fig:MarcheAleatoireZ2}, one finds that the matrix of transition probabilities is

\begin{equation}
\label{eq:MatriceTransitionMASimple}
 P_{\{A, B, C\}} 
 = \frac{1}{-\pi^2 +8\pi-4} \left(
    \begin{array}{ccc}
        -\frac{1}{2}\pi^2+4\pi-4 & -\frac{1}{2}\pi^2+3\pi & \pi \\   
        -\frac{1}{2}\pi^2+3\pi & -\pi^2+6\pi-4 & \frac{1}{2}\pi^2-\pi \\
        \pi & \frac{1}{2}\pi^2-\pi & -\frac{3}{2}\pi^2+8\pi-4 \\
    \end{array}
\right).
\end{equation}

A natural generalization of random walks is the class of $\Z^d$-extensions of dynamical systems. Starting from a probability-preserving ergodic dynamical system $(A, \mu, T)$ and a jump function $F : A \to \Z^d$, its $\Z^d$-extensions is the system $([\Z^d], \widetilde{\mu}, \widetilde{T})$ defined by:
\begin{itemize}
 \item State space: $[\Z^d] := A \times \Z^d$;
 \item Measure: $\widetilde{\mu} := \mu \otimes \sum_{p \in \Z^d} \delta_p$;
 \item Transformation: $\widetilde{T} (x,p) := (T(x), p+F(x)) \in [\Z^d]$ for all $(x,p) \in [\Z^d]$.
\end{itemize}
This class of systems include $\Z^d$-valued random walks, $\Z^d$-extensions of Markov chains~\cite{KramliSzasz:1983, KramliSzasz:1984, KramliSimanyiSzasz:1986} (sometimes used as toy models for more realistic processes), $\Z^d$-extensions of subshifts of finite type~\cite{PollicottSharp:1994a} as well as the collision map for Lorentz gases~\cite{SzaszVarju:2004, SzaszVarju:2007}. Using well-chosen sections, the study of the Lorentz gas flow or of the geodesic flow on Abelian covers of negatively curved compact manifolds may also be reduced to the study of such extensions~\cite{PollicottSharp:1994b, DolgopyatNandori:2020}.

\smallskip

In the setting of $\Z^d$-extensions, the study of hitting probabilities can be transposed as follows. Given a finite subset $\Sigma \subset \Z^d$ and $p$, $q \in \Sigma$, what is the probability (under $\mu$) that the trajectory $(\widetilde{T}^n (x, p))_{n \geq 0}$ returns to $A \times \Sigma$ at the site $A \times \{q\}$?

\smallskip

While an explicit computation of the matrix of transition probabilities, as in Equation~\eqref{eq:MatriceTransitionMASimple}, was possible for the simple random walk, such an explicit formula seems to be out of reach for more general dynamical systems. As for the central limit theorem, the local central limit theorem or the principle of large deviations, elementary formulas are available in an asymptotic regime. For hitting probabilities, we propose to study the asymptotics of the transition probabilities for families $(\Sigma_t)_{t>0}$ of subsets of $\Z^d$ whose points get farther and farther apart when the parameter $t$ goes to infinity:

\begin{figure}[!h]
\centering

\scalebox{0.18}{

\begin{tikzpicture}
   \foreach \m in {-10,-8,-6,-4,-2,0,2,4,6,8,10} 
      \foreach \n in {-11,-9,-7,-5,-3,-1,1,3,5,7,9,11}
         \draw [fill = black, opacity=0.2] (\n,\m) circle (0.8) ;
   \foreach \m in {-11,-9,-7,-5,-3,-1,1,3,5,7,9,11} 
      \foreach \n in {-10,-8,-6,-4,-2,0,2,4,6,8,10}
         \draw [fill = black, opacity=0.2] (\n,\m) circle (0.5) ;
    \fill [blue!100!] (1,0) circle (0.8) ;
    \fill [blue!100!] (2,1) circle (0.5) ;
    \fill [blue!80!] (-1,0) circle (0.8) ;
    \fill [blue!80!] (0,1) circle (0.5) ;
    \fill [blue!60!] (3,2) circle (0.8) ;
    \fill [blue!60!] (4,3) circle (0.5) ;
    \draw (1,0) -- (-1,0) -- (3,2) -- (1,0) ;
\end{tikzpicture}

\hspace{0.5cm}

\begin{tikzpicture}
   \foreach \m in {-10,-8,-6,-4,-2,0,2,4,6,8,10} 
      \foreach \n in {-11,-9,-7,-5,-3,-1,1,3,5,7,9,11}
         \draw [fill = black, opacity=0.2] (\n,\m) circle (0.8) ;
   \foreach \m in {-11,-9,-7,-5,-3,-1,1,3,5,7,9,11} 
      \foreach \n in {-10,-8,-6,-4,-2,0,2,4,6,8,10}
         \draw [fill = black, opacity=0.2] (\n,\m) circle (0.5) ;
    \fill [blue!100!] (1,0) circle (0.8) ;
    \fill [blue!100!] (2,1) circle (0.5) ;
    \fill [blue!80!] (-3,0) circle (0.8) ;
    \fill [blue!80!] (-2,1) circle (0.5) ;
    \fill [blue!60!] (5,4) circle (0.8) ;
    \fill [blue!60!] (6,5) circle (0.5) ;
    \draw (1,0) -- (-3,0) -- (5,4) -- (1,0) ;
\end{tikzpicture}

\hspace{0.5cm}

\begin{tikzpicture}
   \foreach \m in {-10,-8,-6,-4,-2,0,2,4,6,8,10} 
      \foreach \n in {-11,-9,-7,-5,-3,-1,1,3,5,7,9,11}
         \draw [fill = black, opacity=0.2] (\n,\m) circle (0.8) ;
   \foreach \m in {-11,-9,-7,-5,-3,-1,1,3,5,7,9,11} 
      \foreach \n in {-10,-8,-6,-4,-2,0,2,4,6,8,10}
         \draw [fill = black, opacity=0.2] (\n,\m) circle (0.5) ;
    \fill [blue!100!] (1,0) circle (0.8) ;
    \fill [blue!100!] (2,1) circle (0.5) ;
    \fill [blue!80!] (-5,0) circle (0.8) ;
    \fill [blue!80!] (-4,1) circle (0.5) ;
    \fill [blue!60!] (7,6) circle (0.8) ;
    \fill [blue!60!] (8,7) circle (0.5) ;
    \draw (1,0) -- (-5,0) -- (7,6) -- (1,0) ;
\end{tikzpicture}

\hspace{0.5cm}

\begin{tikzpicture}
   \foreach \m in {-10,-8,-6,-4,-2,0,2,4,6,8,10} 
      \foreach \n in {-11,-9,-7,-5,-3,-1,1,3,5,7,9,11}
         \draw [fill = black, opacity=0.2] (\n,\m) circle (0.8) ;
   \foreach \m in {-11,-9,-7,-5,-3,-1,1,3,5,7,9,11} 
      \foreach \n in {-10,-8,-6,-4,-2,0,2,4,6,8,10}
         \draw [fill = black, opacity=0.2] (\n,\m) circle (0.5) ;
    \fill [blue!100!] (1,0) circle (0.8) ;
    \fill [blue!100!] (2,1) circle (0.5) ;
    \fill [blue!80!] (-7,0) circle (0.8) ;
    \fill [blue!80!] (-6,1) circle (0.5) ;
    \fill [blue!60!] (9,8) circle (0.8) ;
    \fill [blue!60!] (10,9) circle (0.5) ;
    \draw (1,0) -- (-7,0) -- (9,8) -- (1,0) ;
\end{tikzpicture}
}
\caption{Some elements of a family of subset $(A \times \Sigma_t)_{t > 0}$ growing farther apart while keeping the same shape.}
\label{fig:FamilleFormes}
\end{figure}
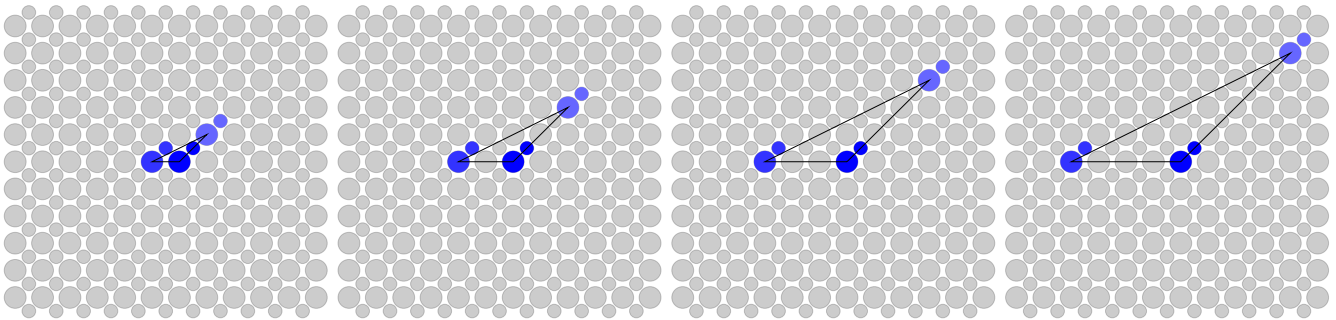

One particularly interesting case, pictured in Figure~\ref{fig:FamilleFormes}, occurs when those points keep the same shape, that is
\begin{equation*}
 \Sigma_t 
 = \left\{ t\sigma(i) + o(t) : \ i \in I\right\}
\end{equation*}
for some finite set $I$ and some injective function $\sigma : I \to \R^d$. This kind of limit theorems is in line with a few former studies~\cite{KramliSzasz:1984, PeneThomine:2019}. 

\smallskip

In this setting, hitting probabilities could be obtained in~\cite{PeneThomine:2019} for families of subsets of size $2$. However, these results were byproducts of distributional limit theorems for additive functionals of the process $(\widetilde{T}^n)_{n \geq 0}$. As a consequence, the methods involved were extremely heavy, limited in scope as they could not produce the matrix of transition probabilities but only its symmetrization, and obscured the underlying potential theory. A second article~\cite{PeneThomine:2020} laid down the foundation for the use of probabilistic potential theory for hyperbolic dynamical systems; in particular, we need to use the transfer operator as a Markov operator instead of the Koopman operator. 

\smallskip

The current contribution has two main parts. First, parametrizing by $\varepsilon = t^{-1}$ instead of $t$, we define \emph{potential operators} $(Q_\varepsilon)_{\varepsilon > 0}$ acting on Lipschitz functions on $A \times \Sigma_\varepsilon$. We are able to relate these potential operators to the matrices of transition probabilities when $\varepsilon$ vanishes:

\begin{theorem}\quad
\label{thm:MainTheorem}
 
 Let $([\Z^d], \widetilde{\mu}, \widetilde{T})$ be an ergodic and recurrent Markov $\Z^d$-extension of a Gibbs-Markov map. Let $I$ be finite. For $\varepsilon >0$, let $\sigma_\varepsilon : I \hookrightarrow \Z^d$. Let $P_\varepsilon$ be the matrix of transition probabilities on $\Sigma_\varepsilon$.
 
 \smallskip
 
 The following properties are equivalent:
 \begin{itemize}
  \item there exists an irreducible (in the sense of Definition~\ref{def:LMatrice}) bi-L-matrix $R$ such that $P_\varepsilon = \Id-\varepsilon R+o(\varepsilon)$;
  \item there exists an irreducible (in the sense of Definition~\ref{def:QMatriceIrreductible}) matrix $S$ such that $Q_\varepsilon = \varepsilon^{-1} S + o(\varepsilon^{-1})$.
 \end{itemize}
 If any of these properties holds, then, in addition, $S = R_0^{-1}$.
\end{theorem}

One of the main features of this theorem is that its hypotheses are quite general: recurrence is necessary for $P_\varepsilon$ to be a stochastic matrix, while ergodicity ensures (among other things) that $P_\varepsilon$ is irreducible. No assumption is made on the shape of the subsets $(\Sigma_\varepsilon)_{\varepsilon >0}$ nor on the jumps $F$. Note that, for random walks, we have $Q_\varepsilon = (\Id-P_\varepsilon)^{-1}$, so that the conclusion of Theorem~\ref{thm:MainTheorem} then holds even in the non-asymptotic regime.

\smallskip

The proof of Theorem~\ref{thm:MainTheorem} relies on the theory of \emph{hitting time of small targets}~\cite{Pitskel:1991, Hirata:1993, ChernovDolgopyat:2009, BalintGilbertNandoriSzaszToth:2017}: when $\varepsilon$ vanishes, the transitions between different sites of $A \times \Sigma_\varepsilon$ become rare events\footnote{This point of view explains the use of the parameter $\varepsilon$, which is the size of the targets, instead of $t$, for most of this article.}, so the process induced by $(\widetilde{T}^n)_{n \geq 0}$ looks like a Markov process on $\Sigma_\varepsilon$ with transitions at times of order $\varepsilon^{-1}$. We give a probabilistic proof of Theorem~\ref{thm:MainTheorem}, inspired by previous works on fast-slow  systems~\cite{ChernovDolgopyat:2009, DeSimoiLiverani:2015} and using a coupling argument to control these transitions. Note however, that a more analytic approach could be possible, along the line adopted by G.~Keller and C.~Liverani~\cite{KellerLiverani:2009a}, or by D.~Dolgopyat and P.~Wright~\cite{DolgopyatWright:2012} (the later work being closer to our setting, as the operators involved have a degeneracy of higher order than those studied in~\cite{KellerLiverani:2009a}).

\smallskip

The potential operators $(Q_\varepsilon)_{\varepsilon>0}$ can be estimated, in the asymptotic regime $\varepsilon \to 0$, from the perturbations of the main eigenvalue of the transfer operator associated with $(A, \mu, T)$. When the jump function $F$ is well-behaved, we get explicit expressions for the operator $S$ in Theorem~\ref{thm:MainTheorem}, and from there, a first order expansion of the transition matrices $(P_\varepsilon)_{\varepsilon >0}$. We investigate four cases, when $d=1$ and $F$ is in the basin of attraction of a L\'evy stable distribution (distinguishing the $\Lbb^2$, general L\'evy and Cauchy cases), and when $d=2$ and $F$ is square integrable. For instance, in the later case, we get

\begin{proposition}\quad
\label{prop:ApplicationD2CarreIntegrable}
 
 Let $([\Z^2], \widetilde{\mu}, \widetilde{T})$ be an ergodic and recurrent Markov $\Z^2$-extension of a Gibbs-Markov map $(A, \mu, T)$ with jump function $F : A \to \Z^2$. Assume furthermore that $\norm{F}{} \in \Lbb^2 (A, \mu)$. Let $\Cov > 0$ be that asymptotic covariance matrix of $F$ defined by Equation~\eqref{eq:D2GreenKubo}.

\smallskip

Let $I$ be non-empty and finite and $\sigma : I \hookrightarrow \R^2$ be injective. For $t > 0$, let $\sigma_t : I \hookrightarrow \Z^2$ be such that $\sigma_t (i) = t \sigma(i) + o(t)$ for all $i \in I$. Let $P_t$ be the transition matrix of the Markov chain induced on $\Sigma_t := \sigma_t (I)$. Then
\begin{equation*}
 P_t 
 =_{t \to + \infty} \Id - \frac{\pi \sqrt{\det(\Cov)}}{\ln(t)}  R + o (\ln(t)^{-1}),
\end{equation*}
where $R$ is the irreducible, symmetric bi-L-matrix defined by Equation~\eqref{eq:DefinitionRD1Cauchy}:
\begin{equation*}
\left\{
 \begin{array}{lcll}
 R_{ij} & = & -\frac{1}{|I|} & \quad \text{whenever } i \neq j \in I, \\
 R_{ii} & = & \frac{|I|-1}{|I|} & \quad \forall i \in I.
 \end{array}
 \right.
\end{equation*}
\end{proposition}

This text also includes a few results which may be of independent interest, such as a structural result on the periodic components of a $\Z^d$-extension of a Gibbs-Markov map, which has the curious particularity of being different for $\Z$ and $\Z^2$-extensions (see Proposition~\ref{prop:StructureColorationInduits}).

\smallskip

The structure of this article is as follows. Section~\ref{sec:ContexteEtResultats} is an introduction. We present a general method to compute transition matrices for random walks, ours setting ($\Z^d$-extensions, induced systems, Gibbs-Markov maps and transfer operators), and give our strategy to get ours results. Sections~\ref{sec:MasterLemma}, \ref{sec:ErgodicStructure} and~\ref{sec:Inversion} give a proof of Theorem~\ref{thm:MainTheorem}; this decomposition in three steps shall be explained in Sub-subsection~\ref{subsubsec:MainTheorem}. Afterwards, we shall obtain Proposition~\ref{prop:ApplicationD2CarreIntegrable} and similar propositions in Sections~\ref{sec:FourierTransform} and~\ref{sec:Calculs}; the former will give a general formula for the potential operator $Q_\varepsilon$ using Fourier transform, while the later will apply this formula to various settings. Finally, an appendix compiles some technical results on Gibbs-Markov maps.

\section*{Acknowledgements}
\label{sec:Remerciements}

The authors thanks Fran\c{c}oise P\`ene for the useful discussions while working on this project and her comments on the current version of this article.

\section{Setting and main results}
\label{sec:ContexteEtResultats}

\subsection{Hitting probabilities for recurrent random walks}
\label{subsec:RandomWalks}

\subsubsection{General theory}
\label{subsubsec:GeneralTheory}

Consider a sequence of independent and identically distributed $\Z^d$-valued random variables $(X_n)_{n \geq 0}$, and set $S_n := \sum_{k = 0}^{n-1} X_k$. The sequence of random variables $(S_n)_{n \geq 0}$ is a random walk in $\Z^d$ starting from $0$. Let $P$ be its transition kernel:
\begin{equation*}
 P_{pq} = \Pbb (X_0 = q-p).
\end{equation*}
This transition kernel acts on functions by $P (f) (p) = \sum_{q \in \Z^d} P_{pq} f(q)$.

\smallskip

In what follows, we shall assume that $(S_n)_{n \geq 0}$ is recurrent and ergodic, which implies in particular that $d \in \{1, 2\}$. Given $\Sigma \subset \Z^d$ and $p \in \Z^d \setminus \Sigma$, we define the hitting time of $\Sigma$ as the random variable
\begin{equation*}
 \overline{\varphi}_\Sigma (p) 
 := \inf \{n \geq 0: \ p+S_n \in \Sigma\},
\end{equation*}
and the \emph{hitting point} as $p+S_{\overline{\varphi}_\Sigma (p)} \in \Sigma$. Then, by recurrence, the hitting point is a $\Sigma$-valued random variable. Its distribution can be computed in the following way: given any bounded function $f : \Sigma \to \R$, extend it to a bounded function $g : \Z^d \to \R$ such that
\begin{equation}
\label{eq:EquationPoissonAvecFrontiere}
 \left\{
 \begin{array}{lll}
  g(r) & = f(r) & \quad \forall r \in \Sigma, \\
  (\Id-P) (g) (r) & = 0 & \quad \forall r \in \Z^d \setminus \Sigma.
 \end{array}
 \right.
\end{equation}
Then, by a martingale argument,
\begin{equation*}
 g(p) 
 = \mathbb{E} (f (p+S_{\overline{\varphi}_\Sigma (p)}));
\end{equation*}
specializing e.g. to $f = \mathbf{1}_q$ for $q \in \Sigma$ gives in principle the distribution of the hitting point, as long as one is able to solve Equation~\eqref{eq:EquationPoissonAvecFrontiere}.

\smallskip

In this article, we are interested in a slightly different problem: $\Sigma$ is finite, and we want to understand transition probabilities ; that is, the starting point $p$ belongs to $\Sigma$, and we consider the \emph{hitting time}
\begin{equation*}
 \varphi_\Sigma (p) 
 := \inf \{n \geq 1: \ p+S_n \in \Sigma\}.
\end{equation*}
The problem of computing transition probabilities can be reframed in the following way. Taking a trajectory $(S_n)_{n \geq 0}$ of the random walk and keeping only the times at which the random walk belongs to $\Sigma$ yields, by the strong Markov property, a $\Sigma$-valued \emph{induced Markov chain}.

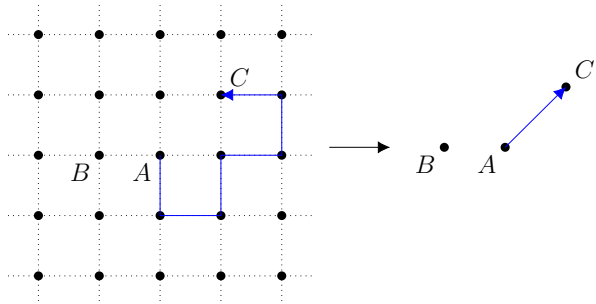
\begin{figure}[!h]
\centering

\scalebox{0.8}{
\begin{tikzpicture}
\foreach \x in {-2,...,2}{
    \draw [dotted] (\x,-2.5) -- (\x,2.5) ;
    }
\foreach \y in {-2,...,2}{
    \draw [dotted] (-2.5,\y) -- (2.5,\y) ;
    }
\foreach \x in {-2,...,2}{
    \foreach \y in {-2,...,2}{
        \node at (\x,\y) [circle,fill,inner sep=1.5pt]{} ;
        }
    }
\draw node[below left] (A) at (0,0) {$A$};
\draw node[below left] (B) at (-1,0) {$B$};
\draw node[above right] (C) at (1,1) {$C$};
\draw [-{Latex[length=2mm,width=2mm]},blue] (0,0) -- (0,-1) -- (1,-1) -- (1,0) -- (2,0) -- (2,1) -- (1,1) ;
\end{tikzpicture}
\begin{tikzpicture}
\draw [white] node at (0,-2.5) {} ;
\draw [white] node at (-0.5,0) {} ;
\draw [-{Latex[length=2mm,width=2mm]}] (-0.5,0) -- (0.5,0) ;
\draw [white] node at (0.5,0) {} ;
\draw [white] node at (0,2.5) {} ;
\end{tikzpicture}
\begin{tikzpicture}
\draw [white] node at (0,-2.5) {} ;
\foreach \P in {(-1,0), (0,0), (1,1)}{
    \node at \P [circle,fill,inner sep=1.5pt]{} ;
    }
\draw node[below left] (A) at (0,0) {$A$};
\draw node[below left] (B) at (-1,0) {$B$};
\draw node[above right] (C) at (1,1) {$C$};
\draw [white] node at (0,2.5) {} ;

\draw [-{Latex[length=2mm,width=2mm]}, blue]  (0,0) -- (1,1) ;
\end{tikzpicture}
}
\caption{A transition on $\{A, B, C\}$ induced by a path in $\Z^2$.}
\label{fig:MarcheAleatoireZ2Induit}
\end{figure}

This induced Markov chain has transition matrix
\begin{equation*}
 P_{\Sigma,pq}
 = \Pbb (p+S_{\varphi_\Sigma (p)} = q),
\end{equation*}
and we want to compute the operator $P_\Sigma$. This can be done thanks to a variation of Equation~\eqref{eq:EquationPoissonAvecFrontiere} adapted to starting points belonging to $\Sigma$, the so-called \emph{balayage identity}. The version we quote is close to~\cite[Corollary~1.11]{Revuz:1975}, and can be found in~\cite{PeneThomine:2020}:

\begin{theorem}[Balayage identity]\quad
\label{thm:IdentiteBalayage}
 
 Let $\Omega$ be a Polish space, $P$ a Markov kernel on $\Omega$, and $\mu$ a measure which is stationary, recurrent and ergodic for $P$.
 
 \smallskip
 
 Let $\Sigma \subset \Omega$ be measurable, with $\mu (\Sigma) > 0$. Let $g \in \Lbb^\infty (\Omega, \mu)$ and $f \in \Lbb^\infty (\Sigma, \mu_{|\Sigma})$. Let $f \mathbf{1}_\Sigma : \Z^d \to \C$ be the function which coincides with $f$ on $\Sigma$ and is null everywhere else.
 
 \smallskip
 
 If
 \begin{equation*}
  (\Id-P) g = f\mathbf{1}_\Sigma,
 \end{equation*}
 then:
 \begin{equation*}
  (\Id-P_\Sigma) g_{|\Sigma} = f.
 \end{equation*} 
\end{theorem}

Theorem~\ref{thm:IdentiteBalayage} relates the solution of the Poisson 
equation for the kernel $P$ on $\Z^d$ to the solutions of the Poisson equation for the kernel 
$P_\Sigma$ on $\Sigma$. In the framework of electrodynamics, $f$ is an intensity, while $g$ is an associated potential.

\smallskip

A first remark is that $(S_n)_{n \geq 0}$, seen as a Markov chain on $\Z^d$, preserves the counting measure, is ergodic and recurrent. Moreover, for general recurrent Markov chains with stationary measure $\mu$, the restricted measure $\mu_{|\Sigma}$ is stationary for the Markov chain induced on $\Sigma$. Consequently, in the context of random walks, $P_\Sigma$ preserves the uniform measure $\mu_\Sigma$ on $\Sigma$.

\smallskip

When acting on functions, $P_\Sigma$ preserves constant functions. Since $(S_n)_{n \geq 0}$ is ergodic, so is $P_\Sigma$, so $\Ker (\Id-P_\Sigma) = \Vect (\mathbf{1})$ has dimension $1$. Finally, since $P_\Sigma$ preserves $\mu_\Sigma$, the eigenprojection onto $\Vect (\mathbf{1})$ is given by $f \mapsto \mu_\Sigma (f) \mathbf{1}$, whence we have a $P_\Sigma$-invariant splitting $\C^\Sigma = \Vect (\mathbf{1}) \oplus \C_0^\Sigma$, with $\C_0^\Sigma := \Ker(\mu_\Sigma)$. Hence, all we need is to understand the action of $P_\Sigma$ on $\C_0^\Sigma$.

\smallskip

Now, take any $f \in \C_0^\Sigma$. Extend it to $f \mathbf{1}_\Sigma : \Z^d \to \C$. Find a function $g : \Z^d \to \C$ such that $(\Id-P) g = f \mathbf{1}_\Sigma$ and $\mu_{|\Sigma} (g) = 0$ ; such a function always exists under our hypotheses. Then $g_{|\Sigma} \in \C_0^\Sigma$ and $(\Id-P_\Sigma) g_{|\Sigma} = f$ by the balayage identity. Doing this operation for all $f$ in a basis of $\C_0^\Sigma$ allows us, in theory, to compute the operator $(\Id-P_\Sigma)^{-1}$ on $\C_0^\Sigma$. Then, invert $(\Id-P_\Sigma)^{-1}$, and from there get the action of $P_\Sigma$ on $\C_0^\Sigma$.

\smallskip

All is left is to effectively compute $g$ given $f$. At this point the algebraic structure of the random walk plays a role, as we can use the Fourier transform. The main idea is that the operator $\widehat{P}$, which corresponds to the action of $P$ in the frequency domain, is merely the multiplication operator by the characteristic function $\Psi (\xi) := \Ebb (e^{i \langle \xi,  X \rangle})$ of $X$. Up to an additive constant,
\begin{equation*}
 \hat{g} (\xi) 
 = \sum_{n = 0}^{+ \infty} \widehat{P}^n (\xi) \hat{f} (\xi) 
 = \sum_{n = 0}^{+ \infty} \Psi (\xi)^n \hat{f} (\xi) 
 = \frac{\hat{f} (\xi)}{1-\Psi(\xi)},
\end{equation*}
whence, writing $\Tbb^d = \R^d_{/ \Z^d}$,
\begin{equation}
 \label{eq:MAFourierASpatial}
 g(p) 
 = \frac{1}{(2 \pi)^d} \int_{\Tbb^d} \frac{\hat{f} (\xi)}{1-\Psi(\xi)} e^{i \langle \xi, p \rangle} \dd \xi,
 \end{equation}
still up to an additive constant. If the kernel $P$ is nice enough, as is the case for the simple random walk on $\Z^d$, such integrals can even be explicitly computed.

\smallskip

To sum up, this strategy to compute the Markov kernel $P_\Sigma$ has three main ingredients:
\begin{itemize}
 \item the balayage identity (Theorem~\ref{thm:IdentiteBalayage}) which relates the Markov kernels $P$ and $P_\Sigma$ through the solutions to their respective Poisson equations;
 \item the ability to solve the Poisson equation on $\Z^d$, using for instance the Fourier transform;
 \item the ability to invert $(\Id-P_\Sigma)^{-1}$.
\end{itemize}
In the case of random walks, $\C_0^\Sigma$ is finite dimensional, so the last step may seem comparatively insignificant. However, in the remainder of this article, we shall work with operators acting on infinite-dimensional Banach spaces, for which getting explicit information on their inverse is much tougher.

\subsubsection{An example with the simple random walk}
\label{subsubsec:MASimpleExample}

Now, let us delve deeper into the method described in Sub-subsection~\ref{subsubsec:GeneralTheory} by applying it to the example of Figure~\ref{fig:MarcheAleatoireZ2}. We take for $(S_n)_{n \geq 0}$ the simple random walk on $\Z^2$, and $\Sigma = \{A, B, C\}$ with $A = (0;0)$, $B=(-1;0)$ and $C=(1;1)$.

\smallskip

The simple random walk has a transition kernel $P_p = \frac{1}{4} \sum_{q \leftrightarrow p} \delta_q$, where $q \leftrightarrow p$ if and only if $p$ an $q$ are neighbours. Its action on functions $f : \Z^2 \to \C$ is thus
\begin{equation*}
 P(f) (p) 
 = \frac{1}{4} \sum_{q \leftrightarrow p} f(q),
\end{equation*}
and the simple random walk is indeed ergodic, recurrent, and preserves the counting measure on $\Z^2$.

\smallskip

Let $P_\Sigma$ be the transition matrix for the Markov chain induced on $\Sigma$. Then $P_\Sigma$ preserves the uniform measure on $\Sigma$ denoted by $\mu_\Sigma$. Thus we have a $P_\Sigma$-invariant splitting $\C^\Sigma = \Ker (\mu_\Sigma) \oplus \Vect(\mathbf{1})$.

\smallskip

Now, we need to understand the action of $P_\Sigma$ on $\Ker (\mu_\Sigma) = \C_0^\Sigma$. It is enough to understand its action on a basis, say, $e_1:=\mathbf{1}_A-\mathbf{1}_B$ and $e_2:=\mathbf{1}_A-\mathbf{1}_C$. Note however that Equation~\eqref{eq:MAFourierASpatial} only let us compute $(\Id-P_\Sigma)^{-1} e_j$ up to an additive constant. To sidestep this issue, we compute the quantities
\begin{equation*}
 \langle e_i, (\Id-P_\Sigma)^{-1} e_j \rangle_{\ell^2 (\Sigma)}
 := \sum_{p \in \Sigma} e_i (p) \cdot \left[ (\Id-P_\Sigma)^{-1} e_j \right] (p)
\end{equation*}
for $i$, $j \in \{1,2\}$, which is enough to reconstruct the action of $(\Id-P_\Sigma)^{-1}$ on $\C_0^\Sigma$. Since all the $e_i$ are in $\C_0^\Sigma$, the indetermination of $(\Id-P_\Sigma)^{-1}$ up to constants does not matters. We illustrate this computation with $i = j = 1$. By Theorem~\ref{thm:IdentiteBalayage}, 
\begin{align*}
 \langle e_1, (\Id-P_\Sigma)^{-1} e_1 \rangle_{\ell^2 (\Sigma)} 
 & = \sum_{p \in \Sigma} e_1 (p) \cdot \left[ (\Id-P_\Sigma)^{-1} e_1 \right] (p) \\
 & = \sum_{p \in \Z^2} (e_1 \mathbf{1}_\Sigma) (p) \cdot \left[ (\Id-P)^{-1} (e_1 \mathbf{1}_\Sigma) \right] (p) \\
 & = \frac{1}{(2 \pi)^2} \int_{\Tbb^2} \widehat{(\mathbf{1}_A - \mathbf{1}_B)} (-\xi) \cdot \frac{\widehat{(\mathbf{1}_A - \mathbf{1}_B)} (\xi)}{1-\Psi (\xi)} \dd \xi \\
 & = \frac{1}{(2 \pi)^2} \int_{\Tbb^2} (1-e^{-i \xi_1}) \cdot \frac{(1-e^{i \xi_1})}{1-\frac{e^{i \xi_1} + e^{-i \xi_1} + e^{i \xi_2} + e^{-i \xi_2}}{4}} \dd \xi \\
 & = \frac{1}{\pi^2} \int_{\Tbb^2} \frac{1-\cos (\xi_1)}{2-\cos(\xi_1)-\cos(\xi_2)} \dd \xi \\
 & = 2.
\end{align*}
The other three computations are similar. We recover from there the action of $(\Id-P_\Sigma)^{-1}$ on $\C_0^\Sigma$, then the action on $P_\Sigma$ on $\C_0^\Sigma$ after inversion. Finally, adding the rank one map $\mathbf{1} \otimes \mu_\Sigma$ yields the matrix of Equation~\eqref{eq:MatriceTransitionMASimple}.

\subsection{Dynamical systems: extension and induction}
\label{subsec:ExtensionEtInduction}

\subsubsection{$\Z^d$-extensions of dynamical systems}

A common generalization of random walks on $\Z^d$, in the context of ergodic theory, is that of \emph{$\Z^d$-extension}. Let $(A, \mu, T)$ be a probability-preserving dynamical system, and $F : A \to \Z^d$ a measurable function (which we call the \textit{jump function}). Let $([\Z^d], \widetilde{\mu}, \widetilde{T})$ be the corresponding $\Z^d$-extension as defined at the beginning of this text.

\smallskip

Let us write $S_n F (x) := \sum_{k=0}^{n-1} F (T^k (x))$ the Birkhoff sums of $F$. Then 
\begin{equation*}
 \widetilde{T}^n (x, p) 
 = (T^n (x), p+S_n F (x)).
\end{equation*}
This setting includes random walks: take $A = (\Z^d)^\N$, then $T((x_n)_{n \geq 0}) = (x_{n+1})_{n \geq 0}$ and $\mu = \mu_0^{\otimes \N}$ a product measure on $A$. Choose $F((x_n)_{n \geq 0}) = x_0$; then, under $\mu \otimes \nu \in \Pcal([\Z^d])$, the second coordinate is a random walk whose starting point has distribution $\nu$ and whose jumps have distribution $\mu_0$. If one replaces the full shift by a subshift, or take a Markov measure instead of the product measure, we get a \textit{random walk with internal states}, which is a special case of hidden Markov models~\cite{KramliSzasz:1983, KramliSzasz:1984, KramliSimanyiSzasz:1986}. If $(A, \mu, T)$ is the collision map of a Sinai billiard, then $([\Z^d], \widetilde{\mu}, \widetilde{T})$ is the collision map of a Lorentz gas, a popular diffusion model.

\subsubsection{Inductions of dynamical systems}
\label{subsubsec:Induction}

Let $(B, \nu, S)$ be a measure-preserving dynamical system. For any measurable $C \subset B$ 
with positive measure, we define the \textit{first hitting time} of $C$ as:
\begin{equation*}
 \varphi_C (x) 
 := \inf \{ n \geq 1 : \widetilde{T} (x) \in C \ \}. 
\end{equation*}
The system $(B, \nu, S)$ is said to be ergodic if it has no non-trivial strictly invariant subset, and \textit{recurrent} if $\varphi_C (x) < + \infty$ for almost every $x \in C$, for any measurable $C \subset [\Z^d]$.

\smallskip

If $(B, \nu, S)$ is recurrent and $C \subset B$ has positive measure, the \textit{induced map}
\begin{equation*}
 T_C : \left \{ 
 \begin{array}{rcl}
 C & \to & C \\
 x & \mapsto & \widetilde{T}^{\varphi_C (x)} (x)                
 \end{array}
 \right.
\end{equation*}
is well-defined almost everywhere, recurrent, and preserves the restricted measure $\nu_C := \nu_{|C}$~\cite[Corollary~1.1]{PeneThomine:2020}. If in addition $(B, \nu, S)$ is ergodic, then so is $(C, \nu_C, T_C)$. In what follows, we shall often make this assumption:

\begin{hypothesis}[Recurrence]\quad
\label{hyp:Recurrence}
 
 The $\Z^d$-extension $([\Z^d], \widetilde{\mu}, \widetilde{T})$ is ergodic and recurrent.
\end{hypothesis}

Let $([\Z^d], \widetilde{\mu}, \widetilde{T})$ be a recurrent $\Z^d$-extension. For any $\Sigma \subset \Z^d$, we shall write $[\Sigma] := A \times \Sigma \subset [\Z^d]$. If $\Sigma = \{\sigma\}$ is a singleton, we shall dispense with the brackets, and write $[p] = A \times \{p\}$. In this article, we will work with three types of systems:
\begin{itemize}
 \item the initial probability-preserving dynamical system $(A, \mu, T)$;
 \item a $\Z^d$-extension $([\Z^d], \widetilde{\mu}, \widetilde{T})$ with jump function $F$;
 \item induced maps $([\Sigma], \mu_{[\Sigma]}, T_{[\Sigma]})$ for $\Sigma \subset \Z^d$.
\end{itemize}
We try as much as possible to keep our notations coherent: objects with a tilde refer to the $\Z^d$-extension, and objects with a set as an index refer to induced maps.

\smallskip

In this context, we define transition probabilities as follows:

\begin{definition}[Transition probabilities]\quad

 Let $([\Z^d], \widetilde{\mu}, \widetilde{T})$ be a recurrent $\Z^d$-extension of a dynamical system $(A, \mu, T)$. Let $\Sigma \subset \Z^d$ be finite.
 
 \smallskip
 
 The \textit{matrix of transition probabilities} $P_\Sigma$ is defined by:
 \begin{equation*}
  P_{\Sigma,pq} 
  := \mu \left(x \in A : T_{[\Sigma]} (x,p) \in [q] \right).
 \end{equation*}
\end{definition}

As long as the extension is recurrent, the matrix $P_\Sigma$ is \textit{bi-stochastic}: the sum of each of its rows and of each of its columns is $1$.

\begin{lemma}\quad
 \label{lem:MatriceBiStochastique}
 
 Let $([\Z^d], \widetilde{\mu}, \widetilde{T})$ be a recurrent $\Z^d$-extension of a probability preserving 
 dynamical system. Let $\Sigma \subset \Z^d$ be non-empty. Then $P_\Sigma$ is bi-stochastic.
\end{lemma}

\begin{proof}
 
 The fact that $P_\Sigma$ is stochastic is a translation of the fact that $\widetilde{T}$ is recurrent. 
 For any $p \in \Sigma$,
 \begin{align*}
  \sum_{q \in \Sigma} P_{\Sigma, pq} 
  & = \sum_{q \in \Sigma} \mu \left(x \in A : T_{[\Sigma]} (x,p) \in [q] \right) \\
  & = \mu \left(x \in A : T_{[\Sigma]} (x,p) \in [\Sigma] \right) \\
  & = \widetilde{\mu} \left([p] \cap \{ \varphi_{[\Sigma]} < +\infty \} \right) \\
  & = 1.
 \end{align*}

 The bi-stochasticity of $P_\Sigma$ comes from the fact that $T_{[\Sigma]}$ preserves $\mu_{[\Sigma]}$. 
 For any $q \in \Sigma$,
 \begin{align*}
  \sum_{p \in \Sigma} P_{\Sigma, pq} 
  & = \sum_{p \in \Sigma} \mu \left(x \in A : T_{[\Sigma]} (x,p) \in [q] \right) \\
  & = \mu_{[\Sigma]} \left( T_{[\Sigma]}^{-1} ([q]) \right) \\
  & = \mu_{[\Sigma]} \left( [q] \right) \\
  & = 1. \qedhere
 \end{align*}
\end{proof}

\subsection{Gibbs-Markov maps}

If no assumption is made on $(A, \mu, T)$ nor on $F$, the behaviour of the trajectories of the system $([\Z^d], \widetilde{\mu}, \widetilde{T})$ can be extremely diverse: for instance, the rate of diffusion in Ehrenfest's wind-tree model depends on the geometry of the scatterers~\cite{DelecroixHubertLelievre:2014, Fougeron:2020}. However, if $(A, \mu, T)$ is in some sense chaotic and $F$ regular, the trajectories of the system $([\Z^d], \widetilde{\mu}, \widetilde{T})$ will be very similar to random walks on $\Z^d$.

\smallskip

In this article, we choose to work with Gibbs-Markov maps. This family of transformations is Markovian, which brings the benefits of avatars of the strong Markov property, for instance Corollary~\ref{cor:BorneUniformeTempsArret}. These transformations are still flexible enough to give rise to a wide variety of examples and, using Young's construction~\cite{Young:1998}, they appear in many geometric examples such as the geodesic flow on negatively curved manifolds, finite horizon Sinai billiards, and Collet-Eckmann unimodal transformations.

\smallskip

The family of $\Z^d$-extensions we get is much more general than random walks and random walks with internal states as in~\cite{KramliSzasz:1984, KramliSimanyiSzasz:1986}, while still more restrictive than Sinai billiards or unimodal maps, since we do not work with Young towers.

\subsubsection{Definition}

In what follows, we shall assume that $(A, \mu, T)$ is a \textit{Gibbs-Markov map}:

\begin{definition}[Measure-preserving Gibbs-Markov maps]\quad
\label{def:ApplicationGibbsMarkov}

Let $(A, d')$ be a metric bounded Polish space, $\Acal$ its Borel $\sigma$-algebra and $\mu$ a probability measure on $(A, \Acal)$. Let $\alpha$ be a partition of $A$ in subsets of positive measure (up to a null set). Let $T : \ A \to A$ be a $\mu$-preserving map, Markov with respect to the partition $\alpha$, and such that $\alpha$ is generating. Such a map is said to be \emph{Gibbs-Markov} if it also satisfies:
\begin{itemize}
\item Big image property: $\inf_{a \in \alpha} \mu (Ta) > 0$;
\item Expansivity: there exists $\Lambda > 1$ such that $d' (Tx, Ty) \geq \Lambda d(x, y)$ 
for all $a \in \alpha$ and $(x,y) \in a \times a$;
\item Bounded distortion: there exists $C > 0$ such that, for all $a \in \alpha$, for almost every 
$(x,y) \in a \times a$: 
\end{itemize}
\begin{equation}
\left| \frac{\dd \mu}{\dd \mu \circ T} (x) - \frac{\dd \mu}{\dd \mu \circ T} (y) \right| 
\leq C d'(Tx, Ty) \frac{\dd \mu}{\dd \mu \circ T} (x).
\end{equation}
\end{definition}

Given a Gibbs-Markov map $(A, \alpha, d', \mu, T)$, a function $F : A \to \Z^d$ is said to be \emph{Markov} if it is constant on elements of the partition $\alpha$. In this case, we also say that the corresponding $\Z^d$-extension $([\Z^d], \widetilde{\mu}, \widetilde{T})$ is Markov.

\smallskip

By encoding, Gibbs-Markov maps can model sequences of independent and identically distributed $\Z^d$-valued random variables, as well as finite state Markov chains. As such, Markov extensions of Gibbs-Markov maps include all random walks, with or without internal states, on $\Z^d$. They are however more general, including for instance the Gauss maps
\begin{equation*}
 T : \left\{ 
 \begin{array}{rcl}
   (0,1] & \to & [0,1] \\
   x & \mapsto & 1/x-\lfloor 1/x \rfloor
 \end{array}
\right. ,
\end{equation*}
and appear in various encoding schemes used to study geodesic flows in negative sectional curvature as well as Sinai billiards.

\smallskip

A comprehensive presentation of the properties of Gibbs-Markov maps we use in this article can be found in Appendix~\ref{sec:GibbsMarkovAppendice}, in particular their properties related to the extension and induction operations described in Subsection~\ref{subsec:ExtensionEtInduction}.

\subsubsection{Transfer operators}
\label{subsubsec:TransferOperators}

One of our main tools is the \textit{transfer operator} $\Lcal$, defined as the dual of the composition (or Koopman) operator $\Kcal : g \mapsto g \circ T$:
\begin{equation}
\label{eq:DefinitionTransferOperator}
 \int_A \Lcal (f) \cdot g \dd \mu 
 := \int_A f \cdot g \circ T \dd \mu,
\end{equation}
for all $p \in [1, \infty]$, all $f \in \Lbb^p (A, \mu; \C)$ and all $g \in \Lbb^q (A, \mu; \C)$, with $p^{-1}+q^{-1} = 1$. Since the Koopman operator is an isometry on $\Lbb^q$, the transfer operator is a weak contraction on $\Lbb^p$.

\smallskip

As specified in Sub-subsection~\ref{subsubsec:Induction}, the operator $\Lcal$ acts on $\Lbb^p (A, \mu; \C)$, 
the operator $\widetilde{\Lcal}$ on $\Lbb^p ([\Z^d], \widetilde{\mu}; \C)$ and the operator $\Lcal_{[\Sigma]}$ 
on $\Lbb^p ([\Sigma], \mu_{[\Sigma]}; \C)$.

\smallskip

One of the main feature of Gibbs-Markov maps is that the transfer operator acts quasi-compactly on a Banach space of 
Lipschitz functions. We present here the most important properties of the action; again, details and references can be found in Appendix~\ref{sec:GibbsMarkovAppendice}.

\begin{definition}[Lipschitz functions]\quad
\label{def:Lipschitz}

  Let $\alpha^* := \bigvee \{T(a) : \ a \in \alpha\}$ be the image partition\footnote{In many applications, $\alpha^*$ is trivial -- 
  consider for instance the Gauss map, where any interval $((n+1)^{-1}, n^{-1}]$ is sent onto $[0,1)$.} of $\alpha$ by $T$. The 
  \textit{Lipschitz seminorm} of a function $f : A \to \C$ is
  \begin{equation*}
    |f|_\Bcal 
    := \sup_{a \in \alpha^*} |f|_{\Lip (a, d')},
  \end{equation*}
  where $|f|_{\Lip (a, d')}$ is the Lipschitz constant of $f$ on $a$. The \textit{Lipschitz norm} is then given by
  \begin{equation*}
    \norm{f}{\Bcal}
    := \norm{f}{\infty} + |f|_\Bcal,
  \end{equation*}
  and the space of Lipschitz functions $\Bcal$ by $\{f : A \to \C, \ \norm{f}{\Bcal} < +\infty\}$.
  
  \smallskip
  
  Finally, we write $\Bcal_0$ for the kernel of $\mu$ in $\Bcal$.
\end{definition}

A typical result (see~\cite[Proposition~1.1.17 and Corollary~1.1.21]{Gouezel:2008e}) is then:

\begin{theorem}\quad
 
 Let $(A, \alpha, d', \mu, T)$ be a Gibbs-Markov map with expansivity constant $\Lambda >1$. The transfer operator $\Lcal$ acts on $\Bcal$, with essential spectral radius $\rho_{\ess} (\Lcal \acts \Bcal) \leq \Lambda^{-1}$.
 
 \smallskip
 
 If $(A, \mu, T)$ is also mixing, then the spectral radius of $\Lcal \acts \Bcal_0$ is strictly less than $1$: there exist constants $C>0$ and $\rho \in (0,1)$ such that $\norm{\Lcal^n}{\Bcal_0 \to \Bcal_0} \leq C \rho^n$. 
\end{theorem}

Another important feature of Gibbs-Markov maps is that they behave very well under induction. Let $([\Z^d], \widetilde{\mu}, \widetilde{T})$ be a recurrent and ergodic Markov $\Z^d$-extension of a Gibbs-Markov map, and $\Sigma \subset \Z^d$ be finite. Following the strategy of e.g.~\cite[Lemma~1.6]{Thomine:2015}, one may endow the system $([\Sigma], \mu_{[\Sigma]}, T_{[\Sigma]})$ with a Gibbs-Markov structure, with expansivity and distortion constants independent from $\Sigma$. From there one could get a family of Banach spaces depending on $\Sigma$. However, we shall not use directly this construction, but choose a closely related strategy: we make all the transfer operators $\Lcal_{[\Sigma]}$ act on the same Banach space, with relevant bounds (such as the essential spectral radius) uniform in $\Sigma$. This let us compare more easily all the operators $\Lcal_{[\Sigma]}$.

\smallskip

We associate to $([\Sigma], \mu_{[\Sigma]}, T_{[\Sigma]})$ a Banach space of Lipschitz functions 
\begin{equation}
\label{eq:DefinitionBSigma}
 \Bcal_{[\Sigma]} 
 := \left\{f : [\Sigma] \to \C, \ f (\cdot, i) \in \Bcal \ \forall i \in \Sigma \right\},
\end{equation}
endowed with the norm $\norm{f}{\Bcal_{[\Sigma]}} = \max_{i\in I} \norm{f (\cdot, i)}{\Bcal}$. We denote by $\Bcal_{[\Sigma], 0}$ the kernel of $\mu_{[\Sigma]}$.

\smallskip

Not that the spaces $\Bcal_{[\Sigma]}$ depend only on $\Sigma$ through its cardinality: if two subsets $\Sigma_1$, $\Sigma_2$ have the same cardinality, then $\Bcal_{[\Sigma_1]}$ and $\Bcal_{[\Sigma_2]}$ are naturally isomorphic.
By Proposition~\ref{prop:ActionInduiteQuasiCompacte}, the operator $\Lcal_{[\Sigma]}$ acts quasi-compactly on $\Bcal_{[\Sigma]}$, with essential spectral radius at most $\Lambda^{-1}$. In particular, by ergodicity, $1$ is a simple eigenvalue of $\Lcal_{[\Sigma]}$ corresponding to constant functions, so that $(\Id-\Lcal_{[\Sigma]}) \acts \Bcal_{[\Sigma],0}$ is invertible.

\subsection{Shapes and matrices}

\subsubsection{Limit shapes}
\label{subsec:LimitShape}

As we argued in the introduction, it seems extremely hard to state anything non-trivial about the matrix of transition probabilities $P_{[\Sigma]}$ for any fixed $\Sigma$. Instead, we shall be interested in subsets $\Sigma$ whose sites are very far apart. We formalize this by fixing a finite set $I$, considering a family of injections $\sigma_\varepsilon : I \hookrightarrow \Z^d$ whose values get apart when $\varepsilon$ vanishes, and looking at the matrix of transition probabilities of $[\Sigma_\varepsilon] := [\sigma_\varepsilon (I)]$ . 

\smallskip

A simple way of generating such functions $\sigma_\varepsilon$ is to fix some injection $\sigma : I \hookrightarrow \R^d$, and to define $\sigma_\varepsilon (i)$ as one of the points of $\Z^d$ closest to $\varepsilon^{-1} \sigma (i)$ -- or, more generally, any point at distance $o(\varepsilon^{-1})$ of $\varepsilon^{-1} \sigma (i)$. As $\varepsilon$ vanishes, the subsets $\Sigma_\varepsilon = \sigma_\varepsilon (I)$ become growing realizations of a fixed shape (see Figure~\ref{fig:FamilleFormes}).

\smallskip

For any $\varepsilon > 0$, there is a canonical bijection $\sigma_{\varepsilon *}: (x,i) \mapsto (x, \sigma_\varepsilon (i))$ between $[I] := A \times I$ and $[\Sigma_\varepsilon] = A \times \Sigma_\varepsilon$. We use this bijection to transport objects from $[\Sigma_\varepsilon]$ to $[I]$:
\begin{itemize}
 \item $\mu_I := |I|^{-1} \sum_{i \in I} \mu \otimes \delta_i$ is conjugated with $\mu_{[\Sigma_\varepsilon]}$ and does not depend on $\varepsilon$;
 \item $T_\varepsilon : [I] \to [I]$ is conjugated with $T_{[\Sigma_\varepsilon]}$;
 \item $\Lcal_\varepsilon \acts \Lbb^p (|I], \mu_I)$ is conjugated with $\Lcal_{[\Sigma_\varepsilon]} \acts \Lbb^p (|\Sigma], \mu_{[\Sigma]})$;
 \item $\Bcal_I$ is conjugated with $\Bcal_{[\Sigma_\varepsilon]}$, and by definition does not depend on $\varepsilon$;
 \item $P_\varepsilon$ is conjugated with $P_{[\Sigma_\varepsilon]}$, etc.
\end{itemize}
The general rule is that objects defined on $[I]$ have $I$ as a subscript if they do not depend on $\varepsilon$, 
and $\varepsilon$ as a subscript if they do. This point of view allows us to keep some objects fixed (the space, the measure, the Banach space $\Bcal_I$, and later on functions) while $\varepsilon$, and thus $\Sigma_\varepsilon$, changes.

\subsubsection{Transfer operator and transition probabilities}

The matrices $(P_\varepsilon)_{\varepsilon > 0}$ can be recovered from the Koopman operators 
$(\Kcal_\varepsilon)_{\varepsilon >0}$. Given $\varepsilon > 0$ and $i$, $j \in I$, 
\begin{align}
 P_{\varepsilon, ij} 
 & = \mu \left([i] \cap T_\varepsilon^{-1} ([j]) \right) \nonumber \\
 & = |I| \int_{[I]} \mathbf{1}_{[i]} \cdot \Kcal_\varepsilon(\mathbf{1}_{[j]}) \dd \mu_I. \label{eq:TransitionToKoopman}
\end{align}
Using transfer operators, this relationship can be rewritten
\begin{equation}
\label{eq:TransitionToOperator}
 P_{\varepsilon, ij} 
 = |I| \int_{[I]} \Lcal_\varepsilon (\mathbf{1}_{[i]}) \cdot \mathbf{1}_{[j]} \dd \mu_I,
\end{equation}
and made more systematic by introducing an averaging operator. Let:
\begin{equation*}
 \Pi_* : \left\{ 
 \begin{array}{rcl}
  \Lbb^1 ([I], \mu_I) & \to & \C^I \\
  f & \mapsto & \left( \int_A f(x,i) \dd \mu (x) \right)_{i \in I}
 \end{array}
 \right. ,
\end{equation*}
which to a function $f$ associate its averages over each $[i]$, and define a right inverse $\Pi^*$ 
of $\Pi_*$:
\begin{equation*}
 \Pi^* : \left\{ 
 \begin{array}{rcl}
  \C^I & \to & \Lbb^1 ([I], \mu_I) \\
  f & \mapsto & \sum_{i \in I} f_i \mathbf{1}_{[i]}
 \end{array}
 \right. .
\end{equation*}
The operators $\Pi_*$ and $\Pi^*$ satisfy the relation $\Pi_* \Pi^* = \Id_{\C^I}$, while 
$\Pi^* \Pi_*$ is a rank $|I|$ averaging operator on $\Lbb^1 ([I], \mu_I)$:
\begin{equation*}
 (\Pi^* \Pi_* f) (x,i) 
 = \int_A f(y,i) \dd \mu (y).
\end{equation*}
Equations~\eqref{eq:TransitionToKoopman} and~\eqref{eq:TransitionToOperator} can then be recast as
\begin{equation}
    \label{eq:TransferOperatorAveraging}
   P_\varepsilon = \Pi_* \Kcal_\varepsilon \Pi^* \text{ and } \transposee{P}_\varepsilon = \Pi_* \Lcal_\varepsilon \Pi^*.
\end{equation}

\subsubsection{Asymptotics of the transition matrices and irreductibility}
\label{subsubsec:AsymptotiqueP}

By~\cite[Lemma~3.17]{PeneThomine:2019}, if the minimal distance between distinct sites in $\Sigma_\varepsilon$ increases to $+\infty$ when $\varepsilon$ vanishes, then the transitions from one site to another becomes increasingly unlikely, so that $\lim_{\varepsilon \to 0} P_\varepsilon = \Id$. 
In this article, we are interested in the next order asymptotics of the family $(P_\varepsilon)_{\varepsilon>0}$ when $\varepsilon$ vanishes. By Lemma~\ref{lem:MatriceBiStochastique}, the matrices $P_\varepsilon$ are bi-stochastic, so that the differences $(\Id - P_\varepsilon)$ are what we shall call \textit{bi-L-matrices}. In other words, up to a sign, bi-L-matrices make up the tangent cone at $\Id$ to the space of bi-stochastic matrices.

\begin{definition}[L-matrix]\quad
\label{def:LMatrice}
 
 A \emph{L-matrix} is a square matrix such that\footnote{The usual definition of L-matrix does not require that the sum of each line be zero.}:
 \begin{itemize}
  \item $R_{ii} \geq 0$ for all $i \in I$;
  \item $R_{ij} \leq 0$ for $i \neq j \in I$;
  \item $\sum_{j \in I} R_{ij} = 0$ for all $i \in I$.
 \end{itemize}
 
 We shall call a L-matrix a \emph{bi-L-matrix} if, in addition, $\sum_{i \in I} R_{ij} = 0$ for all $j \in I$.
 
 \smallskip
 
 A L-matrix is \emph{irreducible} if, for all $i$, $j \in I$, there exists a finite sequence $i = i_0$, $i_1$, $\ldots$, $i_n = j$ such that $R_{i_k i_{k+1}} < 0$ for all $0\leq k < n$.
\end{definition}

Using Perron-Frobenius theory, eigenvalues of a L-matrix always have non-negative real part. For irreducible L-matrices, we can be more accurate.

\begin{lemma}\quad
 \label{lem:SpectreLMatrice}
 
 Let $R$ be a $n \times n$ irreducible L-matrix. Then:
 \begin{itemize}
  \item $0$ is a simple eigenvalue of $R$ with eigenfunction $\mathbf{1}$.
  \item there exists a unique probability vector $\nu$ such that $\nu R = 0$.
  \item the matrix $R$ acts on $\Ker(\nu)$, and $\Sp (R \acts \Ker(\nu)) \subset \{z \in \C : \ \Re (z) > 0\}$.
 \end{itemize}
\end{lemma}

\begin{proof}
 
 Let $R$ be an irreducible L-matrix. Set $t := \max_i R_{ii}$. Then $\Id-t^{-1} R$ is an 
 irreducible nonnegative matrix. The lemma follows from Perron-Frobenius theory applied to $\Id-t^{-1} R$.
\end{proof}

Let us write $R_0$ for the restriction of $R$ to $\Ker(\nu)$. A L-matrix $R$ is never invertible. However, if $R$ is also irreducible, then $R_0$ is invertible. Given an irreducible L-matrix $R$, we shall write
\begin{equation*}
 \rho_R 
 := \min \{ \Re (\lambda) : \ \lambda \in \Sp (R \acts \Ker(\nu))\} 
 = \min \{ \Re (\lambda) : \ \lambda \in \Sp (R_0)\}.
\end{equation*}

We shall often assume that the family $(P_\varepsilon)_{\varepsilon>0}$ admits a first order asymptotics at $\varepsilon = 0$.

\begin{hypothesis}[Asymptotics of the transition matrices]\quad
\label{hyp:AsymptoticsP}
 
 There exists an irreducible bi-L-matrix $R$ such that
 \begin{equation*}
  P_\varepsilon 
  = \Id-\varepsilon R + o(\varepsilon).
 \end{equation*}
\end{hypothesis}

\subsubsection{Asymptotics of the potential and irreductibility}

One issue with Hypothesis~\ref{hyp:AsymptoticsP} is that we do not know \textit{a priori} the matrices $(P_\varepsilon)_{\varepsilon >0}$; indeed, our goal is to compute them! As described in Subsection~\ref{subsec:RandomWalks}, what can be computed are potentials. We shall now define matrices related to these potential, as well as a related irreducibility condition.

\smallskip

Let $([\Z^d], \widetilde{\mu}, \widetilde{T})$ be a Markov $\Z^d$-extension satisfying Hypothesis~\ref{hyp:Recurrence}. By the discussion at the very end of Sub-subsection~\ref{subsubsec:TransferOperators}, for any function $f \in \Bcal_{I, 0}$ and any $\varepsilon > 0$, there exists a unique function $g \in \Bcal_{I, 0}$, its associated potential, such that $(\Id-\Lcal_\varepsilon) g = f$.

\smallskip

Let $\langle \cdot , \cdot \rangle_{\ell^2 (I)}$ be the canonical scalar -- and not Hermitian -- product on $\C^I$. In order to get a finite-dimensional version of $(\Id-\Lcal_\varepsilon)^{-1}$, we define an operator acting on $\C^I_0 = \left\{f : I \to \C, \ \sum_{i \in I} f_i = 0 \right\} = \Ker (\mu_I)$ by:
\begin{equation*}
 \langle f, Q_\varepsilon (g) \rangle_{\ell^2 (I)}
 := \langle \Pi_* (\Id-\Lcal_\varepsilon)^{-1} \Pi^* (f), g \rangle_{\ell^2 (I)} 
\end{equation*}
where the inverse $(\Id-\Lcal_\varepsilon)^{-1}$ is taken in $\Bcal_{I, 0}$. This operator is well-defined 
since $\Pi^*(\C^I_0) \subset \Bcal_{I, 0}$.

\smallskip

The operator $Q_\varepsilon$ is characterized by the integrals:
\begin{equation}
 \label{eq:ExpressionIntegraleQ}
 \langle f, Q_\varepsilon (g) \rangle_{\ell^2 (I)} 
 = |I| \int_{[I]} (\Id-\Lcal_\varepsilon)^{-1} \left(\Pi^* (f)\right) \cdot \Pi^* (g) \dd \mu_I
 = |I| \langle (\Id-\Lcal_\varepsilon)^{-1} \Pi^* (f), \Pi^* (g) \rangle_{\Lbb^2 ([I], \mu_I)}.
\end{equation}

\begin{remark}\quad
 
 Let us justify our choice of definition for $Q_\varepsilon$. We want to adapt the construction we have for Markov maps to dynamical systems, using the Koopman operator $\Kcal_\varepsilon$ instead of the Markov operator. Since the Koopman operator does not acts well on function spaces, we introduce its dual, defined by Equation~\eqref{eq:DefinitionTransferOperator}. The operator $\Pi_* (\Id-\Lcal_\varepsilon)^{-1} \Pi^*$ is thus a discretization of the potential operator associated with the dual of $\Kcal_\varepsilon$. We make it act of the left to cancel those dualities, so that the resulting operator $Q_\varepsilon$ acts as a discretization of the potential operator associated with $\Kcal_\varepsilon$.
\end{remark}

As we formulated asymptotics and irreducibility conditions for the family $(P_\varepsilon)_{\varepsilon >0}$, 
let us do so for the family $(Q_\varepsilon)_{\varepsilon > 0}$.

\begin{hypothesis}[Asymptotics of the potential matrices]\quad
\label{hyp:AsymptoticsQ}
 
 There exists a matrix $S$ acting on $\C_0^I$ such that 
 \begin{equation*}
  Q_\varepsilon 
  = \varepsilon^{-1} S + o(\varepsilon^{-1}).
 \end{equation*}
\end{hypothesis}

\begin{definition}[Irreducible potentials]\quad
\label{def:QMatriceIrreductible}
 
 A matrix $S$ acting on $\C_0^I$ is said to be \textit{irreducible} if, for all 
 $i \neq j \in I$,
 \begin{equation*}
   \langle (\mathbf{1}_i-\mathbf{1}_j), S (\mathbf{1}_i-\mathbf{1}_j) \rangle_{\ell^2 (I)} > 0.
 \end{equation*}
\end{definition}

The notion of irreducibility of Definition~\ref{def:QMatriceIrreductible} is very easy to check in the computations of Section~\ref{sec:Calculs}, and quite convenient in the proof of Theorem~\ref{thm:MainTheorem}. In the later theorem, we shall conclude that $S = R_0^{-1}$ for an irreducible bi-L-matrix $R$; it is thus expected that this notion of irreducibility is related to the notion of irreducibility of bi-L-matrices of Sub-subsection~\ref{subsubsec:AsymptotiqueP}.

\begin{lemma}\quad
 \label{lem:R2QIrred}
 
 Let $R$ be an irreducible bi-L-matrix in the sense of Definition~\ref{def:LMatrice}. Then $R_0^{-1}$ is irreducible in the sense of Definition~\ref{def:QMatriceIrreductible}.
\end{lemma}

\begin{proof}
 
 Our proof is probabilistic in nature, using potential theory for finite state Markov chains. Let $I$ be the index set of the matrix $R$. Fix $\varepsilon >0$ such that $P := \Id-\varepsilon R$ is a stochastic matrix. Since $R$ is irreducible, so is $P$. Since $R$ is a bi-L-matrix, $P$ preserves the counting measure 
 on $I$.
 
 \smallskip
 
 Let $i \neq j \in I$. Let $P_{\{i,j\}}$ be the transition matrix for the induced Markov chain on $\{i,j\}$. Then, by~\cite[Corollary~1.1]{PeneThomine:2020}, and using the fact that all solutions of the Poisson equation on $\{i,j\}$ differ by a constant since $P$ (and thus $P_{\{i,j\}}$) is irreducible,
 \begin{equation}
 \label{eq:PotentielMarkov2Sites}
  \sum_{k \in \{i,j\}} (\mathbf{1}_i-\mathbf{1}_j) (k) \cdot (\Id-P)^{-1} (\mathbf{1}_i-\mathbf{1}_j) (k)
  = \sum_{k \in \{i,j\}} (\mathbf{1}_i-\mathbf{1}_j) (k) \cdot (\Id-P_{\{i,j\}})^{-1} (\mathbf{1}_i-\mathbf{1}_j) (k).
 \end{equation}
 On the one hand, $(\Id-P)^{-1} = \varepsilon^{-1} R_0^{-1}$. On the other hand, 
 since $P_{\{i,j\}}$ is stochastic, irreducible and preserves the counting measure, there exist $t>0$ 
 such that
 \begin{equation*}
  P_{\{i,j\}} 
  = \left( \begin{array}{cc} 1-t & t \\ t & 1-t \end{array}\right).
 \end{equation*}
 In particular, $P_{\{i,j\}} f = (1-2t) f$ for all $f \in \C_0^{\{i,j\}}$. 
 Equation~\eqref{eq:PotentielMarkov2Sites} becomes:
 \begin{equation*}
  \langle (\mathbf{1}_i-\mathbf{1}_j), R_0^{-1} (\mathbf{1}_i-\mathbf{1}_j) \rangle_{\ell^2 (I)} 
  = \varepsilon \langle (\mathbf{1}_i-\mathbf{1}_j), (2t)^{-1} (\mathbf{1}_i-\mathbf{1}_j) \rangle_{\ell^2 (I)} 
  = \frac{\varepsilon}{t} 
  > 0. \qedhere
 \end{equation*}
\end{proof}

\subsection{Strategy and main results}

Now that all the necessary objects have been introduced, let us present the strategy we follow 
and the main results that come along. We recall that, as sketched at the end of Sub-subsection~\ref{subsubsec:GeneralTheory}, the strategy to compute the Markov kernel $P_\Sigma$ on $[\Sigma]$ from a Markov kernel $\widetilde{P}$ on $[\Z^d]$ has three main steps:
\begin{itemize}
 \item the balayage identity, which relates $(\Id-\widetilde{P})^{-1}$ and $(\Id-P_\Sigma)^{-1}$;
 \item the ability to solve the Poisson equation on $[\Z^d]$, using the Fourier transform;
 \item the ability to invert $(\Id-P_\Sigma)^{-1}$, or at least extract some information from it.
\end{itemize}

\subsubsection{Koopman operator and Poisson equation}

First, one must choose the right Markov kernel. A first choice is to use
\begin{equation*}
 \widetilde{P}_x 
 := \delta_{\widetilde{T}(x)},
\end{equation*}
whose action on functions is the Koopman operator $\widetilde{\Kcal}$.

\smallskip

However, we then want to apply Theorem~\ref{thm:IdentiteBalayage}, which requires us to find solutions $g \in \Lbb^\infty ([\Z^d], \widetilde{\mu})$ to the equation
\begin{equation}
\label{eq:Cobord}
 f \mathbf{1}_{[\Sigma]} 
 = (\Id-\widetilde{P}) (g) 
 = (\Id-\widetilde{\Kcal}) (g)
 = g - g \circ \widetilde{T} \quad \widetilde{\mu}-\text{almost everywhere},
\end{equation}
where $f\mathbf{1}_{[\Sigma]}$ is supported on $[\Sigma]$. The existence of solutions to this equation means that $f\mathbf{1}_{[\Sigma]}$ is a coboundary for $\widetilde{T}$. Let us assume for now that the system $(A, \mu, T)$ is one of the simplest examples of Gibbs-Markov map: a subshift of finite type endowed with a Gibbs measure. Then the regularity of $g$ can be improved and the coboundary equation holds everywhere using a classical argument~\cite[Chapter~1.2]{Gouezel:2008e}.

\smallskip

On the other hand, we want to take $f \in \Bcal_{[\Sigma],0}$, hence Lipschitz. By Theorem~\ref{thm:IdentiteBalayage}, we then have $\mu_{[\Sigma]}$-almost everywhere
\begin{equation*}
 f 
 = g_{|[\Sigma]} - g_{|[\Sigma]} \circ T_{[\Sigma]},
\end{equation*}
so that
\begin{equation*}
 \Lcal_{[\Sigma]} (f) 
 = (\Lcal_{[\Sigma]}- \Id) (g_{|[\Sigma]}).
\end{equation*}
Then, since $(\Id - \Lcal_{[\Sigma]})$ in invertible on $\Bcal_{[\Sigma],0}$, the later equation admits a solution
\begin{equation*}
 g_0 
 := -(\Id - \Lcal_{[\Sigma]})^{-1} \Lcal_{[\Sigma]} (f) 
 \in \Bcal_{[\Sigma]}.
\end{equation*}
As the kernel of the action of $(\Id - \Lcal_{[\Sigma]})$ on $\Lbb^\infty ([\Sigma], \mu_{[\Sigma]})$ is the space of constant functions by ergodicity, $g_{|[\Sigma]}-g_0$ is constant almost everywhere, so that $g_{|[\Sigma]}$ admits a Lipschitz version.

\smallskip

The same reasoning goes through with any finite $\Sigma \subset \Sigma'$ and yields that $g_{|[\Sigma']}$ admits a Lipschitz version on $[\Sigma']$, whence $g$ admits a locally Lipschitz version on $[\Z^d]$. In particular, taking for $g$ this continuous version, the equality
\begin{equation*}
 f \mathbf{1}_{[\Sigma]} 
 = g - g \circ \widetilde{T}
\end{equation*}
holds everywhere, and not only $\widetilde{\mu}$-almost everywhere. As a consequence, for any point $(x,p) \in [\Z^d]$ of period $n \geq 1$,
\begin{equation}
\label{eq:ContrainteCobord}
 \sum_{k=0}^{n-1} (f \mathbf{1}_{[\Sigma]}) \circ \widetilde{T}^k (x, p) 
 = 0.
\end{equation}

This gives countably many obstructions to the existence of solutions to Equation~\eqref{eq:Cobord}. In addition, if the sites of $\Sigma$ are far apart and $p \in \Sigma$, then there exist periodic orbits which intersect $[\Sigma]$ only in $[p]$. Then, if $f$ is constant on $[p]$, the sum of Equation~\eqref{eq:ContrainteCobord} is zero if and only if $f(\cdot, p) \equiv 0$. As this holds for all $p \in \Sigma$, we get $f \equiv 0$: the only function $f$ which is constant on all sites and satisfy the coboundary equation is the null function.

\smallskip

A solution of this quandary is to use not the Koopman operator $\widetilde{\Kcal}$, but its dual $\widetilde{\Lcal}$.

\subsubsection{Transfer operator and Poisson equation}

The transformation $\widetilde{T}$ is assumed to be ergodic and recurrent, which implies that the Markov kernel $\widetilde{\Kcal}$ is also ergodic and recurrent. The operator $\widetilde{\Lcal}$ is dual to $\widetilde{\Kcal}$ with respect to $\widetilde{\mu}$, and thus is also ergodic and recurrent by~\cite{PeneThomine:2020}. The relevant version of the balayage identity is~\cite[Proposition~0.1]{PeneThomine:2020}:

\begin{proposition}[Balayage identity for the transfer operator]\quad

 Let $(\widetilde{A}, \widetilde{\mu}, \widetilde{T})$ be measure-preserving, with $\widetilde{\mu}$ a recurrent $\sigma$-finite measure. Let $B \subset \widetilde{A}$ be such that $0 < \widetilde{\mu}(B) \leq +\infty$. Let $f$, $g \in \Lbb^\infty(\widetilde{A}, \widetilde{\mu})$ be such that $f \equiv 0$ on $B^c$ and:
 \begin{equation*}
  (\Id - \widetilde{\Lcal}) (g) = f.
 \end{equation*}
 Then:
 \begin{equation*}
  (\Id - \Lcal_B) (g_{|B}) = f_{|B}.
 \end{equation*}
\end{proposition}

Note that, in Section~\ref{sec:FourierTransform}, we shall actually use the more general~\cite[Lemma~1.7]{PeneThomine:2020}, which applies to equations such as $(\Id - z \widetilde{\Lcal}) (g) = f$ with $|z| < 1$. This generalization is only used to deal with convergence issues.

\smallskip

By the discussion at the end of Sub-subsection~\ref{subsubsec:TransferOperators}, for any $f \in \Bcal_{[\Sigma], 0}$, one can find a locally Lipschitz solution $g$ to the equation $(\Id - \widetilde{\Lcal}) (g) = f$, which makes the balayage identity potentially useful. We still have two hurdles to cross:
\begin{itemize}
 \item If we can construct (approximations of) solutions to the equation $(\Id - \Lcal_\Sigma) (g) = f_{|[\Sigma]}$, how do we use them to extract data on $\Lcal_\Sigma$, and in particular an approximation of the transition matrix $P_\Sigma$?
 \item How do we construct (approximations of) solutions to the equation $(\Id - \widetilde{\Lcal}) (g \mathbf{1}_{[\Sigma]}) = f$?
\end{itemize}
Both of these problems will be addressed in an asymptotic regime, when the sites of $\Sigma$ are far away one from the others. We delve deeper into these questions in the next two sub-subsections.

\subsubsection{Towards the main theorem}
\label{subsubsec:MainTheorem}

Our solution to the first of these two problem is Theorem~\ref{thm:MainTheorem}. Let us recall that it states that, for an ergodic and recurrent Markov $\Z^d$-extension of a Gibbs-Markov map and a family $\sigma_\varepsilon : I \hookrightarrow \Z^d$, there is equivalence between:
 \begin{itemize}
  \item there exists an irreducible (in the sense of Definition~\ref{def:LMatrice}) bi-L-matrix $R$ such that $P_\varepsilon = \Id-\varepsilon R+o(\varepsilon)$;
  \item there exists an irreducible (in the sense of Definition~\ref{def:QMatriceIrreductible}) matrix $S$ such that $Q_\varepsilon = \varepsilon^{-1} S + o(\varepsilon^{-1})$.
 \end{itemize}
 If in addition any of these properties holds, then $S = R_0^{-1}$.

 \begin{remark}[Theorem~\ref{thm:MainTheorem} and averaging operator]\quad
  
Using the operators $\Pi_*$ and $\Pi^*$ as well as transfer operators, Theorem~\ref{thm:MainTheorem} 
states that, under an irreducibility condition, the finite-dimensional operators 
$(\Id-\Pi_*\Lcal_\varepsilon \Pi^*)^{-1}$ and $\Pi_* (\Id-\Lcal_\varepsilon)^{-1} \Pi^*$ are equivalent.
 \end{remark}

\begin{remark}[Intrinsic version of Theorem~\ref{thm:MainTheorem}]\quad
 
 The statement of Theorem~\ref{thm:MainTheorem} depends both on the parametrization of the subsets $\Sigma_\varepsilon$ and the precise speed of convergence of $P_\varepsilon$ to $\Id$. While easy to state and convenient for computations, a more intrisic version can be crafted. First, notice that the irreductibility conditions~\ref{hyp:AsymptoticsP} and~\ref{hyp:AsymptoticsQ} are satisfied on open cones in the relevant subspaces of $M_n (\R)$, or in other words, on projectively open sets. We can then restate Theorem~\ref{thm:MainTheorem} as follows.
 
 \smallskip
 
 Consider an ergodic and recurrent Markov $\Z^d$-extension of a Gibbs-Markov map and a family of finite subsets $(\Sigma_\varepsilon)_{\varepsilon > 0}$ all of the same size and such that
 \begin{equation*}
  \lim_{\varepsilon \to 0} \min_{p \neq q \in \Sigma_\varepsilon} |p-q| = + \infty.
 \end{equation*}
 Then there is equivalence between:
 \begin{itemize}
  \item $(\Id-P_\varepsilon)_{\varepsilon > 0}$ is projectively relatively compact for Hypothesis~\ref{hyp:AsymptoticsP}.
  \item $(Q_\varepsilon)_{\varepsilon > 0}$ is projectively relatively compact for Hypothesis~\ref{hyp:AsymptoticsQ}.
 \end{itemize}
 If either condition is satisfied, then $(I-P_\varepsilon)^{-1}_{| \C_0^I} \sim_{\varepsilon \to 0} Q_\varepsilon$.
\end{remark}

 The matrix $Q_\varepsilon$ can be computed when one evaluates the quantities
 \begin{equation*}
  \langle \Pi_* (\Id-\Lcal_\varepsilon)^{-1} \Pi^* (f), g \rangle_{\ell^2 (I)},
 \end{equation*}
 for $f$, $g \in \C_0^I$. Theorem~\ref{thm:MainTheorem} gives the desired asymptotics for $R$. The initial infinite-dimensional problem of inverting $(\Id-\Lcal_\varepsilon)^{-1}$ is reduced to the finite-dimensional problem of inverting $Q_\varepsilon$.

 \smallskip
 
The proof of Theorem~\ref{thm:MainTheorem} spans three sections. First, it will be very convenient to assume that the induced system $(A, \mu, T_{[0]})$ is mixing. This shall be denoted as a separate hypothesis:

\begin{hypothesis}[Mixing]\quad
\label{hyp:Mixing}
 
 The system $(A, \mu, T_{[0]})$ is mixing.
\end{hypothesis}

Section~\ref{sec:MasterLemma} is devoted to the proof of the direct implication in Theorem~\ref{thm:MainTheorem} under Hypothesis~\ref{hyp:Mixing}:

\begin{lemma}\quad
\label{lem:MasterLemma}
 
 Let $([\Z^d], \widetilde{\mu}, \widetilde{T})$ be an ergodic and recurrent Markov $\Z^d$-extension of a Gibbs-Markov map (Hypothesis~\ref{hyp:Recurrence}). Let $I$ be finite and, and let $\sigma_\varepsilon : I \hookrightarrow \Z^d$ for all $\varepsilon > 0$.
 
 \smallskip
 
 Assume furthermore that $([0], \mu_{[0]}, T_{[0]})$ is mixing (Hypothesis~\ref{hyp:Mixing}).
 
 \smallskip
 
 If there exists an irreducible bi-L-matrix $R$ such that $P_\varepsilon = \Id-\varepsilon R+o(\varepsilon)$ (Hypothesis~\ref{hyp:AsymptoticsP}), then $Q_\varepsilon = \varepsilon^{-1} R_0^{-1} + o(\varepsilon^{-1})$. In addition, $R_0^{-1}$ is irreducible in the sense of Definition~\ref{def:QMatriceIrreductible}.
\end{lemma}

 Lemma~\ref{lem:MasterLemma} is the heart of the proof of Theorem~\ref{thm:MainTheorem}. Its proof is inspired by works on the hitting time of small target for hyperbolic systems~\cite{Pitskel:1991, Hirata:1993} and invariant cones for fast-slow systems~\cite{DeSimoiLiverani:2015}; we refer the reader to the introduction of Section~\ref{sec:MasterLemma} for more details.

\smallskip

Once this is done, in Section~\ref{sec:ErgodicStructure} we shall remove this mixing assumption. 
 If $([0], \mu_{[0]}, T_{[0]})$ is not mixing, then it has a decomposition into periodic components. This step requires us to understand what effects the existence of such periodic components have on the whole extension $([\Z^d], \widetilde{\mu}, \widetilde{T})$; to our knowledge, these structural results (in particular Propositions~\ref{prop:StructureColoration} and~\ref{prop:StructureColorationInduits}) are new and of independent interest.

\smallskip
 
This analysis is independent from Section~\ref{sec:MasterLemma}. That said, we think it is clearer to give the proof of Lemma~\ref{lem:MasterLemma} with the additional mixing assumption, and then point out the differences non-mixing implies, than to write directly the proof of Lemma~\ref{lem:MasterLemma} in a non-mixing context. At the end of Section~\ref{sec:ErgodicStructure}, we shall have proved:

\begin{lemma}\quad
\label{lem:MasterLemmaErgodic}
 
 Let $([\Z^d], \widetilde{\mu}, \widetilde{T})$ be an ergodic and recurrent Markov 
 $\Z^d$-extension of a Gibbs-Markov map (Hypothesis~\ref{hyp:Recurrence}). 
 Let $I$ be finite and, and choose $\sigma_\varepsilon : I \hookrightarrow \Z^d$ 
 for all $\varepsilon > 0$.
 
 \smallskip
 
 If there exists an irreducible bi-L-matrix $R$ such that $P_\varepsilon = \Id-\varepsilon R+o(\varepsilon)$ 
 (Hypothesis~\ref{hyp:AsymptoticsP}), then $Q_\varepsilon = \varepsilon^{-1} R_0^{-1} + o(\varepsilon^{-1})$. 
 In addition, $R_0^{-1}$ is irreducible in the sense of Definition~\ref{def:QMatriceIrreductible}.
\end{lemma}

The final step is to reverse the implication, so as to get Theorem~\ref{thm:MainTheorem}. This is the goal of Section~\ref{sec:Inversion}. This uses the $|I| = 2$ case of the theorem (which is deduced directly from Lemma~\ref{lem:MasterLemmaErgodic}) and a compactness argument.

\subsubsection{Applications}
\label{subsubsec:Examples}

We are given an ergodic an recurrent Markov $\Z^d$-extension of a Gibbs-Markov map and a family $(\sigma_\varepsilon)_{\varepsilon >0}$ of injection from $I$ into $\Z^d$. All is left is, given $f$, $g \in \C_0^\Sigma$, to estimate the integrals
\begin{equation*}
 \langle f, Q_\varepsilon (g) \rangle_{\ell^2 (I)}
 = |I| \langle (\Id-\Lcal_\varepsilon)^{-1} \Pi^* (f), \Pi^* (g) \rangle_{\Lbb^2 ([I], \mu_I)} 
\end{equation*}
when $\varepsilon$ vanishes. Thanks to the balayage identity, this reduces formally to the computation of the quantities (identifying $I$ with $\Sigma_\varepsilon$)
\begin{equation*}
 \langle \Pi_* (\Id-\widetilde{\Lcal})^{-1}  \left( \Pi^* (f)\mathbf{1}_{[\Sigma_\varepsilon]} \right), \Pi^* (g) \mathbf{1}_{[\Sigma_\varepsilon]} \rangle_{\Lbb^2 ([\Z^d], \widetilde{\mu})}.
\end{equation*}
As with random walks, this can be done using the Fourier transform. The main difference is that this operation make so-called \textit{twisted transfer operators} appear.

\begin{definition}[Twisted transfer operators]\quad
 
 For $w \in \Tbb^d := \R^d / 2 \pi \Z^d$, define the twisted transfer operator:
 \begin{equation}
  \label{eq:OperateurTransfertTordu}
  \Lcal_w (\cdot)
  := \Lcal (e^{i \langle w, F \rangle} \cdot),
 \end{equation}
 which acts on $\Lbb^p (A, \mu)$ for all $p \in [1, \infty]$, as well as on $\Bcal$. 
\end{definition}

Our applications are done in two steps. In Section~\ref{sec:FourierTransform}, we shall use these twisted transfer operators to give a general expression for the quantities $\langle f, Q_\varepsilon (g) \rangle_{\ell^2 (I)}$. Proposition~\ref{prop:LimitePoissonTF} states that, in our setting, for all $\varepsilon >0$,
\begin{equation*}
 \langle f, Q_\varepsilon (g) \rangle_{\ell^2 (I)} 
 = \lim_{\rho \to 1^-} \frac{1}{(2\pi)^d} \int_{\Tbb^d} \widecheck{\Fcal}_{\sigma_\varepsilon} (f) \Fcal_{\sigma_\varepsilon} (g) \left( \int_A (\Id-\rho \Lcal_w )^{-1} (\mathbf{1})\dd \mu \right) \dd w,
\end{equation*}
where $\Fcal_{\sigma_\varepsilon} (g)$ is the Fourier transform of $(g \circ \sigma_\varepsilon^{-1}) \mathbf{1}_{\Sigma_\varepsilon}$.

\smallskip

In the same spirit as the Nagaev-Guivarc'h proof of the central limit theorem for dynamical systems~\cite{Nagaev:1957, Nagaev:1961, GuivarchHardy:1988}, when the sites are far apart, the main contribution to these integrals will be given by the main eigenvalue $\lambda_w$ of $\Lcal_w$, for $w$ close to $0$. We shall prove that, for any small enough neighborhood $U$ of $0$,
\begin{equation*}
  \langle f, Q_\varepsilon (g) \rangle_{\ell^2 (I)} 
  = \lim_{\rho \to 1^-} \frac{1}{(2\pi)^d} \int_U \widecheck{\Fcal}_{\sigma_\varepsilon} (f) \Fcal_{\sigma_\varepsilon} (g) \frac{1+\delta_w}{1-\rho \lambda_w} \dd w + O_U (1) \norm{f}{} \cdot \norm{g}{},
\end{equation*}
with $\delta_w$ a small error term.

\smallskip

In Section~\ref{sec:Calculs}, we apply the later equation to different classes of jump function $F$ when $\sigma_t (i) = t \sigma (i) + o(1)$ for all $i \in I$. For $d=2$ and $\norm{F}{} \in \Lbb^2 (A, \mu)$, this yields Proposition~\ref{prop:ApplicationD2CarreIntegrable}. For $d=1$ and $F$ in the basin of attraction of a Cauchy random variable\footnote{In addition to some technical constraints}, we get a similar asymptotics, stated in Proposition~\ref{prop:ApplicationD1Cauchy}. We are also able to give asymptotics for $P_\varepsilon$ when $d=1$ and $F$ is square-integrable, in Proposition~\ref{prop:ApplicationD1CarreIntegrable}. Finally, we give some partial results when $d=1$ and $F$ is in the basin of attraction of a L\'evy random variable (i.e. when $F$ has regularly varying tails of index in $(0,1)$). We are able to give an explicit expression for $S$, stated in Proposition~\ref{prop:ApplicationD1Levy}, but we could not find an explicit expression for its inverse. We are still able to compute the asymptotics for a fixed extension and fixed $(\sigma_t)_{t > 0}$, as done in Examples~\ref{ex:LevyAsymetrique} and~\ref{ex:LevySymetrique}.

\section{From transition probabilities to potential}
\label{sec:MasterLemma}

The goal of this section is to prove Lemma~\ref{lem:MasterLemma}: if $([0], \mu_{[0]}, T_{[0]})$ is mixing and $P_\varepsilon = \Id-\varepsilon R + o(\varepsilon)$ for some irreducible bi-L-matrix $R$, then $Q_\varepsilon = \varepsilon^{-1} R_0^{-1} + o(\varepsilon^{-1})$. 

\smallskip

The transition matrix $P_\varepsilon$ can be computed from the transfer operator $\Lcal_\varepsilon$ using the relation $\transposee{P}_\varepsilon = \Pi_* \Lcal_\varepsilon \Pi^*$ (Equation~\eqref{eq:TransferOperatorAveraging}). From another point of view, $\Pi^* \transposee{P}_\varepsilon \Pi_* = (\Pi^* \Pi_*) \Lcal_\varepsilon$ could be seen as a finite-rank approximation of $\Lcal_\varepsilon$. Such an approximation is reminiscent of works on metastable states by Keller and Liverani~\cite{KellerLiverani:2009a} (for two communicating systems), as well as by Dolgopyat and Wright~\cite{DolgopyatWright:2012} (for interval maps).

\smallskip

At first sight, this approximation is not very good: in general, $\norm{(\Id-\Pi^* \Pi_*) \Lcal_\varepsilon}{\Bcal \to \Bcal}$ is of order $1$, because of the non-trivial behaviour of $\Lcal_\varepsilon$ on each site $[i]$. A solution would be to focus on functions which are constant on each site; however, this subspace is not stable under $\Lcal_\varepsilon$.

\smallskip

On each site, the effect of $\Lcal_\varepsilon$ is roughly that of $\Lcal_{[0]}$, which mixes exponentially quickly. The transitions between sites introduce an error of order $\varepsilon$. The interaction of these two dynamical effects imply the existence of cones of functions $C_K (\varepsilon)$ (functions which are $\varepsilon$-close to being constant on each site) which are invariant under $\Lcal_\varepsilon$. By restricting ourselves to such cones, we shall be able improve the bounds on $(\Id-\Pi^* \Pi_*) \Lcal_\varepsilon$. The goal of Subsection~\ref{subsec:ConesTime1} is thus to prove that such cones are stable under the action of $\Lcal_\varepsilon$, and moreover to prove a cone contraction property 
which shall be instrumental in the next steps.

\begin{remark}[Fast-slow systems]\quad
 
 Such families of cones of functions have some precedent in ergodic theory. One feature of the sequence of operators $\Lcal_\varepsilon$ is that the evolution of $(\Lcal_\varepsilon)_{n \geq 0}$ has two characteristic time scales. At a time scale of $\Theta(1)$, the mixing property of $([0], \mu_{[0]}, T_{[0]})$ homogeneizes each site $[i]$, but there is little communication between distinct sites. The transitions between sites occur at a time scale of $\Theta(\varepsilon^{-1})$. When $\varepsilon$ vanishes, these two time scales become decoupled, so that the transitions become independent. 
 
 \smallskip
 
 The existence of two distinct time scales is also a feature of fast-slow systems. For instance, 
 consider a family of maps $F_\varepsilon$ acting on $\Omega \times \Sbb_1$ such that
 \begin{equation*}
  F_\varepsilon (x, \theta) 
  = (f_\varepsilon (x, \theta), \theta + \varepsilon g (x, \theta))
 \end{equation*}
 and $x \mapsto f_\varepsilon (x, \theta)$ is a family of uniformly hyperbolic transformations on $\Omega$. Then the first variable, $x$, evolves on a time scale of $\Theta(1)$, and the second variable, $\theta$, on a time scale of $\Theta(\varepsilon^{-1})$. When $\varepsilon$ vanishes, these two time scales become decoupled, so that the evolution of $\theta$ is averaged over the values of $x$. 
 
 \smallskip
 
 In the context of fast-slow systems, such parametric families of cones -- or, similarly, standard families -- have been introduced. We refer the interested reader to~\cite{DeSimoiLiverani:2015} for an exposition of these techniques, and~\cite{Dolgopyat:2005, ChernovDolgopyat:2009, DeSimoiLiverani:2018} for applications to partially hyperbolic systems and Sinai billiard maps. In particular, we stress the similarity between~\cite[Proposition~3.1]{DeSimoiLiverani:2015} and Corollary~\ref{cor:ConeStability}.
\end{remark}

Once we have proved the stability of these cones, the error term between $\Lcal_\varepsilon$ and $(\Pi^* \Pi_*) \Lcal_\varepsilon$ will be improved. In Subsection~\ref{subsec:ProduitMatrices}, together with a version of the Baker-Campbell-Hausdorff formula, this shall let us control the error between their products on a time scale of $\Theta(\varepsilon^{-1})$. In particular, if $f$ is constant on each site, then, for all positive $t$, the operator $\Lcal_\varepsilon^{\lfloor \varepsilon^{-1} t \rfloor} (f)$ converges to $\Pi^* e^{- t \transposee{R}} \Pi_* f$. This exponential decay is typical of previous work on the hitting time of small targets~\cite{Pitskel:1991, Hirata:1993}. Let us also point out~\cite{KellerLiverani:2009b, BalintGilbertNandoriSzaszToth:2017} for settings where rare events correspond to strong changes in the system, and where the behaviour of the system after those changes has to be controlled.

\smallskip

Finally, in Subsection~\ref{subsec:ConesTimeEpsilon}, the exponential contraction at time scale $\Theta(1)$ together with a coupling argument gets us exponential bounds on the decay of correlations which are uniform in $\varepsilon^{-1} t$. Note that this coupling argument is classical for Markov chains~\cite[Chapter~5]{LevinPeresWilmer:2009} and has a long tradition in ergodic theory (present for instance in most applications of standard pairs; see e.g.~\cite{ChernovMarkarian:2006} for applications to Sinai billiards maps). Once this is done, we are able to relate $(\Id-\Lcal_\varepsilon)^{-1}$ with the inverse of $R^{-1}$ through a Laplace transform, finishing the proof of Lemma~\ref{lem:MasterLemma}.

\subsection{Cone contraction at time scale $\Theta(1)$}
\label{subsec:ConesTime1}

 Let $([\Z^d], \widetilde{\mu}, \widetilde{T})$ be an ergodic and recurrent Markov $\Z^d$-extension of a Gibbs-Markov map (Hypothesis~\ref{hyp:Recurrence}). Let $I$ be finite and, and choose $\sigma_\varepsilon : I \hookrightarrow \Z^d$ 
 for all $\varepsilon > 0$.
 
 \smallskip

 As announced earlier, we introduce cones of functions which are $\varepsilon$-close to being constant on each site, and prove that they are stable under some iterate of $\Lcal_\varepsilon$, as well as a cone contraction property.

\begin{definition}[Cones of functions]\quad
 
 For all $K \geq 0$ and $\varepsilon >0$, let:
 \begin{equation*}
  C_K (\varepsilon) 
  := \left\{ f: \ [I] \to \R : \ f \geq 0 \text{ and } \norm{(\Id - \Pi^* \Pi_*) f}{\Bcal_I} \leq K \varepsilon \norm{f}{\Lbb^1 ([I], \mu_I)} \right\}.
 \end{equation*}
\end{definition}

These subsets of $\Bcal_I$ are indeed convex cones, as we can quickly check:

\begin{lemma}\quad
 \label{lem:ConesConvexes}
 
 For all $K$, $\varepsilon \geq 0$, the set $C_K (\varepsilon)$ is a convex cone.
\end{lemma}

\begin{proof}
 
 Let $K$, $\varepsilon \geq 0$. Given $f \in C_K (\varepsilon)$ and $\lambda \geq 0$,
\begin{equation*}
 \norm{(\Id - \Pi^* \Pi_*) (\lambda f)}{\Bcal_I} 
 = \lambda \norm{(\Id - \Pi^* \Pi_*) f}{\Bcal_I} 
 \leq \lambda K \varepsilon \norm{f}{\Lbb^1 ([I], \mu_I)} 
 = K \varepsilon \norm{\lambda f}{\Lbb^1 ([I], \mu_I)},
\end{equation*}
so that $\lambda f \in \Bcal_I$. In addition, for all $f$, $g \in C_K (\varepsilon)$ and $t \in [0,1]$:
\begin{align*}
 \norm{(\Id - \Pi^* \Pi_*) ((1-t)f+tg)}{\Bcal_I} 
 & \leq (1-t) \norm{(\Id - \Pi^* \Pi_*) f}{\Bcal_I} + t \norm{(\Id - \Pi^* \Pi_*) g}{\Bcal_I} \\
 & \leq K \varepsilon \left( (1-t) \norm{f}{\Lbb^1 ([I], \mu_I)} + t \norm{g}{\Lbb^1 ([I], \mu_I)} \right) \\
 & = K \varepsilon \norm{(1-t)f+tg}{\Lbb^1 ([I], \mu_I)},
\end{align*}
where we used the fact that $\norm{\cdot}{\Lbb^1 ([I], \mu_I)}$ is linear on nonnegative functions. Hence $(1-t)f+tg$ also belongs to $C_K (\varepsilon)$.
\end{proof}

Another fundamental property of these cones is that, since any function in $C_K (\varepsilon)$ is close for the $\Bcal_I$ norm to the finite dimensional space of functions constant on each $[i]$, many norms are equivalent on $C_K (\varepsilon)$. In particular, 
we can control the $\Bcal_I$ norm (the strongest norm we shall use) with the $\Lbb^1 ([I], \mu_I)$ norm (the weakest norm we shall use).

\begin{lemma}\quad
 \label{lem:ControleNorme}
 
 Let $K$, $\varepsilon \geq 0$. For all $f \in C_K (\varepsilon)$,
 \begin{equation}
  \label{eq:BorneBIL1}
  \norm{f}{\Bcal_I} 
  \leq (|I|+K \varepsilon) \norm{f}{\Lbb^1 ([I], \mu_I)}.
 \end{equation}
\end{lemma}

\begin{proof}
 
Let $K$, $\varepsilon \geq 0$. For all $f \in C_K (\varepsilon)$,
\begin{align*}
 \norm{f}{\Bcal_I} 
 & \leq \norm{(\Id-\Pi^* \Pi_*) f}{\Bcal_I} + \norm{\Pi^* \Pi_* f}{\Bcal_I} \\
 & \leq K \varepsilon \norm{f}{\Lbb^1 ([I], \mu_I)} + |I| \norm{\Pi^* \Pi_* f}{\Lbb^1 ([I], \mu_I)} \\
 & = (|I|+K \varepsilon) \norm{f}{\Lbb^1 ([I], \mu_I)}, 
\end{align*}
where we used the fact that $\norm{\cdot}{\Bcal_I} = \norm{\cdot}{\infty} \leq |I| \norm{\cdot}{\Lbb^1 ([I], \mu_I)}$ on functions which are constant on each $[i]$.
\end{proof}

Given these preliminary results, we shall now prove that, under the assumptions of Lemma~\ref{lem:MasterLemma}, the cones $C_K (\varepsilon)$ are contracted under some iterate of $\Lcal_\varepsilon$ whenever $K$ is large enough. Taking $\sigma = 1$ in the following proposition also yield their stability.

\begin{proposition}\quad
\label{prop:FastAsymptotics}
 
 Assume Hypotheses~\ref{hyp:Recurrence}, \ref{hyp:Mixing} and \ref{hyp:AsymptoticsP}.
 
 \smallskip
 
 For all $\sigma >0$, there exist $K_\sigma$, $n_\sigma \geq 0$ such that, for all $K \geq K_\sigma$, for all $\varepsilon>0$ and $n \geq n_\sigma$,
 \begin{equation*}
  \Lcal_\varepsilon^n (C_K (\varepsilon)) 
  \subset C_{\sigma K} (\varepsilon).
 \end{equation*}
\end{proposition}

\begin{proof}
 
 Let $K\geq 0$ and $\varepsilon$, $n$ be positive. Define:
 \begin{equation*}
  B_n 
  := \{(x,i) \in [I] : \ T_\varepsilon^k (x,i) \in [i] \ \forall \ 0 \leq k \leq n\} 
  = \bigsqcup_{i \in I} \bigcap_{k=0}^n T_\varepsilon^{-k} ([i]),
 \end{equation*}
 the set of points which stay in the same site until time $n$. By Hypothesis~\ref{hyp:AsymptoticsP}, there exists a constant $C$ such that, for all $i \in I$ and all $\varepsilon > 0$,
 \begin{equation*}
  \mu (x \in A : \ T_\varepsilon (x,i) \notin [i]) 
  = \sum_{j \neq i} (\delta_{ij}-\varepsilon R_{ij} + o(\varepsilon)) 
  = - \varepsilon \sum_{j \neq i}  R_{ij} + o(\varepsilon) 
  \leq C \varepsilon.
 \end{equation*}
 Hence, 
 \begin{align*}
  \mu_I \left( \left( \bigcap_{k=0}^n T_\varepsilon^{-k} ([i]) \right)^c \right) 
  & = \sum_{k=1}^n \mu_I \left( x : \  T_\varepsilon^0 (x, i), \ldots, T_\varepsilon^{k-1} (x, i) \in [i], \ T_\varepsilon^k (x, i) \notin [i] \right) \\ 
  & \leq \sum_{k=1}^n \mu_I \left( x : T_\varepsilon^{k-1} (x, i) \in [i], \ T_\varepsilon^k (x, i) \notin [i] \right) \\ 
  & = n \mu_I \left( x : T_\varepsilon (x, i) \notin [i] \right) \\ 
  & \leq C |I|^{-1} n \varepsilon.
 \end{align*}
 Summing over $i \in I$ yields $\mu_I (B_n^c) \leq C n \varepsilon$.
 
 \smallskip
 
 Let $i \in I$ and $f \in C_K (\varepsilon)$. As the transfer operator preserves nonnegative functions, $\Lcal_\varepsilon^n (f) \geq 0$. Let us focus on the regularity of $\Lcal_\varepsilon^n (f)$.
 
 \smallskip
 
 First, by definition of $B_n$, given any point $(x,i) \in B_n$, the trajectories of the points $(x,\sigma_\varepsilon(i))$, $T_{[\Sigma_\varepsilon]} (x, \sigma_\varepsilon(i))$, $\ldots$, $T_{[\Sigma_\varepsilon]}^{n-1} (x,\sigma_\varepsilon(i))$ under $\widetilde{T}$ all return to $[\sigma_\varepsilon(i)]$ before hitting any $[\sigma_\varepsilon(j)]$ with $j \neq i$. Hence, by recursion, for all $0 \leq k \leq n$ and all $x \in B_n$, 
 \begin{equation*}
  T_{[\Sigma_\varepsilon]}^k (x,\sigma_\varepsilon(i)) = (T_{[0]}^k (x),\sigma_\varepsilon(i)),
 \end{equation*}
 so that:
 \begin{equation}
 \label{eq:InduitEpsilonZero}
  \Lcal_\varepsilon^n (f \mathbf{1}_{B_n}) (\cdot, i)
  = \Lcal_{[0]}^n ((f \mathbf{1}_{B_n}) (\cdot, i)).
 \end{equation}
 From Equation~\eqref{eq:InduitEpsilonZero} we get: 
 \begin{align*}
  \Lcal_\varepsilon^n (f) (\cdot, i) 
  & = \Lcal_\varepsilon^n (f \mathbf{1}_{B_n}) (\cdot, i) + \Lcal_\varepsilon^n (f\mathbf{1}_{B_n^c}) (\cdot, i) \\
  & = \Lcal_{[0]}^n ((f \mathbf{1}_{B_n}) (\cdot, i))  + \Lcal_\varepsilon^n (f\mathbf{1}_{B_n^c}) (\cdot, i) \\
  & = \Lcal_{[0]}^n (f (\cdot, i)) - \Lcal_{[0]}^n ((f \mathbf{1}_{B_n^c}) (\cdot, i)) + \Lcal_\varepsilon^n (f\mathbf{1}_{B_n^c}) (\cdot, i).
 \end{align*}
 We shall now use Corollaries~\ref{cor:BorneUniformeTempsArret}  and~\ref{cor:DecroissanceCorrelationsInduit0} to control each term. Stopping times are defined in terms of the canonical filtration on $(A, \mu, T)$, while $\Lcal_{[0]}$ and $\Lcal_\varepsilon$ are defined \textit{via} the extension $([\Z^d], \widetilde{\mu}, \widetilde{T})$. We thus need to express some characteristics of the trajectories $(T_\varepsilon^n (x,i))_{n \geq 0}$ in terms of cylinders for $(A, \mu, T)$, as is done in the proof of Proposition~\ref{prop:ActionInduiteQuasiCompacte}.
 
 \smallskip
 
 Let $\varphi_{[0]} (x) := \inf \{n \geq 1 : \ \sum_{k=0}^{n-1} S_k F (x) = 0\}$. Then $\varphi_{[0]}$ is a stopping time in the sense of Definition~\ref{def:TempsDArret}. Since $([\Z^d], \widetilde{\mu}, \widetilde{T})$ is measure-preserving and recurrent, the map $T_{\varphi_{[0]}} = T_{[0]}$ preserves $\mu$ and is ergodic. Finally, by Hypothesis~\ref{hyp:Mixing}, the map $T_{[0]}$ is also mixing for $\mu$. Hence we can apply Corollary~\ref{cor:DecroissanceCorrelationsInduit0}, which asserts that there exist constants $C >0$ and $\rho \in (0,1)$ such that, for all $i \in I$,
 \begin{equation}
 \label{eq:ConeContraction1}
  \norm{\Lcal_{[0]}^n (f (\cdot, i)) - \int_A f (\cdot, i) \dd \mu}{\Bcal} 
  \leq C \rho^n \norm{f(\cdot, i) - \int_A f (\cdot, i) \dd \mu}{\Bcal} 
  \leq C \rho^n \norm{(\Id-\Pi^* \Pi_*) f}{\Bcal_I}.
 \end{equation}
 Let $i \in I$. The set $B_n^c \cap [I]$ is measurable with respect to the $\sigma$-algebra 
 $\alpha (\varphi_{[0]}^{(n)})$, since, using the identification $[i] \simeq A$,
 \begin{equation*}
  B_n \cap [I] 
  = \bigcup_{k = 0}^{n-1} T_{[0]}^{-k} \left( \left\{ \exists 0 \leq \ell < \varphi_{[0]} (x), 
  \ \exists j \in I \setminus \{i\}, \ S_\ell F (x) = \sigma_\varepsilon(j)-\sigma_\varepsilon(i) \right\} \right).
 \end{equation*}
 As $\varphi_{[0]}^{(n)}$ is a stopping time by Lemma~\ref{lem:SommeTempsArret}, we can apply Corollary~\ref{cor:BorneUniformeTempsArret} to get, for some constant $C$ and all $i \in I$,
\begin{equation}
 \label{eq:ConeContraction2}
  \norm{\Lcal_{[0]}^n ((f \mathbf{1}_{B_n^c}) (\cdot, i))}{\Bcal} 
  \leq C \mu(B_n^c \cap [i]) \norm{f(\cdot, i)}{\Bcal}.
 \end{equation}
 In the same way, fix $i$, $j \in I$ and let $\varphi_{j, n, \varepsilon} (x) := \inf \{\ell \geq 1 : \ \Card \{ S_\ell F (x) \in \{j'-j : \ j' \in I\} = n\}$. The set $B_n^c \cap [j]$ is measurable with respect to the $\sigma$-algebra $\alpha (\varphi_{j, n, \varepsilon})$. By Corollary~\ref{cor:BorneUniformeTempsArret}, for some constant $C$,
 \begin{align*}
  \norm{\Lcal_\varepsilon^n (f \mathbf{1}_{B_n^c \cap [j]}) (\cdot, i)}{\Bcal} 
  & = \norm{\Lcal_{\varphi_{j, n, \varepsilon}} (f \mathbf{1}_{B_n^c \cap [j] \cap \{S_{\varphi_{j, n, \varepsilon}} F = i-j \}})}{\Bcal} \\
  & \leq C \mu(B_n^c \cap [j]) \norm{f(\cdot, j)}{\Bcal} \\
  & = C |I| \mu_I (B_n^c \cap [j]) \norm{f(\cdot, j)}{\Bcal}. 
 \end{align*}
 Summing over $j \in I$ finally yields
 \begin{equation}
 \label{eq:ConeContraction3}
  \norm{\Lcal_\varepsilon^n (f \mathbf{1}_{B_n^c}) (\cdot, i)}{\Bcal} 
  \leq C |I| \mu_I (B_n^c) \norm{f}{\Bcal_I}.
 \end{equation}
 Equations~\eqref{eq:BorneBIL1},~\eqref{eq:ConeContraction1},~\eqref{eq:ConeContraction2} and~\eqref{eq:ConeContraction3} together yield, for some positive constant $C$ and all $n \geq 1$:
 \begin{align}
  \norm{\Lcal_\varepsilon^n (f) (\cdot, i) - \int_A f (\cdot, i) \dd \mu}{\Bcal} 
  & \leq C \rho^n \norm{(\Id-\Pi^* \Pi_*) f}{\Bcal_I} + C \mu_I (B_n^c) \norm{f}{\Bcal_I} \nonumber \\
  & \leq C (K \rho^n + n) \varepsilon \norm{f}{\Lbb^1 ([I], \mu_I)}. \label{eq:BorneSuperieureKRhon+n}
 \end{align}
 As $T_\varepsilon$ preserves $\mu_I$ and $f$ is nonnegative, $\norm{\Lcal_\varepsilon^n (f)}{\Lbb^1 ([I], \mu_I)} = \norm{f}{\Lbb^1 ([I], \mu_I)}$. By Equation~\eqref{eq:BorneSuperieureKRhon+n},
 \begin{equation}
  \label{eq:EquationConePrincipale}
  \Lcal_\varepsilon^n (C_K (\varepsilon)) 
  \subset C_{C (K \rho^n + n)} (\varepsilon).
 \end{equation}
 
 \smallskip
  
 Now, we finish the proof of Proposition~\ref{prop:FastAsymptotics}. Let $\sigma >0$. Without loss of generality, we assume that $\sigma \leq 1$. Let $n_\sigma \geq 1$ be such that $C \rho^{n_\sigma} \leq \sigma/2$. Let $K_\sigma := 4 C \sigma^{-1} n_\sigma$. Then, for all $K \geq K_\sigma$ and $\varepsilon >0$, for all $n_\sigma \leq n \leq 2n_\sigma$, by Equation~\eqref{eq:EquationConePrincipale},
 \begin{equation*}
  C ( K \rho^n + n) 
  \leq C ( K \rho^{n_\sigma} + 2 n_\sigma ) 
  \leq \frac{\sigma K}{2} + \frac{\sigma K_\sigma}{2} 
  \leq \sigma K,
 \end{equation*}
 so that $\Lcal_\varepsilon^n (C_K (\varepsilon)) \subset C_{\sigma K} (\varepsilon)$.
 
 \smallskip
 
 By recursion, let us prove that $\Lcal_\varepsilon^{m n_\sigma} (C_K (\varepsilon)) \subset C_K (\varepsilon)$ for all $m \geq 0$. This is true for $m=0$, and if it holds for some $m \geq 0$, then 
 \begin{equation*}
  \Lcal_\varepsilon^{(m+1) n_\sigma} (C_K (\varepsilon)) 
  = \Lcal_\varepsilon^{n_\sigma} \Lcal_\varepsilon^{m n_\sigma} (C_K (\varepsilon)) 
  \subset \Lcal_\varepsilon^{n_\sigma} (C_K (\varepsilon))
  \subset C_{\sigma K} (\varepsilon) 
  \subset C_K (\varepsilon),
 \end{equation*}
 which finishes this recursion.
 
 \smallskip
 
 Finally, for all $n\geq n_\sigma$, since $n_\sigma \leq n-n_\sigma \lfloor n/n_\sigma \rfloor+n_\sigma \leq 2 n_\sigma$,
 \begin{equation*}
  \Lcal_\varepsilon^n (C_K (\varepsilon)) 
  = \Lcal_\varepsilon^{n-n_\sigma \lfloor n/n_\sigma \rfloor+n_\sigma} \Lcal_\varepsilon^{n_\sigma (\lfloor n/n_\sigma \rfloor-1)} (C_K (\varepsilon))  
  \subset \Lcal_\varepsilon^{n-n_\sigma \lfloor n/n_\sigma \rfloor+n_\sigma} (C_K (\varepsilon))
  \subset C_{\sigma K} (\varepsilon). \qedhere
 \end{equation*}
\end{proof}

As a consequence, if $K$ is large enough and up to the loss of a constant factor, $C_K (\varepsilon)$ is stable under $\Lcal_\varepsilon^n$ for all $n \geq 0$. This corollary shall give us a good control on $\norm{\Lcal_\varepsilon^n f}{\Bcal_I}$ provided we know that $f \in C_K (\varepsilon)$.

\begin{corollary}\quad
 \label{cor:ConeStability}
 
 Assume Hypotheses~\ref{hyp:Recurrence}, \ref{hyp:Mixing} and \ref{hyp:AsymptoticsP}.
 
 \smallskip
 
 There exists $K_1$ and $\lambda_1 \geq 1$ such that, for all $K \geq K_1$, for all $n \geq 0$ and $\varepsilon >0$,
 \begin{equation*}
  \Lcal_\varepsilon^n (C_K (\varepsilon)) 
  \subset C_{\lambda_1 K} (\varepsilon).
 \end{equation*}
\end{corollary}

\begin{proof}
 
 Let $n_1$, $K_1$ be given by Proposition~\ref{prop:FastAsymptotics}. By Proposition~\ref{prop:FastAsymptotics}, for all $K \geq K_1$, for all $n \geq n_1$ and all $\varepsilon >0$,
 \begin{equation*}
  \Lcal_\varepsilon^n (C_K (\varepsilon)) 
  \subset C_K (\varepsilon).
 \end{equation*}
 In addition, by Equation~\eqref{eq:EquationConePrincipale}, for all $n \leq n_1$ and $\varepsilon >0$,
 \begin{equation*}
  \Lcal_\varepsilon^n (C_K (\varepsilon)) 
  \subset C_{C (K + n_1)} (\varepsilon).
 \end{equation*}
 We choose $\lambda_1 := C (1 + n_1/K_1)$.
\end{proof}

\subsection{Product of matrices}
\label{subsec:ProduitMatrices}

Given $f \in C_K (\varepsilon)$, Corollary~\ref{cor:ConeStability} asserts that $\Lcal_\varepsilon^n f$ is close to being constant on each site for all $n \geq 0$; the operator $\Lcal_\varepsilon$ is then well approximated by the finite rank operator $\Pi^* \transposee{P}_\varepsilon \Pi_*$. However, to understand what happens at time scale $\varepsilon^{-1}$, we need to control a composition of a large number (of order $\varepsilon^{-1}$) of such operators. To keep reasonably sharp error bounds, we use a 
version of the Baker-Campbell-Hausdorff formula.

\begin{lemma}\quad
\label{lem:BorneCommutateurs}

Let $N \geq 1$. There exists a neighbourhood $U$ of $0$ in $M_N (\C)$ and a constant $C > 0$ such that, for all matrices $A$ and $B$ in $U$,
\begin{equation}
\label{eq:BorneCommutateurs}
 \norm{\ln(e^A e^B) - A - B}{} 
 \leq C \norm{[A,B]}{}.
\end{equation} 
\end{lemma}

The ability to bound $\ln(e^A e^B) - A - B$ using only commutators is natural if one thinks about the Baker-Campbell-Hausdorff formula:
\begin{equation*}
 e^A e^B = e^{A + B + \frac{1}{2} [A,B] + \frac{1}{12} [A, [A,B]] + \frac{1}{12} [B, [B,A]] + \ldots}.
\end{equation*}
We prove it using an integral version of this formula.

\begin{proof}
 
 Whenever $A$ and $B$ are small enough, $\ln(e^A e^B)$ can be written with an integral Baker-Campbell-Hausdorff formula~\cite[Theorem~3.3]{Hall:2015}:
 \begin{equation*}
  \ln (e^A e^B) 
  = A + \int_0^1 u \left( e^{\ad_A} e^{t\ad_B} \right) \dd t \ (B),
 \end{equation*}
 where $u(z) = \frac{z \ln(z)}{z-1}$. Set:
 \begin{equation*}
  v(z) := \frac{z \ln (z) - z + 1}{(z-1)^2} = \frac{u(z)-1}{z-1}, \quad w(z) := \frac{e^z-1}{z}.
 \end{equation*}
 The functions $v$ and $w$ are analytic on a neighbourhood of the identity, and for all small enough $A$ and $B$:
 \begin{align*}
  \ln (e^A e^B) -A - B 
  & = \int_0^1 u \left( e^{\ad_A} e^{t\ad_B} \right) \dd t \ (B) - B \\
  & = \int_0^1 (u-1) \left( e^{\ad_A} e^{t\ad_B} \right) \dd t \ (B) \\
  & = \int_0^1 v \left( e^{\ad_A} e^{t\ad_B} \right) \ (e^{\ad_A} e^{t\ad_B}-1) \ (B) \dd t   \\
  & = \int_0^1 v \left( e^{\ad_A} e^{t\ad_B} \right) \ (e^{\ad_A}-1) \ (B) \dd t \\
  & = \int_0^1 v \left( e^{\ad_A} e^{t\ad_B} \right) \dd t \ w(\ad_A) \ ([A,B]).
 \end{align*}
 Since both $v \left( e^{\ad_A} e^{t\ad_B} \right)$ and $w(\ad_A)$ are uniformly bounded for 
 $A$ and $B$ on a neighbourhood of the identity, we get the claim. 
\end{proof}

Using Lemma~\ref{lem:BorneCommutateurs} recursively, we can control the product of a large number (of order $\Theta(\varepsilon^{-1})$) of matrices which are $o(\varepsilon)$-close from a reference matrix $R$.

\begin{lemma}\quad
\label{lem:ConvergenceProduitMatrices}
 
 Let $N \geq 1$ and $R \in M_N (\C)$. Let $\omega$ be a nonnegative function such that $\omega (\varepsilon) =_{\varepsilon \to 0} o(\varepsilon)$. Let $(R_{\varepsilon, n})_{\varepsilon >0, n \geq 0}$ be a family of matrices in $M_N (\C)$ such that
 \begin{equation*}
  \sup_{n \geq 0} \norm{R_{\varepsilon, n}}{} 
  \leq \omega(\varepsilon).
 \end{equation*}
 For all $\varepsilon >0$ and $t \geq 0$, let:
 \begin{equation*}
  M_\varepsilon (t) 
  := \prod_{n=0}^{\lfloor \varepsilon^{-1} t \rfloor-1} (\Id-\varepsilon R + R_{\varepsilon, n}), 
 \end{equation*}
 where the matrices are multiplied in any order. Then the family of matrix-valued functions $(M_\varepsilon)_{\varepsilon >0}$ converges to $t \mapsto e^{-tR}$ when $\varepsilon$ vanishes, uniformly on all compacts of $\R_+$, and for fixed $\omega$ uniformly in $(R_{\varepsilon, n})_{\varepsilon >0, n \geq 0}$.
\end{lemma}

\begin{proof}
 
 Let $R$ and $(R_{\varepsilon, n})_{\varepsilon >0, n \geq 0}$ be as in the hypotheses of the lemma. 
 Define:
 \begin{equation*}
  R'_{\varepsilon, n} 
  := \ln(\Id-\varepsilon R + R_{\varepsilon, n}) + \varepsilon R.
 \end{equation*}
 Note that $\omega' (\varepsilon) := \sup_{n \geq 0} \norm{R'_{\varepsilon, n}}{} = O(\omega(\varepsilon))$.

 \smallskip
 
 Let $U$ be a neighbourhood of $0$ in $M_N (\C)$ and $C > 0$ as given by Lemma~\ref{lem:BorneCommutateurs}. 
 We put $K := 2C \max \{\norm{R}{}, 1\}$. Let $t_0$ be small enough that $U$ contains $\{-tR : \ 0 \leq t \leq t_0\}$. 
 
 \smallskip
 
 We claim that, for all small enough $\varepsilon$, for all $n \leq \lfloor \varepsilon^{-1} t_0 \rfloor$, 
 the matrix $\prod_{k=0}^{n-1} (\Id-\varepsilon R + R_{\varepsilon, k})$ belongs to $U$, and:
 \begin{equation}
  \label{eq:BorneSke}
  \norm{\ln \left( \prod_{k=0}^{n-1} (\Id-\varepsilon R + R_{\varepsilon, k}) \right) + n \varepsilon R}{} 
  \leq (1+Kn \varepsilon) \omega'(\varepsilon) \frac{(1+2K\varepsilon)^n-1}{2K \varepsilon}.
 \end{equation}
 
 The claim is true for $n = 0$, as the left-hand side vanishes. Assume that the claim holds for 
 some $n \leq \lfloor \varepsilon^{-1} t_0 \rfloor-1$, and put:
  \begin{equation*}
  S_{\varepsilon, n} 
  := \ln \left( \prod_{k=0}^{n-1} (\Id-\varepsilon R + R_{\varepsilon, k}) \right) + n \varepsilon R.
 \end{equation*}
 Define in the same way $S_{\varepsilon, n+1}$. We assume without loss of generality that $\varepsilon$ is small enough that $\omega' (\varepsilon) \leq \varepsilon$. By Lemma~\ref{lem:BorneCommutateurs}, 
 \begin{align*}
  \norm{S_{\varepsilon, n+1}}{} 
  & = \norm{\ln \left( \prod_{k=0}^n (\Id-\varepsilon R + R_{\varepsilon, k}) \right) + (n+1) \varepsilon R}{} \\
  & = \norm{\ln \left( e^{- \varepsilon R + R'_{\varepsilon, n}} e^{- n \varepsilon R + S_{\varepsilon, n}} \right) + (n+1) \varepsilon R}{} \\
  & \leq \norm{- \varepsilon R + R'_{\varepsilon, n} - n \varepsilon R + S_{\varepsilon, n} + (n+1) \varepsilon R}{} + C \norm{[- \varepsilon R + R'_{\varepsilon, n},- n \varepsilon R + S_{\varepsilon, n}]}{} \\
  & \leq \norm{R'_{\varepsilon, n}}{} + \norm{S_{\varepsilon, n}}{} + 2C \left( \varepsilon \norm{R}{} \norm{S_{\varepsilon, n}}{} + \norm{R'_{\varepsilon, n}}{} \norm{S_{\varepsilon, n}}{} + n \varepsilon \norm{R'_{\varepsilon, n}}{} \norm{R}{} \right) \\
  & \leq \omega'(\varepsilon) + \norm{S_{\varepsilon, n}}{} + K \varepsilon \norm{S_{\varepsilon, n}}{} + K n \varepsilon \omega'(\varepsilon) + K \omega'(\varepsilon) \norm{S_{\varepsilon, n}}{} \\
  & \leq (1+2K\varepsilon) \norm{S_{\varepsilon, n}}{} + (1+Kn \varepsilon) \omega'(\varepsilon).
 \end{align*}
 Using the recursion hypothesis, we get:
 \begin{align*}
  \norm{S_{\varepsilon, n+1}}{} 
  & \leq (1+2K\varepsilon) (1+Kn \varepsilon) \omega'(\varepsilon) \frac{(1+2K\varepsilon)^n-1}{2K \varepsilon} + (1+Kn \varepsilon) \omega'(\varepsilon) \\
  & \leq (1+K(n+1) \varepsilon) \omega'(\varepsilon) \frac{(1+2K\varepsilon)^{n+1}-1}{2K \varepsilon}.
 \end{align*}
 Hence, we get the bound we want on $S_{\varepsilon, n+1}$. In addition, $n+1 \leq \lfloor \varepsilon^{-1} t_0 \rfloor$, so that:
 \begin{equation*}
  \norm{S_{\varepsilon, n+1}}{} 
  \leq (1+Kt_0) e^{2 K t_0}\frac{\omega'(\varepsilon)}{2K \varepsilon},
 \end{equation*}
 which converges to $0$ uniformly for $1 \leq n+1 \leq \lfloor \varepsilon^{-1} t_0 \rfloor$. In particular, if $\varepsilon$ is small enough (depending on $t_0$ but not on $n$), $S_{\varepsilon, n+1}$ belongs to $U$. This finishes the proof of our intermediate claim.
 
 \smallskip

 As a consequence, for all small enough $\varepsilon$ and all $t \in [0,t_0]$, the matrix $M_\varepsilon (t)$ belongs to $U$ and satisfies the inequality:
 \begin{equation*}
  \norm{M_\varepsilon (t) - e^{-tR}}{} 
  \leq (1+Kt_0) e^{2 K t_0}\frac{\omega'(\varepsilon)}{2K \varepsilon} + \norm{e^{-tR}}{} \max_{s\in [0, \varepsilon]} \norm{e^{sR}-\Id}{},
 \end{equation*}
 where the right-hand side converges to $0$ when $\varepsilon$ vanishes. The family of functions $(F_\varepsilon)_{\varepsilon >0}$ thus converges to $t \mapsto e^{-tR}$ uniformly on $[0,t_0]$. In addition, since the error term only depends on $(R_{\varepsilon, n})_{\varepsilon >0, n \geq 0}$ through $\omega$, so the convergence is uniform in $(R_{\varepsilon, n})_{\varepsilon >0, n \geq 0}$ at fixed $\omega$.
  
 \smallskip
 
 Finally, let $T >0$. Divide the interval $[0,T]$ into a finite number of subintervals of length less that $t_0$, 
 taking an increasing sequence $0 = s_0 < s_1 < \ldots < s_\ell = T$ with $s_{i+1}-s_i < t_0$. On each such subinterval, the intermediate claim yields:
 \begin{equation*}
  \lim_{\varepsilon \to 0} \sup_{s \in [s_i, s_{i+1}]} \norm{M_\varepsilon (s) M_\varepsilon (s_i)^{-1} - e^{-(s-s_i) R} }{} 
  = 0.
 \end{equation*}
 The concatenation of these bounds for $0 \leq i \leq \ell$ yields the uniform convergence of $M_\varepsilon$ on $[0,T]$, 
 and ends the proof of Lemma~\ref{lem:ConvergenceProduitMatrices}. 
\end{proof}

\begin{remark}\quad
 
 If the family $(R_{\varepsilon, n})_{\varepsilon >0, n \geq 0}$ does not depend on $n$, which in our future application is the case for random walks, then the conclusion of Lemma~\ref{lem:ConvergenceProduitMatrices} follows directly from the identity
 \begin{equation*}
  M_\varepsilon (t) 
  = e^{\lfloor \varepsilon^{-1} t \rfloor \ln \left( \Id-\varepsilon R + R_{\varepsilon, 0} \right)},
 \end{equation*}
 and the fact that $\ln \left( \Id-\varepsilon R + R_{\varepsilon, 0} \right) = \Id-\varepsilon R + o(\varepsilon)$. 
\end{remark}

\subsection{Cone contraction at time scale $\Theta(\varepsilon^{-1})$}
\label{subsec:ConesTimeEpsilon}

In this subsection we prove Lemma~\ref{lem:MasterLemma}. Let us recall briefly the strategy: thanks to the stability estimates of Subsection~\ref{subsec:ConesTime1}, in particular Corollary~\ref{cor:ConeStability}, and to Lemma~\ref{lem:ConvergenceProduitMatrices}, we shall prove that, for all piecewise constant $f$,
\begin{equation*}
 \lim_{\varepsilon \to 0} \Lcal_\varepsilon^{\lfloor \varepsilon^{-1} t \rfloor} f
 = e^{-t(\transposee{R})}f.
\end{equation*}
Then, a coupling argument referenced at the start of the section (for similar arguments, see e.g.~\cite[Chapter~5]{LevinPeresWilmer:2009} in the context of Markov chains, and~\cite{ChernovMarkarian:2006} in the context of hyperbolic dynamics and standard pairs) yields tension: for some constants $C$, $c>0$, if in addition $\int_{[I]}f \dd \mu_I = 0$, then
\begin{equation*}
 \norm{ \Lcal_\varepsilon^{\lfloor \varepsilon^{-1} t \rfloor} f }{\infty}
 \leq C e^{-c t} \norm{f}{\infty}.
\end{equation*}
Integrating over $t \geq 0$ then relates $(\Id-\Lcal_\varepsilon)^{-1}$, and thus $Q_\varepsilon$, with $R_0^{-1}$.

\begin{remark}[Perturbative argument and the spectrum of $\Lcal_\varepsilon$]\quad
 
 As already mentioned, another direction of attack for this problem is to try and adapt the perturbative methods of G.~Keller and C.~Liverani~\cite{KellerLiverani:2009a}, in particular in the form used by D.~Dolgopyat and P.~Wright~\cite{DolgopyatWright:2012}. In this setting, one views $\Lcal_\varepsilon$ as a perturbation of
 \begin{equation*}
  \Lcal_{[0]}^{\oplus |I|} : \left\{
  \begin{array}{rcl}
  \Bcal_I & \to & \Bcal_I \\
  (f (\cdot, i))_{i \in I} & \mapsto & (\Lcal_{[0]} (f (\cdot, i)))_{i \in I}             
 \end{array}
  \right. ,
 \end{equation*}
 whose spectrum $\Sp (\Lcal_{[0]}^{\oplus |I|} \acts \Bcal_I)$ is a $|I|$-fold copy of $\Sp (\Lcal_{[0]} \acts \Bcal)$.
 
 \smallskip
 
 This approach is likely to give a more accurate description of $\Lcal_\varepsilon$, although with no incidence on Theorem~\ref{thm:MainTheorem}. For instance, under Hypothesis~\ref{hyp:Mixing}, we can hope to get an eigendecomposition
 \begin{equation*}
  \Lcal_\varepsilon 
  = \Mcal_\varepsilon + \Ncal_\varepsilon,
 \end{equation*}
 with $\norm{\Mcal_\varepsilon^n}{\Bcal_I} \leq C e^{-cn}$ for some $C$, $c>0$ and all $\varepsilon>0$, and $\Ncal_\varepsilon$ of finite rank, and almost conjugated with $(\Id - \varepsilon \transposee{R})$. In particular, the spectrum of $\Ncal_\varepsilon$ should be $o(\varepsilon)$-close to that of $(\Id - \varepsilon \transposee{R})$, with eigenfunctions close to those of $\Pi^* \transposee{R} \Pi_*$. If this holds, the bound of Equation~\ref{eq:BorneNormeInitiale} below may be improved to 
 \begin{equation*}
  \norm{\Lcal_\varepsilon^n (f)}{\Bcal_{I, 0}} 
  \leq C e^{-\rho \varepsilon n} \norm{f}{\Bcal_{I, 0}},
 \end{equation*}
 with $\rho < \rho_R$ and $C$ depends only on $\rho$.
 
 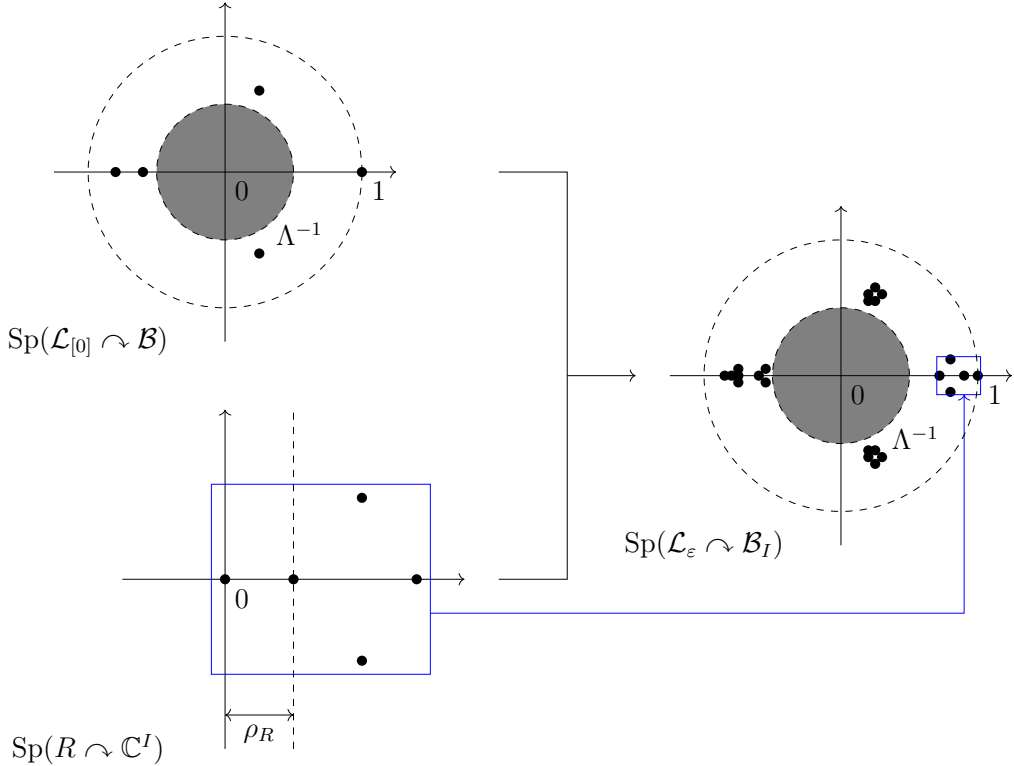
\begin{figure}[!h]
\centering

\scalebox{0.9}{
 \begin{tikzpicture}
 \begin{scope}[yshift = +3cm]
   \draw[->] (-2.5,0) -- (2.5,0) ;
   \draw[->] (0,-2.5) -- (0,2.5) ;
   \foreach \P in {(2,0),(-1.6,0),(0.5,1.2),(0.5,-1.2), (-1.2,0)}
      \node at \P [circle,fill,inner sep=1.5pt]{} ;
   \draw node[below right] at (2,0) {$1$} ;
   \draw node[below right] at (0.6,-0.6) {$\Lambda^{-1}$} ;
   \draw node[below right] at (0,0) {$0$} ;
   \draw [dashed] (0,0) circle (2) ;
   \draw [dashed] (0,0) circle (1) ;
   \fill [black, opacity=0.5] (0,0) circle (1) ;
   \draw node at (-2,-2.5) {$\Sp(\Lcal_{[0]} \acts \Bcal)$} ;
 \end{scope}

 \begin{scope}[yshift = -3cm]
   \draw[->] (-1.5,0) -- (3.5,0) ;
   \draw[->] (0,-2.5) -- (0,2.5) ;
   \foreach \P in {(0,0),(1,0),(2,1.2),(2,-1.2),(2.8,0)}
      \node at \P [circle,fill,inner sep=1.5pt]{} ;
   \draw node[below right] at (0,0) {$0$} ;
   \draw [dashed] (1,-2.5) -- (1,2.5) ;
   \draw[<->] (0,-2) -- (1,-2) ;
   \draw node[below] at (0.5,-2) {$\rho_R$} ;
   \draw[blue] (-0.2, 1.4) -- (3, 1.4) -- (3, -1.4) -- (-0.2, -1.4) -- (-0.2, 1.4) ;
   \draw node at (-2,-2.5) {$\Sp(R \acts \C^I)$} ;
 \end{scope}
 
 \begin{scope}[xshift = +4cm]
   \draw (0,3) -- (1,3) -- (1,-3) -- (0,-3) ;
   \draw[->] (1,0) -- (2,0) ;
   \draw[->, blue] (-1, -3.5) -- (6.8, -3.5) -- (6.8, -0.28) ;
 \end{scope}
 
 \begin{scope}[xshift = +9cm]
   \draw[->] (-2.5,0) -- (2.5,0) ;
   \draw[->] (0,-2.5) -- (0,2.5) ;
   \foreach \P in {(2,0), (1.8, 0), (1.6,0.24), (1.6,-0.24), (1.44,0)}
      \node at \P [circle,fill,inner sep=1.5pt]{} ;
   \draw[blue] (2.04, 0.28) -- (1.4, 0.28) -- (1.4, -0.28) -- (2.04, -0.28) -- (2.04, 0.28) ;
   \foreach \P in {(-1.6, 0), (-1.7, 0), (-1.5, 0), (-1.5, 0.1), (-1.5, -0.1)}
      \node at \P [circle,fill,inner sep=1.5pt]{} ;
   \foreach \P in {(0.4, 1.2), (0.4, 1.1), (0.5, 1.1), (0.6, 1.2), (0.5, 1.3)}
      \node at \P [circle,fill,inner sep=1.5pt]{} ;
   \foreach \P in {(0.4, -1.2), (0.4, -1.1), (0.5, -1.1), (0.6, -1.2), (0.5, -1.3)}
      \node at \P [circle,fill,inner sep=1.5pt]{} ;
   \foreach \P in {(-1.2, 0), (-1.1, 0.1), (-1.1, -0.1)}
      \node at \P [circle,fill,inner sep=1.5pt]{} ;
   \draw node[below right] at (2,0) {$1$} ;
   \draw node[below right] at (0.6,-0.6) {$\Lambda^{-1}$} ;
   \draw node[below right] at (0,0) {$0$} ;
   \draw [dashed] (0,0) circle (2) ;
   \draw [dashed] (0,0) circle (1) ;
   \fill [black, opacity=0.5] (0,0) circle (1) ;
   \draw node at (-2,-2.5) {$\Sp(\Lcal_\varepsilon \acts \Bcal_I)$} ;
 \end{scope}
 \end{tikzpicture}
}
\caption{Conjectural picture for the relationship between the spectra of $\Lcal_{[0]}$, $R$ and $\Lcal_\varepsilon$ with $|I| = 5$. The spectrum of $\Lcal_\varepsilon$ is a perturbation of magnitude $\Theta(\varepsilon)$ of a five-fold copy of $\Sp(\Lcal_{[0]} \acts \Bcal)$, and the spectrum of $R$ can be recovered by zooming in a window of size $\varepsilon$ around $1 \in \Sp(\Lcal_\varepsilon \acts \Bcal_I)$, then applying a central symmetry.}
\end{figure}

Let us finally mention the recent work of H.~Tanaka~\cite{Tanaka:2020}, who recently investigated such multidimensional spectral degenerescences in the context of equilibrium measures of subshifts of finite type.
\end{remark}

Our first step is to show that we can indeed apply Lemma~\ref{lem:ConvergenceProduitMatrices} in our context.

\begin{proposition}\quad
\label{prop:ApproximationParMatrices}
 
 Assume Hypotheses~\ref{hyp:Recurrence}, \ref{hyp:Mixing} and \ref{hyp:AsymptoticsP}.
 
 \smallskip
 
 For all $K \geq 0$, there exists a nonnegative function $\omega$ such that $\omega (\varepsilon) =_{\varepsilon \to 0} o(\varepsilon)$, and with the following property.
 
 \smallskip
 
 Let $(f_\varepsilon)_{\varepsilon >0}$ be a family of functions from $[I]$ to $\R_+$ such that $f_\varepsilon \in C_K (\varepsilon)$ for all small enough $\varepsilon$. Then there exists a sequence of square matrices $(R_{\varepsilon, n})_{\varepsilon >0, n \geq 0}$ such that
 \begin{equation*}
  \sup_{n \geq 0} \norm{R_{\varepsilon, n}}{} 
  \leq \omega(\varepsilon)
\end{equation*}
 and, for all $n \geq 0$,
 \begin{equation*}
  \Pi_* \Lcal_\varepsilon^{n+1} f_\varepsilon 
  = (\Id-\varepsilon (\transposee{R}) + R_{\varepsilon, n}) (\Pi_* \Lcal_\varepsilon^n f_\varepsilon).
 \end{equation*}
\end{proposition}

\begin{proof}
 
 Let $K \geq 0$ and $(f_\varepsilon)_{\varepsilon >0}$ be as in the hypotheses of Proposition~\ref{prop:ApproximationParMatrices}. Without loss of generality, we assume that $K \geq K_1$, where $K_1$ is given by Corollary~\ref{cor:ConeStability}. By Corollary~\ref{cor:ConeStability}, there exists $\lambda_1 \geq 1$ such that $\Lcal_\varepsilon^n f_\varepsilon \in C_{\lambda_1 K} (\varepsilon)$ for all small enough $\varepsilon$. Since $\Lcal_\varepsilon^n f_\varepsilon \in C_{\lambda_1 K} (\varepsilon)$ and $\norm{\Lcal_\varepsilon^n f_\varepsilon}{\Lbb^1([I], \mu_I)} = \norm{f_\varepsilon}{\Lbb^1([I], \mu_I)}$, 
 \begin{equation}
  \label{eq:ApproximationParMatricesIntermediaire}
  \norm{(\Id-\Pi^* \Pi_*) \Lcal_\varepsilon^n f_\varepsilon}{\infty}
  \leq \lambda_1 K \varepsilon \norm{f_\varepsilon}{\Lbb^1([I], \mu_I)}.
 \end{equation}
 
 Set $g_{\varepsilon,n} := (\Id-\Pi^* \Pi_*) \Lcal_\varepsilon^n f_\varepsilon$. By construction, $\Pi_* g_{\varepsilon,n} = 0$. For all $i$, $j \in I$,
 \begin{equation*}
  \left| \int_{[j]} g_{\varepsilon,n} \cdot \mathbf{1}_{[j]} \circ T_\varepsilon \dd \mu \right|
  \leq \norm{g_{\varepsilon,n}}{\infty} \mu (x \in A : \ T_\varepsilon (x,i) \in [j]) 
  = \norm{g_{\varepsilon,n}}{\infty} P_{\varepsilon, ij}.
 \end{equation*}
 Since $\Lcal_\varepsilon$ preserves $\mu_I$, for all $i \in I$,
 \begin{align*}
  \int_{[i]} \Lcal_\varepsilon (g_{\varepsilon,n}) \dd \mu 
  & = |I| \int_{[I]} \mathbf{1}_{[i]} \Lcal_\varepsilon (g_{\varepsilon,n}) \dd \mu_I \\
  & = |I| \int_{[I]} g_{\varepsilon,n} \cdot \mathbf{1}_{[i]} \circ T_\varepsilon \dd \mu_I \\
  & = \sum_{j \in I} \int_{[j]} g_{\varepsilon,n} \cdot \mathbf{1}_{[i]} \circ T_\varepsilon \dd \mu
    - \sum_{j \in I} \int_{[i]} g_{\varepsilon,n} \cdot \mathbf{1}_{[j]} \circ T_\varepsilon \dd \mu \\
  & = \int_{[i]} g_{\varepsilon,n} \dd \mu + \sum_{j \neq i} \left( \int_{[j]} g_{\varepsilon,n} \cdot \mathbf{1}_{[i]} \circ T_\varepsilon \dd \mu - \int_{[i]} g_{\varepsilon,n} \cdot \mathbf{1}_{[j]} \circ T_\varepsilon \dd \mu \right),
 \end{align*}
 so that, by Equation~\eqref{eq:ApproximationParMatricesIntermediaire},
 \begin{align*}
  \left| \int_{[i]} \Lcal_\varepsilon (g_{\varepsilon,n}) \dd \mu - \int_{[i]} g_{\varepsilon,n} \dd \mu \right| 
  & \leq \norm{g_{\varepsilon,n}}{\infty} \sum_{j \neq i} \left( P_{\varepsilon, ij} + P_{\varepsilon, ji} \right) \\
  & \leq \lambda_1 K \varepsilon \norm{\Lcal_\varepsilon^n f_\varepsilon}{\Lbb^1([I], \mu_I)} O(\varepsilon).
 \end{align*}
 Since $\Pi_* g_{\varepsilon,n} = 0$, this later inequality can be read as 
 $\norm{\Pi_* \Lcal_\varepsilon g_{\varepsilon,n}}{\infty} = K O(\varepsilon^2) \norm{\Lcal_\varepsilon^n f_\varepsilon}{\Lbb^1([I], \mu_I)}$.
 
 \smallskip
 
 Recall that $\transposee{P}_\varepsilon = \Pi_* \Lcal_\varepsilon \Pi^*$. We apply 
 $\Pi_* \Lcal_\varepsilon$ to $g_{\varepsilon,n}$:
 \begin{equation*}
  \Pi_* \Lcal_\varepsilon^{n+1} f_\varepsilon 
  = \transposee{P}_\varepsilon \Pi_* \Lcal_\varepsilon^n f_\varepsilon + \Pi_* \Lcal_\varepsilon g_{\varepsilon,n}
  = (\Id-\varepsilon (\transposee{R}) + o(\varepsilon)) \Pi_* \Lcal_\varepsilon^n f_\varepsilon + K O(\varepsilon^2) \norm{\Lcal_\varepsilon^n f_\varepsilon}{\Lbb^1([I], \mu_I)}.
 \end{equation*}
 As $f_\varepsilon$ is nonnegative, $\norm{\Lcal_\varepsilon^n f_\varepsilon}{\Lbb^1([I], \mu_I)} = \norm{\Pi_* \Lcal_\varepsilon^n f_\varepsilon}{\ell^1 (I)}$, 
 so that 
 \begin{equation*}
  \Pi_* \Lcal_\varepsilon^{n+1} f_\varepsilon 
  = (\Id-\varepsilon (\transposee{R}) + o(\varepsilon)+ K O(\varepsilon^2)) \Pi_* \Lcal_\varepsilon^n f_\varepsilon,
 \end{equation*}
 and the error term $o(\varepsilon)+ K O(\varepsilon^2)$ only depends on $K$.
\end{proof}

 Our proximate goal is to understand the sequence $(\Lcal_\varepsilon^n f_\varepsilon)_{n \geq 0}$ when $f_\varepsilon$ has average $0$. In order to apply Proposition~\ref{prop:ApproximationParMatrices}, we shall write a function with average $0$ as the difference of two functions in $C_K (\varepsilon)$. This can be efficiently formalized by introducing an auxiliary family of norms on $\Bcal_I$.
 
 \smallskip
 
 Let $K$, $\varepsilon > 0$. Given a real function $f \in \Bcal_I$, we define its $(K, \varepsilon)$-norm as
 \begin{equation*}
  \norm{f}{K,\varepsilon} 
  := \inf_{\substack{f_+,f_- \in C_K (\varepsilon) \\ f = f_+ - f_-}} \max\{\norm{f_+}{\Lbb^1 ([I], \mu_I)},\norm{f_-}{\Lbb^1 ([I], \mu_I)}\}.
 \end{equation*}
 For complex-valued $f \in \Bcal_I$, we put $\norm{f}{K,\varepsilon} := \max\{ \norm{\Re f}{K, \varepsilon}, \norm{\Im f}{K, \varepsilon} \}$. For instance, if $f \in \Bcal_{I,0}$ is constant on each $[i]$, then $\norm{f}{K,\varepsilon} = \max\{ \norm{\Re f}{\Lbb^1([I], \mu_I)}, \norm{\Im f}{\Lbb^1([I], \mu_I)} \}/2$.

 \begin{lemma}\quad
  \label{lem:EquivalenceNormes}
  
  For any $K$, $\varepsilon >0$, the $(K, \varepsilon)$-norm is equivalent to the $\Bcal_I$-norm.
 \end{lemma}

 \begin{proof}
  
  Let $K$, $\varepsilon >0$. 
  
  \smallskip
  
  Let $f \in \Bcal_I$ with $\norm{f}{K,\varepsilon} <+ \infty$. Assume that $f$ is real-valued. Let $\delta>0$, and write $f = f_+ - f_-$ with $\norm{f_+}{\Lbb^1}$, $\norm{f_-}{\Lbb^1} \leq \norm{f}{K,\varepsilon}+\delta$, and $f_+$, $f_- \in C_K (\varepsilon)$. Then $\norm{(\Id-\Pi^* \Pi_*)f_+}{\Bcal_I} \leq K \varepsilon \norm{f_+}{\Lbb^1}$, so that, by Lemma~\ref{lem:ControleNorme},
  \begin{equation*}
   \norm{f_+}{\Bcal_I} 
   \leq (|I|+K \varepsilon) \norm{f_+}{\Lbb^1 ([I], \mu_I)} 
   \leq (|I|+K \varepsilon)(\norm{f}{K,\varepsilon}+\delta)
  \end{equation*}
  and similarly for $f_-$. Hence, $\norm{f}{\Bcal_I} \leq 2 (|I|+K \varepsilon) (\norm{f}{K,\varepsilon}+\delta)$. Since this is true for all $\delta>0$, we get $\norm{f}{\Bcal_I} \leq 2 (|I|+K \varepsilon) \norm{f}{K,\varepsilon}$. If $f$ is complex-valued, by taking real and imaginary parts, we get 
  \begin{equation}
   \label{eq:BorneBcalI<Keps}
   \norm{f}{\Bcal_I} 
   \leq 4 (|I|+K \varepsilon) \norm{f}{K,\varepsilon}.
  \end{equation}
  
  \smallskip
  
  Let $f \in \Bcal_I$ be real. Let $C := 1+2/(K \varepsilon)$. We choose $f_+ := f+C\norm{f}{\Bcal_I}$ and $f_- := C \norm{f}{\Bcal_I}$. Then $f_- \in C_0 (\varepsilon)$ and 
  \begin{equation*}
   \frac{2 \norm{f}{\Bcal_I}}{K \varepsilon} 
   \leq f_+ 
   \leq \frac{2 (1+K \varepsilon) \norm{f}{\Bcal_I}}{K \varepsilon}.
  \end{equation*}
  We deduce from the lower bound on $f_+$:
  \begin{equation*}
   \norm{(\Id-\Pi^* \Pi_*) f_+}{\Bcal_I} 
   = \norm{(\Id-\Pi^* \Pi_*) f}{\Bcal_I} 
   \leq 2 \norm{f}{\Bcal_I} 
   \leq K \varepsilon \norm{f_+}{\Lbb^1 ([I], \mu_I)}.
  \end{equation*}
  Hence $f_+ \in C_K (\varepsilon)$. Using the upper bound on $f_+$, 
  \begin{equation*}
   \norm{f_+}{\Lbb^1 ([I], \mu_I)} \leq \frac{2 (1+K \varepsilon)}{K \varepsilon} \norm{f}{\Bcal_I}.
  \end{equation*}
  Doing the same with the imaginary part of $f$ finally yields:
  \begin{equation*}
   \norm{f}{K, \varepsilon} 
   \leq \frac{2 (1+K \varepsilon)}{K \varepsilon} \norm{f}{\Bcal_I}. \qedhere
  \end{equation*}  
 \end{proof}

 Thanks to Lemma~\ref{lem:ConvergenceProduitMatrices}, we now apply 
 Proposition~\ref{prop:ApproximationParMatrices}.

\begin{corollary}\quad
\label{cor:ConvergenceUniformeCompacts}
 
 Assume Hypotheses~\ref{hyp:Recurrence}, \ref{hyp:Mixing} and \ref{hyp:AsymptoticsP}.
 
 \smallskip
 
 Let $K$, $C$, $T>0$.
 
 \smallskip
 
 Let $(f_\varepsilon)_{\varepsilon >0}$ be a family of functions in $\Bcal_I$. Assume that $\sup_{\varepsilon>0} \norm{f_\varepsilon}{K, \varepsilon} \leq C$. Then 
 \begin{equation}
 \label{eq:ConvergenceUniformeCompactsDynamique}
  \lim_{\varepsilon \to 0} \sup_{t \in [0,T]} \norm{ \Pi_* \Lcal_\varepsilon^{\lfloor \varepsilon^{-1} t \rfloor} f_\varepsilon - e^{-t(\transposee{R})} \Pi_* f_\varepsilon }{\ell^1 (I)} 
  = 0,
 \end{equation}
 where the convergence is uniform for fixed $K$ and $C$.
\end{corollary}

\begin{proof}
 
 Let $K$, $C$, $T>0$. Let $(f_\varepsilon)_{\varepsilon >0}$ be a family of functions in $\Bcal_I$ such that $\sup_{\varepsilon>0} \norm{f_\varepsilon}{K, \varepsilon} \leq C$.
 
 \smallskip
 
 Let $(\varepsilon_k)$ be a positive sequence converging to $0$. For all $k$, there exist $f_{k,+}$, $f_{k,-}$, $f_{k,i+}$, $f_{k,i-}$ in $C_K (\varepsilon)$ such that $f_{\varepsilon_k} = f_{k,+}-f_{k,-}+if_{k,i+}-if_{k,i-}$ and $\norm{f_{k,+}}{\Lbb^1 ([I], \mu_I)} \leq 2 C$, and likewise for $f_{k,-}$, $f_{k,i+}$, $f_{k,i-}$.
 
 \smallskip
 
 Note that
 \begin{align*}
  \sup_{t \in [0,T]} & \norm{ \Pi_* \Lcal_\varepsilon^{\lfloor \varepsilon_k^{-1} t \rfloor} f_{k,+} - e^{-t(\transposee{R})} \Pi_* f_{k,+} }{\ell^1 (I)} \\
  & \leq \sup_{t \in [0,T]}\norm{ \prod_{n=0}^{\lfloor \varepsilon^{-1} t \rfloor-1} (\Id-\varepsilon_k (\transposee{R}) + R_{\varepsilon_k, n}) - e^{-t(\transposee{R})} }{\ell^1 (I) \to \ell^1 (I)} \norm{\Pi_* f_{k,+}}{\ell^1 (I)}, 
 \end{align*}
 where $(R_{\varepsilon_k, n})_{k,n \geq 0}$ is a sequence of matrices obtained by Proposition~\ref{prop:ApproximationParMatrices}. 
 The quantity $\norm{\Pi_* f_{k,+}}{\ell^1 (I)}$ is equal to $|I| \cdot \norm{f_{k,+}}{ \Lbb^1 ([I], \mu_I)}$, and thus is bounded by $2C|I|$.
 By Lemma~\ref{lem:ConvergenceProduitMatrices}, 
 \begin{equation*}
  \lim_{k \to + \infty} \sup_{t \in [0,T]} \norm{ \Pi_* \Lcal_\varepsilon^{\lfloor \varepsilon_k^{-1} t \rfloor} f_{k,+} - e^{-t(\transposee{R})} \Pi_* f_{k,+} }{\ell^1 (I)} 
  = 0,
 \end{equation*}
 where the convergence is uniform in $C$. The same holds for $f_{k,-}$, $f_{k,i+}$, $f_{k,i-}$. Summing all four limits yields
 \begin{equation*}
  \lim_{k \to + \infty} \sup_{t \in [0,T]} \norm{ \Pi_* \Lcal_\varepsilon^{\lfloor \varepsilon_k^{-1} t \rfloor} f_{\varepsilon_k} - e^{-t(\transposee{R})} \Pi_* f_{\varepsilon_k} }{\ell^1 (I)}
  = 0.
 \end{equation*}
 Since the subsequence $(\varepsilon_k)$ is arbitrary, this finishes the proof.
\end{proof}

Corollary~\ref{cor:ConvergenceUniformeCompacts} gives us convergence of $\Lcal_\varepsilon^{\lfloor \varepsilon^{-1} t \rfloor} f_\varepsilon$ to $e^{-t(\transposee{R})} \Pi_* f_\varepsilon$, in a relatively weak sense, but uniformly on 
all compacts in $t$. As announced, we now implement a coupling argument to show that $(\Lcal_\varepsilon^{\lfloor \varepsilon^{-1} t \rfloor} f_\varepsilon)_{t \geq 0, \varepsilon >0}$ is exponentially tight.

\begin{proposition}\quad
\label{prop:Couplage}
 
 Assume Hypotheses~\ref{hyp:Recurrence}, \ref{hyp:Mixing} and \ref{hyp:AsymptoticsP}.
 
 \smallskip
 
 For any $\rho \in (0, \rho_R)$ and $K \geq 0$, there exist constants $C$, $\varepsilon_0>0$ such that, for all $\varepsilon \leq \varepsilon_0$ and $n \geq 0$, the operator norm of $\Lcal_\varepsilon^n$ acting on $(\Bcal_{I,0}, \norm{\cdot}{K, \varepsilon})$ 
 is no larger than  $C e^{- \rho \varepsilon n}$.
\end{proposition}

\begin{proof}
 
 Let $\rho' \in (\rho, \rho_R)$. Since $R$ is an irreducible bi-L-matrix, so is $\transposee{R}$. By Lemma~\ref{lem:SpectreLMatrice}, there exists a constant $C$ such that, for all $t \geq 0$,
 \begin{equation*}
  \norm{e^{-t(\transposee{R})}-\pi}{\ell^1 (I) \to \ell^1 (I)} 
  \leq C e^{-\rho' t},
 \end{equation*}
 where $\pi (F) = |I|^{-1} \left(\sum_I F\right) \mathbf{1}$ is the eigenprojection associated with the eigenvalue $0$ of $\transposee{R}$. 
 
 \smallskip
  
 Let $\sigma \in (0,1)$, to be fixed at the end of this proof. Let $K_\sigma$ and $n_\sigma$ be given by Lemma~\ref{prop:FastAsymptotics}. Let $T >0$ be such that $C |I| e^{-\rho' T} \leq \sigma/3$. We assume without loss of generality that $K \geq K_\sigma$.
 
 \medskip
 \textbf{A renormalization map}
 \smallskip
 
 Given $f_+ \in C_K (\varepsilon)$, we define~:
 \begin{equation*}
  \Phi_{\sigma,\varepsilon} (f_+) 
  := \sigma^{-1} \left( \Lcal_\varepsilon^{\lfloor \varepsilon^{-1} T \rfloor} f_+ - (1-\sigma) \int_{[I]} f_+ \dd \mu_I \right).
 \end{equation*}
 Our intermediate goal is to prove that $\Phi_{\sigma,\varepsilon}$ maps $C_K (\varepsilon)$ into itself if $\varepsilon$ is small enough.
 
 \smallskip
 
 Let $C'>0$. Let $(\varepsilon_k)_{k \geq 0}$ be a sequence of positive numbers converging to $0$, and $(f_{k,+})_{k \geq 0} \in C_K (\varepsilon_k)$ with $\norm{f_{k,+}}{\Lbb^1 ([I], \mu_I)} \leq C'$. By Corollary~\ref{cor:ConvergenceUniformeCompacts}, 
 \begin{equation*}
  \lim_{k \to + \infty} \sup_{t \in [0,T]} \norm{  \Pi_* \Lcal_{\varepsilon_k}^{\lfloor \varepsilon_k^{-1} t \rfloor} f_{k,+} - e^{-t(\transposee{R})} \Pi_* f_{k,+}}{\ell^1 (I)} 
  = 0,
 \end{equation*}
 uniformly for fixed $K$ and $C'$. By our choice of $T$,
 \begin{equation*}
  e^{-T(\transposee{R})} \Pi_* f_{k,+} 
  \geq \pi (\Pi_* f_{k,+} ) - C e^{-\rho' T} \norm{ \Pi_* f_{k,+}}{\ell^1 (I)} 
  \geq \left(1-\frac{\sigma}{3}\right) \int_{[I]} f_{k,+} \dd \mu_I.
 \end{equation*}
 Hence, if $k$ is large enough, then 
 \begin{equation*}
  \Pi_* \Lcal_{\varepsilon_k}^{\lfloor \varepsilon_k^{-1} T \rfloor} f_{k,+} 
  \geq \left(1-\frac{2\sigma}{3}\right) \int_{[I]} f_{k,+} \dd \mu_I.
 \end{equation*}
 
 \smallskip
 
 If $k$ is large enough, then $\varepsilon_k^{-1} T \geq n_\sigma$, so that $\Lcal_{\varepsilon_k}^{\lfloor \varepsilon_k^{-1} T \rfloor} f_{k,+} \in C_{\sigma K} (\varepsilon_k)$. This has two consequences. The first is that, if $k$ is large enough,
 \begin{align*}
  \Lcal_{\varepsilon_k}^{\lfloor \varepsilon_k^{-1} T \rfloor} f_{k,+} 
  & \geq \Pi_* \Lcal_{\varepsilon_k}^{\lfloor \varepsilon_k^{-1} T \rfloor} f_{k,+} - \sigma K \varepsilon_k \int_{[I]} f_{k,+} \dd \mu_I \\
  & \geq (1-\sigma) \int_{[I]} f_{k,+} \dd \mu_I.
 \end{align*}
 In particular, $\Phi_{\sigma, \varepsilon_k} (f_{k,+}) \geq 0$. The second consequence is that, if $k$ is large enough, then 
 \begin{align*}
  \norm{(\Id-\Pi^* \Pi_*) \Phi_{\sigma, \varepsilon_k} (f_{k,+})}{\Bcal_I} 
  & = \sigma^{-1} \norm{(\Id-\Pi^* \Pi_*) \Lcal_{\varepsilon_k}^{\lfloor \varepsilon_k^{-1} T \rfloor} f_{k,+}}{\Bcal_I} \\
  & \leq \sigma^{-1} \sigma K \varepsilon_k \norm{\Lcal_{\varepsilon_k}^{\lfloor \varepsilon_k^{-1} T \rfloor} f_{k,+}}{\Lbb^1 ([I], \mu_I)} \\
  & = K \varepsilon_k \norm{\Phi_{\sigma, \varepsilon_k} (f_{k,+})}{\Lbb^1 ([I], \mu_I)}.
 \end{align*}
 Hence, $\Phi_{\sigma, \varepsilon_k} (f_{k,+}) \in C_K (\varepsilon_k)$. Since $\Phi_{\sigma, \varepsilon_k}$ is linear, $\Phi_{\sigma,\varepsilon_k} ( C_K (\varepsilon_k)) \subset C_K (\varepsilon_k)$ for all large enough $k$. Finally, since $(\varepsilon_k)_{k \geq 0}$ is arbitrary,  $\Phi_{\sigma,\varepsilon} ( C_K (\varepsilon)) \subset C_K (\varepsilon)$ 
 for all small enough $\varepsilon$.
 
 \medskip
 \textbf{Bound on the operator norm}
 \smallskip
 
 Let $\varepsilon$ be small enough that $\Phi_{\sigma,\varepsilon}$ preserves $C_K (\varepsilon)$. Let $f \in \Bcal_{I,0}$, which we assume without loss of generality to be real. Let $f_+$, $f_- \in C_K (\varepsilon)$ be such that $f = f_+-f_-$ and $\norm{f_+}{\Lbb^1([I],\mu_I)}$, $\norm{f_-}{\Lbb^1([I],\mu_I)} \leq 2 \norm{f}{K, \varepsilon}$. Then 
 \begin{align*}
  \Lcal_\varepsilon^{\lfloor \varepsilon^{-1} T \rfloor} f 
  & = \Lcal_\varepsilon^{\lfloor \varepsilon^{-1} T \rfloor} f_+ - \Lcal_\varepsilon^{\lfloor \varepsilon^{-1} T \rfloor} f_- \\
  & = \left( \Lcal_\varepsilon^{\lfloor \varepsilon^{-1} T \rfloor} f_+  - (1-\sigma) \int_{[I]} f_+ \dd \mu_I \right) - \left( \Lcal_\varepsilon^{\lfloor \varepsilon^{-1} T \rfloor} f_+  - (1-\sigma) \int_{[I]} f_+ \dd \mu_I \right) \\
  & = \sigma \left( \Phi_{\sigma,\varepsilon} (f_+) - \Phi_{\sigma,\varepsilon} (f_-)\right).
 \end{align*}
 By recursion, we get, for all $m \geq 0$:
 \begin{equation*}
  \Lcal_\varepsilon^{m \lfloor \varepsilon^{-1} T \rfloor} f 
  = \sigma^m \left( \Phi_{\sigma,\varepsilon}^m (f_+) - \Phi_{\sigma,\varepsilon}^m (f_-)\right).
 \end{equation*}
 The functions $\Phi_{\sigma,\varepsilon}^m (f_+)$ and $\Phi_{\sigma,\varepsilon}^m (f_-)$ belong to $C_K ( \varepsilon)$ and have the same integral as $f_+$. 
 
 \smallskip
 
 Let $k \geq 0$. By Corollary~\ref{cor:ConeStability}, without loss of generality, there exists a constant $\lambda_1 \geq 1$ 
 such that, and all $\varepsilon >0$,  $k \geq 0$, the functions $\Lcal_\varepsilon^k \Phi_{\sigma,\varepsilon}^m (f_+)$ and $\Lcal_\varepsilon^k \Phi_{\sigma,\varepsilon}^m (f_-)$ belong to $C_{\lambda_1 K} (\varepsilon)$. Hence, setting 
 \begin{equation*}
  \tilde{f}_+ 
  := (\lambda_1-1)\norm{f_+}{\Lbb^1 ([I], \mu_I)} + \Lcal_\varepsilon^k \Phi_{\sigma,\varepsilon}^m (f_+)
 \end{equation*}
 and likewise for $f_-$,
 \begin{equation*}
  \Lcal_\varepsilon^{m \lfloor \varepsilon^{-1} T \rfloor+k} f 
  = \sigma^m ( \tilde{f}_+ - \tilde{f}_- ).
 \end{equation*}
 By construction, 
 \begin{align*}
  \norm{(\Id-\Pi^* \Pi_*) \tilde{f}_+}{\Bcal_I} 
  & = \norm{(\Id-\Pi^* \Pi_*) \Lcal_\varepsilon^k \Phi_{\sigma,\varepsilon}^m (f_+)}{\Bcal_I} \\
  & \leq \lambda_1 K \varepsilon \norm{ \Lcal_\varepsilon^k \Phi_{\sigma,\varepsilon}^m (f_+)}{\Lbb^1 ([I], \mu_I)} \\
  & = \lambda_1 K \varepsilon \norm{ f_+}{\Lbb^1 ([I], \mu_I)} \\
  & = K \varepsilon \norm{\tilde{f}_+}{\Lbb^1 ([I], \mu_I)}.
 \end{align*}
 Since the same estimate holds for $\tilde{f}_-$, the functions $\tilde{f}_+$, $\tilde{f}_-$ belong to $C_K (\varepsilon)$ and their $\Lbb^1$-norm is bounded by $\lambda_1 \norm{f_+}{\Lbb^1 ([I], \mu_I)} \leq 2 \lambda_1 \norm{f}{K, \varepsilon}$. 
 Hence,
 \begin{equation}
  \label{eq:CouplageRapide}
  \norm{\Lcal_\varepsilon^{m \lfloor \varepsilon^{-1} T \rfloor+k} f}{K, \varepsilon}
  \leq 2 \lambda_1 \sigma^m \norm{f}{K, \varepsilon}.
 \end{equation}
 Let $n \geq 0$. Take $m := \lfloor \varepsilon T^{-1} n \rfloor$ and $k := n - m \lfloor \varepsilon^{-1} T \rfloor$.
 Equation~\eqref{eq:CouplageRapide} implies~:
 \begin{equation*}
  \norm{\Lcal_\varepsilon^n f}{K, \varepsilon}
  \leq 2 \lambda_1 \sigma^{\lfloor \varepsilon T^{-1} n \rfloor} \norm{f}{K, \varepsilon}.
 \end{equation*}
 Finally, recall that the only constraint on $T$ is that $C |I| e^{-\rho' T} \leq \sigma/3$. Hence, 
 we can take $T := \rho'^{-1} |\ln(\sigma/(3C|I|))|$, so that:
 \begin{align*}
  \norm{\Lcal_\varepsilon^n f}{K, \varepsilon}
  & \leq 2 \lambda_1 e^{-|\ln(\sigma)|\lfloor \rho' \varepsilon |\ln(\sigma/(3C|I|))|^{-1} n \rfloor} \norm{f}{K, \varepsilon} \\
  & \leq 2 \lambda_1 \sigma^{-1} e^{- \rho' \frac{\ln(\sigma)}{\ln(\sigma) - \ln (3C|I|)} \varepsilon n} \norm{f}{K, \varepsilon}.
 \end{align*}
 Choose $\sigma$ small enough that 
 \begin{equation*}
  \rho' \frac{\ln(\sigma)}{\ln(\sigma) - \ln (3C|I|)} 
  \leq \rho \quad ;
 \end{equation*}
 then 
 \begin{equation*}
  \norm{\Lcal_\varepsilon^n f}{K, \varepsilon} 
  \leq 2 \lambda_1 \sigma^{-1}  e^{- \rho \varepsilon n} \norm{f}{K, \varepsilon}. \qedhere
 \end{equation*}
\end{proof}

\begin{remark}[Bound in operator norm]\quad
\label{rmk:BorneNormeInitiale}
 
 The bound on the operator norm of $\Lcal_\varepsilon^n$ in Proposition~\ref{prop:Couplage} is given for the norm $\norm{\cdot}{K, \varepsilon}$ For the original Lipschitz norm $\norm{\cdot}{\Bcal_I}$, which does not depend on $\varepsilon$, for any $\rho < \rho_R$, the bounds in the proof of Lemma~\ref{lem:EquivalenceNormes} yield a constant $C$ such that:
 \begin{equation}
 \label{eq:BorneNormeInitiale}
  \norm{\Lcal_\varepsilon^n}{\Bcal_{I,0} \to \Bcal_{I,0}} 
  \leq \frac{C}{\varepsilon} e^{- \rho \varepsilon n}.
 \end{equation}
\end{remark}

While not necessary for the remainder of this article, notice that Lemma~\ref{lem:EquivalenceNormes} and Proposition~\ref{prop:Couplage} give an explicit lower bound on the spectral gap of $\Lcal_\varepsilon$. By ergodicity of $([I], \mu_I, T_\varepsilon)$, the operator $\Lcal_\varepsilon$ has $1$ as a simple eigenvalue, corresponding to constant eigenfunctions. 
Define the \emph{spectral gap} of $\Lcal_\varepsilon$ acting on $\Bcal_I$ as, equivalently, one minus the spectral radius of $\Lcal_\varepsilon$ acting on $\Bcal_{I,0}$, that is  
\begin{equation*}
 \min\{ 1-|\lambda| : \ \lambda \in \Sp (\Lcal_\varepsilon \acts \Bcal_I), \ \lambda \neq 1  \}.
\end{equation*}

\begin{corollary}\quad
\label{cor:TrouSpectral}
 
 Assume Hypotheses~\ref{hyp:Recurrence}, \ref{hyp:Mixing} and \ref{hyp:AsymptoticsP}.
 
 \smallskip
 
 When $\varepsilon$ goes to $0$, the spectral gap of $\Lcal_\varepsilon$ acting on $\Bcal_I$ 
 is at least $\rho_R \varepsilon (1-o(1))$.
\end{corollary}

Integrating over $t \geq 0$, we can now relate $S$ with $R_0^{-1}$.

\begin{corollary}\quad
\label{cor:EquivalencePotentielInduit}
 
 Assume Hypotheses~\ref{hyp:Recurrence}, \ref{hyp:Mixing} and \ref{hyp:AsymptoticsP}.
 
 \smallskip
 
 Let $K \geq 0$. Let $(f_\varepsilon)_{\varepsilon >0}$ be a family of functions in $\Bcal_{I,0}$.
 Assume that $\sup_{\varepsilon >0} \norm{f_\varepsilon}{K, \varepsilon} < +\infty$. Then, in $\Bcal_{I,0}$,
 \begin{equation}
  \label{eq:EquivalencePotentielInduit}
  \sum_{n = 0}^{+ \infty} \Lcal_\varepsilon^n f_\varepsilon 
  = \frac{1}{\varepsilon} \Pi^* (\transposee{R}_0^{-1}) \Pi_* f_\varepsilon + o (\varepsilon^{-1}).
 \end{equation}
\end{corollary}

\begin{proof}

 Changing the time scale by a factor of $\varepsilon^{-1}$, we get:
 \begin{equation*}
  \sum_{n = 0}^{+ \infty} \Lcal_\varepsilon^n f_\varepsilon 
  = \frac{1}{\varepsilon} \int_0^{+\infty} \Lcal_\varepsilon^{\lfloor \varepsilon^{-1} t \rfloor} f_\varepsilon \dd t.
 \end{equation*}

 By Corollary~\ref{cor:ConvergenceUniformeCompacts}, on all compact subsets of $\R_+$,
 \begin{equation*}
   \lim_{\varepsilon \to 0} \sup_{t \in [0,T]} \norm{ \Pi_* \Lcal_\varepsilon^{\lfloor \varepsilon^{-1} t \rfloor} f_\varepsilon - e^{-t(\transposee{R})} \Pi_* f_\varepsilon }{\ell^1 (I)} 
   = 0.
 \end{equation*}
 By Proposition~\ref{prop:Couplage}, there exists a constant $C$ such that $\norm{\Lcal_\varepsilon^{\lfloor \varepsilon^{-1} t \rfloor} f_\varepsilon}{K, \varepsilon} 
 \leq C \norm{f}{K, \varepsilon}$. We write $\Lcal_\varepsilon^{\lfloor \varepsilon^{-1} t \rfloor} f_\varepsilon = f_+-f_-$ 
 with $f_+$, $f_- \in C_K (\varepsilon)$ and $\norm{f_+}{\Lbb^1 ([I], \mu_I)}$, $\norm{f_-}{\Lbb^1 ([I], \mu_I)} \leq 2 \norm{f}{K, \varepsilon}$. 
 Up to taking a larger value of $K$, the function $\Lcal_\varepsilon^{\lfloor \varepsilon^{-1} t \rfloor} f_+$ belongs to 
 $C_{\lambda_1 K} (\varepsilon)$ for all $t$ by Corollary~\ref{cor:ConeStability}. Hence,
 \begin{equation*}
  \norm{(\Id-\Pi^* \Pi_*) \Lcal_\varepsilon^{\lfloor \varepsilon^{-1} t \rfloor} f_+}{\Bcal_I} 
  \leq \lambda_1 K \varepsilon \norm{\Lcal_\varepsilon^{\lfloor \varepsilon^{-1} t \rfloor} f_+}{\Lbb^1 ([I], \mu_I)} 
  \leq 2 \lambda_1 K \varepsilon \norm{f}{K, \varepsilon}.
 \end{equation*}
 The function $f_-$ satisfies a similar estimate. It follows that 
 \begin{equation*}
  \norm{(\Id- \Pi^*\Pi_*) \Lcal_\varepsilon^{\lfloor \varepsilon^{-1} t \rfloor} f_\varepsilon }{\Bcal_I} 
  \leq 4 \lambda_1 K \varepsilon \norm{f}{K, \varepsilon}.
 \end{equation*}
 Hence,
 \begin{equation}
  \lim_{\varepsilon \to 0} \sup_{t \in [0,T]} \norm{ \Lcal_\varepsilon^{\lfloor \varepsilon^{-1} t \rfloor} f_\varepsilon - \Pi^* e^{-t(\transposee{R})} \Pi_* f_\varepsilon }{\Bcal_I} 
  = 0.
 \end{equation}
 In addition, by Proposition~\ref{prop:Couplage} again, there exist constants $C$, $\rho>0$ such that the operator norm of $\Lcal_\varepsilon^n$ acting on $(\Bcal_{I,0}, \norm{\cdot}{K, \varepsilon})$ is no larger than  $C e^{- \rho \varepsilon n}$. Using Equation~\eqref{eq:BorneBcalI<Keps},
 \begin{equation*}
  \norm{\Lcal_\varepsilon^{\lfloor \varepsilon^{-1} t \rfloor} f_\varepsilon}{\Bcal_I} 
  \leq 4(1+K\varepsilon) \norm{\Lcal_\varepsilon^{\lfloor \varepsilon^{-1} t \rfloor} f_\varepsilon}{K, \varepsilon}
  \leq 4(1+K\varepsilon) C e^{\rho \varepsilon} e^{- \rho t} \norm{f}{K, \varepsilon}.
 \end{equation*}
 This gives the tightness needed for the convergence of integrals:
 \begin{equation*}
  \lim_{\varepsilon \to 0} \norm{ \int_0^{+\infty} \Lcal_\varepsilon^{\lfloor \varepsilon^{-1} t \rfloor} f_\varepsilon \dd t - \int_0^{+\infty} \Pi^* e^{-t(\transposee{R})} \Pi_* f_\varepsilon \dd t}{\Bcal_I} 
  = 0.
 \end{equation*}
 Finally, denoting by $\transposee{R}_0$ the action of $\transposee{R}$ on $\C_0^I$, to which $\Pi_* f_\varepsilon$ belongs,
 \begin{equation*}
  \int_0^{+\infty} \Pi^* e^{-t(\transposee{R})} \Pi_* f_\varepsilon \dd t 
  = \Pi^* \left( \int_0^{+\infty} e^{-t(\transposee{R}_0)} \dd t \right) \Pi_* f_\varepsilon 
  = \Pi^* (\transposee{R}_0^{-1}) \Pi_* f_\varepsilon.
 \end{equation*}
 We get the claim by multiplying both side by $\varepsilon^{-1}$:
 \begin{equation*}
  \norm{ \frac{1}{\varepsilon} \int_0^{+\infty} \Lcal_\varepsilon^{\lfloor \varepsilon^{-1} t \rfloor} f_\varepsilon \dd t - \frac{1}{\varepsilon} \Pi^* (\transposee{R}_0^{-1}) \Pi_* f_\varepsilon}{\Bcal_I} 
  = o(\varepsilon^{-1}). \qedhere
 \end{equation*}
\end{proof}

We now finish the proof of Lemma~\ref{lem:MasterLemma}.

\begin{proof}[Proof of Lemma~\ref{lem:MasterLemma}]
 Given $f \in \C_0^I$, the quantity $\norm{\Pi^* f}{K, \varepsilon}$ 
is bounded jointly in $K$ and $\varepsilon$, and Corollary~\ref{cor:EquivalencePotentielInduit} applies. 
By Lemma~\ref{lem:R2QIrred}, $\transposee{R}_0^{-1}$ is irreducible in the sense of Definition~\ref{def:QMatriceIrreductible}. This yields, for all $f$, $g \in \C_0^I$,
\begin{align*}
 \varepsilon^{-1} \langle f, S (g) \rangle_{\ell^2 (I)} + o(\varepsilon^{-1}) 
 & = \langle f, Q_\varepsilon (g) \rangle_{\ell^2 (I)} \\
 & = \langle \Pi_* (\Id - \Lcal_\varepsilon)^{-1} \Pi^* (f), g \rangle_{\ell^2 (I)} \\
 & = \varepsilon^{-1} \langle \Pi_* \Pi^* (\transposee{R}_0^{-1}) \Pi_* \Pi^* (f), g \rangle_{\ell^2 (I)}  + o (\varepsilon^{-1}) \\
 & = \varepsilon^{-1} \langle f, R_0^{-1} (g) \rangle_{\ell^2 (I)}  + o (\varepsilon^{-1}),
\end{align*}
whence $S = R_0^{-1}$.
\end{proof}

\section{The ergodic case}
\label{sec:ErgodicStructure}

All the work in Section~\ref{sec:MasterLemma} was done under Hypothesis~\ref{hyp:Mixing}, that is, that the first return map $(A, \mu, T_{[0]})$ is mixing. This hypothesis is inconvenient, as we would like to work only with conditions involving the initial data: the system $(A, \mu, T)$ and the function $F$.

\smallskip

If $([\Z^d], \widetilde{\mu}, \widetilde{T})$ is recurrent and ergodic, then $(A, \mu, T_{[0]})$ is well-defined and ergodic. Hence, ergodicity comes freely, and we would like to use this property instead of the stronger mixing hypothesis. This is the goal of this section.

\subsection{Structure of extensions with non-mixing first return maps}

Let $([\Z^d], \widetilde{\mu}, \widetilde{T})$ be a Markov $\Z^d$-extension of a Gibbs-Markov map $(A, \alpha, d, \mu, T)$ with jump function $F$. Assume that $([\Z^d], \widetilde{\mu}, \widetilde{T})$ is ergodic and recurrent. As the induced map $(A, \mu, T_{[0]})$ is well-defined and ergodic, by Corollary~\ref{cor:DecroissanceCorrelationsInduit0Ergodique}, there exists an integer $M \geq 1$ and a partition $A = \bigsqcup_{\ell \in \Z_{/ M \Z}} A_\ell$ such that each $A_\ell$ is $\alpha^*$-measurable, 
$T_{[0]} (A_\ell) \subset A_{\ell+1}$ for all $\ell$, and $T_{[0]}^M : A_\ell \to A_\ell$ is mixing for all $\ell$. 

\smallskip

At first, we show the constraints that the existence of a non-trivial period for $(A, \mu, T_{[0]})$ impose on the whole extension $([\Z^d], \widetilde{\mu}, \widetilde{T})$. We introduce a colouration of $[\Z^d]$.

\begin{definition}[Color of a site]\quad
 
 Let $([\Z^d], \widetilde{\mu}, \widetilde{T})$ be a Markov $\Z^d$-extension of a Gibbs-Markov map $(A, \alpha, d, \mu, T)$. Assume that $([\Z^d], \widetilde{\mu}, \widetilde{T})$ is ergodic and recurrent. Let $A = \bigsqcup_{\ell \in \Z_{/M\Z}} A_\ell$ be the decomposition of $(A, \mu, T_{[0]})$ associated with its period $M$. For $\widetilde{\mu}$-almost every $(x,p) \in [\Z^d]$, let
 \begin{equation}
  \widetilde{\varphi}_{[0]} (x,p) 
  := \inf \{n \geq 1 : \ \widetilde{T}^n (x,p) \in [0]\}.
 \end{equation}
 The \emph{colour} of $(x,p)$ is defined as the element $\ell \in \Z_{/M\Z}$ such that $\widetilde{T}^{\widetilde{\varphi}_{[0]} (x,p)} (x,p) \in A_\ell$. We shall denote it by $C(x,p)$.
\end{definition}

Interestingly, the structural constraints we get are different depending on whether $d=1$ or $d=2$, the later case being simpler. As the following proposition states, in the $d=2$ case, all sites have the same colour, up to a neighbourhood of $[0]$. In the $d=1$ case, we may need two colours. This difference comes from the fact that the plane minus a point is connected, while the line minus a point has two connected components.

\begin{proposition}\quad
\label{prop:StructureColoration}
 
 Let $([\Z^d], \widetilde{\mu}, \widetilde{T})$ be a Markov $\Z^d$-extension of a Gibbs-Markov map $(A, \alpha, d, \mu, T)$ with jump function $F$. Assume that $([\Z^d], \widetilde{\mu}, \widetilde{T})$ is ergodic and recurrent.
 
 \smallskip
 
 If $d = 2$, or $d = 1$ and $\norm{F}{\Lbb^\infty (A, \mu)} = +\infty$, then there 
 exists $R > 0$ and $\ell \in \Z_{/M\Z}$ such that $C(x,p) = \ell$ whenever $\norm{p}{} \geq R$.
 
 \smallskip
 
 If $d = 1$ and $\norm{F}{\Lbb^\infty (A, \mu)} < +\infty$, then there 
 exists $R > 0$ and $\ell_-$, $\ell_+ \in \Z_{/M\Z}$ (not necessarily distinct) 
 such that $C(x,p) = \ell_-$ whenever $p \leq -R$ and $C(x,p) = \ell_+$ whenever $p \geq R$.
\end{proposition}

\begin{proof}

We work under the hypotheses of Proposition~\ref{prop:StructureColoration}. First, let us note that $C \circ \widetilde{T} (x,p) = C (x,p)$ for all $(x,p) \in [\Z^d] \setminus [0]$. 

\smallskip

By ergodicity and recurrence, $[\Z^d]$ is swept by the forward images of $[0]$, that is, 
\begin{equation*}
 [\Z^d] 
 = \bigcup_{x \in A} \bigsqcup_{k=0}^{\varphi_{[0]} (x)-1} \widetilde{T}^k (x,0).
\end{equation*}
Let $\ell \in \Z_{/M\Z}$. The property that $T_{[0]} (A_{\ell-1}) = A_\ell$ implies that, for all $x \in A_{\ell-1}$ and all $0 \leq k < \varphi_{[0]} (x)$, 
\begin{equation*}
 \widetilde{T}^{\widetilde{\varphi}_{[0]} (\widetilde{T}^k (x,0))} \circ \widetilde{T}^k (x,0) 
 = \widetilde{T}^{\varphi_{[0]} (x)-k} \circ \widetilde{T}^k (x,0) 
 = T_{[0]} (x) 
 \in A_\ell,
\end{equation*}
so that $C (\widetilde{T}^k (x,0)) = \ell$. Doing this for all $\ell$ yields $\{C=\ell\} = \bigcup_{x \in A_{\ell-1}} \bigsqcup_{k=0}^{\varphi_{[0]} (x)-1} \widetilde{T}^k (x,0)$: modulo a negligible subset, the points of colour $\ell$ are exactly the forward images of $A_{\ell-1}$ up to time $\varphi_{[0]}-1$.

\smallskip

Let $\ell \in \Z_{/M\Z}$. We can also write the later fact as
\begin{equation*}
 \{C=\ell\} 
 = \bigcup_{k = 0}^{+ \infty} \widetilde{T}^k (A_{\ell-1} \cap \{\varphi_{[0]} > k\}).
\end{equation*}
Since $\varphi_{[0]}$ is a stopping time, $\{\varphi_{[0]} > k\}$ is $\sigma(\alpha^{(k)})$-measurable. 
The set $A_{\ell-1}$ is also $\sigma(\alpha^*)$-measurable by Corollary~\ref{cor:DecroissanceCorrelationsInduit0Ergodique}, and thus also $\sigma(\alpha^{(k)})$-measurable. The function $S_k F$ is constant on each cylinder of length $k$, so that we can write it as a function of the cylinder. Hence, setting $B(p, \ell) := \{x \in A : \ C(x,p) =\ell \}$, we have
\begin{equation}
\label{eq:ErgodiciteImagesDirectes}
 B(p, \ell) 
 = \bigcup_{k = 0}^{+ \infty} \bigsqcup_{\substack{a \in \alpha^{(k)} \\ a \subset A_{\ell-1} \cap \{\varphi_{[0]} > k\} \\ S_k (a) = p}} T^k (a).
\end{equation}
By the big image property of Gibbs-Markov maps, for all $k \geq 0$ and nonempty $a \in \alpha^{(k)}$, we have $\mu (T^k (a)) \geq \delta_0 := \inf_{a \in \alpha} \mu (T (a)) > 0$. Each subset $T^k (a)$ in Equation~\eqref{eq:ErgodiciteImagesDirectes} has measure at least $\delta_0$. Hence, for any $p \in \Z^d$ and $\ell \in \Z_{/M\Z}$, we have a dichotomy:
\begin{itemize}
 \item either $\mu(B(p, \ell)) = 0$;
 \item or $\mu(B(p, \ell)) \geq \delta_0 >0$.
\end{itemize}
In other words, on each site, either a colour is absent, or it is present in a quantity of at least $\delta_0$.

\smallskip

Let $p \in \Z^d \setminus \{0\}$ and $\ell$ such that $\mu(B(p, \ell)) >0$. From the expression in Equation~\eqref{eq:ErgodiciteImagesDirectes}, the set $B(p, \ell)$ is $\sigma(\alpha^*)$-measurable, so that $\norm{\mathbf{1}_{B(p, \ell)}}{\Bcal} = 1$. Since $T_{[0]}$ is ergodic, by Corollary~\ref{cor:DecroissanceCorrelationsInduit0Ergodique}, there exists $M \geq 1$,  $C>0$ and $\rho \in (0,1)$, independent from $p$ and $\ell$, such that
\begin{equation*}
 \norm{\frac{1}{M} \sum_{k=0}^{M-1} \Lcal_{[0]}^{n+k} (\mathbf{1}_{B(p, \ell)}) - \mu(B(p, \ell))}{\Lbb^\infty} 
 \leq C \rho^n \norm{\mathbf{1}_{B(p, \ell)}}{\Bcal} 
 = C \rho^n.
\end{equation*}
Hence, there exists $n_0 > 0$ such that $\sum_{k=0}^{n_0-1} \Lcal_{[0]}^k (\mathbf{1}_{B(p, \ell)}) \geq \delta_0/2 > 0$ almost everywhere for all $p \in \Z^d \setminus \{0\}$ and $\ell$ such that $\mu(B(p, \ell)) >0$. But the support of the function $\sum_{k=0}^{n_0-1} \Lcal_{[0]}^k (\mathbf{1}_{B(p, \ell)})$ is $\bigcup_{k=0}^{n_0-1} T_{[0]}^k (B(p, \ell))$, so that, almost everywhere, 
\begin{equation*}
 \bigcup_{k=0}^{n_0-1} T_{[0]}^k (B(p, \ell)) 
 = A.
\end{equation*}

Let us define $\varphi_{[p] \to [0]} (x) := \inf \{n \geq 1 : T^n_{[\{p,0\}]} (x,p) \in [0]\}$. Since the probability of hitting $[0]$ before going back to $[p]$ converges to $0$ as $p$ goes to infinity by~\cite[Lemma~4.17]{PeneThomine:2019}, the random variable $\varphi_{[p] \to [0]}$ converges in distribution to $+ \infty$ when $p$ goes to infinity. In particular, there exists $r_0 >0$ such that $\mu (\varphi_{[p] \to [0]} < n_0) < \delta_0/n_0$ whenever $\norm{p}{} > r_0$.

\smallskip

Assume that $\norm{p}{} > r_0$. Then 
\begin{equation*}
 \mu \left( \bigcup_{k=0}^{n_0-1} T_{[0]}^k \left(B(p, \ell) \cap \{\varphi_{[p] \to [0]} \geq n_0\} \right) \right) 
 \geq 1-n_0 \mu (\varphi_{[p] \to [0]} < n_0) 
 > 1-\delta_0.
\end{equation*}
In addition, for $x \in B(p, \ell) \cap \{\varphi_{[p] \to [0]} \geq n_0\}$ and $0 \leq k < n_0$, we have $C(T_{[\{0,p\}]}^k (x,p)) = C(x,p) = \ell$, and thus $\mu (B(p, \ell)) > 1-\delta_0$. By the dichotomy on $\mu (B(p, \ell))$, we thus have $\mu (B(p, \ell')) = 0$ for $\ell' \neq \ell$, so that $\mu (B(p, \ell)) = 1$. In other words, if $\norm{p}{} > r_0$, then $[p]$ is monochromatic.

\smallskip

Let $\leftrightarrow$ be the nearest neighbor relation in $\Z^d$. For all $q \leftrightarrow 0$, the random variable $\varphi_{[p] \to [p+q]}$ does not depend on $p$, while $\varphi_{[p] \to [0]}$ converges in distribution to $+ \infty$ as $p$ goes to infinity. In particular, there exists $r_1$ such that $\mu (\varphi_{[p] \to [p+q]} < \varphi_{[p] \to [0]}) > 0$ for all $\norm{p}{} > r_1$ and $q \leftrightarrow 0$.

\smallskip

Let $p \in \Z^d$ with $\norm{p}{}> R := \max\{r_0+1,r_1\}$. Then $[p]$ is monochromatic; let $\ell$ be its colour. For each $p' \leftrightarrow p$, we have $\mu (\varphi_{[p] \to [p']} < \varphi_{[p] \to [0]}) > 0$, so that $[p']$ also contains the colour $\ell$. But $\norm{p'}{} > r_0$, so $[p']$ is also monochromatic, and thus is the same colour as $[p]$.

\smallskip

At this point, our conclusion depends on the dimension.
\begin{itemize}
 \item For the case $d = 2$: the graph $\Z^2 \setminus B(0, R)$ is connected, so that $[\Z^2 \setminus B(0, R)]$ is monochromatic.
 \item For the case $d = 1$: the graph $\Z \setminus B(0, R)$ has two connected components, so that the sets $[(-\infty, -R) \cap \Z]$ and $[(R, +\infty) \cap \Z]$ are both monochromatic.
\end{itemize}
In the later case, assume furthermore that $\norm{f}{\Lbb^\infty (A, \mu)} = + \infty$; without loss of generality, assume that the essential supremum of $F$ is $+ \infty$. Then there exists $p \in (-\infty, -R) \cap \Z$ and $q \in (R, +\infty) \cap \Z$ such that $\mu ([p] \cap \widetilde{T}^{-1} ([q])) > 0$. In particular, $[p]$ and $[q]$ have the same colour, so $[\Z \setminus B(0, R)]$ is monochromatic.
\end{proof}

Proposition~\ref{prop:StructureColoration} gives very strong constraints on the structure of a Markov $\Z^d$-extension whose induced map in $[0]$ is not mixing. We now prove that, under these constraints, the induced map $T_{[\Sigma]}$ itself has a decomposition in periodic components, as long as the elements of $\Sigma$ are far enough apart.

\begin{proposition}\quad
\label{prop:StructureColorationInduits}
 
  Let $([\Z^d], \widetilde{\mu}, \widetilde{T})$ be a Markov $\Z^d$-extension of a Gibbs-Markov map $(A, \alpha, d, \mu, T)$ with jump function $F$. Assume that $([\Z^d], \widetilde{\mu}, \widetilde{T})$ is ergodic and recurrent. Let $M \geq 1$ be the period of $([0], \mu, T_{[0]})$. Then there exists $R >0$ with the following property. Let $\Sigma \subset \Z^d$ be non-empty and such that $\min_{p \neq q \in \Sigma} \norm{p-q}{} \geq R$. 
  
 \smallskip
 
 If $[\Z^d \setminus B(0,r)]$ is monochromatic for large enough $r$ (and in particular if $d = 2$, or $d = 1$ and $\norm{F}{\Lbb^\infty (A, \mu)} = +\infty$), then $T_{[\Sigma]} (A_k \times \Sigma) \subset A_{k+1} \times \Sigma$ for all $k \in \Z_{/M\Z}$. In addition, there exists $\ell \in \Z_{/M\Z}$ such that, almost surely, transitions between sites occur only from $A_{\ell-1} \times \Sigma$ to $A_\ell \times \Sigma$.
 
 \smallskip
 
 If $[\Z \setminus B(0,r)]$ is bichromatic for large enough $r$, sort $\Sigma$ in increasing order, for instance: 
 \begin{equation*}
  \Sigma = \{\ldots, \sigma_{-1}, \sigma_0, \sigma_1, \ldots\} = \{\sigma_n : n \in I\}.
 \end{equation*}
 Then $T_{[\Sigma]}$ sends $\bigsqcup_{n \in I} A_{k+n(\ell_+ - \ell_-)} \times \{\sigma_n\}$ to $\bigsqcup_{n \in I} A_{k+n(\ell_+ - \ell_-)+1} \times \{\sigma_n\}$ for all $k \in \Z_{/M\Z}$. In addition, transitions occur only from $A_{\ell_+ -1} \times \{\sigma_n\}$ to $A_{\ell_-} \times \{\sigma_{n+1}\}$ or from $A_{\ell_- -1} \times \{\sigma_n\}$ to $A_{\ell_+} \times \{\sigma_{n-1}\}$.
\end{proposition}

\begin{proof}
 
 Let $([\Z^d], \widetilde{\mu}, \widetilde{T})$ be such an extension, and choose $R$ as in Proposition~\ref{prop:StructureColoration}. Let $\Sigma \subset \Z^d$ satisfying the conditions of 
 Proposition~\ref{prop:StructureColorationInduits}.
 
 \medskip
 \textbf{First case: monochromatic extensions}
 \smallskip
 
 We start with the first case. Since the extension is monochromatic away from $[0]$, let $\ell \in \Z_{/M\Z}$ be its colour. Let $p \in \Sigma$ and $(x,p) \in [p]$.
 
 \smallskip
 
 If $T_{[\Sigma]} (x,p) \in [p]$, then $T_{[\Sigma]} (x,p) = (T_{[0]} (x), p)$. Given $k$ such that $x \in A_k$, we get $T_{[0]} (x) \in A_{k+1}$, and thus $T_{[\Sigma]} (x,p) \in A_{k+1} \times \Sigma$.

 \smallskip
 
 Otherwise, let $q \in \Sigma$ be such that $T_{[\Sigma]} (x,p) \in [q]$. Then $T_{[\Sigma-q]} (x,p-q)$ is the first return to $[0]$ of the orbit of $(x,p-q)$ under $\widetilde{T}$, and $(x,p-q) \in [\Z^d \setminus B(0,R)]$. Since the colour of the later is $\ell$, this means that $T_{[\Sigma-q]} (x,p-q) \in A_\ell \times \{0\}$, and thus $T_{[\Sigma]} (x,p) \in A_\ell \times \{q\}$. In addition, using again the monochromaticity of $[\Z^d \setminus B(0,R)]$, the first return of $T_{[\Sigma]} (x,p)$ to $[p]$ belongs to $A_\ell \times \{p\}$. But, since this first return is $(T_{[0]} (x),p)$, this implies that $x \in A_{\ell-1}$.
 
 \medskip
 \textbf{Second case: bichromatic extensions}
 \smallskip
 
  We now tackle the second case. By Proposition~\ref{prop:StructureColoration}, the extension is a $\Z$-extension. It has colour $\ell_-$ on the large enough negatives and $\ell_+\neq \ell_-$ on the large enough positives numbers. Let $p \in \Sigma$ and $(x,p) \in [p]$.
  
  \smallskip
  
  If there exists non-adjacent sites $p<q<r$ in $\Sigma$ such that $\mu ([p] \cap T_{[\Sigma]}^{-1} ([r])) > 0$, then $p-q$ and $r-q$ have the same colour, so $\ell_+ = \ell_-$; the same holds if $\mu ([r] \cap T_{[\Sigma]}^{-1} ([q])) > 0$. This contradicts the bichromaticity of the extension. Hence the only allowed transitions are between adjacent sites in $\Sigma$.
  
  \smallskip
 
  Again, if $T_{[\Sigma]} (x,p) \in [p]$, then $T_{[\Sigma]} (x,p) = (T_{[0]} (x), p)$. Given $k$ such that $x \in A_k$, we get $T_{[0]} (x) \in A_{k+1}$, and thus $T_{[\Sigma]} (x,p) \in A_{k+1} \times \Sigma$.

  \smallskip
  
  Otherwise, let $q \in \Sigma$ be such that $T_{[\Sigma]} (x,p) \in [q]$. Assume without loss of generality that $q>p$. Then the same arguments as in the first case yield $T_{[\Sigma]} (x,p) \in A_{\ell_-} \times \{q\}$ and $(x,p) \in A_{\ell_+ -1} \times [p]$.  
  
  \smallskip
  
  Let us sort the elements of $\Sigma$ in increasing order, indexing them either by $I = \{0, \ldots, |\Sigma|-1 \}$, $I = \N$ or $I = \Z$; this let us write, for instance, $\Sigma = \{\ldots, \sigma_{-1}, \sigma_0, \sigma_1, \ldots\}$. A transition from $\sigma_n$ to $\sigma_{n+1}$ starts from $A_{\ell_+ -1} \times \{\sigma_n\}$ and ends in $A_{\ell_-} \times \{\sigma_{n+1}\}$, while a transition from $\sigma_n$ to $\sigma_{n-1}$ starts from $A_{\ell_- -1} \times \{\sigma_n\}$ and ends in $A_{\ell_+} \times \{\sigma_{n-1}\}$. In particular, $T_{[\Sigma]}$ sends $\bigsqcup_{n \in I} A_{k+n(\ell_+ - \ell_-)} \times \{\sigma_n\}$ 
  to $\bigsqcup_{n \in I} A_{k+n(\ell_+ - \ell_-)+1} \times \{\sigma_n\}$ for all $k \in \Z_{/M\Z}$.
\end{proof}

\begin{remark}[Examples of monochromatic and bichromatic extensions]\quad
 
 Constructing a monochromatic $\Z^d$-extension is easy. Choose a Gibbs-Markov map $(A, \alpha, d, \mu, T)$ with a non-trivial period, and jump only at a given instant in the period. For instance, take
 \begin{itemize}
  \item $A = \Tbb \times \Z_{/5\Z}$,
  \item $\mu$ is proportional to the Lebesgue measure on $A$,
  \item $T(x,k) = (3x, k+1)$,
  \item $F(x,0) = -1$ if $x \in [0, 1/3)$, then $F(x,0) = +1$ if $x \in [2/3, 1)$ and $F(x,k) = 0$ otherwise.
\end{itemize}
 Then $A_k = \Tbb \times \{k\}$ and $\ell = 1$.
 
 \begin{figure}[!h]
 \centering
 \scalebox{1}{
 \begin{tikzpicture}
 
   \draw (0:0.75cm) circle (0.25cm);
   \draw (72:0.75cm) circle (0.25cm);
   \draw (144:0.75cm) circle (0.25cm);
   \draw (216:0.75cm) circle (0.25cm);
   \draw (288:0.75cm) circle (0.25cm);
   
   \draw[-Latex] plot [smooth, tension=1] coordinates {(16:0.88cm) (36:0.95cm) (56:0.88cm)};
   \draw[-Latex] plot [smooth, tension=1] coordinates {(88:0.88cm) (108:0.95cm) (128:0.88cm)};
   \draw[-Latex] plot [smooth, tension=1] coordinates {(160:0.88cm) (180:0.95cm) (200:0.88cm)};
   \draw[-Latex] plot [smooth, tension=1] coordinates {(232:0.88cm) (252:0.95cm) (272:0.88cm)};
   \draw[-Latex] plot [smooth, tension=1] coordinates {(304:0.88cm) (324:0.95cm) (344:0.88cm)};
   
   \draw node at (0:0.75cm) {\small{$A_1$}};
   \draw node at (72:0.75cm) {\small{$A_2$}};
   \draw node at (144:0.75cm) {\small{$A_3$}};
   \draw node at (216:0.75cm) {\small{$A_4$}};
   \draw node at (288:0.75cm) {\small{$A_0$}};
 \end{tikzpicture}
 
 \hspace{1cm}
 
  \begin{tikzpicture}
 
   \draw[fill = blue!30!] (0:0.75cm) circle (0.25cm);
   \draw[fill = blue!50!] (72:0.75cm) circle (0.25cm);
   \draw[fill = blue!70!] (144:0.75cm) circle (0.25cm);
   \draw[fill = blue!90!] (216:0.75cm) circle (0.25cm);
   \draw[fill = blue!10!] (288:0.75cm) circle (0.25cm);
   
   \draw[-Latex] plot [smooth, tension=1] coordinates {(16:0.88cm) (36:0.95cm) (56:0.88cm)};
   \draw[-Latex] plot [smooth, tension=1] coordinates {(88:0.88cm) (108:0.95cm) (128:0.88cm)};
   \draw[-Latex] plot [smooth, tension=1] coordinates {(160:0.88cm) (180:0.95cm) (200:0.88cm)};
   \draw[-Latex] plot [smooth, tension=1] coordinates {(232:0.88cm) (252:0.95cm) (272:0.88cm)};
   \draw[-Latex] plot [smooth, tension=1] coordinates {(304:0.88cm) (324:0.95cm) (344:0.85cm)};
   
   \draw[-Latex] plot [smooth, tension=1] coordinates {(288:1cm) ([xshift=1.25cm] 300:0.95cm) ([xshift=2.5cm] 0:0.5cm)};
   \draw[-Latex] plot [smooth, tension=1] coordinates {(288:1cm) ([xshift=-1.25cm] 280:1cm) ([xshift=-2.5cm] 344:0.9cm)};
   
   \begin{scope}[shift={(2.5cm,0)}]
   \draw[fill = blue!30!] (0:0.75cm) circle (0.25cm);
   \draw[fill = blue!50!] (72:0.75cm) circle (0.25cm);
   \draw[fill = blue!70!] (144:0.75cm) circle (0.25cm);
   \draw[fill = blue!90!] (216:0.75cm) circle (0.25cm);
   \draw[fill = blue!10!] (288:0.75cm) circle (0.25cm);
   
   \draw[-Latex] plot [smooth, tension=1] coordinates {(16:0.88cm) (36:0.95cm) (56:0.88cm)};
   \draw[-Latex] plot [smooth, tension=1] coordinates {(88:0.88cm) (108:0.95cm) (128:0.88cm)};
   \draw[-Latex] plot [smooth, tension=1] coordinates {(160:0.88cm) (180:0.95cm) (200:0.88cm)};
   \draw[-Latex] plot [smooth, tension=1] coordinates {(232:0.88cm) (252:0.95cm) (272:0.88cm)};
   \draw[-Latex] plot [smooth, tension=1] coordinates {(304:0.88cm) (324:0.95cm) (344:0.85cm)};
   
   \draw[-Latex] plot [smooth, tension=1] coordinates {(288:1cm) ([xshift=1.25cm] 300:0.95cm) ([xshift=2.5cm] 0:0.5cm)};
   \draw[-Latex] plot [smooth, tension=1] coordinates {(288:1cm) ([xshift=-1.25cm] 280:1cm) ([xshift=-2.5cm] 344:0.9cm)};
   \end{scope}
   
   \begin{scope}[shift={(-2.5cm,0)}]
   \draw[fill = blue!30!] (0:0.75cm) circle (0.25cm);
   \draw[fill = blue!50!] (72:0.75cm) circle (0.25cm);
   \draw[fill = blue!70!] (144:0.75cm) circle (0.25cm);
   \draw[fill = blue!90!] (216:0.75cm) circle (0.25cm);
   \draw[fill = blue!10!] (288:0.75cm) circle (0.25cm);
   
   \draw[-Latex] plot [smooth, tension=1] coordinates {(16:0.88cm) (36:0.95cm) (56:0.88cm)};
   \draw[-Latex] plot [smooth, tension=1] coordinates {(88:0.88cm) (108:0.95cm) (128:0.88cm)};
   \draw[-Latex] plot [smooth, tension=1] coordinates {(160:0.88cm) (180:0.95cm) (200:0.88cm)};
   \draw[-Latex] plot [smooth, tension=1] coordinates {(232:0.88cm) (252:0.95cm) (272:0.88cm)};
   \draw[-Latex] plot [smooth, tension=1] coordinates {(304:0.88cm) (324:0.95cm) (344:0.85cm)};
   
   \draw[-Latex] plot [smooth, tension=1] coordinates {(288:1cm) ([xshift=1.25cm] 300:0.95cm) ([xshift=2.5cm] 0:0.5cm)};
   \draw[-Latex] plot [smooth, tension=1] coordinates {(288:1cm) ([xshift=-1.25cm] 280:1cm) ([xshift=-2.5cm] 344:0.9cm)};
   \end{scope}
   
   \begin{scope}[shift={(5cm,0)}]
   \draw[fill = blue!30!] (0:0.75cm) circle (0.25cm);
   \draw[fill = blue!50!] (72:0.75cm) circle (0.25cm);
   \draw[fill = blue!70!] (144:0.75cm) circle (0.25cm);
   \draw[fill = blue!90!] (216:0.75cm) circle (0.25cm);
   \draw[fill = blue!10!] (288:0.75cm) circle (0.25cm);
   
   \draw[-Latex] plot [smooth, tension=1] coordinates {(16:0.88cm) (36:0.95cm) (56:0.88cm)};
   \draw[-Latex] plot [smooth, tension=1] coordinates {(88:0.88cm) (108:0.95cm) (128:0.88cm)};
   \draw[-Latex] plot [smooth, tension=1] coordinates {(160:0.88cm) (180:0.95cm) (200:0.88cm)};
   \draw[-Latex] plot [smooth, tension=1] coordinates {(232:0.88cm) (252:0.95cm) (272:0.88cm)};
   \draw[-Latex] plot [smooth, tension=1] coordinates {(304:0.88cm) (324:0.95cm) (344:0.85cm)};
   
   \draw[-Latex, dashed] plot [smooth, tension=1] coordinates {(288:1cm) ([xshift=1.25cm] 300:0.95cm)};
   \draw[-Latex] plot [smooth, tension=1] coordinates {(288:1cm) ([xshift=-1.25cm] 280:1cm) ([xshift=-2.5cm] 344:0.9cm)};
   \end{scope}

   \begin{scope}[shift={(-5cm,0)}]
   \draw[fill = blue!30!] (0:0.75cm) circle (0.25cm);
   \draw[fill = blue!50!] (72:0.75cm) circle (0.25cm);
   \draw[fill = blue!70!] (144:0.75cm) circle (0.25cm);
   \draw[fill = blue!90!] (216:0.75cm) circle (0.25cm);
   \draw[fill = blue!10!] (288:0.75cm) circle (0.25cm);
   
   \draw[-Latex] plot [smooth, tension=1] coordinates {(16:0.88cm) (36:0.95cm) (56:0.88cm)};
   \draw[-Latex] plot [smooth, tension=1] coordinates {(88:0.88cm) (108:0.95cm) (128:0.88cm)};
   \draw[-Latex] plot [smooth, tension=1] coordinates {(160:0.88cm) (180:0.95cm) (200:0.88cm)};
   \draw[-Latex] plot [smooth, tension=1] coordinates {(232:0.88cm) (252:0.95cm) (272:0.88cm)};
   \draw[-Latex] plot [smooth, tension=1] coordinates {(304:0.88cm) (324:0.95cm) (344:0.85cm)};
   
   \draw[-Latex] plot [smooth, tension=1] coordinates {(288:1cm) ([xshift=1.25cm] 300:0.95cm) ([xshift=2.5cm] 0:0.5cm)};
   \draw[-Latex, dashed] plot [smooth, tension=1] coordinates {(288:1cm) ([xshift=-1.25cm] 280:1cm)};
   \end{scope}
 \end{tikzpicture}
 }
 \caption{On the left: the allowed transitions for the transformation $T$ and the five elements of the partition in periodic components $\{A_0, A_1, \ldots, A_4\}$ of $T_{[0]}$. On the right: the $\Z$-extension of $T$ and its allowed transitions. The five shades of blue correspond to the five elements of the partition $\{A_0 \times \Z, A_1 \times \Z, \ldots, A_4 \times \Z\}$ of $[\Z]$. All allowed transitions lead to a subset one shade darker, or from the darkest shade to the lightest.}
 \label{fig:DecompositionPeriodiqueStandard}
 \end{figure}
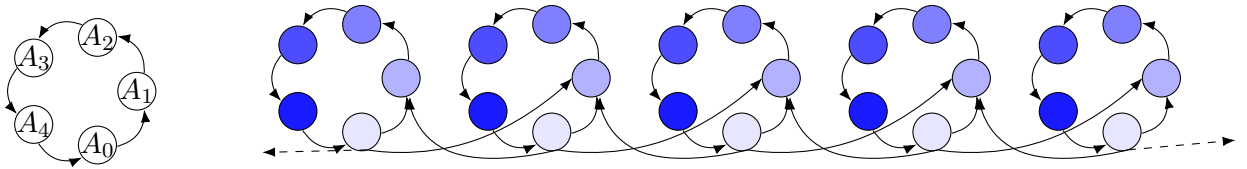
 
 The analyses in the proofs of Propositions~\ref{prop:StructureColoration} and~\ref{prop:StructureColorationInduits} give a guide to construct a bichromatic extension. We tweak the Gibbs-Markov map defined above. We still take $A = \Tbb \times \Z_{/5\Z}$ and 
 $\mu$ proportional to the Lebesgue measure, but:
 \begin{equation*}
  T(x,k) = \left\{ \begin{array}{lcl}
                   (2x, 3) & \text{ if } & x \in [0, 1/2) \text{ and } k = 0 \\
                   (2x, 1) & \text{ if } & x \in [0, 1/2) \text{ and } k = 2 \\
                   (2x, k+1) & \text{ otherwise} & \\
                   \end{array}
\right. .
 \end{equation*}
 We choose for $F$:
 \begin{equation*}
  F(x,k) = \left\{ \begin{array}{lcl}
                   +1 & \text{ if } & x \in [0, 1/2) \text{ and } k = 0 \\
                   -1 & \text{ if } & x \in [0, 1/2) \text{ and } k = 2 \\
                   0 & \text{ otherwise} & \\
                   \end{array}
\right. .
 \end{equation*}
 Then $A_k = \Tbb \times \{k\}$, $\ell_+ = 1$ and $\ell_- = 3$.

  \begin{figure}[!h]
 \centering
 \scalebox{1}{
 \begin{tikzpicture}
 
   \draw (0:0.75cm) circle (0.25cm);
   \draw (72:0.75cm) circle (0.25cm);
   \draw (144:0.75cm) circle (0.25cm);
   \draw (216:0.75cm) circle (0.25cm);
   \draw (288:0.75cm) circle (0.25cm);
   
   \draw[-Latex] plot [smooth, tension=1] coordinates {(16:0.88cm) (36:0.95cm) (56:0.88cm)};
   \draw[-Latex] plot [smooth, tension=1] coordinates {(88:0.88cm) (108:0.95cm) (128:0.88cm)};
   \draw[-Latex] plot [smooth, tension=1] coordinates {(160:0.88cm) (180:0.95cm) (200:0.88cm)};
   \draw[-Latex] plot [smooth, tension=1] coordinates {(232:0.88cm) (252:0.95cm) (272:0.88cm)};
   \draw[-Latex] plot [smooth, tension=1] coordinates {(304:0.88cm) (324:0.95cm) (344:0.88cm)};
   \draw[-Latex] plot [smooth, tension=1] coordinates {(56:0.5cm) (36:0.4cm) (16:0.5cm)};
   \draw[-Latex] plot [smooth, tension=1] coordinates {(288:0.5cm) (216:0.2cm) (144:0.5cm)};
   
   \draw node at (0:0.75cm) {\small{$A_1$}};
   \draw node at (72:0.75cm) {\small{$A_2$}};
   \draw node at (144:0.75cm) {\small{$A_3$}};
   \draw node at (216:0.75cm) {\small{$A_4$}};
   \draw node at (288:0.75cm) {\small{$A_0$}};
 \end{tikzpicture}
 
 \hspace{1cm}
 
  \begin{tikzpicture}
 
   \draw[fill = blue!30!] (0:0.75cm) circle (0.25cm);
   \draw[fill = blue!50!] (72:0.75cm) circle (0.25cm);
   \draw[fill = blue!70!] (144:0.75cm) circle (0.25cm);
   \draw[fill = blue!90!] (216:0.75cm) circle (0.25cm);
   \draw[fill = blue!10!] (288:0.75cm) circle (0.25cm);
   
   \draw[-Latex] plot [smooth, tension=1] coordinates {(16:0.88cm) (36:0.95cm) (56:0.88cm)};
   \draw[-Latex] plot [smooth, tension=1] coordinates {(88:0.88cm) (108:0.95cm) (128:0.88cm)};
   \draw[-Latex] plot [smooth, tension=1] coordinates {(160:0.88cm) (180:0.95cm) (200:0.88cm)};
   \draw[-Latex] plot [smooth, tension=1] coordinates {(232:0.88cm) (252:0.95cm) (272:0.88cm)};
   \draw[-Latex] plot [smooth, tension=1] coordinates {(304:0.88cm) (324:0.95cm) (344:0.85cm)};
   
   \draw[-Latex] plot [smooth, tension=1] coordinates {(288:1cm) (340:1.3cm) ([xshift=2.5cm] 166:0.8cm)};
   \draw[-Latex] plot [smooth, tension=1] coordinates {(88:0.9cm) (144:1.2cm) ([xshift=-2.5cm] 0:1cm)};
   
   \begin{scope}[shift={(2.5cm,0)}]
   \draw[fill = blue!90!] (0:0.75cm) circle (0.25cm);
   \draw[fill = blue!10!] (72:0.75cm) circle (0.25cm);
   \draw[fill = blue!30!] (144:0.75cm) circle (0.25cm);
   \draw[fill = blue!50!] (216:0.75cm) circle (0.25cm);
   \draw[fill = blue!70!] (288:0.75cm) circle (0.25cm);
   
   \draw[-Latex] plot [smooth, tension=1] coordinates {(16:0.88cm) (36:0.95cm) (56:0.88cm)};
   \draw[-Latex] plot [smooth, tension=1] coordinates {(88:0.88cm) (108:0.95cm) (128:0.88cm)};
   \draw[-Latex] plot [smooth, tension=1] coordinates {(160:0.88cm) (180:0.95cm) (200:0.88cm)};
   \draw[-Latex] plot [smooth, tension=1] coordinates {(232:0.88cm) (252:0.95cm) (272:0.88cm)};
   \draw[-Latex] plot [smooth, tension=1] coordinates {(304:0.88cm) (324:0.95cm) (344:0.85cm)};
   
   \draw[-Latex] plot [smooth, tension=1] coordinates {(288:1cm) (340:1.3cm) ([xshift=2.5cm] 166:0.8cm)};
   \draw[-Latex] plot [smooth, tension=1] coordinates {(88:0.9cm) (144:1.2cm) ([xshift=-2.5cm] 0:1cm)};
   \end{scope}
   
   \begin{scope}[shift={(-2.5cm,0)}]
   \draw[fill = blue!70!] (0:0.75cm) circle (0.25cm);
   \draw[fill = blue!90!] (72:0.75cm) circle (0.25cm);
   \draw[fill = blue!10!] (144:0.75cm) circle (0.25cm);
   \draw[fill = blue!30!] (216:0.75cm) circle (0.25cm);
   \draw[fill = blue!50!] (288:0.75cm) circle (0.25cm);
   
   \draw[-Latex] plot [smooth, tension=1] coordinates {(16:0.88cm) (36:0.95cm) (56:0.88cm)};
   \draw[-Latex] plot [smooth, tension=1] coordinates {(88:0.88cm) (108:0.95cm) (128:0.88cm)};
   \draw[-Latex] plot [smooth, tension=1] coordinates {(160:0.88cm) (180:0.95cm) (200:0.88cm)};
   \draw[-Latex] plot [smooth, tension=1] coordinates {(232:0.88cm) (252:0.95cm) (272:0.88cm)};
   \draw[-Latex] plot [smooth, tension=1] coordinates {(304:0.88cm) (324:0.95cm) (344:0.85cm)};
   
   \draw[-Latex] plot [smooth, tension=1] coordinates {(288:1cm) (340:1.3cm) ([xshift=2.5cm] 166:0.8cm)};
   \draw[-Latex] plot [smooth, tension=1] coordinates {(88:0.9cm) (144:1.2cm) ([xshift=-2.5cm] 0:1cm)};
   \end{scope}
   
   \begin{scope}[shift={(5cm,0)}]
   \draw[fill = blue!50!] (0:0.75cm) circle (0.25cm);
   \draw[fill = blue!70!] (72:0.75cm) circle (0.25cm);
   \draw[fill = blue!90!] (144:0.75cm) circle (0.25cm);
   \draw[fill = blue!10!] (216:0.75cm) circle (0.25cm);
   \draw[fill = blue!30!] (288:0.75cm) circle (0.25cm);
   
   \draw[-Latex] plot [smooth, tension=1] coordinates {(16:0.88cm) (36:0.95cm) (56:0.88cm)};
   \draw[-Latex] plot [smooth, tension=1] coordinates {(88:0.88cm) (108:0.95cm) (128:0.88cm)};
   \draw[-Latex] plot [smooth, tension=1] coordinates {(160:0.88cm) (180:0.95cm) (200:0.88cm)};
   \draw[-Latex] plot [smooth, tension=1] coordinates {(232:0.88cm) (252:0.95cm) (272:0.88cm)};
   \draw[-Latex] plot [smooth, tension=1] coordinates {(304:0.88cm) (324:0.95cm) (344:0.85cm)};
   
   \draw[-Latex, dashed] plot [smooth, tension=1] coordinates {(288:1cm) (340:1.3cm)};
   \draw[-Latex] plot [smooth, tension=1] coordinates {(88:0.9cm) (144:1.2cm) ([xshift=-2.5cm] 0:1cm)};
   \end{scope}

   \begin{scope}[shift={(-5cm,0)}]
   \draw[fill = blue!10!] (0:0.75cm) circle (0.25cm);
   \draw[fill = blue!30!] (72:0.75cm) circle (0.25cm);
   \draw[fill = blue!50!] (144:0.75cm) circle (0.25cm);
   \draw[fill = blue!70!] (216:0.75cm) circle (0.25cm);
   \draw[fill = blue!90!] (288:0.75cm) circle (0.25cm);
   
   \draw[-Latex] plot [smooth, tension=1] coordinates {(16:0.88cm) (36:0.95cm) (56:0.88cm)};
   \draw[-Latex] plot [smooth, tension=1] coordinates {(88:0.88cm) (108:0.95cm) (128:0.88cm)};
   \draw[-Latex] plot [smooth, tension=1] coordinates {(160:0.88cm) (180:0.95cm) (200:0.88cm)};
   \draw[-Latex] plot [smooth, tension=1] coordinates {(232:0.88cm) (252:0.95cm) (272:0.88cm)};
   \draw[-Latex] plot [smooth, tension=1] coordinates {(304:0.88cm) (324:0.95cm) (344:0.85cm)};
   
   \draw[-Latex] plot [smooth, tension=1] coordinates {(288:1cm) (340:1.3cm) ([xshift=2.5cm] 166:0.8cm)};
   \draw[-Latex, dashed] plot [smooth, tension=1] coordinates {(88:0.9cm) (144:1.2cm)};
   \end{scope}
 \end{tikzpicture}
 }
 \caption{On the left: the allowed transitions for the transformation $T$ and the five elements of the partition in periodic components $\{A_0, A_1, \ldots, A_4\}$ of $T_{[0]}$. On the right: the $\Z$-extension of $T$ and its allowed transitions. The five shades of blue correspond to the five elements of the partition $\left\{ \bigsqcup_{n \in \Z} A_{k-2n} \times \{n\}: \ k \in \Z_{/5\Z}\right\}$ of $[\Z]$. All allowed transitions lead to a subset one shade darker, or from the darkest shade to the lightest.}
 \label{fig:DecompositionPeriodiqueExceptionnelle}
 \end{figure}
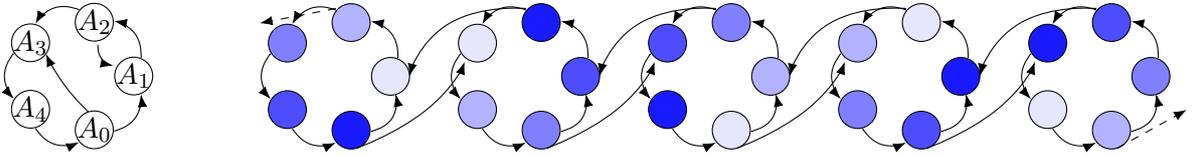
\end{remark}

\begin{remark}[Structure of bichromatic extensions]\quad
 
 As per Proposition~\ref{prop:StructureColoration}, a bichromatic extension can occur only if $d = 1$ and $\norm{F}{\Lbb^\infty (A, \mu)} < +\infty$. A reasonable conjecture is that such extensions can only occur if $d=1$ and there exists a bounded coboundary $u \circ T - u$ such that $F + u \circ T - u$ takes its values in $\{-1, 0, 1\}$.
\end{remark}

\begin{remark}[Non Gibbs-Markov maps]\quad
 
 The structure of $\Z^d$-extensions when the induced map $(A, \mu, T_{[0]})$ has some periodicity may be different for maps which are not Gibbs-Markov. An interesting example is that of piecewise expanding maps of the interval with a compatible jump function. Then $(A, \mu, T_{[0]})$ is still a piecewise expanding map of the interval (with infinitely many pieces). 
 
 \smallskip
 
 Up to additional conditions to check~\cite{LasotaYorke:1973, Rychlik:1983}, we may expect its transfer operator to act quasi-compactly on the space of functions with bounded variation, from which we can develop a similar description of its spectrum. In particular, if $(A, \mu, T_{[0]})$ is not mixing, then it has a decomposition in periodic components, from which we can construct a colouring of $[\Z^d]$. However, our argument in this article relies heavily on the big image property of Gibbs-Markov maps, which is not satisfied for general piecewise expanding maps of the interval. In this case, how are Propositions~\ref{prop:StructureColoration} and~\ref{prop:StructureColorationInduits} to be adapted? Are $[\Z^d]$-extensions of piecewise expanding maps either monochromatic or bichromatic?
\end{remark}

\subsection{End of the proof of Lemma~\ref{lem:MasterLemmaErgodic}}

We now finish the proof of Lemma~\ref{lem:MasterLemmaErgodic}. We explain the main adaptations to make to the proof of Lemma~\ref{lem:MasterLemma}, first in the case of monochromatic extensions, and then in the case of bichromatic extensions.

\begin{proof}[Proof of Lemma~\ref{lem:MasterLemmaErgodic}]
 
 Assume Hypotheses~\ref{hyp:Recurrence} and~\ref{hyp:AsymptoticsP}. Let $M \geq 1$ be the period of $([0], \mu, T_{[0]})$ and $(A_k)_{k \in \Z_{/M\Z}}$ a decomposition of $A$ in periodic components for $T_{[0]}$.
 
 \medskip
 \textbf{First case: monochromatic extensions}
 \smallskip
 
 By Hypothesis~\ref{hyp:AsymptoticsP}, the probability of transition between different sites in $[\Sigma_\varepsilon]$ 
 converges to $0$. By recurrence and ergodicity, $\lim_{\varepsilon \to 0} \min_{p \neq q \in \Sigma_\varepsilon} \norm{p-q}{} = + \infty$. We can thus apply Proposition~\ref{prop:StructureColorationInduits}: there exists $\ell \in \Z_{/M\Z}$ such that, for all small enough $\varepsilon > 0$, the map $T_\varepsilon$ sends each $A_k \times I$ into $A_{k+1} \times I$, and all transitions between distinct sites occur from $A_{\ell-1} \times I$ to $A_\ell \times I$. Hence, $T_\varepsilon^M (A_k \times I) = A_k \times I$ for all $k$.
 
 \smallskip
 
 Let us fix some $k \in \Z_{/M\Z}$, and work with the measure-preserving dynamical system $(A_k \times I, M\mu_{I|A_k \times I}, T_\varepsilon^M)$. A suitable Banach space is given by the space $\Bcal_{I, k}$ of functions in $\Bcal_I$ supported on $A_k \times I$ (recall that $A_k$ is $\sigma(\alpha^*)$-measurable, so multiplication by $\mathbf{1}_{A_k \times I}$ is continuous on $\Bcal_I$). We define operators $\Pi_{k,*} : \Bcal_{I, k} \to \C^I$ and $\Pi_k^* : \C^I \to \Bcal_{I, k}$ in the same way as we defined $\Pi_*$ and $\Pi^*$, that is respectively by averaging in each $A_k \times \{i\}$ and by finding the preimage that is constant on each $A_k \times \{i\}$. From there a definition of cones $C_{k, K} (\varepsilon)$ is straightforward.
 
 \smallskip
 
 The operator $\Lcal_{[0]}$ acts quasi-compactly on $\Bcal$, and thus so does $\Lcal_{[0]}^M$. It follows that the action of $\Lcal_{[0]}^M$ on the space $\Bcal_k$ of functions in $\Bcal$ supported on $A_k$ is also quasi-compact. Since $(A_k, M \mu_{|A_k}, T_{[0]}^M)$ is mixing, the action $\Lcal_{[0]}^M \acts \Bcal_k$ has a spectral gap.
  
  \smallskip
  
  Let $i \neq j \in I$. Let $\ell'$ be to smallest integer larger than $k$ and congruent to $\ell$ modulo $M$. A point $(x, i) \in A_k \times \{i\}$ belongs to $T_\varepsilon^{-M} (A_k \times \{j\})$ if and only if $T_{[0]}^{\ell' - k-1} (x, i)$ belongs to $T_\varepsilon^{-1} ([j])$, so that
  \begin{align*}
   M \mu_{|A_k} (x : \ (x,i) \in T_\varepsilon^{-M} (A_k \times \{j\})) 
   & = M \mu \left( T_{[0]}^{1+k-\ell'} (\{ x : \ (x, i) \in T_\varepsilon^{-1} ([j])\}) \right) \\
   & = M \mu ([i] \cap T_\varepsilon^{-1} ([j])\}) \\
   & = \varepsilon M R_{ij} + o (\varepsilon).
  \end{align*}
  In addition, if $(x, i) \in A_k \times \{i\}$ and $T_\varepsilon^M (x, i) \in A_k \times \{i\}$, then $T_\varepsilon^n (x, i) \in A_k \times \{i\}$ for all $0 \leq n \leq M$, so that $T_\varepsilon^M (x, i) = T_{[0]}^M (x, i)$.
  
  \smallskip
  
  From there, the proof of Proposition~\ref{prop:FastAsymptotics} follows mostly unchanged, replacing $\Lcal_\varepsilon$ by $\Lcal_\varepsilon^M$ and $C_K (\varepsilon)$ by $C_{k,K} (\varepsilon)$. More explicitly, for all $\sigma >0$, there exist $K_\sigma$, $n_\sigma \geq 0$ such that, for all $K \geq K_\sigma$, for all $\varepsilon>0$ and $n \geq n_\sigma$,
 \begin{equation*}
  \Lcal_\varepsilon^{Mn} (C_{k,K} (\varepsilon)) 
  \subset C_{k, \sigma K} (\varepsilon).
 \end{equation*}
 Corollary~\ref{cor:ConeStability} follows directly. Subsection~\ref{subsec:ProduitMatrices} is unchanged as it makes no reference to the dynamics. Corollary~\ref{cor:ConvergenceUniformeCompacts} follows with only the obvious adaptations, and so does Proposition~\ref{prop:Couplage}. We get a variant of Equation~\eqref{eq:ConvergenceUniformeCompactsDynamique}:
 \begin{equation*}
  \lim_{\varepsilon \to 0} \sup_{t \in [0,T]} \norm{ \Pi_{k,*} \Lcal_\varepsilon^{M \lfloor \varepsilon^{-1} M^{-1} t \rfloor} f_\varepsilon - e^{-t(\transposee{R})} \Pi_{k,*} f_\varepsilon }{\ell^1 (I)} 
  = 0 \quad \forall (f_\varepsilon)_{\varepsilon > 0} \in C_{k, K} (\varepsilon).
 \end{equation*}
 The map $\Lcal_{[0]}^n : \Bcal_{k, I} \to \Bcal_{k+n, I}$ is continuous with a norm of at most $\norm{\Lcal_{[0]}^n}{ \Bcal_I \to \Bcal_I}$, since we are only restricting $\Lcal_{[0]}^n$ to functions supported on $A_k$. We can also apply Proposition~\ref{prop:ApproximationParMatrices} to get that 
 \begin{equation*}
  \Pi_{k+n,*} \Lcal_\varepsilon^n f_\varepsilon 
  = (1+O(\varepsilon)) \Pi_{k,*} f_\varepsilon,
 \end{equation*}
 where $n = \lfloor \varepsilon^{-1} t \rfloor - M \lfloor \varepsilon^{-1} M^{-1} t \rfloor$ and the $O(\varepsilon)$ term depends only on the choice of $K$. This yields
 \begin{equation*}
  \lim_{\varepsilon \to 0} \sup_{t \in [0,T]} \norm{ \Pi_{k+\lfloor \varepsilon^{-1} t \rfloor,*} \Lcal_\varepsilon^{\lfloor \varepsilon^{-1} t \rfloor} f_\varepsilon - e^{-t(\transposee{R})} \Pi_{k+\lfloor \varepsilon^{-1} t \rfloor,*} f_\varepsilon }{\ell^1 (I)} 
  = 0 \quad \forall (f_\varepsilon)_{\varepsilon > 0} \in C_{k, K} (\varepsilon).
 \end{equation*}
 
 Proposition~\ref{prop:Couplage} also carries over:
 \begin{equation*}
  \norm{\Lcal_\varepsilon^{Mn}}{(\Bcal_{k, I, 0}, \norm{\cdot}{k, K, \varepsilon}) \to (\Bcal_{k, I, 0}, \norm{\cdot}{k, K, \varepsilon})} 
  \leq C e^{- \rho M \varepsilon n}
 \end{equation*}
 for some constants $C$, $\rho >0$ (which can be taken arbitrarily close to $\rho_R$), all $n \geq 0$ and all small enough $\varepsilon >0$. Again, using the boundedness of $\Lcal_\varepsilon^n : \Bcal_{k, I} \to \Bcal_{k+n, I}$, up to increasing the constants $C$ and $K$, we get
 \begin{equation*}
  \norm{\Lcal_\varepsilon^n}{(\Bcal_{k, I, 0}, \norm{\cdot}{k, K, \varepsilon}) \to (\Bcal_{k+n, I, 0}, \norm{\cdot}{k+n, K, \varepsilon})} 
  \leq C e^{- \rho \varepsilon n}.
 \end{equation*}
 
 The statement of Corollary~\ref{cor:EquivalencePotentielInduit} can be adapted by replacing $\Bcal_{I,0}$ by $\Bcal_{k, I, 0}$. However, we want in the end a statement for functions constant on each $[i] \in [I]$, so this is not satisfactory. Some caution is needed at this point; what we get is that, given $K \geq 0$ and a family on functions $(f_\varepsilon)_{\varepsilon > 0} \in \bigoplus_{k \in \Z_{/M\Z}} \Bcal_{k, I, 0}$ such that $\sup_{\varepsilon > 0} \sup_{k \in \Z_{/M\Z}} \norm{f_\varepsilon}{k, K, \varepsilon}$,
 \begin{equation*}
  \sum_{n = 0}^{+ \infty} \Lcal_\varepsilon^n f_\varepsilon 
  = \frac{1}{\varepsilon} \Pi^* (\transposee{R}_0^{-1}) \Pi_* f_\varepsilon + o(\varepsilon^{-1}),
 \end{equation*}
 where the equality holds in $\Bcal_{I, 0}$. The space $\bigoplus_{k \in \Z_{/M\Z}} \Bcal_{k, I, 0}$ is of codimension $M-1$ in $\Bcal_{I, 0}$: if we want the series $\sum_{n = 0}^{+ \infty} \Lcal_\varepsilon^n f_\varepsilon$ to actually be summable, we need $f_\varepsilon$ to have zero average on each $A_k \times I$, and not just on the whole of $[I]$. Under this slight modification, the proof of Corollary~\ref{cor:EquivalencePotentielInduit} can be adapted up to straightforward modifications, starting from
 \begin{equation*}
  \sum_{n = 0}^{+ \infty} \Lcal_\varepsilon^n f_\varepsilon 
  = \sum_{k \in \Z_{/M\Z}} \sum_{n = 0}^{+ \infty} \mathbf{1}_{A_{k+n} \times I} \Lcal_\varepsilon^n ( f_\varepsilon \mathbf{1}_{A_k \times I}),
 \end{equation*}
 and working with each $k$ separately. Finally, given $f \in \C^I_0$, the function $\Pi^* f$ do belong to the space $\bigoplus_{k \in \Z_{/M\Z}} \Bcal_{k, I, 0}$, which finishes the proof of Lemma~\ref{lem:MasterLemmaErgodic} in the case of monochromatic extensions.
 
 \medskip
 \textbf{Second case: bichromatic extensions}
 \smallskip
 
 As with monochromatic extensions, under Hypothesis~\ref{hyp:AsymptoticsP}, the minimal distance between sites in $\Sigma_\varepsilon$ converges to $+\infty$, so that we can apply Proposition~\ref{prop:StructureColorationInduits}.
 
 \smallskip
 
 Let us index the sites in $\Sigma_\varepsilon$ in increasing order by $I := \{0, \ldots, |\Sigma_\varepsilon|-1\}$. This reindexing conjugate both the matrices $P_\varepsilon$ and $Q_\varepsilon$ by the same permutation matrix, which depends on $\varepsilon >0$. Since there are only finitely many permutation matrices, it is enough to prove the theorem on the subsequences of $\varepsilon > 0$ corresponding to any fixed permutation. Hence, up to restriction to suitable subsequences, we assume simultaneously that Hypothesis~\ref{hyp:AsymptoticsP} holds and that the indexing of $\Sigma_\varepsilon$ is increasing.
 
 \smallskip
 
 By Proposition~\ref{prop:StructureColorationInduits}, there exist $\ell_- \neq \ell_+ \in \Z_{/M\Z}$ such that 
 all transitions for $T_\varepsilon$ occur from $A_{\ell_+ -1} \times \{n\}$ to $A_{\ell_-} \times \{n+1\}$ or from $A_{\ell_- -1} \times \{n\}$ to $A_{\ell_+} \times \{n-1\}$. We set $\widetilde{A}_k := \bigsqcup_{n \in I} A_{k+n(\ell_+ - \ell_-)} \times \{n\}$ so that $T_\varepsilon^M$ preserves all the subsets $\widetilde{A}_k$ for all small enough $\varepsilon$.
 
 \smallskip
 
 The adaptation is then mostly the same as in the monochromatic case, replacing $A_k \times I$ by $\widetilde{A}_k$. There are two significant differences. The first is that jumps upwards and jumps downwards happen at different times, so the argument which in the monochromatic case yielded
 \begin{equation*}
   M \mu_{|A_k} (x : \ (x,i) \in T_\varepsilon^{-M} (A_k \times \{j\})) 
   = \varepsilon M R_{ij} + o (\varepsilon)
  \end{equation*}
 does not work immediately. Here is how this argument can be fixed. First, note that, from times $k=0$ to $k = M$, a trajectory cannot get up then down, or down then up. Indeed, a trajectory of the first kind would have to go through $A_{\ell_+-1} \times \{n\}$, then $A_{\ell_-} \times \{n+1\}$, then $A_{\ell_--1} \times \{n+1\}$, then $A_{\ell_+} \times \{n\}$, which requires at least $M+1$ steps. The same reasoning applies to trajectories of the second kind.
 
 \smallskip
 
 Let $k \in \Z_{/M\Z}$. Let $\ell'_-$ (respectively $\ell'_+$) be the least integer strictly larger than $k$ and congruent to $\ell_-$ (respectively $\ell_+$) modulo $M$. Assume without loss of generality that $\ell'_- < \ell'_+$. Then 
 \begin{align*}
   M \mu_{|A_k} (x : \ (x,n) \in T_\varepsilon^{-M} ([\{0, \ldots, n-1\}])) 
   & = M \mu (x : \ (x,n) \in T_\varepsilon^{-1} ([n-1])) \\
   & = \varepsilon M R_{n, n-1} + o (\varepsilon),
  \end{align*}
  with the same argument as in the monochromatic case. In addition,
  \begin{align*}
   M \mu_{|A_k} & \left(x : \ (x,n) \in T_\varepsilon^{-M} ([\{n+1, \ldots, |I|\}])\right) \\ 
   & = M \mu_{|A_k} \left(x : \ (x,n) \in T_\varepsilon^{-M} ([\{n+1, \ldots, |I|\}]) \text{ and } (x,n) \notin T_\varepsilon^{-M} ([\{0, \ldots, n-1\}]) \right) \\
   & = M \int_{A_{\ell_+-1}} \mathbf{1}_{\{x : (x,n) \in T_\varepsilon^{-1} ([n+1])\}} \Lcal_{[0]}^{\ell_+-\ell_-} \left(\mathbf{1}-\mathbf{1}_{\{x : (x,n) \in T_\varepsilon^{-1} ([n-1])\}}\right) \dd \mu \\
   & = M \mu (x : \ (x,n) \in T_\varepsilon^{-1} ([n-1])) \\ & \hspace{2cm} - M \int_{A_{\ell_+-1}} \mathbf{1}_{\{x : (x,n) \in T_\varepsilon^{-1} ([n+1])\}} \Lcal_{[0]}^{\ell_+-\ell_-} \left(\mathbf{1}_{\{x : (x,n) \in T_\varepsilon^{-1} ([n-1])\}}\right) \dd \mu.
  \end{align*}
  By Corollary~\ref{cor:BorneUniformeTempsArret}, 
  \begin{equation*}
   \norm{\Lcal_{[0]}^{\ell_+-\ell_-} \left(\mathbf{1}_{\{x : (x,n) \in T_\varepsilon^{-1} ([n-1])\}}\right)}{\Bcal} 
   \leq C \mu \left(x : (x,n) \in T_\varepsilon^{-1} ([n-1]) \right) 
   = O(\varepsilon),
  \end{equation*}
  so that 
  \begin{align*}
   M \mu_{|A_k} (x : \ (x,n) \in T_\varepsilon^{-M} ([\{n+1, \ldots, |I|\}])) 
   & = M (1+O(\varepsilon)) \mu \left(x : \ (x,n) \in T_\varepsilon^{-1} ([n+1]) \right) \\
   & = \varepsilon M R_{n, n+1} + o(\varepsilon).
  \end{align*}
  Finally, the same reasoning let us show that the measure of points which undergo at least two jumps before time $M$ is in $O(\varepsilon^2)$, so that 
 \begin{align*}
   M \mu_{|A_k} (x : \ (x,n) \in T_\varepsilon^{-M} ([n-1])) 
   & = M \mu_{|A_k} (x : \ (x,n) \in T_\varepsilon^{-M} ([\{0, \ldots, n-1\}])) + O(\varepsilon^2) \\
   & = \varepsilon M R_{n, n-1} + o (\varepsilon),
  \end{align*}
  and the same for $M \mu_{|A_k} (x : \ (x,n) \in T_\varepsilon^{-M} ([n+1]))$.
  
  \smallskip
  
  The second difference is that, when adapting Proposition~\ref{prop:FastAsymptotics}, we work with the action of $\Lcal_{[0]}^M$ on $\Bcal_{k+n(\ell_+-\ell_-)}$, which depends on the site $n$. In the end, this does not matter: as $k$ takes only finitely many values, there are finitely many such actions, so that all relevant controls are uniform in $n$.
\end{proof}

\section{From potential to transition probabilities}
\label{sec:Inversion}

We now finish the proof of Theorem~\ref{thm:MainTheorem}. 

\smallskip

In Sections~\ref{sec:MasterLemma} (under the mixing Hypothesis~\ref{hyp:Mixing}) and~\ref{sec:ErgodicStructure}, 
knowing an asymptotic development $P_\varepsilon = \Id-\varepsilon R + o(\varepsilon)$, 
we were able to deduce an asymptotic development $Q_\varepsilon = \varepsilon^{-1} R_0^{-1} + o(\varepsilon^{-1})$. However, in practice, we can compute $Q_\varepsilon$ and want to estimate $P_\varepsilon$. 
Since the direct direction comes from Lemma~\ref{lem:MasterLemmaErgodic}, we shall 
work under Hypothesis~\ref{hyp:AsymptoticsQ}:
\begin{equation*}
 Q_\varepsilon = \varepsilon^{-1} S + o(\varepsilon^{-1}),
\end{equation*}
where $S$ is irreducible in the sense of Definition~\ref{def:QMatriceIrreductible}.

 \smallskip
 
 In order to deduce the asymptotics of $P_\varepsilon$ from those of $Q_\varepsilon$, we shall 
 consider an adherence point $R$ of $\varepsilon^{-1} (\Id-P_\varepsilon)$. Then, an application of  
 Lemma~\ref{lem:MasterLemma} yields that $R_0 = S^{-1}$, and thus Theorem~\ref{thm:MainTheorem}. 
 The two main obstructions are:
 \begin{itemize}
  \item the existence of an adherence point of $\varepsilon^{-1} (\Id-P_\varepsilon)$. We shall prove that $\varepsilon^{-1} (\Id-P_\varepsilon)$ stays bounded when $\varepsilon$ vanishes. This is the object of Subsection~\ref{subsec:Bootstrapping}.
  \item in order to apply Lemma~\ref{lem:MasterLemma}, the matrix $R$ needs to be irreducible. This is the object of Subsection~\ref{subsec:IrredR}.
 \end{itemize}
 Our main tool is that Lemma~\ref{lem:MasterLemma} implies Theorem~\ref{thm:MainTheorem} whenever we induce on two sites, i.e.\ whenever $|I|=2$. This let us transfer precious information from $Q_\varepsilon$ to $P_\varepsilon$.

\subsection{Bootstrapping}
\label{subsec:Bootstrapping}

We start our analysis with the case $|I| = 2$, which can be deduced directly from Section~\ref{sec:MasterLemma}. The general case will follow.

\smallskip

If $|I| = 2$, then $\C_0^I$ is $1$-dimensional, so $Q_\varepsilon$ is a scalar, which we shall denote by $q_\varepsilon$. The matrix $P_\varepsilon$ is bi-stochastic, 
and thus can be written as
\begin{equation}
\label{eq:ExpressionRI2}
 P_\varepsilon 
  = \left( \begin{array}{cc}
           1-p_\varepsilon & p_\varepsilon \\
           p_\varepsilon & 1-p_\varepsilon
          \end{array}
  \right)
 = \Id - p_\varepsilon \left( \begin{array}{cc}
           1 & -1 \\
           -1 & 1
          \end{array}
  \right)
 =: \Id-p_\varepsilon R
\end{equation}
for some constant $p_\varepsilon>0$.

\begin{lemma}\quad
\label{lem:MixingLemmaI2}
 
 Assume Hypotheses~\ref{hyp:Recurrence}, \ref{hyp:Mixing} and \ref{hyp:AsymptoticsQ}. Assume furthermore that the matrix $S$ from Hypothesis~\ref{hyp:AsymptoticsQ} is irreducible in the sense of Definition~\ref{def:QMatriceIrreductible}. Assume finally that $|I| = 2$, so that $S = s$ is a scalar. Then
 \begin{equation*}
  p_\varepsilon 
  \sim_{\varepsilon \to 0} \frac{\varepsilon}{2s}.
 \end{equation*}
\end{lemma}

\begin{proof}
 
 In the $|I| = 2$ case, Hypothesis~\ref{hyp:AsymptoticsQ} reads $q_\varepsilon \sim_{\varepsilon \to 0} \varepsilon^{-1} s$. 
 The irreducibility of $S$ in the sense of Definition~\ref{def:QMatriceIrreductible} is equivalent to the property that $s>0$.
 
 \smallskip
 
 As a consequence, $\lim_{\varepsilon \to 0} q_\varepsilon = + \infty$. This is only possible if the distance between the two sites of $\Sigma_\varepsilon$ grows to $+ \infty$ when $\varepsilon$ vanishes. By~\cite[Lemma~4.17]{PeneThomine:2019}, we have $\lim_{\varepsilon \to 0} p_\varepsilon = 0$: indeed, if the two sites are very far away one from the other, then the probability of transitioning from one to the other gets very small.
 
 \smallskip
 
 The matrix $R$ in equation~\eqref{eq:ExpressionRI2} is irreducible. Hence, by Lemma~\ref{lem:MasterLemma},
 \begin{equation*}
    q_\varepsilon 
    \sim_{\varepsilon \to 0} \frac{1}{p_\varepsilon} R_0^{-1}.
 \end{equation*}
 But $R(f) = 2f$ for all $f \in \C_0^I$, or in other words $R_0 = 2\Id$, so
 \begin{equation*}
    q_\varepsilon 
    \sim_{\varepsilon \to 0} \frac{1}{2 p_\varepsilon} 
    \sim_{\varepsilon \to 0} \frac{1}{\varepsilon} s,
 \end{equation*}
 from which the claim follows.
\end{proof}

\begin{remark}\quad
 
 The $|I|=2$ case was one of the main results of~\cite{PeneThomine:2019}, where it was proved using an argument on the distributional asymptotics of the Birkhoff sums for $\widetilde{T}$ of $\mathbf{1}_{[\sigma_\varepsilon(i)]}-\mathbf{1}_{[\sigma_\varepsilon(j)]}$. Lemma~\ref{lem:MixingLemmaI2} is already a more general variant of~\cite[Corollary~2.9]{PeneThomine:2019}.
 
 \smallskip
 
 In addition, much of the proof of Lemma~\ref{lem:MasterLemma} can be simplified in the case $|I|=2$, since we can replace 
 matrices by scalars, giving a more streamlined proof of Lemma~\ref{lem:MixingLemmaI2}. More generic tools, for instance from~\cite{KellerLiverani:2009a}, may be applicable.
\end{remark}

The boundedness of the family $\varepsilon^{-1}(\Id-P_\varepsilon)$ in the general case readily follows.

\begin{lemma}\quad
\label{lem:MatriceTransitionBornee}
 
 Assume Hypotheses~\ref{hyp:Recurrence}, \ref{hyp:Mixing} and \ref{hyp:AsymptoticsQ}. Assume furthermore that the matrix $S$ from Hypothesis~\ref{hyp:AsymptoticsQ} is irreducible in the sense of Definition~\ref{def:QMatriceIrreductible}.
 
 \smallskip
 
 Then $\Id-P_\varepsilon = O(\varepsilon)$.
\end{lemma}

\begin{proof}
 
 Let $i \neq j \in I$. Let $\widehat{\Sigma}_\varepsilon := \{\sigma_\varepsilon (i), \sigma_\varepsilon (j)\} \subset \Sigma_\varepsilon$. The potential operator $\widehat{Q}_\varepsilon$ associated with the induced map on $\widehat{\Sigma}_\varepsilon$ can be recovered from $Q_\varepsilon$ in the following way:
 \begin{itemize}
  \item Take a function $f \in \C_0^{\{i, j\}}$. Extend $f$ to a function $f \mathbf{1}_{\{i,j\}} \in \C_0^I$ by setting $f \mathbf{1}_{\{i,j\}} (k) := 0$ for $k \notin \{i,j\}$.
  \item Apply $Q_\varepsilon$ to $f \mathbf{1}_{\{i,j\}}$ to get a function $g \in \C_0^I$.
  \item Restrict $g$ to $\{i,j\}$, to get a function $g_{|\{i,j\}} \in \C^{\{i,j\}}$.
  \item Add a constant to $g_{|\{i,j\}}$ to get the function $\widehat{Q}_\varepsilon f \in \C_0^{\{i, j\}}$.
 \end{itemize}
 The family $(\widehat{Q}_\varepsilon)_{\varepsilon > 0}$ thus has the asymptotics
 \begin{equation*}
  \widehat{Q}_\varepsilon 
  = \varepsilon^{-1} \widehat{S} + o(\varepsilon^{-1}),
 \end{equation*}
 where $\widehat{S}$ is the restriction of $S$ to $\C_0^{\{i, j\}}$ composed with the projection onto $\C_0^{\{i, j\}}$ parallel to constants. Since $(\mathbf{1}_i - \mathbf{1}_j)$ is supported on $\{i,j\}$ and has zero sum,
 \begin{equation*}
  \langle (\mathbf{1}_i-\mathbf{1}_j), \widehat{S} (\mathbf{1}_i-\mathbf{1}_j) \rangle_{\ell^2 (I)} 
  = \langle (\mathbf{1}_i-\mathbf{1}_j), S (\mathbf{1}_i-\mathbf{1}_j) \rangle_{\ell^2 (I)} 
  > 0,
 \end{equation*}
 so that $\widehat{S}$ is also irreducible in the sense of Definition~\ref{def:QMatriceIrreductible}. By Lemma~\ref{lem:MixingLemmaI2},
 \begin{equation}
 \label{eq:BorneRetourTheta}
    \mu \left( (\widetilde{T}^n (x,\sigma_\varepsilon (i)))_{n \geq 1} \text{ reaches } [\sigma_\varepsilon (j)] \text{ before going back to } [\sigma_\varepsilon (i)] \right) 
    \sim_{\varepsilon \to 0} \frac{\varepsilon}{\langle (\mathbf{1}_i-\mathbf{1}_j), \widehat{S} (\mathbf{1}_i-\mathbf{1}_j) \rangle_{\ell^2 (I)}}.
\end{equation}
But the left hand-side is larger than or equal to $(P_\varepsilon)_{ij}$, so that all the off-diagonal terms of $P_\varepsilon$ are in $O(\varepsilon)$.
\end{proof}

\subsection{Irreducibility of the bi-L-matrix}
\label{subsec:IrredR}

We now end the proof of Theorem~\ref{thm:MainTheorem} when there are more than two sites. Up to extraction of a subsequence, we may assume that $\Id-P_\varepsilon = \varepsilon R + o(\varepsilon)$. The matrix $R$ is a bi-L-matrix. We want to show that $R_0^{-1} = S$. This is a consequence of Lemma~\ref{lem:MasterLemma} if $R$ is irreducible. Our next focus is to prove 
that $R$ is indeed irreducible.

\begin{lemma}\quad
\label{lem:SToIrredR}
 
 Assume Hypotheses~\ref{hyp:Recurrence}, \ref{hyp:Mixing} and~\ref{hyp:AsymptoticsQ}. 
Assume furthermore that the matrix $S$ from Hypothesis~\ref{hyp:AsymptoticsQ} is irreducible 
 in the sense of Definition~\ref{def:QMatriceIrreductible}.
 
 \smallskip
 
 Let $R$ be an adherence point of $(\varepsilon^{-1} (\Id-P_\varepsilon))_{\varepsilon >0}$. 
 Then $R$ is an irreducible bi-L-matrix.
\end{lemma}

\begin{proof}
 
Let $R$ be such an adherence point. Since $P_\varepsilon$ is bi-stochastic by Lemma~\ref{lem:MatriceBiStochastique} and the space of bi-L-matrices is closed, $R$ is a bi-L-matrix.

\smallskip

By Equation~\eqref{eq:BorneRetourTheta}, under these hypotheses, for any two distinct $i$, $j \in I$,
\begin{equation*}
  \mu \left( (\widetilde{T}^n (x,\sigma_\varepsilon (i)))_{n \geq 1} \text{ reaches } [\sigma_\varepsilon (j)] \text{ before going back to } [\sigma_\varepsilon (i)] \right) 
  = \Theta (\varepsilon).
\end{equation*}

Inducing on $[\Sigma_\varepsilon]$ yields 
\begin{equation*}
  \mu \left( (T_\varepsilon^n (x,i))_{n \geq 1} \text{ reaches } [j] \text{ before going back to } [i] \right) 
  = \Theta (\varepsilon).
\end{equation*}

We proceed by contradiction: if $R$ is reducible, we prove that there exist $i$ and $j$ such that the above probability is an $o(\varepsilon)$.

\smallskip

Let $\sim_R$ be the relation on $I$ generated by $i \sim_R j$ if $R_{ij} <0$ and $i \sim_R i$. The relation $\sim_R$ is reflexive and transitive. Since the counting measure is in the kernel of $R$, there are no absorbing subsets of $I$, so the relation $\sim_R$ is also symmetric.

\smallskip

Assume that $R$ is not irreducible. Let $I_1$ be an equivalence class for $\sim_R$. Let $i \in I_1$ and $j \in I \setminus I_1$. By definition of $\sim_R$, we have $P_{\varepsilon, k \ell} = o(\varepsilon)$ for all $k \in I_1$ and $\ell \in I \setminus I_1$. Setting
\begin{equation*}
B_{\varepsilon,1} = \{(x,k) \in [I_1]: \ T_\varepsilon (x,k) \in [I \setminus I_1]\},
\end{equation*}
the set of points in $[I_1]$ which exit $[I_1]$ under $T_\varepsilon$, we get $\mu_{I_1} (B_{\varepsilon,1}) = o (\varepsilon)$.

\smallskip

We want to estimate the probability, starting from $[i]$, to hit $[j]$ before going 
back to $[i]$. Let $T_{\varepsilon,1} : [I_1] \to [I_1]$ be the map obtained by inducing $T_\varepsilon$ on $[I_1]$. Let 
\begin{equation*}
\varphi_{[i]}^{T_{\varepsilon, 1}} (x) 
:= \inf \{n \geq 1 : \ T_{\varepsilon, 1}^n (x,i) \in [I] \}.
\end{equation*}
Note that
\begin{equation*}
 \left\{ (T_\varepsilon^n (x,i))_{n \geq 1} \text{ reaches } [j] \text{ before going back to } [i] \right\}
 \subset \left\{ \exists 0 \leq n < \varphi_{[i]}^{T_{\varepsilon, 1}} (x) : \ T_{\varepsilon, 1}^n (x,i) \in B_{\varepsilon,1} \right\},
\end{equation*}
whence, using Kac's formula,
\begin{align*}
 \mu \left( (T_\varepsilon^n (x,i))_{n \geq 1} \text{ reaches } [j] \text{ before going back to } [i] \right)
 & \leq \mu \left( \exists 0 \leq n < \varphi_{[i]}^{T_{\varepsilon, 1}} (x) : \ T_{\varepsilon, 1}^n (x,i) \in B_{\varepsilon,1} \right) \\
 & = \int_A \max_{0 \leq n < \varphi_{[i]}^{T_{\varepsilon, 1}} (x)} \mathbf{1}_{B_{\varepsilon,1}} \circ T_{\varepsilon, 1}^n (x,i) \dd \mu (x) \\
 & \leq \int_A \sum_{0 \leq n < \varphi_{[i]}^{T_{\varepsilon, 1}} (x)} \mathbf{1}_{B_{\varepsilon,1}} \circ T_{\varepsilon, 1}^n (x,i) \dd \mu (x) \\
 & = |I_1| \mu_{I_1} (B_{\varepsilon,1}) \\
 & = o(\varepsilon).
\end{align*}
This gives the contradiction we want. Hence, our additional assumption is false: $R$ is indeed irreducible. 
\end{proof}

Let us wrap up to proof of Theorem~\ref{thm:MainTheorem}.

\begin{proof}[Proof of Theorem~\ref{thm:MainTheorem}]
 
 The direct implication is the content of Lemma~\ref{lem:MasterLemma}; we prove the converse direction. Assume Hypotheses~\ref{hyp:Recurrence} and~\ref{hyp:AsymptoticsQ}. Assume furthermore that the matrix $S$ from Hypothesis~\ref{hyp:AsymptoticsQ} is irreducible in the sense of Definition~\ref{def:QMatriceIrreductible}.
 
 \smallskip
 
 By Lemma~\ref{lem:MatriceTransitionBornee}, the family $(\varepsilon^{-1} (\Id-P_\varepsilon))_{\varepsilon>0}$ is bounded. Let $R$ be one if its adherence points. By Lemma~\ref{lem:SToIrredR}, $R$ is an irreducible bi-L-matrix.
 
 \smallskip
 
 Let $(\varepsilon_k)_{k \geq 0}$ be converging to $0$ such that $P_{\varepsilon_k} = \Id-\varepsilon_k R + o(\varepsilon_k)$. Since $R$ is a bi-L-matrix, it is null on constant functions and preserves $\C_0^I$. By Lemma~\ref{lem:MasterLemma}, we have $S = R_0^{-1}$, whence $R_0 = S^{-1}$. This characterizes the action of $R$ on $\C_0^I$. Hence the matrix $R$ is entirely characterized by $S$, and thus the adherence point of $(\varepsilon^{-1} (\Id-P_\varepsilon))_{\varepsilon>0}$ is unique. 
 This proves that $P_\varepsilon = \Id-\varepsilon R + o(\varepsilon)$, and finishes the proof of Theorem~\ref{thm:MainTheorem}. 
\end{proof}

\section{Fourier transform and perturbations of the transfer operator}
\label{sec:FourierTransform}

Let $(A, \mu, T)$ be an ergodic Gibbs-Markov map with transfer operator $\Lcal$ acting on $\Bcal$. Let $F : A \to \Z^d$ be Markov, and assume that the $\Z^d$-extension $([\Z^d], \widetilde{\mu}, \widetilde{T})$ is ergodic. Let $(\Sigma_\varepsilon)_{\varepsilon > 0}$ be a family of subsets of $\Z^d$. Theorem~\ref{thm:MainTheorem} relates the transition matrices $(P_\varepsilon)_{\varepsilon > 0}$ to the potential operators $(Q_\varepsilon)_{\varepsilon > 0}$.

\smallskip

Our goal is to compute, in a few cases, the asymptotics of the operators $(Q_\varepsilon)_{\varepsilon > 0}$, and in particular the operator $S$ appearing in Hypothesis~\ref{hyp:AsymptoticsQ}; from there we shall deduce, at least in some cases, the asymptotics of the transition matrices $(P_\varepsilon)_{\varepsilon > 0}$. First, we shall develop general tools involving perturbations of transfer operators.

\smallskip

In order to circumvent some summability issues which may arise in the computation of 
\begin{equation*}
 (\Id-\widetilde{\Lcal})^{-1} 
 = \sum_{n \geq 0} \widetilde{\Lcal}^n,
\end{equation*}
we shall approximate this operator by 
\begin{equation*}
 (\Id-\rho \widetilde{\Lcal})^{-1} 
 = \sum_{n \geq 0} \rho^n\widetilde{\Lcal}^n
\end{equation*}
with $|\rho|<1$, and then take the limit as $\rho \to 1$. However, the requires working with a twisted Poisson equation and complicates significantly the relationship between $\widetilde{\Lcal}$ and $\Lcal_\varepsilon$. This will be the object of Subsection~\ref{subsec:PoissonTordue}, where we shall prove the following:

\begin{proposition}\quad
\label{prop:LimitePoissonTordue}
 
 Let $([\Z^d], \widetilde{\mu}, \widetilde{T})$ be an ergodic and recurrent Markov $\Z^d$-extension of a Gibbs-Markov map. Let $I$ be non-empty and finite and $\sigma : I \hookrightarrow \Z^d$ with image $\Sigma$. 
 
 \smallskip
 
 Let $Q_\sigma : \C_0^I \to \C_0^I$ be defined, as in Equation~\eqref{eq:ExpressionIntegraleQ}, by :
 \begin{equation*}
  \langle f, Q_\sigma (g) \rangle_{\ell^2 (I)} 
  := |I| \int_{[\Z^d]} (\Id-\Lcal_{[\Sigma]})^{-1} \left( \sum_{i \in I} f_i \mathbf{1}_{[\sigma(i)]} \right) \cdot  \left(\sum_{j \in I} g_j \mathbf{1}_{[\sigma(j)]}\right) \dd \mu_{[\Sigma]}.
 \end{equation*}
 Then:
 \begin{equation*}
  \langle f, Q_\sigma (g) \rangle_{\ell^2 (I)} 
  = \lim_{\rho \to 1^-} \sum_{n \geq 0} \rho^n \sum_{i, j \in I} f_i g_j \mu\left(S_n F = \sigma(j) - \sigma(i)\right).
 \end{equation*}
\end{proposition}

For random walks, the computation of $Q_\varepsilon = Q_{\sigma_\varepsilon}$ involves taking the Fourier transform of the distribution of $F$, or in other words, uses the characteristic function of $F$. As mentioned in Sub-subsection~\ref{subsubsec:Examples}, in a dynamical context, the role of this characteristic function is taken by the main eigenvalue of twisted transfer operators. We recall that these operators are defined, for $w \in \Tbb^d := \R^d_{2 \pi \Z^d}$, by
\begin{equation}
 \Lcal_w (\cdot)
 = \Lcal (e^{i \langle w, F \rangle} \cdot),
\end{equation}
which acts on $\Lbb^p (A, \mu)$ for all $p \in [1, \infty]$, as well as on $\Bcal$. 

\smallskip

In addition, we also define a Fourier transform for functions defined on $I$, transporting them first to $\Sigma$ before taking their Fourier transform on $\Z^d$:

\begin{definition}[Fourier transform operator]\quad
 
 Let $d \geq 1$, $I$ be finite and $\sigma : I \hookrightarrow \Z^d$ an injective map. We define for all $f\in \C^I$:
 \begin{equation}
  \Fcal_\sigma (f) (w) 
  := \sum_{j \in I} f_j e^{-i\langle w, \sigma (j)\rangle} \quad \forall w \in \Tbb^d
 \end{equation}
 the Fourier transform of the function taking value $f_j$ on $\sigma (j)$ and $0$ everywhere else. We also set $\widecheck{\Fcal}_\sigma (f) (w) := \Fcal_\sigma (f) (-w)$.
 
 \smallskip
 
 Given a family $\sigma_\varepsilon : I \hookrightarrow \Z^d$ of injective maps, we denote by $\Fcal_\varepsilon = \Fcal_{\sigma_\varepsilon}$ the corresponding Fourier transform operators.
\end{definition}

We want to understand the behaviour of $\sum_{n \geq 0} \widetilde{\Lcal}^n$ using Fourier transform, which implies understanding the behaviour of $\Lcal_w$ for all $w \in \Tbb^d$. This will be done in two steps. In Subsection~\ref{subsec:Periodicites}, 
we explore the eigenvalues of modulus $1$ of $\Lcal_w$ in order to ensure that $(\Id-\Lcal_w)^{-1}$ is well-defined. Then, in Subsection~\ref{subsec:TransformeeFourier}, we apply this study together with Proposition~\ref{prop:LimitePoissonTordue} to get the proposition mentioned in Sub-subsection~\ref{subsubsec:Examples}:

\begin{proposition}\quad
\label{prop:LimitePoissonTF}
 
 Let $([\Z^d], \widetilde{\mu}, \widetilde{T})$ be an ergodic and recurrent Markov $\Z^d$-extension of a Gibbs-Markov map. Let $I$ be non-empty and finite and $\sigma : I \hookrightarrow \Z^d$ with image $\Sigma$. Then, for all $f$, $g \in \C_0^I$:
 
 \smallskip
 
 \begin{equation}
 \label{eq:ExpressionPotentielFourier}
 \langle f, Q_\sigma (g) \rangle_{\ell^2 (I)} 
 = \lim_{\rho \to 1^-} \frac{1}{(2\pi)^d} \int_{\Tbb^d} \widecheck{\Fcal}_\sigma (f) \Fcal_\sigma (g) \left( \int_A (\Id-\rho \Lcal_w )^{-1} (\mathbf{1})\dd \mu \right) \dd w.
 \end{equation}

 \smallskip
 
 In particular, let $U$ be a small enough neighbourhood of $0 \in \Tbb^d$. Let $\lambda_w$ be the main eigenvalue of $\Lcal_w$ such that $\lambda_0 = 1$, which is well-defined and continuous on $U$. Then:
 \begin{equation}
 \label{eq:BonneBornePourTF}
  \langle f, Q_\sigma (g) \rangle_{\ell^2 (I)} 
  = \lim_{\rho \to 1^-} \frac{1}{(2\pi)^d} \int_U \widecheck{\Fcal}_\sigma (f) \Fcal_\sigma (g) \frac{1+\delta_w}{1-\rho \lambda_w} \dd w + O_U (1) \norm{f}{} \cdot \norm{g}{},
 \end{equation}
 where $w \mapsto \delta_w$ is continuous on a neighbourhood of $0\in \Tbb^d$ with $\delta_0 = 0$, and the error term $O_U (1)$ is uniform in $\rho$ and $\sigma$ but depends on $U$. 
\end{proposition}

\begin{conjecture}\quad
 
 If $|1-\lambda_w| \geq c \norm{w}{}^2$ for some $c > 0$, or even $|1-\lambda_w| \geq c \norm{w}{}^{3-\iota}$ for some $c$, $\iota > 0$, then Equation~\eqref{eq:BonneBornePourTF} can be simplified, using the fact that $\Fcal_\sigma (f)$ and $\Fcal_\sigma (f)$ are both $O(\norm{w}{})$ near $0$:
 \begin{equation*}
 \lim_{\rho \to 1^-} \frac{1}{(2\pi)^d} \int_U \widecheck{\Fcal}_\sigma (f) \Fcal_\sigma (g) \frac{1+\delta_{\rho, w}}{1-\rho \lambda_w} \dd w
 = \frac{1}{(2\pi)^d} \int_U \widecheck{\Fcal}_\sigma (f) \Fcal_\sigma (g) \frac{1+\delta_w}{1-\lambda_w} \dd w.
 \end{equation*}
 The condition $|1-\lambda_w| \geq c \norm{w}{}^2$ is true for ergodic random walks, and all examples in Section~\ref{sec:Calculs}. A general argument still eludes the author. 
\end{conjecture}

Equation~\eqref{eq:BonneBornePourTF} is exactly what we need to compute the asymptotics of $(Q_\varepsilon)_{\varepsilon > 0}$ from the asymptotics of the main eigenvalue of the twisted transfer operators $\Lcal_w$. Applications will be given in Section~\ref{sec:Calculs}, where we shall explore a few examples with more constraints on the family $(\Sigma_\varepsilon)_{\varepsilon > 0}$ and the jump function $F$: extensions with square-integrable jumps in dimensions $1$ and $2$, as well as extensions with jumps in the basin of attraction of a L\'evy stable distribution in dimension $d=1$.

\subsection{Twisted Poisson equations}
\label{subsec:PoissonTordue}

Our first goal is to prove Proposition~\ref{prop:LimitePoissonTordue}. Let $([\Z^d], \widetilde{\mu}, \widetilde{T})$ be an ergodic and recurrent Markov $\Z^d$-extension of a Gibbs-Markov map. Let $I$ be non-empty and finite and $\sigma : I \hookrightarrow \Z^d$ with image $\Sigma$.

\begin{definition}\quad
 
 For all $\rho \in B(0,1)$, let $Q_{\sigma, \rho} : \C_0^I \to \C_0^I$ be the operator such that, for all $f$, $g \in \C_0^I$,
 \begin{equation*}
   \langle f, Q_{\sigma, \rho} (g) \rangle_{\ell^2 (I)} 
   := \int_{[\Z^d]} \sum_{n=0}^{+ \infty} \rho^n \widetilde{\Lcal}^n \left( \sum_{i \in I} f_i \mathbf{1}_{[\sigma(i)]} \right) \cdot  \left(\sum_{j \in I} g_j \mathbf{1}_{[\sigma(j)]}\right) \dd \widetilde{\mu}.
 \end{equation*}
 Note that the sum is well-defined since $\rho \widetilde{\Lcal}$ is a strict contraction on $\Lbb^1 ([\Z^d], \widetilde{\mu})$. 
 We define $Q_{\sigma, 1}$ as in Equation~\eqref{eq:ExpressionIntegraleQ}, replacing $\Sigma_\varepsilon$ 
 by $\Sigma$.
\end{definition}

Note that, for all $n \geq 0$,
\begin{align*}
 \int_{[\Z^d]} \widetilde{\Lcal}^n \left( \sum_{i \in I} f_i \mathbf{1}_{[\sigma(i)]} \right) \cdot  \left(\sum_{j \in I} g_j \mathbf{1}_{[\sigma(j)]}\right) \dd \widetilde{\mu} 
 & = \int_{[\Z^d]} \left( \sum_{i \in I} f_i \mathbf{1}_{[\sigma(i)]} \right) \cdot \left(\sum_{j \in I} g_j \mathbf{1}_{[\sigma(j)]}\right) \circ \widetilde{T}^n \dd \widetilde{\mu} \\
 & = \sum_{i, j \in I} f_i g_j \int_{[\sigma(i)]} \mathbf{1}_{[\sigma(j)]} \circ \widetilde{T}^n \dd \mu \\
 & = \sum_{i, j \in I} f_i g_j \mu \left(S_n F = \sigma(j)-\sigma(i)\right),
\end{align*}
so we need to prove that $\lim_{\rho \to 1} Q_{\sigma, \rho} = Q_{\sigma, 1}$.

\smallskip

Note that $[I]$ is endowed with a partition $\alpha_\Sigma$ generated by $\varphi_{[\Sigma]}$. More precisely, elements of $\alpha_\Sigma$ in $[i] \subset [I]$ are elements of the partition $\alpha_{i, \Sigma}$ generated by the stopping time
\begin{equation*}
 \varphi_{i, [\Sigma]} 
  := \inf \{n \geq 1 : \ \sigma(i) + S_n F \in \Sigma\}.
\end{equation*}
That said, we will need the following

\begin{lemma}\quad
\label{lem:Tricherie}
 
 Let $([\Z^d], \widetilde{\mu}, \widetilde{T})$ be an ergodic and recurrent Markov $\Z^d$-extension of a Gibbs-Markov map. Let $I$ be non-empty and finite and $\sigma : I \hookrightarrow \Z^d$ with image $\Sigma$. 
 
 \smallskip
 
 Let $\varphi : [I] \to \R_+$ a nonnegative function constant on the elements of the partition $\alpha_\Sigma$ of $[I]$. Then:
 \begin{itemize}
  \item the family of operators $(\Lcal_k)_{k \geq 0}$ acting on $\Bcal_I$ and defined by $\Lcal_k (f) := \Lcal_{[\Sigma]} (e^{-k \varphi}f)$ is continuous in the operator topology.
  \item for $k$ small enough, there exists a continuous family of eigenvalues $\lambda_k$ of $\Lcal_k$ such that $\lambda_0 = 1$.
  \item setting $H(k) := \int_{[I]} (1-e^{-k\varphi}) \dd \mu_I$, we have $(1-\lambda_k) \sim_{k \to 0} H(k)$.
  \item $\norm{\Lcal-\Lcal_k}{\Bcal_I \to \Bcal_I} = O(H(k))$.
 \end{itemize}
\end{lemma}

\begin{proof}
 
 We start by proving a strengthened version of the last point. We introduce the auxiliary semi-norm:
 \begin{equation*}
  |f|_{\Lip^1 ([I])} 
  := \sum_{p \in \Sigma} \sum_{a \in \alpha_{p, \Sigma}} \mu (a) |f (\cdot, p)|_{\Lip(a)},
 \end{equation*}
 and the auxiliary norm $\norm{f}{\Lip^1 ([I])} := \norm{f}{\Lbb^1 ([I], \mu_I)} + |f|_{\Lip^1 ([I])}$.
 
 \smallskip
 
 Let $k$, $k' \geq 0$ and $f \in \Bcal_I$. Then
 \begin{align*}
  \norm{e^{-k \varphi}f - e^{-k' \varphi}f}{\Lbb^1 ([I], \mu_I)} 
  & \leq \norm{e^{-\min\{k, k'\} \varphi} }{\Lbb^\infty ([I], \mu_I)} \norm{1-e^{-|k-k'| \varphi}}{\Lbb^1 ([I], \mu_I)} \norm{f}{\Lbb^\infty ([I], \mu_I)} \\
  & \leq H(|k-k'|) \norm{f}{\Lbb^\infty ([I], \mu_I)}.
 \end{align*}
 In addition, seeing $\varphi$ as a function from $\alpha_\Sigma$ to $\R_+$,
 \begin{align*}
  \left| e^{-k \varphi}f - e^{-k' \varphi}f \right|_{\Lip^1 ([I])} 
  & = \sum_{a \in \alpha_\Sigma} \mu_I (a) \left| (e^{-k \varphi} - e^{-k' \varphi}) f \right|_{\Lip (a)} \\
  & = \sum_{a \in \alpha_\Sigma} \mu_I (a) \left| e^{-k \varphi(a)} - e^{-k' \varphi(a)} \right| \cdot |f|_{\Lip (a)} \\
  & \leq \left( \sum_{a \in \alpha_\Sigma} \mu_I (a) \left| 1-e^{-|k-k'| \varphi(a)} \right| \right) \sup_{a \in \alpha_\Sigma} |f|_{\Lip (a)} \\ 
  & \leq H(|k-k'|) \max_{i \in I} \sup_{a \in \alpha^*} |f(\cdot, i)|_{\Lip (a)},
 \end{align*}
 where we used in the last step that $\alpha^*$ is coarser than $\alpha_{i, \Sigma}$ for all $i \in I$. Writing $(e^{- k \varphi} \cdot )$ the operator of multiplication by $e^{- k \varphi}$, these two computations together yield
 \begin{equation*}
  \norm{(e^{- k \varphi} \cdot ) - (e^{- k' \varphi} \cdot )}{\Bcal_I \to \Lip^1 ([I])} 
  \leq H(|k-k'|).
 \end{equation*}
 By Proposition~\ref{prop:ContinuiteL1LInfty}, the operator $\Lcal_{[\Sigma]} : \Lip^1 ([I]) \to \Bcal_I$ is bounded. Composition of the multiplication operators and $\Lcal_{[\Sigma]}$ gives 
 \begin{equation*}
  \norm{\Lcal_k - \Lcal_{k'}}{\Bcal_I \to \Bcal_I} 
  = O(H(|k-k'|)),
 \end{equation*}
 finishing the proof of the fourth point. 
 
 \smallskip
 
 By the dominated convergence theorem, $\lim_{0^+} H = 0$, finishing the proof of the first point.
 
 \smallskip
 
 Since $([\Z^d], \widetilde{\mu}, \widetilde{T})$ is ergodic, so is $([I], \mu_I, T_{\Sigma})$, so $1$ is an isolated eigenvalue of multiplicity $1$ of $\Lcal_{[\Sigma]}$. The second point follows from~\cite[Part~IV.3.5]{Kato:1995}.
 
 \smallskip
 
 All is left to prove is the third point. Let $(h_k)_{k \geq 0}$ be the family of eigenfunctions of $(\Lcal_k)_{k \geq 0}$ corresponding the the eigenvalue $\lambda_k$ and such that $\mu_I (h_k) = 1$, for all small enough $k$. Then $(h_k)_{k \geq 0}$ is continuous and $h_0 = \mathbf{1}$. Furthermore,
 \begin{align*}
  1-\lambda_k 
  & = \int_{[I]} (\Lcal_{[\Sigma]} - \Lcal_k) h_k \dd \mu_I \\
  & = \int_{[I]} (1-e^{-k \varphi}) h_k \dd \mu_I \\ 
  & = H(k) + \int_{[I]} (1-e^{-k \varphi}) (h_k - \mathbf{1}) \dd \mu_I.
 \end{align*}
 Again using~\cite[Part~IV.3.5]{Kato:1995}, $\norm{\mathbf{1} - h_k}{\Bcal_I} = O \left( \norm{\Lcal_{[\Sigma]} - \Lcal_k}{\Bcal_I \to \Bcal_I} \right) = O(H(k))$, 
 so that
 \begin{align*}
  \left| \int_{[I]} (1-e^{-k \varphi}) (h_k - \mathbf{1}) \dd \mu_I \right| 
  & \leq \norm{1-e^{-k \varphi}}{\Lbb^1 ([I], \mu_I)} \norm{\mathbf{1} - h_k}{\Lbb^\infty ([I], \mu_I)} \\
  & \leq H(k)  \norm{\mathbf{1} - h_k}{\Bcal_I} \\
  & = O(H(k)^2),
 \end{align*}
 whence $(1-\lambda_k) \sim_{k \to 0} H(k)$.
\end{proof}

We now proceed to the proof of Proposition~\ref{prop:LimitePoissonTordue}.

\begin{proof}[Proof of Proposition~\ref{prop:LimitePoissonTordue}]

By~\cite[Lemma~1.7]{PeneThomine:2020}, for all $\rho \in (0,1)$ and $f$, $g \in \C_0^I$,
\begin{equation*}
 \frac{1}{|I|}\int_{[\Sigma]} (\Id-\rho \widetilde{\Lcal})^{-1} \left( \sum_{i \in I} f_i \mathbf{1}_{[\sigma(i)]} \right) \cdot \left( \sum_{j \in I} g_j \mathbf{1}_{[\sigma(j)]} \right) \dd \widetilde{\mu}
 = \int_{[I]} (\Id-\Lcal_I (e^{\ln (\rho) \varphi_{[\Sigma]}} \cdot ) )^{-1} \Pi^* (f) \cdot \Pi^* (g) \dd \mu_I
\end{equation*}
where the operator $(\Id-\rho \widetilde{\Lcal})^{-1}$ is defined on $\Lbb^\infty ([\Z^d], \widetilde{\mu})$ and the operator $(\Id-\Lcal_I (e^{\ln (\rho) \varphi_{[\Sigma]}} \cdot ) )^{-1}$ 
on $\Lbb^\infty ([I], \mu_I)$.

\smallskip

We set $k := -\ln(\rho) \geq 0$. Let $\Lcal_k := \Lcal_I (e^{-k \varphi_{[\Sigma]}} \cdot)$. By Lemma~\ref{lem:Tricherie}, there exists an eigendecomposition of $\Lcal_k$, continuous in $k$, such that 
\begin{equation*}
 \Lcal_k 
 = \lambda_k h_k \otimes \mu_{I, k} + Z_k,
\end{equation*}
 with $\mu_{I, k} (h_k) = \mu_I (h_k) = 1$ for all $k$ and $Z_k$ the restriction of $\Lcal_k$ to $\Ker(\mu_k) \subset \Bcal_I$. Note that $Z_k$ converges to the restriction of $\Lcal_I$ to $\Bcal_{I, 0}$, and $h_k$ to $\mathbf{1}$ as well as $\mu_{I, k}$ to $\mu_I$.
 
 \smallskip
 
 We then get for $f$, $g \in \C_0^I$,
 \begin{align*}
  \lim_{\rho \to 1^-} \langle f, Q_{\sigma, \rho} (g) \rangle_{\ell^2 (I)} 
  & = \lim_{k \to 0^+} |I| \int_{[I]} (\Id-\Lcal_k)^{-1} \Pi^* (f) \cdot \Pi^* (g) \dd \mu_I \\ 
  & = \lim_{k \to 0^+} \left( |I| \int_{[I]} (\Id-Z_k)^{-1} (\Pi^* (f) - \mu_{I, k} (\Pi^* (f))h_k ) \cdot \Pi^* (g) \dd \mu_I \right. \\ 
  & \hspace{2cm} \left.+ |I| \int_{[I]} \frac{\Pi^* (g) h_k \mu_{I,k} (\Pi^* (f))}{1-\lambda_k} \dd \mu_I \right).
 \end{align*}
 On the one hand, since $\Pi^* (f) \in \Bcal_{I, 0}$,
 \begin{align*}
  \lim_{k \to 0^+} |I| \int_{[I]} (\Id-Z_k)^{-1} & (\Pi^* (f) - \mu_{I, k} (\Pi^* (f))h_k ) \cdot  \Pi^* (g) \dd \mu_I \\
  & = \lim_{k \to 0^+} |I| \int_{[I]} (\Id-Z_k)^{-1} (\Pi^* (f)) \cdot  \Pi^* (g) \dd \mu_I \\
  & = |I| \int_{[I]} (\Id-\Lcal_{[\Sigma]})^{-1} \Pi^* (f) \cdot \Pi^* (g) \dd \mu_I \\
  & = \langle f, Q_\sigma (g) \rangle_{\ell^2 (I)}.
 \end{align*}
 On the other hand, $\norm{\mathbf{1}-h_k}{\Bcal_I} = O(H(k))$, 
 as well as $\mu_{I,k} (\Pi^* (g)) = O(H(k))$ and $(1-\lambda_k) \sim_{k \to 0^+} H(k)$, 
 where $H(k) = \int_{[I]} (1-e^{-k \varphi_{[\Sigma]}} ) \dd \mu_I$. 
 Hence
 \begin{align*}
  |I| \int_{[I]} \frac{\Pi^* (g) h_k \mu_{I,k} (\Pi^* (f))}{1-\lambda_k} \dd \mu_I 
  & = \int_{[I]} \Pi^* (g) h_k \dd \mu_I \cdot O(1) \\
  & = \left( \int_{[I]} \Pi^* (g) \dd \mu_I + O(H(k)) \right) O(1) \\
  & = O(H(k)).
 \end{align*}
 The claim follows from the fact that $\lim_{0^+} H = 0$.
 \end{proof}

\subsection{Transfer operators and periodicities}
\label{subsec:Periodicites}

After taking a Fourier transform, we will have to work with the inverse of twisted transfer operators $(\Id - \Lcal_w)^{-1}$. We need to find the values of $w$ such that $1 \in \Sp (\Lcal_w \acts \Bcal)$ in order to control these inverse operators. This leads us to a refinement of the spectral analysis done in ~\cite{PeneThomine:2019}, in particular~\cite[Proposition~3.2]{PeneThomine:2019}, so that we can remove the assumption of aperiodicity present in that article.

\smallskip

The action of $\Lcal$ on $\Bcal$ is quasi-compact: it admits $1$ as a simple eigenvalue, has spectral radius $1$ and essential spectral radius $\Lambda^{-1}<1$, where $\Lambda$ is the expansivity constant of the Gibbs-Markov map. The essential 
spectral radius of the operators $\Lcal_w$ is also bounded by $\Lambda^{-1}$, as they satisfy a Lasota-Yorke inequality with the same constants.

\begin{lemma}\quad
\label{lem:DecompositionSpectrale}

 Let $([\Z^d], \widetilde{\mu}, \widetilde{T})$ be a Markov $\Z^d$-extension of a Gibbs-Markov map 
 $(A, \alpha, d, \mu, T)$ with jump function $F$. Assume that $([\Z^d], \widetilde{\mu}, \widetilde{T})$ is ergodic.
 
 \smallskip
 
 Let $M := |\{\lambda \in \Sp(\Lcal \acts \Bcal) : \ |\lambda| = 1\}|$. Then there exist an integer $n \geq 1$ and an element 
 $w_0 \in \Tbb^d$ of primitive order $n$ such that 
 \begin{equation*}
  H 
  := \{(w,\lambda) \in \Tbb^d \times \Sbb^1 : \ \lambda \in \Sp (\Lcal_w \acts \Bcal) \} 
  = \left\{\left(kw_0, e^{2 \pi i \frac{k}{Mn}} \right) : \ k \in \Z_{/ Mn \Z} \right\}.
 \end{equation*}
 In addition, each of these eigenvalues is simple.
\end{lemma}

\begin{proof}
 
 The argument of~\cite[Lemma~3.1]{Thomine:2020} can be adapted to prove that $H$ is a subgroup of $\Tbb^d \times \Sbb^1$. The map $T$ is not invertible. However, if $h$ is a non-zero function such that $\Lcal_w (h) = \lambda h$ for some $(w, \lambda) \in H$, then, by the equality case of the triangular inequality, $h$ is constant on almost every subset $T^{-1} (x)$ for $x \in A$, 
 so that the computations in the proof of~\cite[Lemma~3.1]{Thomine:2020} stay valid. In addition, each of these eigenvalues is simple.
 
 \smallskip
 
 Let $w \in \Tbb^d$ be such that $1 \in \Sp (\Lcal_w \acts \Bcal)$. Let $h_{w,1}$ be an eigenfunction for this eigenvalue. Then $\Lcal (e^{i \langle w, F \rangle} h_{w,1}) = h_{w,1}$.  Define $f(x,p) := e^{-i \langle w, p \rangle} h_{w,1} (x)$ on $[\Z^d]$. As we check, $f$ is invariant under the transfer operator $\widetilde{\Lcal}$:
 \begin{align*}
  \widetilde{\Lcal} (f) (x,p) 
  & = \sum_{(y,q) \in \widetilde{T}^{-1} (\{(x,p)\})} \frac{f (y,q) }{|\Jac (T) (y)|} \\
  & = \sum_{y \in T^{-1} (\{x\})} \frac{f (y,p-F(y)) }{|\Jac (T) (y)|} \\
  & = e^{-i \langle w, p \rangle} \sum_{y \in T^{-1} (\{x\})} \frac{e^{i \langle w, F \rangle (y)} h_{w,1} (y) }{|\Jac (T) (y)|} \\
  & = e^{-i \langle w, p \rangle} \Lcal_w (h_{w,1}) (x) \\
  & = e^{-i \langle w, p \rangle} h_{w,1} (x) \\ 
  & = f(x,p).
 \end{align*}
 Since $([\Z^d], \widetilde{\mu}, \widetilde{T})$ is ergodic, $f$ is constant, and in particular the map $p \mapsto e^{-i \langle w, p \rangle}$ is constant. Hence $w = 0$ in $\Tbb^d$. In other words, the canonical projection $H \to \Sbb^1$ is injective. Hence $H$ is isomorphic to a subgroup of $\Sbb_1$, and thus cyclic.
 
 \smallskip
 
 Let $\Ubb_M$ be the group of peripheral eigenvalues of $\Lcal$. Then $\{0\} \times \Ubb_M$ 
 is a subgroup of $H$, so $M$ divides $|H|$. Setting $n := |H|/M$, there exists $w_0 \in \Tbb^d$ 
 such that
 \begin{equation*}
  H 
  = \left\{ \left(k w_0, e^{2 \pi i \frac{k}{Mn}} \right) : \ k \in \Z_{/Mn\Z} \right\}.
 \end{equation*}
 Taking $k = n$, we get that $\left(n w_0, e^{2 \pi i \frac{1}{M}} \right) \in H$. But its image under the second projection is $e^{2 \pi i \frac{1}{M}} \in \Sp(\Lcal \acts \Bcal)$, so by injectivity of the second projection, $n w_0 = 0$. In addition, $k w_0 \neq 0$ for $0 < k < n$; otherwise, we would get an eigenvalue of $\Lcal$ not in $\Ubb_M$.
\end{proof}

Lemma~\ref{lem:DecompositionSpectrale} let us control the operators $(\Id-\Lcal_w)^{-1}$.

\begin{corollary}\quad
\label{cor:ControleTFInverse}

 Let $([\Z^d], \widetilde{\mu}, \widetilde{T})$ be a Markov $\Z^d$-extension of a Gibbs-Markov map 
 $(A, \alpha, d, \mu, T)$ with jump function $F$. Assume that $([\Z^d], \widetilde{\mu}, \widetilde{T})$ is ergodic.
 
 \smallskip
 
 Then, for any neighbourhood $U$ of $0 \in \Tbb^d$, 
 \begin{equation}
  \sup_{w \in \Tbb^d \setminus U} \norm{(\Id-\Lcal_w)^{-1}}{\Bcal \to \Bcal}
  < +\infty.
 \end{equation}
\end{corollary}

\begin{proof}
 
 For $w \neq 0$, the spectrum of $\Lcal_w$ does not contain $1$ by Lemma~\ref{lem:DecompositionSpectrale}, so that $(\Id-\Lcal_w)$ is invertible. By the discussion following the proof of~\cite[Corollary~4.1.3]{Gouezel:2008e}, the family of operators $(\Lcal_w)_{w \in \Tbb^d}$ is continuous, so the family of operators $\left((\Id-\Lcal_w)^{-1} \right)_{w \in \Tbb^d \setminus \{0\}}$ is also continuous. The result follows by relative compactness of $\Tbb^d \setminus U$ in $\Tbb^d \setminus \{0\}$.
\end{proof}

\subsection{Fourier transform and twisted transfer operators}
\label{subsec:TransformeeFourier}

As announced, we now prove Proposition~\ref{prop:LimitePoissonTF}.

\begin{proof}[Proof of Proposition~\ref{prop:LimitePoissonTF}]
 
 Let $([\Z^d], \widetilde{\mu}, \widetilde{T})$ be an ergodic and recurrent Markov $\Z^d$-extension of a Gibbs-Markov map. Let $I$ be non-empty and finite and $\sigma : I \hookrightarrow \Z^d$ with image $\Sigma$. 
 
 \smallskip
 
 Let $f$, $g \in \C_0^I$. By Proposition~\ref{prop:LimitePoissonTordue},
 \begin{equation*}
  \langle f, Q_\sigma (g) \rangle_{\ell^2 (I)} 
  = \lim_{\rho \to 1^-} \sum_{n \geq 0} \rho^n \sum_{i, j \in I} f_i g_j \mu \left( S_n F = \sigma(j)-\sigma(i) \right).
 \end{equation*}
 Let $n \geq 0$ and $i$, $j \in I$. Using the Fourier transform,
 \begin{align*}
  \mu \left( S_n F = \sigma(j)-\sigma(i) \right) 
  & = \int_A \mathbf{1}_{0} (S_n F + \sigma(i)-\sigma(j)) \dd \mu \\
  & = \int_A \frac{1}{(2\pi)^d} \int_{\Tbb^d} e^{i \langle w, S_n F + \sigma(i)-\sigma(j)\rangle} \dd w \dd \mu \\
  & = \frac{1}{(2\pi)^d} \int_{\Tbb^d} e^{i \langle w, \sigma(i)-\sigma(j)\rangle} \int_A e^{i \langle w, S_n F\rangle} \dd \mu \dd w\\
  & = \frac{1}{(2\pi)^d} \int_{\Tbb^d} e^{i \langle w, \sigma(i)-\sigma(j)\rangle} \int_A \Lcal_w^n (\mathbf{1}) \dd \mu \dd w.
 \end{align*}
 Summing over $i$, $j \in I$ yields
 \begin{align*}
  \sum_{i, j \in I} f_i g_j & \mu \left( S_n F = \sigma(j)-\sigma(i) \right) \\
  & = \frac{1}{(2\pi)^d} \int_{\Tbb^d}  \left( \sum_{i \in I} f_i e^{i \langle w, \sigma(i)\rangle}\right) \left( \sum_{j \in I} g_j e^{-i \langle w, \sigma(j)\rangle} \right) \left( \int_A \Lcal_w^n (\mathbf{1}) \dd \mu \right) \dd w \\
  & = \frac{1}{(2\pi)^d} \int_{\Tbb^d} \widecheck{\Fcal}_\sigma (f) \Fcal_\sigma (g) \left( \int_A \Lcal_w^n (\mathbf{1}) \dd \mu \right) \dd w,
 \end{align*}
 whence, weighting by $\rho^n$ and summing over $n$,
 \begin{align*}
  \langle f, Q_\sigma (g) \rangle_{\ell^2 (I)} 
  & = \lim_{\rho \to 1^-} \sum_{n \geq 0} \rho^n \frac{1}{(2\pi)^d} \int_{\Tbb^d} \widecheck{\Fcal}_\sigma(f) \Fcal_\sigma (g) \left( \int_A \Lcal_w^n (\mathbf{1}) \dd \mu \right) \dd w, \\
  & = \lim_{\rho \to 1^-} \frac{1}{(2\pi)^d} \int_{\Tbb^d} \widecheck{\Fcal}_\sigma (f) \Fcal_\sigma (g) \left( \int_A (\Id-\rho \Lcal_w )^{-1} (\mathbf{1})\dd \mu \right) \dd w.
 \end{align*}
This proves Equation~\eqref{eq:ExpressionPotentielFourier}.

\smallskip

Since $w \mapsto \Lcal_w$ is continuous and $1$ is an isolated eigenvalue of $\Lcal$, for any small enough neighbourhood $U$ of $0 \in \Tbb^d$, there exists a continuous eigendecomposition:
\begin{equation*}
 \Lcal_w 
 = \lambda_w h_w \otimes \mu_w + Z_w,
\end{equation*}
with $\mu_w (h_w) = \mu(h_w) = 1$ for all $w \in -U$, as well as $\lambda_0 = 1$. In addition, up to taking a smaller neighbourhood $U$, the operators $\left((\Id-\rho Z_w)^{-1}\right)_{\rho \geq 0, w \in U}$ are uniformly bounded by continuity. Then, 
\begin{align*}
 \int_A (\Id-\rho \Lcal_w )^{-1} (\mathbf{1})\dd \mu
 & = \int_A \frac{\mu_w (\mathbf{1}) h_w}{1-\rho \lambda_w} \dd \mu + \int_A (\Id-\rho Z_w )^{-1} (\mathbf{1}-\mu_w (\mathbf{1}) h_w)\dd \mu \\
 & = \frac{\mu_w (\mathbf{1})}{1-\rho \lambda_w} + O (1).
\end{align*}

In addition, the family of operators $\left((\Id-\Lcal_w)^{-1}\right)_{w \in \Tbb^d \setminus U}$ is bounded by Corollary~\ref{cor:ControleTFInverse}. By continuity, the family of operators $\left((\Id-\rho\Lcal_w)^{-1}\right)_{\rho \geq 0, w \in \Tbb^d \setminus U}$ is also bounded, and
\begin{equation*}
 \sup_{\rho \geq 0, w \in \Tbb^d \setminus U} \left| \int_A (\Id-\rho \Lcal_w )^{-1} (\mathbf{1})\dd \mu \right| 
 < +\infty.
\end{equation*}

 Putting everything together, we get
 \begin{equation*}
  \langle f, Q_\sigma (g) \rangle_{\ell^2 (I)} 
  = \lim_{\rho \to 1^-} \frac{1}{(2\pi)^d} \int_{\Tbb^d} \widecheck{\Fcal}_\sigma (f) \Fcal_\sigma (g) \frac{\mu_w (\mathbf{1})}{1-\rho \lambda_w} \dd w + O_U (1) \norm{f}{} \cdot \norm{g}{}.
 \end{equation*}
 Setting $\delta_w := \mu_w (\mathbf{1})-1$, which converges to $0$ when $w$ vanishes, yields  Proposition~\ref{prop:LimitePoissonTF}.
\end{proof}

\section{Computation of the potential operators}
\label{sec:Calculs}

With the groundwork done in Section~\ref{sec:FourierTransform}, we are ready to compute the asymptotics of the transition matrices in a few important cases. We shall consider sets $(\Sigma_\varepsilon)_{\varepsilon >0}$ with a limiting shape as introduced in Subsection~\ref{subsec:LimitShape}. We shall also add classical assumptions on the tails of the distribution the jump $F$, for instance square-integrability or being in the basin of attraction of a L\'evy stable distribution.

\smallskip

In this section, we assume that $([\Z^d], \widetilde{\mu}, \widetilde{T})$ is an ergodic and recurrent Markov $\Z^d$-extension of a Gibbs-Markov map $(A, \mu, T)$ with jump function $F : A \to \Z^d$.

\smallskip

Let $I$ be non-empty and finite and $\sigma : I \hookrightarrow \R^d$. We choose $(\sigma_t)_{t > 0} : I \hookrightarrow \Z^d$ such that, for all $i \in I$,
\begin{equation*}
 \sigma_t (i) 
 = t \sigma(i) + o(t),
\end{equation*}
as in Figure~\ref{fig:FamilleFormes}. This point of view, where $t$ is of the same order as the distance between sites of $\Sigma_t$, seems more intuitive, and shall be used to write the propositions. Letting $\varepsilon := t^{-1}$, the asymptotics $t \to +\infty$ correspond to asymptotic $\varepsilon \to 0$ as in the previous parts of this work, and the variable $\varepsilon$ shall be used in lieu of $t$ in the demonstrations.

\smallskip

 Once this setting is fixed, most computations rely on the same strategies as~\cite[Subsection~3.5]{PeneThomine:2019}. An important difference is that we compute a matrix $S$ such that $Q_\varepsilon \sim a(\varepsilon) S$ for some function $a(\varepsilon)$, and then invert $S$, while~\cite{PeneThomine:2019} only involved the computation of scalars. This adds a layer of linear algebra, which we solve in a few cases, but not in full generality. As a consequence, we split this section in four cases:
 \begin{itemize}
  \item square-integrable jump functions in dimension $1$ (Subsection~\ref{subsec:D1CarreIntegrable});
  \item jump functions in the basin of attraction of a L\'evy stable distribution with parameter in $(1,2)$ in dimension $1$ (Subsection~\ref{subsec:D1Levy}), for which the author was not able to invert $S$ in full generality;
  \item some jump functions in the basin of attraction of a Cauchy distribution (Subsection~\ref{subsec:D1Cauchy});
  \item square-integrable jump functions in dimension $2$ (Subsection~\ref{subsec:D2CarreIntegrable}), with a behaviour \textit{in fine} similar to the Cauchy case but quite different computations.
 \end{itemize}
 There are a few specialized cases which we do not write down, although they may be of some practical interest. In particular, we leave to the interested reader the situation where $F$ has tails with regular variation of index $2$, and thus lies in the basin in attraction of a Gaussian, while not being square-integrable. This specific situation can occur as a toy model for Lorentz gases in infinite horizon, for which the diffusion channels give rise to supercritical diffusion~\cite{Bleher:1992, SzaszVarju:2007}.

\subsection{Computation of the potential: $d=1$, square integrable jumps}
\label{subsec:D1CarreIntegrable}

In addition to the general setting of Section~\ref{sec:Calculs}, we assume in this Subsection that $F$ is square-integrable and centered, the later condition being already implied by the first condition and the ergodicity of the extension. 

\smallskip

Let $\Var$ be the asymptotic variance $F$:
\begin{equation}
\label{eq:GreenKubo}
 \Var
 := \int_A F^2 \dd \mu + 2 \sum_{n \geq 1} \int_A F\circ T^n \cdot F \dd \mu,
\end{equation}
where the infinite series is taken in the sense of Ces\`aro. Since $([\Z], \widetilde{\mu}, \widetilde{T})$ is ergodic, $F$ is not 
a coboundary and thus $\Var$ is positive~\cite[Th\'eor\`eme~4.1.4]{Gouezel:2008e}.

\smallskip

We shall now get our first application of Theorem~\ref{thm:MainTheorem}:

\begin{proposition}\quad
\label{prop:ApplicationD1CarreIntegrable}
 
 Let $([\Z], \widetilde{\mu}, \widetilde{T})$ be an ergodic and recurrent Markov $\Z$-extension of a Gibbs-Markov map $(A, \mu, T)$ with jump function $F : A \to \Z$. Assume furthermore that $F \in \Lbb^2 (A, \mu)$. Let $\Var > 0$ be that asymptotic variance of $F$ defined by Equation~\eqref{eq:GreenKubo}.

\smallskip

Let $I = \{1, \ldots, |I|\}$ and $\sigma : I \hookrightarrow \R$ be injective and increasing. For $t > 0$, let $\sigma_t : I \hookrightarrow \Z$ be such that $\sigma_t (i) = t \sigma(i) + o(t)$ for all $i \in I$. Let $P_t$ be the transition matrix of the Markov chain induced on $\Sigma_t := \sigma_t (I)$. Then
\begin{equation*}
 P_t 
 =_{t \to + \infty} \Id - \frac{\Var}{t} R + o (t^{-1}),
\end{equation*}
where $R$ is the irreducible, tridiagonal, symmetric bi-L-matrix defined by
\begin{equation}
\label{eq:DefinitionRD1CarreIntegrable}
\left\{
 \begin{array}{lcll}
 R_{i, i+1} & = & -\frac{1}{2 (\sigma(i+1) - \sigma(i))} & \quad \forall \ 1 \leq i < |I|, \\
 R_{i, i-1} & = & -\frac{1}{2 (\sigma(i) - \sigma(i-1))} & \quad \forall \ 1 < i \leq |I|, \\
 R_{ij} & = & 0 & \quad \text{whenever } |i-j| \geq 2, \\
 R_{ii} & = & \sum_{\substack{j \in I \\ j \neq i}}  \left| R_{ij} \right| & \quad \forall \ i \in I.
 \end{array}
 \right.
\end{equation}
\end{proposition}

Under the hypotheses of Proposition~\ref{prop:ApplicationD1CarreIntegrable}, and setting $\Delta_i^j := |\sigma(j) - \sigma(i)|$, 
the transition matrices $(P_t)_{t >0}$ have the following graphical representation:
\begin{equation*}
 P_t
 = \left(
 \begin{array}{cccccc}
  1 - \frac{2 \Var}{t\Delta_1^2} & \frac{2 \Var}{t\Delta_1^2} & 0 & \ldots & 0 & 0 \\
  \frac{2 \Var}{t\Delta_1^2} & 1 - \frac{2 \Var}{t\Delta_1^2} - \frac{2 \Var}{t\Delta_2^3} & \frac{2 \Var}{t\Delta_2^3} & \ldots & 0 & 0 \\ 
  0 & \frac{2 \Var}{t\Delta_2^3} & 1 - \frac{2 \Var}{t\Delta_2^3} - \frac{2 \Var}{t\Delta_3^4} & \ldots & 0 & 0 \\
  \vdots & \ddots & \ddots & \ddots & \vdots & \vdots \\
  0 & 0 & 0 & \ldots & 1 - \frac{2 \Var}{t\Delta_{|I|-2}^{|I|-1}} - \frac{2 \Var}{t\Delta_{|I|-1}^{|I|}} & \frac{2 \Var}{t\Delta_{|I|-1}^{|I|}} \\
  0 & 0 & 0 & \ldots & \frac{2 \Var}{t\Delta_{|I|-1}^{|I|}} & 1 - \frac{2 \Var}{t\Delta_{|I|-1}^{|I|}}
 \end{array}
 \right)
 + o(t^{-1}).
\end{equation*}

\begin{proof}
 
We assume the setting of Proposition~\ref{prop:ApplicationD1CarreIntegrable}. By~\cite[Th\'eor\`eme~4.1.4]{Gouezel:2008e}, the main eigenvalue of $\Lcal_w$ admits the expansion
\begin{equation}
\label{eq:DeveloppementVPD1CarreIntegrable}
 \lambda_w 
 = 1 - \frac{\Var w^2}{2} + o(w^2).
\end{equation}

\smallskip

We set $\varepsilon := t^{-1}$. By Proposition~\ref{prop:LimitePoissonTF}, for any small enough neighbourhood $U$ of $0 \in \Tbb$, 
for all $\varepsilon > 0$, for all $f$, $g \in \C_0^I$,
\begin{align*}
 \langle f, Q_\varepsilon (g) \rangle_{\ell^2 (I)} 
 & = \lim_{\rho \to 1^-} \frac{1}{2 \pi} \int_{\Tbb} \widecheck{\Fcal}_\varepsilon(f) \Fcal_\varepsilon(g) \int_A (\Id - \rho \Lcal_w)^{-1} (\mathbf{1}) \dd \mu \dd w \\
 & = \frac{1}{2 \pi} \int_U \widecheck{\Fcal}_\varepsilon(f) \Fcal_\varepsilon(g) \frac{1+\delta_w}{1-\lambda_w} \dd w + O_U (1) \norm{f}{} \norm{g}{}.
\end{align*}
Note that $\widecheck{\Fcal}_\varepsilon(f) (w)$ and $\Fcal_\varepsilon(g) (w)$ are both in $O(|w|)$ while the denominator is in $\Theta (w^2)$ by Equation~\eqref{eq:DeveloppementVPD1CarreIntegrable}, so the integrand is bounded for each $\varepsilon > 0$. This justifies that we took the limit in Equation~\eqref{eq:BonneBornePourTF}.

\smallskip

Let us set
\begin{equation*}
 1+\delta'_w
 := (1+\delta_w) \frac{\Var w^2}{2(1-\lambda_w)},
\end{equation*}
so that $\delta'_w$ is also a continuous function of $w$ with $\lim_{w \to 0} \delta'_w = 0$. In addition, we choose $U = (-\eta, \eta)$ for small enough $\eta > 0$. Then:
\begin{equation}
\label{eq:AControlerD1CarreIntegrable}
 \langle f, Q_\varepsilon (g) \rangle_{\ell^2 (I)} 
 = \frac{1}{\pi \Var} \int_{-\eta}^\eta \widecheck{\Fcal}_\varepsilon(f) \Fcal_\varepsilon(g) \frac{1+\delta'_w}{w^2} \dd w + O_\eta (1) \norm{f}{} \norm{g}{}.
\end{equation}
We split this integral in three parts:
\begin{itemize}
 \item the integral involving the error term $\delta'_w$ on a small neighbourhood $(-\varepsilon, \varepsilon)$ of $0$, which shall be negligible because both $\delta'_w$ and the term $\left( \widecheck{\Fcal}_\varepsilon(f) \Fcal_\varepsilon(g) \right) (w)$ are small near $0$.
 \item the integral involving the error term $\delta'_w$ on $(-\eta, \eta) \setminus (-\varepsilon, \varepsilon)$, which shall be negligible because $\delta'_w$ is small and $w^2$ is large.
 \item the integral without the error term $\delta'_w$, which is the main contribution 
 to $\langle f, Q_\varepsilon (g) \rangle_{\ell^2 (I)}$.
\end{itemize}

\medskip
\textbf{Control of $\delta'_w$ on the neighbourhood $(-\varepsilon, \varepsilon)$.}
\smallskip

To control the part of the integral~\eqref{eq:AControlerD1CarreIntegrable} involving $\delta'_w$ on $(-\varepsilon, \varepsilon)$, remark that, since both $f$ and $g$ have zero sum, $\widecheck{\Fcal}_\varepsilon(f)$ and $\Fcal_\varepsilon(g)$ are in $O(w)$ for all $\varepsilon > 0$. More precisely,
\begin{align}
 \widecheck{\Fcal}_\varepsilon(f) (w) 
 & = \sum_{i \in I} f_i e^{i w\sigma_\varepsilon (i)} \nonumber \\
 & = \sum_{i \in I} f_i \left( e^{i w( \varepsilon^{-1} \sigma (i) + o(\varepsilon^{-1}) )} -1\right) \nonumber \\
 & = O(\varepsilon^{-1} |w| \norm{f}{} ), \label{eq:TFPetiteFrequence}
\end{align}
and likewise for $\Fcal_\varepsilon(g)$. Hence
\begin{align*}
 \int_{-\varepsilon}^\varepsilon \widecheck{\Fcal}_\varepsilon(f) \Fcal_\varepsilon(g) \frac{\delta'_w}{w^2} \dd w 
 & = \int_{(-\varepsilon, \varepsilon)} O \left( \varepsilon^{-2} w^2 \norm{f}{} \norm{g}{} \right) \frac{\max_{w \in [-\eta, \eta]} \left|\delta'_w\right|}{w^2} \dd w \\
 & = O\left(\varepsilon^{-1} \norm{f}{} \norm{g}{} \max_{w \in [-\eta, \eta]} \left|\delta'_w\right|\right).
\end{align*}

\medskip
\textbf{Control of the error term $\delta'_w$ on $(-\eta, \eta) \setminus (-\varepsilon, \varepsilon)$.}
\smallskip

On $(-\eta, \eta) \setminus (-\varepsilon, \varepsilon)$, we get
\begin{align*}
 \left| \int_{(-\eta, \eta) \setminus (-\varepsilon, \varepsilon)} \right. & \left. \widecheck{\Fcal}_\varepsilon(f) \Fcal_\varepsilon(g) \frac{\delta'_w}{w^2} \dd w \right| \\
 & \leq C \int_{(-\eta, \eta) \setminus (-\varepsilon, \varepsilon)} \frac{\max_{w \in [-\eta, \eta]} \left|\delta'_w\right|}{w^2} \dd w \cdot \norm{f}{} \norm{g}{}\\
 & \leq C \varepsilon^{-1} \max_{w \in [-\eta, \eta]} \left|\delta'_w\right| \int_{\R \setminus (-1 , 1)} \frac{1}{w^2} \dd w \cdot \norm{f}{} \norm{g}{} \\
 & = O(\varepsilon^{-1}\norm{f}{} \norm{g}{} \max_{w \in [-\eta, \eta]} \left|\delta'_w\right|).
\end{align*}

\medskip
\textbf{Main contribution to the integral~\eqref{eq:AControlerD1CarreIntegrable}.}
\smallskip

Now let us focus on the remaining term. Let us split $\widecheck{\Fcal}_\varepsilon (f) \Fcal_\varepsilon(g)$ along its frequencies:
\begin{align*}
 \left( \widecheck{\Fcal}_\varepsilon(f) \Fcal_\varepsilon(g) \right) (w)
 & = \left( \sum_{i\in I} f_i e^{i w \sigma_\varepsilon (i)} \right) \left( \sum_{j \in I} g_j e^{-i w \sigma_\varepsilon (j)} \right) \\
 & = \sum_{i, j \in I} f_i g_j e^{i w (\sigma_\varepsilon (i) - \sigma_\varepsilon (j))} \\
 & = \sum_{i, j \in I} f_i g_j \left( e^{i w (\sigma_\varepsilon (i) - \sigma_\varepsilon (j))} - 1 \right).
\end{align*}
Injecting this expression into the main part of Equation~\eqref{eq:AControlerD1CarreIntegrable} yields
\begin{align*}
 \frac{1}{\pi \Var} \int_{-\eta}^\eta & \frac{\widecheck{\Fcal}_\varepsilon(f)\Fcal_\varepsilon(g)}{w^2} \dd w \\
 & = \frac{2}{\pi \Var} \int_0^\eta \sum_{i, j \in I} f_i g_j \frac{ \cos \left( w (\sigma_\varepsilon (i) - \sigma_\varepsilon (j)) \right) -1}{w^2} \dd w \\
 & = \frac{2}{\pi \Var} \sum_{i, j \in I} f_i g_j \int_0^\eta \frac{ \cos \left( \varepsilon^{-1} w (\sigma (i) - \sigma (j) + o(1)) \right) -1 }{w^2} \dd w \\
 & = \frac{2}{\varepsilon \pi \Var}  \sum_{i, j \in I} f_i g_j \left(\sigma (i) - \sigma (j) + o(1) \right) \int_0^{\varepsilon^{-1} (\sigma (i) - \sigma (j) + o(1)) \eta} \frac{ \cos (w) -1 }{w^2} \dd w \\
 & =_{\varepsilon \to 0} - \frac{1}{\varepsilon \Var}  \sum_{i, j \in I} f_i g_j\left|\sigma (i) - \sigma (j) \right| + o(\varepsilon^{-1}),
\end{align*}
were the last line comes from the well-known integral~\cite[Chap.~XVII.4,~\(4.11\)]{Feller:1966}
\begin{equation*}
 \int_0^{+\infty} \frac{\cos (w) -1}{w^2} \dd w 
 = -\frac{\pi}{2}.
\end{equation*}

\medskip
\textbf{Asymptotics of the potential $Q_\varepsilon$.}
\smallskip

Adding the terms involving $\delta'_w$ to the main contribution yields
\begin{align*}
 \limsup_{\varepsilon \to 0} \left| \varepsilon \langle f, Q_\varepsilon (g) \rangle_{\ell^2 (I)} +  \frac{1}{\Var}  \sum_{i, j \in I} f_i g_j\left|\sigma (i) - \sigma (j) \right| \right|
 & = \limsup_{\varepsilon \to 0}  \left( \varepsilon O_\eta (1) + O(\max_{w \in [-\eta, \eta]} \left|\delta_w\right|) \right) \norm{f}{} \norm{g}{} \\
 & = O(\max_{w \in [-\eta, \eta]} |\delta_w|) \norm{f}{} \norm{g}{};
\end{align*}
since the above holds for all small enough $\eta > 0$ and $\lim_{w \to 0} \delta_w = 0$, we get finally that
\begin{equation*}
 \langle f, Q_\varepsilon (g) \rangle_{\ell^2 (I)} 
 \sim_{\varepsilon \to 0} - \frac{1}{\varepsilon \Var}  \sum_{i, j \in I} f_i g_j\left|\sigma (i) - \sigma (j) \right|.
\end{equation*}

Let $S : \C_0^I \to \C^I$ be the operator defined by
\begin{equation*}
 \langle f, S (g) \rangle_{\ell^2 (I)} 
 = -\sum_{i, j \in I} f_i g_j \left|\sigma (i) - \sigma (j) \right|.
\end{equation*}
Then $S$ is symmetric and irreducible in the sense of Definition~\ref{def:QMatriceIrreductible}. Indeed, for all $i \neq j \in I$,
\begin{equation*}
 \langle \mathbf{1}_i-\mathbf{1}_j, S (\mathbf{1}_i-\mathbf{1}_j) \rangle_{\ell^2 (I)}
 = 2\left|\sigma (i) - \sigma (j) \right| 
 > 0.
\end{equation*}
Let $R'$ be the unique irreducible bi-L-matrix such that $R'_0 = S^{-1}$. By Theorem~\ref{thm:MainTheorem}, 
\begin{equation*}
 P_\varepsilon 
 = \Id - \varepsilon \Var R' + o(\varepsilon);
\end{equation*}
setting $t = \varepsilon^{-1}$, all is left is to check that $R' = R$, or in other words, that $R_0 = S^{-1}$.

\medskip
\textbf{Asymptotics of the transition matrix $P_\varepsilon$.}
\smallskip

We now use the assumption that $I = \{1, \ldots, |I|\}$ and that $\sigma$ is increasing, which always holds up to permutation of $I$. We shall work with a convenient basis for $\C_0^I$. Set $e_i^{i+1} := \mathbf{1}_{i+1} - \mathbf{1}_i$ for $1 \leq i < |I|$. Set also $\Delta_i^j := \left|\sigma (i) - \sigma (j) \right|$.

\smallskip

Let $\widehat{S} : \C^I \to \C^I$ such that $\widehat{S}_{ij} = -\Delta_i^j$, so that $S$ is the composition of $\widehat{S}_{|\C_0^I} : \C_0^I \to \C^I$ with the projection parallel to constants from $\C^I$ to $\C_0^I$. Then, for $1 \leq i < |I|$ and $j \in I$,
\begin{equation*}
 (\widehat{S} e_i^{i+1})_j 
 = \Delta_i^j - \Delta_{i+1}^j 
 = \left\{ 
    \begin{array}{lll} 
        - \Delta_i^{i+1} & \text{ if } & j \leq i, \\
        \Delta_i^{i+1} & \text{ if } & j \geq i+1.
        \end{array} \right.
\end{equation*}
For $j < i$, the non-zero coefficients of the $j$th line of $R$ are all in the range $\{1, \ldots, i \}$, so that 
\begin{equation*}
 (R\widehat{S} e_i^{i+1})_j 
 = - \Delta_i^{i+1} (R \mathbf{1})_j 
 = 0.
\end{equation*}
Likewise, $(R\widehat{S} e_i^{i+1})_j = 0$ for $j \geq i+2$. The last two cases $j = i$ and $j = i+1$ finally yield $R\widehat{S} e_i^{i+1} = e_i^{i+1}$ for all $i$. As $R$ preserves $\C_0^I$, its restriction $R_0$ to $\C_0^I$ is the inverse of $S$.
\end{proof}

\begin{remark}[Construction of $R$]\quad

 One can also use a probabilistic argument which illuminates the choice of the matrix $R$. As the asymptotics is, up to the variance, the same for all processes with $\Lbb^2$ increments, it is enough to compute $R$ in a specific example. For the simple 
 random walk, in particular, $R$ can be computed explicitly by a martingale argument.
\end{remark}

\subsection{Computation of the potential: $d=1$, L\'evy jumps}
\label{subsec:D1Levy}

We are now interested in processes in the basin of attraction of a L\'evy stable distribution. We recall the definition of regular variation:

\begin{definition}[Regular and slow variation]\quad
 
 For any real $\alpha$, a real-valued, measurable, positive function $L$ defined on an interval $[x_0, +\infty)$ is said to have \emph{regular variation of index $\alpha$ at infinity} if, for all $x > 0$,
 \begin{equation*}
  \lim_{y \to +\infty} \frac{L(xy)}{L(y)} 
  = x^\alpha.
 \end{equation*}
 A function $L$ defined on an interval $(0, x_0]$ is said to have \emph{regular variation of index $\alpha$ at $0$} if $x \mapsto L(1/x)$ has regular variation of index $-\alpha$ at infinity.
 
 \smallskip
 
 A function with regular variation of index $0$ is also said to have \emph{slow variation}.
\end{definition}

In addition to the general setting of Section~\ref{sec:Calculs}, we assume that there exist $\alpha \in (1,2)$, $c_-$, $c_+ \geq 0$ not both zero and $L$ with regular variation of index $\alpha$ at infinity such that
\begin{align*}
 \mu\left(F \geq x\right) & \sim_{x \to + \infty} \frac{c_+ }{L(x)}, \\
 \mu\left(F \leq -x\right) & \sim_{x \to + \infty} \frac{c_-}{L(x)}.
\end{align*}
In this setting, we are able to give explicit asymptotics for the potential matrices, but not to invert them in general. We aim to prove:

\begin{proposition}\quad
\label{prop:ApplicationD1Levy}
 
 Let $([\Z], \widetilde{\mu}, \widetilde{T})$ be an ergodic and recurrent Markov $\Z$-extension of a Gibbs-Markov map $(A, \mu, T)$ with jump function $F : A \to \Z$. Assume furthermore that there exists $\alpha \in (1,2)$, $c_-$, $c_+ \geq 0$ not both zero and $L$ with regular variation of index $\alpha$ at infinity such that
\begin{align*}
 \mu\left(F \geq x\right) & \sim_{x \to + \infty} \frac{c_+ }{L(x)}, \\
 \mu\left(F \leq -x\right) & \sim_{x \to + \infty} \frac{c_-}{L(x)}.
\end{align*}

\smallskip

Let $I$ be non-empty and finite and $\sigma : I \hookrightarrow \R$ be injective. For $t > 0$, let $\sigma_t : I \hookrightarrow \Z$ be such that $\sigma_t (i) = t \sigma(i) + o(t)$ for all $i \in I$. Let $P_t$ be the transition matrix of the Markov chain induced on $\Sigma_t := \sigma_t (I)$ and $Q_t : \C_0^I \to \C_0^I$ the associated family of potential matrices. Then, for all $f$, $g \in \C_0^I$,
\begin{align}
 \langle f, Q_t(g) \rangle_{\ell^2 (I)} 
 & =_{t \to +\infty} \frac{\sin(\alpha \pi)}{\pi \left[c_+^2 + c_-^2 +2c_+ c_- \cos ( \alpha\pi ) \right]} \frac{L(t)}{t} \nonumber \\ & \hspace{2cm} \cdot \sum_{i, j \in I} f_i g_j \left|\sigma(i) - \sigma(j) \right|^{\alpha-1} 
 \left( c_+ \mathbf{1}_{\sigma (j) > \sigma(i)} + c_- \mathbf{1}_{\sigma (j) < \sigma(i)} \right) + o\left(\frac{L(t)}{t}\right). \label{eq:ExpansionPotentielD1Levy}
\end{align}
As a consequence, there exists an irreducible bi-L-matrix $R$ such that
\begin{equation}
\label{eq:ExpansionTransitionsD1Levy}
 P_t 
 =_{t \to + \infty} \Id - \frac{t}{L(t)} R + o\left(\frac{t}{L(t)}\right).
\end{equation}
\end{proposition}

Note that, as soon as $c_+ \neq c_-$ (i.e.\ $F$ is in the basin of attraction of an asymmetric L\'evy distribution), the operator $S$ is not symmetric, and thus $R$ is not symmetric either.

\begin{proof}
 
The structure of this demonstration follows that of the proof of Proposition~\ref{prop:ApplicationD1CarreIntegrable}.

\smallskip

We assume the setting of Proposition~\ref{prop:ApplicationD1CarreIntegrable}. By \cite[Theorem~5.1]{AaronsonDenker:2001}\footnote{In that theorem, $(A, \mu, T)$ is assumed to be mixing; however, ergodicity is enough, as can be seen e.g.\ by working with $(A, \mu, T^M)$ for $M \geq 1$ the period of $T$. In addition, we correct here a typographic mistake in the expression of $\Re \ln \lambda_w$.}, there exist $\vartheta >0$, $\zeta \in [-1, 1]$ and a neighbourhood $U$ of $0 \in \Tbb^1$ such that, for all $w \in U$,
\begin{equation*}
 \lambda_w 
 = e^{-\vartheta (1-i \zeta \tan (\pi \alpha/2) \sgn (w))L(|w|^{-1})^{-1}} + o(L(|w|^{-1})^{-1}). 
\end{equation*}
Without loss of generality, $L$ is assumed to be continuous. The sign function $\sgn$ in the above equation takes value $1$ on $\R_+^*$ and $-1$ on $\R_-^*$. The constants $\vartheta$ and $\zeta$ can in addition be expressed as functions of the initial data on $F$:
\begin{align*}
 \vartheta & = (c_- + c_+) \Gamma (1-\alpha) \cos \left( \frac{\pi \alpha}{2}\right), \\
 \zeta & = \frac{c_+ - c_-}{c_+ + c_-}.
\end{align*}
In particular, the parameter $\vartheta$ quantifies the dispersion of $F$, while $\zeta$ quantifies its asymmetry. To make computations slightly lighter, we put $\zeta' :=  \zeta \tan (\pi \alpha/2)$.

\smallskip

We set $\varepsilon := t^{-1}$. By Proposition~\ref{prop:LimitePoissonTF}, for any small enough neighbourhood $U$ of $0 \in \Tbb$, 
for all $\varepsilon > 0$, for all $f$, $g \in \C_0^I$,
\begin{align*}
 \langle f, Q_\varepsilon (g) \rangle_{\ell^2 (I)} 
 & = \frac{1}{2 \pi} \int_U \widecheck{\Fcal}_\varepsilon(f) \Fcal_\varepsilon(g) \frac{1+\delta_w}{1-\lambda_w} \dd w + O_U (1) \norm{f}{} \norm{g}{}.
\end{align*}
Note that $\widecheck{\Fcal}_\varepsilon(f) (w)$ and $\Fcal_\varepsilon(g) (w)$ are both in $O(|w|)$ while the denominator is bounded from below by $c'|w|^2$ for some $c'>0$ by Equation~\eqref{eq:DeveloppementVPD1CarreIntegrable} and Potter's bound \cite[Theorem~1.5.6]{BinghamGoldieTeugels:1987}.This justifies that we took the limit in Equation~\eqref{eq:BonneBornePourTF}.

\smallskip

Let us set
\begin{equation*}
 1+\delta'_w
 := (1+\delta_w) \frac{\vartheta (1-i \zeta' \sgn (w))L(|w|^{-1})^{-1}}{1-\lambda_w},
\end{equation*}
so that $\delta'_w$ is also a continuous function of $w$ with $\lim_{w \to 0} \delta'_w = 0$. In addition, we choose $U = (-\eta, \eta)$ for small enough $\eta > 0$. Then:
\begin{equation}
\label{eq:AControlerD1Levy}
 \langle f, Q_\varepsilon (g) \rangle_{\ell^2 (I)} 
 = \frac{1}{2\pi} \int_{-\eta}^\eta \widecheck{\Fcal}_\varepsilon(f) \Fcal_\varepsilon(g) \frac{(1+i \zeta' \sgn (w))(1+\delta'_w)}{\vartheta (1+ \zeta'^2)} L(|w|^{-1}) \dd w + O_\eta (1) \norm{f}{} \norm{g}{}.
\end{equation}

As in the proof of Proposition~\ref{prop:ApplicationD1CarreIntegrable}, we split this integral in three parts, which can be dealt with roughly in the same way, with a few adaptations.

\medskip
\textbf{Control of $\delta'_w$ on the neighbourhood $(-\varepsilon, \varepsilon)$.}
\smallskip

As in the proof of Proposition~\ref{prop:ApplicationD1CarreIntegrable}, for some constant $C \geq 0$,
\begin{equation*}
 \left| \int_{-\varepsilon}^\varepsilon \widecheck{\Fcal}_\varepsilon(f) \Fcal_\varepsilon(g) \delta'_w L(|w|^{-1}) \dd w \right|
 \leq C \norm{f}{} \norm{g}{} \max_{w \in [-\eta, \eta]} |\delta'_w| \varepsilon^{-2} \int_0^\varepsilon w^2 L(w^{-1}) \dd w.
\end{equation*}
But the function $w \mapsto w^2 L(w^{-1})$ has regular variation of index $2-\alpha > 0$ near $0$, so by Karamata's theorem~\cite{Karamata:1933} (alternatively,~\cite[Proposition~1.5.8]{BinghamGoldieTeugels:1987}),
\begin{equation*}
 \left| \int_{-\varepsilon}^\varepsilon \widecheck{\Fcal}_\varepsilon(f) \Fcal_\varepsilon(g) \delta'_w L(|w|^{-1}) \dd w \right| 
 \leq C \norm{f}{} \norm{g}{} \max_{w \in [-\eta, \eta]} |\delta'_w| \varepsilon L(\varepsilon^{-1}).
\end{equation*}

\medskip
\textbf{Control of the error term $\delta'_w$ on $(-\eta, \eta) \setminus (-\varepsilon, \varepsilon)$.}
\smallskip

For all $\iota \in (0, \alpha-1)$,
\begin{align*}
 \Bigg| \int_{(-\eta, \eta) \setminus (- \varepsilon, \varepsilon)} \widecheck{\Fcal}_\varepsilon(f) \Fcal_\varepsilon(g)  & \delta'_w L(|w|^{-1}) \dd w \Bigg| \\
 & \leq C \norm{f}{} \norm{g}{} \max_{w \in [-\eta, \eta]} |\delta'_w| \int_\varepsilon^\eta L(|w|^{-1}) \dd w \\
 & \leq C \norm{f}{} \norm{g}{} \max_{w \in [-\eta, \eta]} |\delta'_w| \varepsilon \int_1^{+ \infty} L(\varepsilon^{-1}w^{-1}) \dd w \\
 & \leq C_\iota \norm{f}{} \norm{g}{} \max_{w \in [-\eta, \eta]} |\delta'_w| \varepsilon \int_1^{+ \infty} \frac{L(\varepsilon^{-1})}{w^{\alpha - \iota}} \dd w \\
 & = O \left( \norm{f}{} \norm{g}{} \max_{w \in [-\eta, \eta]} |\delta'_w| \varepsilon L(\varepsilon^{-1}) \right),
\end{align*}
where we used Potter's bound~\cite[Theorem~1.5.6]{BinghamGoldieTeugels:1987} to obtain the penultimate line, and the fact that $\alpha-\iota > 1$ to get the last line.

\medskip
\textbf{Main contribution to the integral~\eqref{eq:AControlerD1Levy}.}
\smallskip

Now let us focus on the remaining term. Again we split $\widecheck{\Fcal}_\varepsilon (f) \Fcal_\varepsilon(g)$ along its frequencies. For all $i$, $j \in I$:
\begin{align*}
 \int_{-\eta}^\eta & \left( e^{i \varepsilon^{-1} w (\sigma(i) - \sigma(j) + o(1))}-1\right) (1+i \zeta' \sgn (w))L(|w|^{-1}) \dd w \\
 & = \varepsilon \int_{-\varepsilon^{-1} \eta}^{\varepsilon^{-1}\eta} \left( e^{i w (\sigma(i) - \sigma(j) + o(1))}-1\right) (1+i \zeta' \sgn (w))L(\varepsilon^{-1} |w|^{-1}) \dd w \\
 & = \varepsilon L(\varepsilon^{-1}) \int_\R \frac{\left( e^{i w (\sigma(i) - \sigma(j))}-1\right)(1+i \zeta' \sgn (w))}{|w|^\alpha} \dd w + o(\varepsilon L(\varepsilon^{-1})),
 \end{align*}
 where we used on the last line the fact that $L(\varepsilon^{-1} |w|^{-1}) \sim_{\varepsilon \to 0} L(\varepsilon^{-1}) |w|^{-\alpha}$ to ensure pointwise convergence, and Potter's bounds and the dominated convergence theorem to ensure the convergence of the integrals. From there we get
 \begin{align*}
 \int_{-\eta}^\eta & \left( e^{i \varepsilon^{-1} w (\sigma(i) - \sigma(j) + o(1))}-1\right) (1+i \zeta' \sgn (w))L(|w|^{-1}) \dd w \\
 & = \varepsilon L(\varepsilon^{-1}) \int_0^{+\infty} \left[ \frac{\left( e^{i w (\sigma(i) - \sigma(j))}-1\right)(1+i \zeta')}{w^\alpha} + \frac{\left( e^{-i w (\sigma(i) - \sigma(j))}-1\right)(1-i \zeta')}{w^\alpha} \right] \dd w   + o(\varepsilon L(\varepsilon^{-1})) \\
 & = 2 \varepsilon L(\varepsilon^{-1}) |\sigma(i)-\sigma(j)|^{\alpha-1} \Gamma(1-\alpha) \Re \left[ (1+i \zeta') e^{-i\frac{(\alpha-1)\pi}{2} \sgn(\sigma(i)-\sigma(j))} \right] + o(\varepsilon L(\varepsilon^{-1})) \\
 & = 2 \varepsilon L(\varepsilon^{-1}) |\sigma(i)-\sigma(j)|^{\alpha-1} \Gamma(1-\alpha) \left[ \cos \left( \frac{(\alpha-1) \pi}{2} \right) + \zeta' \sgn(\sigma(i)-\sigma(j)) \sin \left( \frac{(\alpha-1) \pi}{2} \right) \right] \\ 
 & \hspace{2cm} + o(\varepsilon L(\varepsilon^{-1})). \\
 & = 2 \varepsilon L(\varepsilon^{-1}) |\sigma(i)-\sigma(j)|^{\alpha-1} \Gamma(1-\alpha) \sin \left( \frac{\alpha \pi}{2} \right) \left[ 1 + \zeta \sgn(\sigma(j)-\sigma(i)) \right] + o(\varepsilon L(\varepsilon^{-1})).
\end{align*}
where the integrals are evaluated using~\cite[Chap.~XVII.4,~\(4.11\)]{Feller:1966}. Reintroducing the constant factors of Equation~\eqref{eq:AControlerD1Levy} and replacing $\vartheta$ and $\zeta$ by their values then yield:
 \begin{align*}
 \frac{1}{2\pi} \int_{-\eta}^\eta & \left( e^{i \varepsilon^{-1} w (\sigma(i) - \sigma(j) + o(1))}-1\right) \frac{1+i \zeta' \sgn (w)}{\vartheta (1+ \zeta'^2)} L(|w|^{-1}) \dd w \\
 & = \varepsilon L(\varepsilon^{-1}) |\sigma(i)-\sigma(j)|^{\alpha-1} \frac{2 \Gamma(1-\alpha) \sin \left( \frac{\alpha \pi}{2} \right) \left[ 1 + \frac{c_+ - c_-}{c_+ + c_-} \sgn(\sigma(j)-\sigma(i)) \right]}{2 \pi (c_+ + c_-) \Gamma(1-\alpha) \cos \left( \frac{\alpha\pi}{2} \right) \left(1+\left(\frac{c_+ - c_-}{c_+ + c_-}\right)^2 \tan^2 \left( \frac{\alpha\pi}{2} \right) \right)} \\ 
 & \hspace{2cm} + o(\varepsilon L(\varepsilon^{-1})) \\
 & = \varepsilon L(\varepsilon^{-1}) |\sigma(i)-\sigma(j)|^{\alpha-1} \frac{\sin \left( \frac{\alpha \pi}{2} \right) \cos \left( \frac{\alpha\pi}{2} \right)  \left[ (c_+ + c_-) + (c_+ - c_-) \sgn(\sigma(j)-\sigma(i)) \right]}{\pi \left((c_+ + c_-)^2 \cos^2 \left( \frac{\alpha\pi}{2} \right)  +(c_+ - c_-)^2 \sin^2 \left( \frac{\alpha\pi}{2} \right) \right)} \\
 & \hspace{2cm} + o(\varepsilon L(\varepsilon^{-1})) \\
& = \frac{\sin (\alpha \pi)  \left|\sigma(i) - \sigma(j) \right|^{\alpha-1} \left( c_+ \mathbf{1}_{\sigma (j) > \sigma(i)} + c_- \mathbf{1}_{\sigma (j) < \sigma(i)} \right)}{\pi \left(c_+^2 + c_-^2 + 2 c_- c_+ \cos (\alpha \pi) \right)} \varepsilon L(\varepsilon^{-1})
+ o(\varepsilon L(\varepsilon^{-1})).
\end{align*}

Let us set
\begin{equation*}
 \langle f, S (g)\rangle_{\ell^2 (I)} 
 := \sum_{i, j \in I} f_i g_j \frac{\sin (\alpha \pi)  \left|\sigma(i) - \sigma(j) \right|^{\alpha-1} \left( c_+ \mathbf{1}_{\sigma (j) > \sigma(i)} + c_- \mathbf{1}_{\sigma (j) < \sigma(i)} \right)}{\pi \left(c_-^2 + c_+^2 + 2 c_- c_+ \cos (\alpha \pi) \right)}.
\end{equation*}
After summing over $i$, $j \in I$, we obtain
\begin{equation*}
 \frac{1}{2\pi} \int_{-\eta}^\eta \widecheck{\Fcal}_\varepsilon(f) \Fcal_\varepsilon(g) \frac{1+i \zeta' \sgn (w)}{\vartheta (1+ \zeta'^2)} L(|w|^{-1}) \dd w 
  = \varepsilon L(\varepsilon^{-1}) \langle f, S (g)\rangle_{\ell^2 (I)} + o(\varepsilon L(\varepsilon^{-1})).
\end{equation*}
Then, as in the proof of Proposition~\ref{prop:ApplicationD1CarreIntegrable}, we can control the error terms involving $\delta'_w$ to get
\begin{equation*}
 \langle f, Q_\varepsilon (g) \rangle_{\ell^2 (I)} 
 = \varepsilon L(\varepsilon^{-1}) \langle f, S (g)\rangle_{\ell^2 (I)} + o(\varepsilon L(\varepsilon^{-1})),
\end{equation*}
which, after replacing $\varepsilon$ by $t^{-1}$, finishes the proof of the first part of Proposition~\ref{prop:ApplicationD1Levy}.

\smallskip

In addition, for all $i \neq j \in I$,
\begin{equation*}
 \langle \mathbf{1}_i-\mathbf{1}_j, S (\mathbf{1}_i-\mathbf{1}_j) \rangle_{\ell^2 (I)} 
 = \frac{( c_+ + c_-) \left|\sin(\alpha \pi) \right|}{\pi \left[c_+^2 + c_-^2 +2c_+ c_- \cos ( \alpha\pi ) \right]} \left|\sigma(i) - \sigma(j) \right|^{\alpha-1} 
 > 0,
\end{equation*}
so that $S$ is irreducible in the sense of Definition~\ref{def:QMatriceIrreductible}. By Theorem~\ref{thm:MainTheorem}, there exists an irreducible bi-L-matrix $R$ such that
\begin{equation*}
 P_\varepsilon 
 =_{\varepsilon \to 0} \Id - \frac{1}{\varepsilon L(\varepsilon^{-1})} R + o\left(\frac{1}{\varepsilon L(\varepsilon^{-1})}\right);
\end{equation*}
replacing $\varepsilon$ by $t^{-1}$ finishes the proof.
\end{proof}

\begin{example}[Asymmetric L\'evy distributions]\quad
\label{ex:LevyAsymetrique}
 
 While we have no general explicit expression for $R$, it can be computed on given examples. Let us consider for instance the case where $c_-=0$ and, without loss of generality, $c_+ = 1$. This corresponds to a L\'evy distribution which is maximally asymmetric, supported of $\R_+$. Take, for instance, $I = \{A, B, C\}$ with $(\sigma(A), \sigma(B), \sigma(C)) = (-1, 0, 1)$.
 
 \smallskip
 
 Set $e_1 := \mathbf{1}_C - \mathbf{1}_B$ and $e_2 := \mathbf{1}_B - \mathbf{1}_A$. One computes
 \begin{equation*}
  \begin{array}{llll}
    \langle e_1, S (e_1) \rangle_{\ell^2 (I)} & = \frac{|\sin(\alpha \pi)|}{\pi};
    & \langle e_1, S (e_2) \rangle_{\ell^2 (I)} & = 0; \\
    \langle e_2, S (e_1) \rangle_{\ell^2 (I)} & = - (2-2^{\alpha-1})\frac{|\sin(\alpha \pi)|}{\pi}; 
    & \langle e_2, S (e_2) \rangle_{\ell^2 (I)} & = \frac{|\sin(\alpha \pi)|}{\pi}.
  \end{array}
 \end{equation*}
  
  Since $\langle e_i, e_j \rangle_{\ell^2 (I)} = 2$ if $i = j$ and $-1$ otherwise, we can write $S$ in the basis $(e_1, e_2)$:
  \begin{equation*}
   S 
   = \frac{|\sin(\alpha \pi)|}{3 \pi}
   \left(
   \begin{array}{cc}
    2^{\alpha-1} & 1 \\
    2^\alpha - 3 & 2
   \end{array}
   \right),
  \end{equation*}
  whence
  \begin{equation*}
   S^{-1} 
   = R_0 
   = \frac{\pi}{|\sin(\alpha \pi)|}
   \left(
   \begin{array}{cc}
    2 & -1 \\
    3 - 2^\alpha & 2^{\alpha-1}
   \end{array}
   \right).
  \end{equation*}
  We can then recover, in the basis $(\mathbf{1}_A, \mathbf{1}_B, \mathbf{1}_C)$,
  \begin{equation}
  \label{eq:ExempleLevyAsymetrique}
   R
   = \frac{\pi}{|\sin(\alpha \pi)|}
   \left(
   \begin{array}{ccc}
    1 & -(2^{\alpha-1}-1) & -(2-2^{\alpha-1}) \\
    -1 & 2^{\alpha-1} & -(2^{\alpha-1}-1) \\
    0 & -1 & 1 
   \end{array}
   \right),
  \end{equation}
  and $P_t =_{t \to + \infty} \Id - \frac{t}{L(t)} R + o(t L(t^{-1}))$ by Proposition~\ref{prop:ApplicationD1Levy}. The first term of the last line is $0$, which means that transitions from $\sigma_t(C) \sim t$ to $\sigma_t(A) \sim -t$ avoiding $0$ have low probability (in $o(t L(t^{-1}))$). Transitions from $\sigma_t(A) \sim -t$ to $\sigma_t(C) \sim t$ avoiding $0$ have a much higher probability.
  
  \smallskip
  
  This is expected: one can find random walks under these constraints whose jumps are supported on $[-1, +\infty)$, in which case transitions from $\sigma_t(C)$ to $\sigma_t(A)$ avoiding $0$ are forbidden.
  
  \smallskip
  
  The matrix $R$ of Equation~\eqref{eq:ExempleLevyAsymetrique} has other interesting features. If $\alpha$ is close to $2$, as expected, the matrix $R$ is close to the one of the $\Lbb^2$ case, written in Proposition~\ref{prop:ApplicationD1CarreIntegrable}. Transitions occur mostly between neighbours, that is, between $\sigma_t(A)$ and $\sigma_t(B)$ or between $\sigma_t(B)$ and $\sigma_t(C)$. If $\alpha$ is close to $1$, on the other hand, the transitions occur mostly in a circular way: from $\sigma_t(A)$ to $\sigma_t(C)$, or from $\sigma_t(C)$ to $\sigma_t(B)$, or from $\sigma_t(B)$ to $\sigma_t(A)$.
\end{example}

\begin{example}[Symmetric L\'evy distributions]\quad
\label{ex:LevySymetrique}
 
 Let us present briefly the symmetric case. Without loss of generality, we assume that $c_- = c_+ = 1$. We take $I = \{A, B, C\}$ with $(\sigma(A), \sigma(B), \sigma(C)) = (-1, 0, 1)$, as in Example~\ref{ex:LevyAsymetrique}. Similar computations yield

  \begin{equation}
  \label{eq:ExempleLevySymetrique}
   R
   = \frac{2\pi}{2^{\alpha-1}(4-2^{\alpha-1})|\tan\left(\frac{\alpha \pi}{2}\right)|}
   \left(
   \begin{array}{ccc}
    2 & -2^{\alpha-1} & -(2-2^{\alpha-1}) \\
    -2^{\alpha-1} & 2^\alpha & -2^{\alpha-1} \\
    -(2-2^{\alpha-1}) & -2^{\alpha-1} & 2
   \end{array}
   \right),
  \end{equation}
  and $P_t =_{t \to + \infty} \Id - \frac{t}{L(t)} R + o(t L(t^{-1}))$ by Proposition~\ref{prop:ApplicationD1Levy}.
  
  \smallskip
  
  If $\alpha$ is close to $2$, as expected, the matrix $R$ is close to the one of the $\Lbb^2$ case, written in Proposition~\ref{prop:ApplicationD1CarreIntegrable}. If $\alpha$ is close to $1$, on the other hand, the transitions are close to the symmetric Cauchy case, written in Proposition~\ref{prop:ApplicationD1Cauchy}.
\end{example}

\subsection{Computation of the potential: $d=1$, Cauchy jumps}
\label{subsec:D1Cauchy}

We include some distributions for the jumps in the basin of attraction of a $1$-stable law. These are of particular interest for at least two reasons. First, they occur in geometric settings, such at the geodesic flow on $\C \setminus \Z$ with a translation-invariant hyperbolic metric~\cite{AaronsonDenker:1998, AaronsonDenker:1999}. In addition, in the context of the current work, we can compute explicitly the asymptotics of their transition matrices $(P_t)_{t > 0}$ and not just the asymptotic potential operator $S$.

\smallskip

Getting exhaustive conditions ensuring that a given distribution is in the basin of attraction of a $1$-stable law is a notoriously difficult problem. As in~\cite[Corollary~2.9]{PeneThomine:2019}, we choose to make a spectral assumption which, while very restrictive\footnote{In particular, the $H_1$ and $H_2$ terms appearing in~\cite[Theorem~2]{AaronsonDenker:1998} must cancel up to an error in $o(L(|w|^{-1})^{-1})$.}, includes symmetric distributions. The work of J.~Aaronson and M.~Denker offers some explicit conditions on the tails of $F$ which guarantee that Hypothesis~\ref{hyp:FonctionCaracteristiqueCauchy} is satisfied~\cite{AaronsonDenker:1998}.

\begin{hypothesis}[Basin of attraction of $1$-stable distributions]\quad
\label{hyp:FonctionCaracteristiqueCauchy}
 
 There exist $\vartheta > 0$, $\zeta \in \R$ and $L$ with regular variation of index $1$ at infinity such that 
 \begin{equation}
  \label{eq:FonctionCaracteristiqueCauchy}
  \int_A e^{i w F} \dd \mu 
  =_{w \to 0} e^{- \vartheta (1-i\zeta \sgn(w)) L(|w|^{-1})^{-1} } + o(L(|w|^{-1})^{-1}).
 \end{equation}
 Furthermore, we then set
 \begin{equation*}
  J(t) 
  := \int_C^t \frac{L(s)}{s^2} \dd s,
 \end{equation*}
 which is well-defined for all large enough $C$ and all $t \geq C$.
\end{hypothesis}

We then get:

\begin{proposition}\quad
\label{prop:ApplicationD1Cauchy}
 
 Let $([\Z], \widetilde{\mu}, \widetilde{T})$ be an ergodic and recurrent Markov $\Z$-extension of a Gibbs-Markov map $(A, \mu, T)$ with jump function $F : A \to \Z$. Assume furthermore that $F$ satisfies Hypothesis~\ref{hyp:FonctionCaracteristiqueCauchy}.

\smallskip

Let $I$ be non-empty and finite and $\sigma : I \hookrightarrow \R$ be injective. For $t > 0$, let $\sigma_t : I \hookrightarrow \Z$ be such that $\sigma_t (i) = t \sigma(i) + o(t)$ for all $i \in I$. Let $P_t$ be the transition matrix of the Markov chain induced on $\Sigma_t := \sigma_t (I)$. 
Then
\begin{equation*}
 P_t 
 =_{t \to + \infty} \Id - \frac{\pi \vartheta (1+\zeta^2)}{J(t)} R + o (J(t)^{-1}),
\end{equation*}
where $R$ is the irreducible, symmetric bi-L-matrix defined by
\begin{equation}
\label{eq:DefinitionRD1Cauchy}
\left\{
 \begin{array}{lcll}
 R_{ij} & = & -\frac{1}{|I|} & \quad \text{whenever } i \neq j \in I, \\
 R_{ii} & = & \frac{|I|-1}{|I|} & \quad \forall \ i \in I.
 \end{array}
 \right.
\end{equation}
\end{proposition}

Under the hypotheses of Proposition~\ref{prop:ApplicationD1Cauchy}, and setting
\begin{equation*}
 a(t) 
 := \frac{\pi \vartheta (1+\zeta^2)}{|I| J(t)},
\end{equation*}
the transition matrices $(P_t)_{t >0}$ have the following graphical representation:
\begin{equation*}
 P_t
 = \left(
 \begin{array}{ccccc}
  1 -(|I|-1) a(t) & a(t) & \ldots & a(t) & a(t) \\
  a(t) & 1 -(|I|-1) a(t) & \ldots & a(t) & a(t) \\ 
  \vdots & \vdots & \ddots & \vdots & \vdots \\
  a(t) & a(t) & \ldots & 1 -(|I|-1) a(t) & a(t) \\
  a(t) & a(t) & \ldots & a(t) & 1 -(|I|-1) a(t)
 \end{array}
 \right)
 + o(J(t)^{-1}).
\end{equation*}
It is possible to get less trivial limit matrices if one looks at shapes whose points lie at different spatial scales. One may consider, for instance, $L(t) \sim Ct$, so that $J(t) \sim C \ln (t)$, and $\norm{\sigma_t(i)- \sigma_t(j)}{} \sim t^{\max\{r(i), r(j)\}}$ for some function $r : I \to (0, +\infty)$.

\smallskip

Before proving Proposition~\ref{prop:ApplicationD1Cauchy}, we introduce a technical lemma:

\begin{lemma}\quad
 \label{lem:DivergenceJCauchy}
 
  Let $([\Z], \widetilde{\mu}, \widetilde{T})$ be an ergodic and recurrent Markov $\Z$-extension of a Gibbs-Markov map $(A, \mu, T)$ with jump function $F : A \to \Z$. Assume furthermore that $F$ satisfies Hypothesis~\ref{hyp:FonctionCaracteristiqueCauchy}. Then, for all large enough $C$,
\begin{equation*}
 \int_C^{+\infty} \frac{L(t)}{t^2} \dd t
 = +\infty.
\end{equation*}
\end{lemma}

\begin{proof}
 
 By Halmos' recurrence theorem~\cite{Halmos:1956}, if $([\Z], \widetilde{\mu}, \widetilde{T})$ is recurrent, then
 \begin{equation*}
  \sum_{n = 0}^{+ \infty} \widetilde{\mu} ((x,p) \in [0] \text{ and } \widetilde{T}^n (x, p) \in [0]) 
  = \sum_{n = 0}^{+ \infty} \mu (S_n F = 0) 
  = + \infty.
 \end{equation*}
 As in the proof of Proposition~\ref{prop:LimitePoissonTF}, we can express this sum using the Fourier transform and twisted transfer operators. For all small enough $\eta > 0$, such that in particular $|\mu_w (\mathbf{1})| \leq 2 \pi$ on $(-\eta, \eta)$:
 \begin{align*}
  \sum_{n = 0}^{+ \infty} \mu (S_n F = 0) 
  & = \lim_{\rho \to 1^-} \sum_{n = 0}^{+ \infty} \rho^n \mu (S_n F = 0) \\
  & = \lim_{\rho \to 1^-} \sum_{n = 0}^{+ \infty} \rho^n \int_A \frac{1}{2\pi} \int_{\Tbb^1} e^{i w S_n F} \dd w \dd \mu \\
  & = \frac{1}{2\pi} \lim_{\rho \to 1^-} \sum_{n = 0}^{+ \infty} \rho^n \int_{\Tbb^1} \int_A \Lcal_w^n (\mathbf{1}) \dd \mu \dd w \\
  & = \frac{1}{2\pi} \lim_{\rho \to 1^-} \int_{\Tbb^1} \int_A \left(\Id - \rho \Lcal_w\right)^{-1} (\mathbf{1}) \dd \mu \dd w \\
  & \leq \frac{1}{2\pi} \limsup_{\rho \to 1^-} \int_{-\eta}^\eta \frac{|\mu_w (\mathbf{1})|}{|1-\rho \lambda_w|} \dd w  + O(1) \\
  & \leq \int_{-\eta}^\eta \frac{1}{|1-\lambda_w|} \dd w  + O(1).
 \end{align*}

 By~\cite[Theorem~2]{AaronsonDenker:1999}, under Condition~\eqref{eq:FonctionCaracteristiqueCauchy} and with the same constants,
 \begin{equation*}
  \lambda_w 
   = e^{- \vartheta (1-i\zeta \sgn(w)) L(|w|^{-1})^{-1} } + o(L(|w|^{-1})^{-1}).
 \end{equation*}
 Taking real parts and possibly a lower value of $\eta$, the function $|1-\lambda_w|$ has a lower bound in $\vartheta L(|w|^{-1})^{-1}/2$, and in addition
 \begin{equation*}
  \int_{-\eta}^\eta \frac{2}{\vartheta L(|w|^{-1})^{-1}} \dd w 
  = \frac{4}{\vartheta} \int_{\eta^{-1}}^{+\infty} \frac{L(w)}{w^2} \dd w,
 \end{equation*}
 which must diverge.
\end{proof}

We can now proceed to the proof of Proposition~\ref{prop:ApplicationD1Cauchy}.

\begin{proof}[Proof of Proposition~\ref{prop:ApplicationD1Cauchy}]

The structure of this demonstration follows that of the proofs of Propositions~\ref{prop:ApplicationD1CarreIntegrable} and~\ref{prop:ApplicationD1Levy}.

\smallskip

We assume the setting of Proposition~\ref{prop:ApplicationD1Cauchy}. In particular, using Hypothesis~\ref{hyp:FonctionCaracteristiqueCauchy}, let $\vartheta > 0$, $\zeta \in \R$ and $L$ with regular variation of index $1$ at infinity such that 
\begin{equation*}
 \int_A e^{i w F} \dd \mu 
 =_{w \to 0} e^{- \vartheta (1-i\zeta \sgn(w)) L(|w|^{-1})^{-1} } + o(L(|w|^{-1})^{-1}).
\end{equation*}

We set $\varepsilon := t^{-1}$. With exact the same reasoning as in the beginning of the proof of Proposition~\ref{prop:ApplicationD1Levy}, we get for all $f$, $g \in \C_0^I$:
\begin{equation}
\label{eq:AControlerD1Cauchy}
 \langle f, Q_\varepsilon (g) \rangle_{\ell^2 (I)} 
 = \frac{1}{2\pi} \int_{-\eta}^\eta \widecheck{\Fcal}_\varepsilon(f) \Fcal_\varepsilon(g) \frac{(1+i \zeta \sgn (w))(1+\delta'_w)}{\vartheta (1 + \zeta^2)} L(|w|^{-1}) \dd w + O_\eta (1) \norm{f}{} \norm{g}{},
\end{equation}
where $\lim_{w \to 0} \delta'_w = 0$.

\smallskip

As in the proof of Proposition~\ref{prop:ApplicationD1Levy}, we split this integral in three parts; however these parts are different:
\begin{itemize}
 \item the whole integral on a small neighbourhood $(-\varepsilon, \varepsilon)$ of $0$, which shall be negligible because $\left( \widecheck{\Fcal}_\varepsilon(f) \Fcal_\varepsilon(g) \right) (w)$ is small near $0$, with a small additional gain thanks to~\cite[Proposition~1.5.9b]{BinghamGoldieTeugels:1987}.
 \item the integral involving the error term $\delta'_w$ on $(-\eta, \eta) \setminus (-\varepsilon, \varepsilon)$, which shall be negligible because $\delta'_w$ is small and $L(|w|^{-1})$ is large.
 \item the integral on $(-\eta, \eta) \setminus (-\varepsilon, \varepsilon)$ without the error term $\delta'_w$, which is the main contribution to $\langle f, Q_\varepsilon (g) \rangle_{\ell^2 (I)}$.
\end{itemize}

Another meaningful difference will reside in the computation of the asymptotics of the main part: the non-zero frequencies of $\widecheck{\Fcal}_\varepsilon(f) \Fcal_\varepsilon(g)$ will yield a negligible contribution by an argument \textit{\`a la} Riemann-Lebesgue.

\medskip
\textbf{Control of the integral on the neighbourhood $(-\varepsilon, \varepsilon)$.}
\smallskip

The exact same computation as in the proof of Proposition~\ref{prop:ApplicationD1Levy} yields
\begin{equation*}
 \int_{-\varepsilon}^\varepsilon \widecheck{\Fcal}_\varepsilon(f) \Fcal_\varepsilon(g) (1+\delta'_w) L(|w|^{-1}) \dd w 
 = \norm{f}{} \norm{g}{} O\left(\varepsilon L(\varepsilon^{-1}) \right).
\end{equation*}
In addition, since $t^{-2} L(t)$ is not integrable by Lemma~\ref{lem:DivergenceJCauchy}, we can apply~\cite[Proposition~1.5.9b]{BinghamGoldieTeugels:1987} to get that $t^{-1} L(t) = o(J(t))$, or by a change of variable
\begin{equation*}
 \int_{-\varepsilon}^\varepsilon \widecheck{\Fcal}_\varepsilon(f) \Fcal_\varepsilon(g) (1+\delta'_w) L(|w|^{-1}) \dd w 
 = \norm{f}{} \norm{g}{} o\left( J(\varepsilon^{-1}) \right).
\end{equation*}

\medskip
\textbf{Control of the error term $\delta'_w$ on $(-\eta, \eta) \setminus (-\varepsilon, \varepsilon)$.}
\smallskip

On $(-\eta, \eta) \setminus (-\varepsilon, \varepsilon)$, we get
\begin{align*}
 \Bigg| \int_{(-\eta, \eta) \setminus (- \varepsilon, \varepsilon)} \widecheck{\Fcal}_\varepsilon(f) \Fcal_\varepsilon(g)  & \delta'_w L(|w|^{-1}) \dd w \Bigg| \\
 & \leq C \norm{f}{} \norm{g}{} \max_{w \in [-\eta, \eta]} |\delta'_w| \int_\varepsilon^\eta L(w^{-1}) \dd w \\
 & = C \norm{f}{} \norm{g}{} \max_{w \in [-\eta, \eta]} |\delta'_w| \int_{\eta^{-1}}^{\varepsilon^{-1}} \frac{L(w)}{w^2} \dd w \\
 & = C \norm{f}{} \norm{g}{} \max_{w \in [-\eta, \eta]} |\delta'_w| \left[ J(\varepsilon^{-1}) - J(\eta^{-1}) \right] \\
 & \leq C \norm{f}{} \norm{g}{} \max_{w \in [-\eta, \eta]} |\delta'_w| J(\varepsilon^{-1}).
\end{align*}

\medskip
\textbf{Main contribution to the integral~\eqref{eq:AControlerD1Cauchy}.}
\smallskip

Now let us focus on the remaining term. Again we split $\widecheck{\Fcal}_\varepsilon (f) \Fcal_\varepsilon(g)$ along its frequencies. Without loss of generality, $L$ can be chosen increasing. The function $w \mapsto L(w^{-1})$ is then decreasing on $(\varepsilon, \eta)$. By the Riemann-Stieltjes version of the integration by parts~\cite[Theorem~7.6]{Apostol:1974}, whenever $i \neq j \in I$,
\begin{align*}
 \int_\varepsilon^\eta & e^{i \varepsilon^{-1} w (\sigma(i) - \sigma(j) + o(1))} L(w^{-1}) \dd w \\
  & = \frac{\varepsilon}{\sigma(i) - \sigma(j)+o(1)} \left[ e^{i \varepsilon^{-1} w (\sigma(i) - \sigma(j) + o(1))} L(w^{-1}) \right]_\varepsilon^\eta \\
 & \hspace{2cm} - \frac{\varepsilon}{\sigma(i) - \sigma(j)+o(1)} \int_\varepsilon^\eta e^{i \varepsilon^{-1} w (\sigma(i) - \sigma(j) + o(1))} \left[ L(w^{-1}) \right]' \dd w \\
 & = O \left( \varepsilon L(\varepsilon^{-1}) \right) \\
 & = o \left( J(\varepsilon^{-1})\right),
\end{align*}
where we again applied Lemma~\ref{lem:DivergenceJCauchy} and~\cite[Proposition~1.5.9b]{BinghamGoldieTeugels:1987}. The integral over $(-\eta, -\varepsilon)$ can be dealt with in the same way. If $i = j$, then 
\begin{align*}
 \int_{(-\eta, \eta) \setminus (- \varepsilon, \varepsilon)} e^{i \varepsilon^{-1} w (\sigma(j) - \sigma(i) + o(1))} (1+i \zeta \sgn (w))L(|w|^{-1}) \dd w
 & = 2 \int_\varepsilon^\eta L(w^{-1}) \dd w \\
 & = 2 \left[ J(\varepsilon^{-1}) - J(\eta^{-1})\right] \\
 & \sim_{\varepsilon \to 0} 2 J(\varepsilon^{-1}).
\end{align*}

Finally, summing up everything,
\begin{equation*}
 \langle f, Q_\varepsilon (g) \rangle_{\ell^2 (I)} 
 = \frac{J(\varepsilon^{-1})}{\pi \vartheta (1+\zeta^2)} \langle f, g \rangle_{\ell^2 (I)} 
+ \norm{f}{} \norm{g}{} o(J(\varepsilon^{-1}) )
+ \norm{f}{} \norm{g}{} \max_{w \in [-\eta, \eta]} |\delta'_w| O(J(\varepsilon^{-1})).
\end{equation*}
By the same argument as in the proof of Proposition~\ref{prop:ApplicationD1CarreIntegrable}, making $\eta$ smaller as $\varepsilon$ vanishes, we get
\begin{equation*}
 \langle f, Q_\varepsilon (g) \rangle_{\ell^2 (I)} 
 = \frac{J(\varepsilon^{-1})}{\pi \vartheta (1+\zeta^2)} \langle f, g \rangle_{\ell^2 (I)}
+ \norm{f}{} \norm{g}{} o(J(\varepsilon^{-1})),
\end{equation*}
or, in other words,
\begin{equation*}
 Q_\varepsilon 
 \sim_{\varepsilon \to 0} \frac{J(\varepsilon^{-1})}{\pi \vartheta (1+\zeta^2)} \Id_{\C_0^I}.
\end{equation*}

\medskip
\textbf{Asymptotics of the transition matrix $P_\varepsilon$.}
\smallskip

The identity is irreducible in the sense of Definition~\ref{def:QMatriceIrreductible}. Hence we can apply Theorem~\ref{thm:MainTheorem} with $S = \Id_{\C_0^I}$:
\begin{equation*}
 (\Id-P_\varepsilon)_{\C_0^I \to \C_0^I} 
 \sim_{\varepsilon \to 0} \frac{\pi \vartheta (1+\zeta^2)}{J(\varepsilon^{-1})} \Id_{\C_0^I}.
\end{equation*}

The last step is to get the asymptotics for $P_\varepsilon$ on $\C^I$. Writing $\varpi_I$ the projection parallel to constants onto $\C_0^I$, 
\begin{align*}
 P_\varepsilon 
 & = \mathbf{1} \otimes \mu_I + (P_\varepsilon)_{\C_0^I \to \C_0^I} \varpi_I \\
 & = \mathbf{1} \otimes \mu_I + \left( 1-\frac{\pi \vartheta (1+\zeta^2)}{J(\varepsilon^{-1})}\right) \varpi_I + o\left(J(\varepsilon^{-1})^{-1}\right) \\
 & = \left( 1-\frac{\pi \vartheta (1+\zeta^2)}{J(\varepsilon^{-1})}\right) \Id + \frac{\pi \vartheta (1+\zeta^2)}{J(\varepsilon^{-1})} \mathbf{1} \otimes \mu_I + o\left(J(\varepsilon^{-1})^{-1}\right).
\end{align*}

In other words, the diagonal terms of $P_\varepsilon$ are all 
\begin{equation*}
 1-\frac{|I|-1}{|I|} \frac{\pi \vartheta (1+\zeta^2)}{J(\varepsilon^{-1})} + o\left(J(\varepsilon^{-1})^{-1}\right),
\end{equation*}
while the off-diagonal terms are all
\begin{equation*}
 \frac{1}{|I|} \frac{\pi \vartheta (1+\zeta^2)}{J(\varepsilon^{-1})} + o\left(J(\varepsilon^{-1})^{-1}\right);
\end{equation*}
replacing $\varepsilon$ by $t^{-1}$ finishes the proof.
\end{proof}

\subsection{Computation of the potential: $d=2$, square integrable jumps}
\label{subsec:D2CarreIntegrable}

Finally, we tackle $\Z^2$-extensions with square-integrable jumps. In addition to the general setting of Section~\ref{sec:Calculs}, we assume in this Subsection that $F$ is square-integrable and centered, the later condition being already implied by the first condition and the ergodicity of the extension.

\smallskip

Let $\Cov$ be the asymptotic covariance matrix of $F$, defined for all $u \in \R^2$ by:
\begin{equation}
\label{eq:D2GreenKubo}
 \langle u, \Cov (u) \rangle_{\R^2}
 := \int_A \langle F, u \rangle_{\R^2}^2 \dd \mu + 2 \sum_{n \geq 1} \int_A \langle F, u \rangle_{\R^2}\circ T^n \cdot \langle F, u \rangle_{\R^2} \dd \mu,
\end{equation}
where the infinite series is taken in the sense of Ces\`aro. Since $([\Z^2], \widetilde{\mu}, \widetilde{T})$ is ergodic, $\langle F, u \rangle_{\R^2}$ is not a coboundary whenever $u \neq 0$ (otherwise the sums $S_n \langle F, u \rangle_{\R^2}$ would be 
almost surely bounded), and thus $\Cov$ is positive definite~\cite[Th\'eor\`eme~4.1.4]{Gouezel:2008e}.

\smallskip

The behaviour of the transition matrices is then similar to the Cauchy case (Proposition~\ref{prop:ApplicationD1Cauchy}), with a statement given in the introduction (Proposition~\ref{prop:ApplicationD2CarreIntegrable}) and which we shall now prove.

\begin{proof}[Proof of Proposition~\ref{prop:ApplicationD2CarreIntegrable}]
 
The structure of this demonstration is very close to that of Proposition~\ref{prop:ApplicationD1Cauchy}. We assume the setting of Proposition~\ref{prop:ApplicationD2CarreIntegrable} and set $\varepsilon := t^{-1}$. The main eigenvalue of $\Lcal_w$ then admits the expansion~\cite[Th\'eor\`eme~4.1.4]{Gouezel:2008e}
\begin{equation}
\label{eq:DeveloppementVPD2CarreIntegrable}
 \lambda_w 
 = 1 - \frac{\langle w, \Cov(w)\rangle}{2} + o(\norm{w}{}^2).
\end{equation}

With the exact same reasoning as in the beginning of the proof of Proposition~\ref{prop:ApplicationD1CarreIntegrable}, but with a two-dimensional Fourier transform, we get for all $f$, $g \in \C_0^I$ and all small enough neighborhood $U$ of $0$:
\begin{equation}
\label{eq:AControlerD2CarreIntegrable}
 \langle f, Q_\varepsilon (g) \rangle_{\ell^2 (I)} 
 = \frac{1}{2 \pi^2} \int_U \left( \widecheck{\Fcal}_\varepsilon(f) \Fcal_\varepsilon(g) \right) (w) \frac{1+\delta'_w}{\langle w, \Cov(w) \rangle} \dd w + O_U (1) \norm{f}{} \norm{g}{},
\end{equation}
where $\lim_{w \to 0} \delta'_w = 0$. 

\smallskip

We choose $U = \sqrt{\Cov} B(0, \eta)$ for small enough $\eta > 0$. Then, after a change of variables:
\begin{equation*}
 \langle f, Q_\varepsilon (g) \rangle_{\ell^2 (I)} 
 = \frac{1}{2 \pi^2 \sqrt{\det(\Cov)}} \int_{B(0, \eta)} \left( \widecheck{\Fcal}_\varepsilon(f) \Fcal_\varepsilon(g) \right) (\sqrt{\Cov}^{-1}w) \frac{1+\delta'_{\sqrt{\Cov}^{-1}w}}{\norm{w}{}^2} \dd w + O_\eta (1) \norm{f}{} \norm{g}{}.
\end{equation*}

As in the proof of Proposition~\ref{prop:ApplicationD1Cauchy}, we split this integral 
in three parts:
\begin{itemize}
 \item in dimension $2$, since both $f$ and $g$ have zero sum, the contribution of a small ball $B(0, \varepsilon)$ is in $O(1)$.
 \item the integral involving the error term $\delta'_w$ on $B(0, \eta) \setminus B(0, \varepsilon)$, which shall be negligible because $\delta'_w$ is small and $\norm{w}{}^2$ is large.
 \item for the remainder, we split $\widecheck{\Fcal}_\varepsilon(f) \Fcal_\varepsilon(g)$ along its different frequencies. Again, the frequency $0$ has to be considered separately, and will give the main contribution to $\langle f, Q_\varepsilon g \rangle_{\ell^2 (I)}$. The other frequencies will have a lower contribution thanks to cancellations \textit{\`a la} Riemman-Lebesgue. Computations will be handled in polar coordinates, thanks to our choice of neighbourhood $U$.
\end{itemize}
For the third part and non-zero frequencies, we will subdivise the region $B(0, \eta) \setminus B(0, \varepsilon)$ in two parts: one close to the line orthogonal to the frequency vector, which shall be small, and the remaining part of the annulus, on which the Riemann-Lebesgue argument works out.

 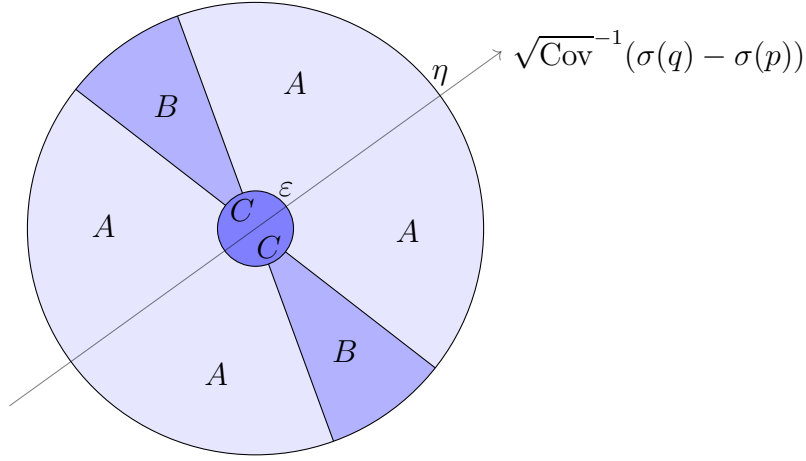
\begin{figure}[!h]
 \centering
 \scalebox{1}{
 \begin{tikzpicture}
   \fill[blue!10!] (0,0) circle (3cm);
   \fill[blue!30!] (110:0.5cm) -- (110:3cm) arc (110:142:3cm) -- (142:0.5cm) arc (142:110:0.5cm) ;
   \fill[blue!30!] (290:0.5cm) -- (290:3cm) arc (290:322:3cm) -- (322:0.5cm) arc (322:290:0.5cm) ;
   \fill[blue!50!] (0,0) circle (0.5cm);
   \draw (0,0) circle (0.5cm);
   \draw (0,0) circle (3cm);
   \draw[->, opacity = 0.5] (216:4cm) -- (36:4cm);
   \draw (110:0.5cm) -- (110:3cm);
   \draw (142:0.5cm) -- (142:3cm);
   \draw (290:0.5cm) -- (290:3cm);
   \draw (322:0.5cm) -- (322:3cm);
   \draw[right] node at (36:4cm) {$\sqrt{\Cov}^{-1} (\sigma(q) - \sigma(p))$};
   \draw node at (75:2cm) {$A$};
   \draw node at (179:2cm) {$A$};
   \draw node at (255:2cm) {$A$};
   \draw node at (359:2cm) {$A$};
   \draw node at (126:2cm) {$B$};
   \draw node at (306:2cm) {$B$};
   \draw node at (126:0.3cm) {$C$};
   \draw node at (306:0.3cm) {$C$};
   \draw[above] node at (36:0.5cm) {$\varepsilon$};
   \draw[above] node at (36:3cm) {$\eta$};
 \end{tikzpicture}
 }
 \caption{Regions of integration for non-zero frequencies. Light blue ($A$): negligible by a Riemann-Lebesgue argument. Middle blue ($B$): near the perpendicular to $\sqrt{\Cov}^{-1} (\sigma(q) - \sigma(p))$; negligible because the area is arbitrarily small. Dark blue ($C$): a disc of radius $\varepsilon$; negligible because the area is small and the function $\widecheck{\Fcal}_\varepsilon(f) \Fcal_\varepsilon(g)$ cancels the denominator.}
 \end{figure}

\medskip
\textbf{Control of the integral on the neighbourhood $B(0, \varepsilon)$.}
\smallskip

Our usual computations, as in the proof of Proposition~\ref{prop:ApplicationD1CarreIntegrable}, give:
\begin{align*}
 \int_{B(0, \varepsilon)} \left| \widecheck{\Fcal}_\varepsilon(f) \Fcal_\varepsilon(g) \right| (\sqrt{\Cov}^{-1}w) \frac{\left|1+\delta_{\sqrt{\Cov}^{-1}w} \right|}{\norm{w}{}^2} \dd w 
 & \leq C \norm{f}{} \norm{g}{} \varepsilon^{-2} \int_{B(0, \varepsilon)} 1 \dd w \\
 & \leq C \norm{f}{} \norm{g}{}.
\end{align*}

\medskip
\textbf{Control of the error term $\delta'_w$ on $B(0, \eta) \setminus B(0, \varepsilon)$.}
\smallskip

On $B(0, \eta) \setminus B(0, \varepsilon)$, we get
\begin{align*}
 \Bigg| \int_{B(0, \eta) \setminus B(0, \varepsilon)} \left( \widecheck{\Fcal}_\varepsilon(f) \Fcal_\varepsilon(g) \right) & (\sqrt{\Cov}^{-1}w) \frac{\delta_{\sqrt{\Cov}^{-1}w}}{\norm{w}{}^2} \dd w \Bigg| \\
 & \leq C \norm{f}{} \norm{g}{} \max_{w \in \sqrt{\Cov} \cdot \overline{B}(0, \eta)} |\delta'_w| \int_{B(0, \eta) \setminus B(0, \varepsilon)} \frac{1}{\norm{w}{}^2} \dd w \\
 & = C \norm{f}{} \norm{g}{} \max_{w \in \sqrt{\Cov} \cdot \overline{B}(0, \eta)} |\delta'_w| \int_\varepsilon^\eta \frac{2 \pi}{r} \dd r \\
 & \leq 2 \pi C \norm{f}{} \norm{g}{} \max_{w \in \sqrt{\Cov} \cdot \overline{B}(0, \eta)} |\delta'_w| \cdot \left|\ln(\varepsilon) \right|.
\end{align*}

\medskip
\textbf{Main contribution to the integral~\eqref{eq:AControlerD2CarreIntegrable}.}
\smallskip

Now let us focus on the remaining term. Again we split $\left( \widecheck{\Fcal}_\varepsilon (f) \Fcal_\varepsilon(g) \right) (\sqrt{\Cov}^{-1}w)$ along its frequencies. Fix some small $\kappa > 0$. For $i \neq j \in I$,
\begin{equation*}
 \int_{B(0, \eta)\setminus B(0, \varepsilon)} \frac{ e^{i \langle \sqrt{\Cov}^{-1} w, \varepsilon^{-1} (\sigma(i) - \sigma(j)) + o(\varepsilon^{-1}) \rangle} }{ \norm{w}{}^2 } \dd w \\
 = \int_{\Sbb_1} \int_\varepsilon^\eta \frac{ e^{i \varepsilon^{-1} r \langle (\cos(\theta), \sin(\theta)), \sqrt{\Cov}^{-1} (\sigma(j) - \sigma(i)) + o(1) \rangle} }{r} \dd r \dd \theta.
\end{equation*}
For fixed $\theta \in \Sbb_1$, the inner integral can be bounded in two ways. First, we have the crude but general:
\begin{equation*}
 \left| \int_\varepsilon^\eta \frac{ e^{i \varepsilon^{-1} r \langle (\cos(\theta), \sin(\theta)), \sqrt{\Cov}^{-1} (\sigma(j) - \sigma(i)) + o(1) \rangle} }{r} \dd r \right|
 \leq \int_\varepsilon^\eta \frac{1}{r} \dd r 
 = \ln (\eta) - \ln(\varepsilon) 
 \leq \left| \ln(\varepsilon) \right|.
\end{equation*}
A second way uses an integration by parts:
\begin{align}
 \left| \int_\varepsilon^\eta \right. & \left. \frac{ e^{i \varepsilon^{-1} r \langle (\cos(\theta), \sin(\theta)), \sqrt{\Cov}^{-1} (\sigma(j) - \sigma(i)) + o(1) \rangle} }{r} \dd r \right| \nonumber \\
 & = \left| \left[ \frac{ e^{i \varepsilon^{-1} r \langle (\cos(\theta), \sin(\theta)), \sqrt{\Cov}^{-1} (\sigma(j) - \sigma(i)) + o(1) \rangle} }{\left( i \varepsilon^{-1} \langle (\cos(\theta), \sin(\theta)), \sqrt{\Cov}^{-1} (\sigma(j) - \sigma(i)) + o(1) \rangle \right)r} \right]_\varepsilon^\eta \right. \nonumber \\ 
 & \hspace{2cm} \left. + \int_\varepsilon^\eta \frac{ e^{i \varepsilon^{-1} r \langle (\cos(\theta), \sin(\theta)), \sqrt{\Cov}^{-1} (\sigma(j) - \sigma(i)) + o(1) \rangle} }{\left( i \varepsilon^{-1} \langle (\cos(\theta), \sin(\theta)), \sqrt{\Cov}^{-1} (\sigma(j) - \sigma(i)) + o(1) \rangle \right)r^2} \dd r \right| \nonumber \\
 & \leq \frac{2\varepsilon}{\left| \langle (\cos(\theta), \sin(\theta)), \sqrt{\Cov}^{-1} (\sigma(j) - \sigma(i)) + o(1) \rangle \right|} \left( \frac{1}{\varepsilon} + \frac{1}{\eta} \right). \label{eq:ApresIPP}
\end{align}
Let $K_{\eta, \varepsilon, i, j}$ be the set of values of $\theta$ such that the right hand-side of Equation~\eqref{eq:ApresIPP} is smaller or equal to $\sqrt{|\ln(\varepsilon)|}$. Then $\lim_{n \to + \infty} \Leb (K_{\eta, \varepsilon, i, j}) = 2 \pi$, whence
\begin{align*}
 \left| \int_{B(0, \eta)\setminus B(0, \varepsilon)} \right. & \left.\frac{ e^{i \langle \sqrt{\Cov}^{-1} w, \varepsilon^{-1} (\sigma(i) - \sigma(j)) + o(\varepsilon^{-1}) \rangle} }{ \norm{w}{}^2 } \dd w \right| \\
 & \leq \int_{K_{\eta, \varepsilon, i, j}} \sqrt{|\ln(\varepsilon)|} \dd \theta + \int_{\Sbb_1 \setminus K_{\eta, \varepsilon, i, j}} |\ln (\varepsilon)| \dd \theta \\
 & \leq 2 \pi \sqrt{|\ln(\varepsilon)|} + \Leb (\Sbb_1 \setminus K_{\eta, \varepsilon, i, j}) |\ln (\varepsilon)| \\
 & = o\left(|\ln (\varepsilon)|\right).
\end{align*}

The zero frequency term is comparatively easier to compute:
\begin{equation*}
 \int_{B(0, \eta)\setminus B(0, \varepsilon)} \frac{1}{ \norm{w}{}^2 } \dd w 
 = 2 \pi \int_\varepsilon^\eta \frac{1}{r} \dd r 
 = 2 \pi \left( \ln(\eta) - \ln(\varepsilon) \right) 
 \sim_{\varepsilon \to 0} 2 \pi |\ln (\varepsilon)|.
\end{equation*}

Finally, summing up everything,
\begin{equation*}
 \langle f, Q_\varepsilon (g) \rangle_{\ell^2 (I)} 
 = \frac{|\ln(\varepsilon) |}{\pi \sqrt{\det(\Cov)}} \langle f, g \rangle_{\ell^2 (I)} 
+ \norm{f}{} \norm{g}{} \max_{w \in \sqrt{\Cov} \cdot \overline{B}(0, \eta)} |\delta'_w| O( |\ln(\varepsilon) |)
+ \norm{f}{} \norm{g}{} o( |\ln(\varepsilon)|).
\end{equation*}
By the same argument as in the proof of Proposition~\ref{prop:ApplicationD1CarreIntegrable}, 
making $\eta$ smaller as $\varepsilon$ vanishes, we get
\begin{equation*}
 \langle f, Q_\varepsilon (g) \rangle_{\ell^2 (I)} 
 = \frac{|\ln(\varepsilon) |}{\pi \sqrt{\det(\Cov)}} \langle f, g \rangle_{\ell^2 (I)} 
 + \norm{f}{} \norm{g}{} o( |\ln(\varepsilon)|),
\end{equation*}
or, in other words,
\begin{equation*}
 Q_\varepsilon 
 \sim_{\varepsilon \to 0} \frac{|\ln(\varepsilon) |}{\pi \sqrt{\det(\Cov)}} \Id_{\C_0^I}.
\end{equation*}
The computation of the asymptotics of $P_\varepsilon$ from those of $Q_\varepsilon$ are then the same as in the end of the proof of Proposition~\ref{prop:ApplicationD1Cauchy}.
\end{proof}

\begin{appendix}

\section{Gibbs-Markov maps}
\label{sec:GibbsMarkovAppendice}

In this appendix, we compile definitions and technical resutls relative to Gibbs-Markov maps. In what follows, $(A, \alpha, d, \mu, T)$ is a Gibbs-Markov map (see Definition~\ref{def:ApplicationGibbsMarkov}).

\subsection{Markov partitions and stopping times}
\label{subsec:TempsArret}

We aim to define various partitions and filtrations related to Gibbs-Markov maps, so as to introduce stopping times. The central definition is that of a \emph{cylinder}.

\begin{definition}[Cylinder]\quad

 For $n \geq 0$, a \emph{cylinder} of length $n$ is a non-trivial element of $\alpha_n := \bigvee_{k=0}^{n-1} T^{-k} \alpha$. 
It is given by a unique finite sequence $(a_k)_{0 \leq k < n}$ of elements of $\alpha$ such that $T (a_k) \cap a_{k+1}$ is non-negligible for all $0 \leq k < n-1$. Such a cylinder shall be denoted by $[a_0, \ldots, a_{n-1}]$.
\end{definition}

We denote by $\alpha^*$ the image partition, generated by $\{T(a) : \ a \in \alpha\}$. Since $T$ is Markov, $\alpha^*$ is coarser than $\alpha$.

\smallskip

With any Markov map comes a natural filtration given by $\Fcal_n := \sigma(\alpha_n)$ for all $n \geq 0$. The $\sigma$-algebra $\Fcal_0$ is trivial, while $\Acal = \Fcal_\infty := \bigvee_{n \geq 0} \Fcal_n$ is the Borel $\sigma$-algebra of $(A,d)$. From this filtration we define \emph{stopping times}.

\begin{definition}[Stopping time]
\label{def:TempsDArret}

Let $(A, \alpha, d, \mu, T)$ be a Gibbs-Markov map. Let $\varphi: \ A \to \N \cup \{+ \infty\}$ be measurable. We say that $\varphi$ is a \emph{stopping time} if $\{\varphi \leq n\} \in \Fcal_n$ for all $n \geq 0$, and $1 \leq \varphi < +\infty$ almost surely.

\smallskip

To any stopping time $\varphi$ we associate a countable partition of $A$ given by:
\begin{equation*}
\alpha_\varphi 
:= \bigcup_{n \geq 1} \{ \overline{a} \in \alpha_n: \ \mu (\overline{a}) > 0 \text{ and } \overline{a} \subset \{\varphi = n\} \},
\end{equation*}
and a transformation defined by:
\begin{equation*}
T_\varphi (x) 
:= T^{\varphi(x)} (x),
\end{equation*}
which is well-defined almost everywhere.
\end{definition}

In this article, and in opposition to other ressources, we shall always require stopping times to be positive and finite, for the following reasons. If we did not require stopping times to be positive, then $\varphi \equiv 0$ would be a stopping time (and the only stopping time such that $\mu (\varphi = 0) > 0$), for which the partition the partition $\alpha_\varphi$ would be trivial. If we did not require stopping times to be finite, then $T_\varphi$ may not be well-defined.

\smallskip

The iterates of $T_\varphi$ can be directly related to stopping times. For any $n \geq 1$, we define:
\begin{equation}
 \varphi^{(n)} 
 = S_n^{T_\varphi} \varphi 
 := \sum_{k=0}^{n-1} \varphi \circ T_\varphi^k.
\end{equation}
Then $T_\varphi^n = T_{\varphi^{(n)}}$. As the following lemma asserts, $\varphi^{(n)}$ is again a stopping time.

\begin{lemma}\quad
\label{lem:SommeTempsArret}

 Let $(A, \alpha, d, \mu, T)$ be a Gibbs-Markov map. Let $\varphi$ be a stopping time and $n \geq 1$. Then $\varphi^{(n)}$ is a stopping time. 
\end{lemma}

\begin{proof}
 
 We proceed by induction on $n$. For $n=1$, the function $\varphi^{(1)} = \varphi$ is a stopping time. Let $n \geq 1$ and assume that $\varphi^{(n)}$ is a stopping time. We have
 \begin{equation*}
  \varphi^{(n+1)} 
  = \sum_{k=0}^n \varphi \circ T_\varphi^k 
  = \varphi + \varphi^{(n)} \circ T_\varphi,
 \end{equation*}
 so that, for all $k \geq 0$:
 \begin{equation*}
  \left\{ \varphi^{(n+1)} \leq k \right\} 
  = \bigcup_{j=0}^k \left\{ \varphi \leq j \right\} \cap T^{-j} \left\{ \varphi^{(n)} \leq k-j \right\}.
 \end{equation*}
 Let $0 \leq j \leq k$. Both $\varphi$ and $\varphi^{(n)}$ are stopping times, so $\left\{ \varphi \leq j \right\} \in \Fcal_j$ and $\left\{ \varphi^{(n)} \leq k-j \right\} \in \Fcal_{k-j}$. It follows that 
 \begin{equation*}
  \left\{ \varphi \leq j \right\} \cap T^{-j} \left\{ \varphi^{(n)} \leq k-j \right\} 
  \in \Fcal_j \vee \Fcal_k 
  \subset \Fcal_k,
 \end{equation*}
 and thus $\left\{ \varphi^{(n+1)} \leq k \right\}$ is $\Fcal_k$-measurable. 
\end{proof}

The Markov partitions $\alpha$ and $\alpha_\varphi$ can be refined by considering iterates of the transformation. For any stopping time $\varphi$ and any $n \geq 1$, we define:
\begin{equation*}
 \alpha_\varphi^{(n)} := \bigvee_{k=0}^{n-1} T_\varphi^{-k} \alpha_\varphi.
\end{equation*}
With these notations, $\alpha_\varphi^{(n)} = \alpha_{\varphi^{(n)}}$. We also denote by $\alpha^{(n)} := \alpha_{\mathbf{1}}^{(n)}$ the partition generated by cylinders of length $n$.

\subsection{Stopping times and regularizing effect of $\Lcal_\varphi$}
\label{subsec:BornitudeCompacite}

To each stopping time $\varphi$ is associated a transformation $T_\varphi$, and thus a transfer operator $\Lcal_\varphi$ defined by
\begin{equation*}
 \int_A \Lcal_\varphi (f) \cdot g \dd \mu 
 = \int_A f \cdot g \circ T_\varphi \dd \mu 
 \quad \forall \ f \in \Lbb^1 (A, \mu), \ \forall \ g \in \Lbb^\infty(A, \mu).
\end{equation*}
The transformation $T_\varphi$ in general does not preserve $\mu$, thus in general $\Lcal_\varphi (\mathbf{1}) \neq \mathbf{1}$. 
The operator $\Lcal_\varphi$ still acts quasi-compactly on the Banach space $\Bcal$ introduced in Definition~\ref{def:Lipschitz}. To make the statement more flexible, we introduce local Lipschitz seminorms:

\begin{definition}[Local Lipschitz seminorms]
 
 Let $(A, \alpha, d, \mu, T)$ be a Gibbs-Markov map and $f : A \to \C$. The function $D_\alpha f : A \to \R_+ \cup \{+ \infty\}$ is defined by
 \begin{equation*}
  D_\alpha f(x) 
  := \sup_{\substack{y, z \in a \\ y \neq z}} \frac{|f(y)-f(z)|}{d(y,z)}
  \quad \forall \ a \in \alpha, \ \forall \ x \in a.
 \end{equation*}
\end{definition}

For instance, $|f|_\Bcal = \norm{D_{\alpha^*} f}{\Lbb^\infty (A, \mu)}$, and $D_{\alpha^*} f \geq D_\alpha f$ since $\alpha^*$ is coarser than $\alpha$.

\begin{proposition}\quad
\label{prop:Lemme1113}
 
 Let $(A, \alpha, d, \mu, T)$ be a Gibbs-Markov map. There exists a constant $C$ such that, for any stopping time $\varphi$ and any $n \geq 1$, for any $\underline{a} \in \sigma(\alpha_\varphi^{(n)})$, for any $f : A \to \C$,
 \begin{equation}
 \label{eq:Lemme1113}
  \norm{\Lcal_\varphi^n (f \mathbf{1}_{\underline{a}})}{\Bcal} 
  \leq C \left[ \Lambda^{-\inf_{\underline{a}} \varphi^{(n)}} \int_{\underline{a}} D_{\alpha_\varphi^{(n)}} f \dd \mu + \int_{\underline{a}} |f| \dd \mu \right]
 \end{equation} 
\end{proposition}

\begin{proof}
 
 Let us assume that $\underline{a} = [a_0, \ldots, a_{k-1}]$ is a cylinder of length $k = \varphi^{(n)} (\underline{a})$. Then $\Lcal_\varphi^n (f \mathbf{1}_{\underline{a}}) = \Lcal^k (f \mathbf{1}_{\underline{a}})$. In addition, $\Lambda^{-k} \leq \Lambda^{-\inf_{\underline{a}} \varphi^{(n)}}$. With these observations in mind, the claim is given by~\cite[Lemme~1.1.13]{Gouezel:2008e}, with the refinement that the local Lipschitz seminorm can be taken on $\underline{a}$ instead of $a_0$.
 
 \smallskip
 
 In general, $\underline{a}$ is a countable union of such cylinders, and the claim follows by countable additivity of $\mu$. 
\end{proof}

Proposition~\ref{prop:Lemme1113} yields a uniform Lasota-Yorke inequality for $\Lcal_\varphi$ for all stopping times $\varphi$:

\begin{corollary}\quad
\label{cor:LasotaYorkeTempsArret}
 
 Let $(A, \alpha, d, \mu, T)$ be a Gibbs-Markov map. There exists a constant $C$ such that, for any stopping time $\varphi$ and any $n \geq 1$, for any $f \in \Bcal$,
 \begin{equation*}
  \norm{\Lcal_\varphi^n (f)}{\Bcal} 
  \leq C \left[ \Lambda^{-n} \norm{f}{\Bcal} + \norm{f}{\Lbb^1 (A, \mu)} \right].
 \end{equation*} 
\end{corollary}

\begin{proof}
 
 We apply Proposition~\ref{prop:Lemme1113} with $\underline{a} = A$. The partition $\alpha_\varphi^{(n)}$ is finer than $\alpha$, which is itself finer than $\alpha^*$, whence $ D_{\alpha_\varphi^{(n)}} f \leq D_{\alpha^*} f$ and
 \begin{equation*}
  \int_A D_{\alpha_\varphi^{(n)}} f \dd \mu 
  \leq \int_A D_{\alpha^*} f \dd \mu 
  \leq \norm{D_{\alpha^*} f}{\Lbb^\infty (A, \mu)} 
  = |f|_\Bcal 
  \leq \norm{f}{\Bcal}.
 \end{equation*}
 Finally, $\inf_A \varphi \geq 1$ so $\inf_A \varphi^{(n)} \geq n$.
\end{proof}

In particular, $\rho_{\ess} (\Lcal_\varphi \acts \Bcal) \leq \Lambda^{-1}$, and the usual arguments imply the existence of finitely many absolutely continuous invariant probability measures for $T_\varphi$. Another application of Proposition~\ref{prop:Lemme1113} yields

\begin{corollary}\quad
\label{cor:BorneUniformeTempsArret}
 
 Let $(A, \alpha, d, \mu, T)$ be a Gibbs-Markov map. There exists a constant $C$ such that, for any stopping time $\varphi$ and any $n \geq 1$, for any $\underline{a} \in \alpha_\varphi^{(n)}$ and $f \in \Bcal$,
 \begin{equation}
  \label{eq:BorneUniformeTempsArret}
  \norm{\Lcal_\varphi^n (f \mathbf{1}_{\underline{a}})}{\Bcal} 
  \leq C \mu(\underline{a}) \norm{f}{\Bcal} .
 \end{equation} 
\end{corollary}

\begin{proof}
 
 We apply Proposition~\ref{prop:Lemme1113}. The same argument as in the proof of Corollary~\ref{cor:LasotaYorkeTempsArret} yields 
 \begin{align*}
  \norm{\Lcal_\varphi^n (f \mathbf{1}_{\underline{a}})}{\Bcal} 
  & \leq C \left[ \Lambda^{-n} \int_{\underline{a}} D_{\alpha_\varphi^{(n)}} f \dd \mu + \int_{\underline{a}} |f| \dd \mu \right] \\
  & \leq C \left[ \Lambda^{-n} \mu(\underline{a}) \norm{D_{\alpha^*} f}{\Lbb^\infty (A, \mu)} + \mu(\underline{a}) \norm{f}{\Lbb^\infty (A, \mu)} \right] \\
  & \leq C \mu(\underline{a}) \norm{f}{\Bcal},
 \end{align*}
 with the same constant $C$.
\end{proof}

\subsection{Induced Gibbs-Markov maps}
\label{subsec:InducedGibbsMarkov}

Let $(A, \alpha, d, \mu, T)$ be a Gibbs-Markov map, and $F : A \to \Z^d$ be constant on the elements of $\alpha$. Let $([\Z^d],\widetilde{\mu}, \widetilde{T})$ be the associated $\Z^d$-extension, which we assume is recurrent. Let $\Sigma \subset \Z^d$ be non-empty and finite. We want to find a good Banach space on which the transfer operator associated with the induced system $([\Sigma], \mu_{[\Sigma]}, T_{[\Sigma]})$ acts. A natural candidate is the space $\Bcal_{[\Sigma]}$ defined by Equation~\eqref{eq:DefinitionBSigma}.

\smallskip

By adapting~\cite[Subsection~4.2.1]{PeneThomine:2019}, it is also possible to endow $([\Sigma], \mu_{[\Sigma]}, T_{[\Sigma]})$ with a Gibbs-Markov structure, which would then yield their own spaces of Lipschitz functions. However, these spaces are much too large, and their Lipschitz norm would depend on $[\Sigma]$. To keep some of our estimates uniform, we do not follow this line. Instead, we adapt the previous results on Gibbs-Markov maps to the specific case where the stopping time $\varphi$ comes from an induction process applied to $([\Z^d],\widetilde{\mu}, \widetilde{T})$.

\smallskip

Since $([\Z^d],\widetilde{\mu}, \widetilde{T})$ is recurrent and measure-preserving, the system $([\Sigma], \mu_{[\Sigma]}, T_{[\Sigma]})$ is well-defined and measure-preserving.

\begin{proposition}\quad
\label{prop:ActionInduiteQuasiCompacte} 

 Let $([\Z^d], \widetilde{\mu}, \widetilde{T})$ be a recurrent Markov $\Z^d$-extension of a Gibbs-Markov map $(A, \alpha, d, \mu, T)$. Let $\Sigma \subset \Z^d$ be non-empty and finite. Then:
 \begin{align*}
  \rho (\Lcal_{[\Sigma]} \acts \Bcal_{[\Sigma]}) & = 1, \\
  \rho_{\ess} (\Lcal_{[\Sigma]} \acts \Bcal_{[\Sigma]}) & \leq \Lambda^{-1},
 \end{align*}
 where $\Lambda >1$ is any expansion constant for $(A, \alpha, d, \mu, T)$.
\end{proposition}

\begin{proof}
 
 The extension $([\Z^d], \widetilde{\mu}, \widetilde{T})$ is Markov, so by definition the jump function $F$ is constant on elements of the partition $\alpha$. Hence, for all $p \in \Sigma$ and $n \geq 1$, we have an almost surely finite stopping time
\begin{equation*}
 \varphi_{p, \Sigma, n} (x) 
 := \inf \left\{k \geq 1 : \Card \{ 0 < \ell \leq k : S_\ell F (x) \in \{q-p: \ q \in \Sigma\} \} = n \right\}.
\end{equation*}
 This stopping time is designed so that, for all $x \in A$,
 \begin{equation*}
  T_{[\Sigma]}^n (x, p) 
  = (T_{\varphi_{p, \Sigma, n}} (x), p + S_{\varphi_{p, \Sigma, n} (x)} F (x)).
 \end{equation*}

  For all $p$, $q \in \Sigma$ and $n \geq 1$, let
 \begin{equation*}
  A_{p, q, n} 
  := \{x \in A : \ (x, p) \in T_{[\Sigma]}^{-n} ([q])\},
 \end{equation*}
 which is $\sigma(\alpha_{\varphi_{p, \Sigma, n}})$-measurable. Then, for all $p \in \Sigma$ and $n \geq 1$, for all $f \in \Bcal_{[\Sigma]}$,
\begin{align*}
 \Lcal_{[\Sigma]}^n (f) (\cdot, p) 
 = \sum_{q \in \Sigma} \Lcal_{\varphi_{q, \Sigma, n}} (f (\cdot, q) \mathbf{1}_{A_{q, p, n}}).
\end{align*}
By Proposition~\ref{prop:Lemme1113}, there exists a constant $C$ depending only on the initial Gibbs-Markov map such that
\begin{align*}
 \norm{\Lcal_{[\Sigma]}^n (f) (\cdot, p)}{\Bcal} 
 & \leq \sum_{q \in \Sigma} \norm{\Lcal_{\varphi_{q, \Sigma, n}} (f (\cdot, q) \mathbf{1}_{A_{q, p, n}})}{\Bcal} \\
 & \leq \sum_{q \in \Sigma} C \left[ \Lambda^{-\inf_{A_{q, p, n}} \varphi_{q, \Sigma, n}} \int_{A_{q, p, n}} D_{\alpha^*} f (\cdot, q) \dd \mu + \int_{A_{q, p, n}} |f| (\cdot, q) \dd \mu \right] \\
 & \leq \sum_{q \in \Sigma} C \left[ \Lambda^{-n} \mu(A_{q, p, n}) \norm{f}{\Bcal_{[\Sigma]}} +\int_{A_{q, p, n}} |f| (\cdot, q) \dd \mu \right] \\
 & = C \left[ \Lambda^{-n} \norm{f}{\Bcal_{[\Sigma]}} + \sum_{q \in \Sigma} \int_{A_{q, p, n}} |f| (\cdot, q) \dd \mu \right].
\end{align*}
Taking the supremum over $p \in \Sigma$ yields
 \begin{align*}
 \norm{\Lcal_{[\Sigma]}^n (f)}{\Bcal_{[\Sigma]}} 
 & \leq C \left[ \Lambda^{-n} \norm{f}{\Bcal_{[\Sigma]}} + \sum_{p, q \in \Sigma} \int_{A_{q, p, n}} |f| (\cdot, q) \dd \mu \right] \\
 & \leq C \left[ \Lambda^{-n} \norm{f}{\Bcal_{[\Sigma]}} + |\Sigma| \int_{[\Sigma]} |f| (\cdot, q) \dd \mu_{[\Sigma]} \right].
\end{align*}

 The bound on the essential spectral radius $\rho_{\ess} (\Lcal_{[\Sigma]} \acts \Bcal_{[\Sigma]})$ comes from Hennion's theorem~\cite{Hennion:1993}. The spectral radius is at least $1$ since $\Lcal_{[\Sigma]} (\mathbf{1}) = \mathbf{1}$, and at most $1$ because of the Lasota-Yorke inequality and the fact that the spectrum of $\Lcal_{[\Sigma]}$ oustide $\overline{B}(0, \Lambda^{-1})$ is made of eigenvalues of $\Lcal_{[\Sigma]}$.
\end{proof}

If in addition $([\Sigma], \mu_{[\Sigma]}, T_{[\Sigma]})$ is mixing, then the operator $\Lcal_{[\Sigma]} \acts \Bcal_{[\Sigma],0}$ decays at exponential speed; otherwise, the induced system $([\Sigma], \mu_{[\Sigma]}, T_{[\Sigma]})$ has a period, which we need to take into account. The proofs of the following theorems are classical, albeit lengthy; we refer the reader, for instance, to~\cite[Corollaire~1.1.18]{Gouezel:2008e} and~\cite[Corollaire~1.1.19]{Gouezel:2008e}.

\begin{corollary}\quad
\label{cor:DecroissanceCorrelationsInduit0} 

 Let $([\Z^d], \widetilde{\mu}, \widetilde{T})$ be an ergodic and recurrent Markov $\Z^d$-extension of a Gibbs-Markov map with data $(A, \alpha, d, \mu, T)$. Assume furthermore that $([0], \mu, T_{[0]})$ is mixing.
 
 \smallskip
 
 Then there exist constants $C >0$ and $\rho \in (0,1)$ such that, for all $f \in \Bcal_0$ and $n \geq 0$,
 \begin{equation}
 \label{eq:DecroissanceCorrelationsInduit0}
  \norm{\Lcal_{[0]}^n f}{\Bcal} 
  \leq C \rho^n \norm{f}{\Bcal}.
 \end{equation}
\end{corollary}

\begin{corollary}\quad
\label{cor:DecroissanceCorrelationsInduit0Ergodique} 

 Let $([\Z^d], \widetilde{\mu}, \widetilde{T})$ be an ergodic and recurrent Markov $\Z^d$-extension of a Gibbs-Markov map with data $(A, \alpha, d, \mu, T)$. Let $M \geq 1$ be such that $\Sp (\Lcal_{[0]} \acts \Bcal) \cap \Sbb_1$ is the group of $M$th roots of the unit.
 
 \smallskip
 
 Then there exists a partition $A = \bigsqcup_{\ell \in \Z_{/ M \Z}} A_\ell$ such that each $A_\ell$ is $\sigma(\alpha^*)$-measurable and $T_{[0]} (A_\ell) \subset A_{\ell+1}$ for all $\ell$. Setting $\Bcal_\ell := \{f \in \Bcal: \ \Supp (f) \subset A_\ell\}$, 
 \begin{equation*}
  \Bcal = \bigoplus_{{\ell \in \Z_{/ M \Z}}} \Bcal_\ell,
 \end{equation*}
and each restriction $\Lcal_{[0]} : \Bcal_\ell \to \Bcal_{\ell+1}$ is continuous. In addition, $T_{[0]}^M : A_\ell \to A_\ell$ is mixing for all $\ell$.
 
 \smallskip
 
 Finally, there exist constants $C >0$ and $\rho \in (0,1)$ such that, for all $f \in \Bcal$ and $n \geq 0$,
 \begin{equation}
 \label{eq:DecroissanceCorrelationsInduit0Ergodique}
  \norm{\frac{1}{M} \sum_{\ell=0}^{M-1} \Lcal_{[0]}^{n+\ell} f - \int_{[0]} f \dd \mu}{\Bcal} 
  \leq C \rho^n \norm{f}{\Bcal}.
 \end{equation}
\end{corollary}

We may also adapt Proposition~\ref{prop:Lemme1113} to get that $\Lcal_{[\Sigma]}$ maps a much larger Banach space into $\Bcal_{[\Sigma]}$. As in Section~\ref{sec:FourierTransform}, let $\alpha_{p, \Sigma}$ be the partition associated with the stopping time 
\begin{equation*}
 \varphi_{p, \Sigma} (x) 
 := \inf \{n \geq 1 : \ S_n F (x) \in \{q-p : \ q \in \Sigma\} \}.
\end{equation*}
Set $\alpha_{[\Sigma]} := \{a \times \{p\} : \ a \in \alpha_{p, \Sigma}, \ p \in \Sigma\}$. We define a seminorm
\begin{equation*}
 |f|_{\Lip^1 ([\Sigma])} 
 := \sum_{a \times \{p\} \in \alpha_{[\Sigma]}} \mu (a) |f (\cdot, p)|_{\Lip (a)},
\end{equation*}
and a norm $\norm{f}{\Lip^1 ([\Sigma])} := |f|_{\Lip^1 ([\Sigma])} + \norm{f}{\Lbb^1 ([\Sigma], \mu_{[\Sigma]})}$. Then we get:

\begin{proposition}\quad
\label{prop:ContinuiteL1LInfty}
 
  Let $([\Z^d], \widetilde{\mu}, \widetilde{T})$ be a recurrent Markov $\Z^d$-extension of a Gibbs-Markov map $(A, \alpha, d, \mu, T)$. Let $\Sigma \subset \Z^d$ be non-empty and finite. Then $\Lcal_{[\Sigma]} : \ \Lip^1 ([\Sigma]) \to \Bcal_{[\Sigma]}$ is continuous.
\end{proposition}

\begin{proof}
 
 We use the same notation as in the proof of Proposition~\ref{prop:ActionInduiteQuasiCompacte} with $n = 1$. For all $p$, $q \in \Sigma$, 
 for all $a \times \{p\} \in \alpha_{[\Sigma]}$, for all $f \in \Lip^1 ([\Sigma])$, by Proposition~\ref{prop:Lemme1113},
\begin{align*}
 \norm{\Lcal_{[\Sigma]} (f) (\cdot, p)}{\Bcal} 
 & \leq \sum_{q \in \Sigma} \sum_{\substack{a \times \{q\} \in \alpha_{[\Sigma]} \\ a \subset A_{q, p, 1}}} \norm{\Lcal_{\varphi_{q, \Sigma, 1}} (f (\cdot, q) \mathbf{1}_a )}{\Bcal} \\
 & \leq \sum_{q \in \Sigma} \sum_{\substack{a \times \{q\} \in \alpha_{[\Sigma]} \\ a \subset A_{q, p, 1}}} C \left[ \Lambda^{-1} \mu(a) |f|_{\Lip(a)}  + \int_a |f| (\cdot, q) \dd \mu \right] \\
 & = C \Lambda^{-1} |f|_{\Lip^1 ([\Sigma])} + C |\Sigma| \norm{f}{\Lbb^1 ([\Sigma], \mu_{[\Sigma]})} \\
 & \leq C |\Sigma| \norm{f}{\Lip^1 ([\Sigma])}. \qedhere
\end{align*}
\end{proof}

\end{appendix}

\end{document}